%% file: main.tex
\documentclass[preprint,3p,sort&compress]{elsarticle}
\journal{{\tt arXiv.org}}

\usepackage{graphicx}
\usepackage[dvipsnames]{xcolor}
\usepackage{amsmath,amsthm,amssymb}
\usepackage{bbm}
\usepackage{upgreek}
\usepackage{aligned-overset}
\usepackage[obeyFinal]{todonotes}
\usepackage{comment}
\usepackage{relsize}

\usepackage{subcaption}
\usepackage[linesnumbered,ruled]{algorithm2e}
\usepackage{booktabs}

\usepackage[hypertexnames=false]{hyperref}
\usepackage[sort&compress,capitalize,noabbrev]{cleveref}

\usepackage{pgfplots}
\pgfplotsset{compat=newest}

\usepackage{externaltikz}

\newcommand{\beginsupplement}{%
        \setcounter{table}{0}
        \renewcommand{\thetable}{S\arabic{table}}%
        \setcounter{figure}{0}
        \renewcommand{\thefigure}{S\arabic{figure}}%
     }

\title{A new entropy-variable-based discretization method for minimum entropy moment approximations of linear kinetic
equations\texorpdfstring{\footnote{{\textbf{Funding:}} The authors acknowledge funding by the Deutsche Forschungsgemeinschaft (DFG, German Research Foundation)
  under Germany's Excellence Strategy EXC 2044 –390685587, Mathematics Münster: Dynamics–Geometry–Structure.}{}
}}
\author[tl]{Tobias Leibner}
\author[mo]{Mario Ohlberger}
\address[tl]{Fachbereich Mathematik und Informatik, WWU M\"unster, Einsteinstrasse 62, 48149 M\"unster, {\tt tobias.leibner@uni-muenster.de}}
\address[mo]{Fachbereich Mathematik und Informatik, WWU M\"unster, Einsteinstrasse 62, 48149 M\"unster, {\tt mario.ohlberger@uni-muenster.de}}
\date{\today}

\input{Sections/definitions}

\begin{document}

\begin{abstract}
  In this contribution we derive and analyze a new numerical method for kinetic equations based on a variable transformation of the moment approximation.
  Classical minimum-entropy moment closures are a class of reduced models for kinetic equations that conserve many of the fundamental physical properties of solutions.
  However, their practical use is limited by their high computational cost, as an optimization problem has to be solved for every cell in the space-time grid.
  In addition, implementation of numerical solvers for these models is hampered by the fact that the optimization problems are only
  well-defined if the moment vectors stay within the realizable set. For the same reason, further reducing these models by,
  e.g., reduced-basis methods is not a simple task.
  Our new method overcomes these disadvantages of classical approaches.
  The transformation is performed on the semi-discretized level which makes them
  applicable to a wide range of kinetic schemes and replaces the nonlinear optimization problems by inversion of the positive-definite Hessian matrix.
  As a result, the
  new scheme gets rid of the realizability-related problems. Moreover, a discrete entropy law can be enforced by modifying the time stepping scheme.
  Our numerical experiments demonstrate that our new method is often several times faster than the standard optimization-based scheme.
\end{abstract}
\maketitle

\noindent


\newcommand{\tikzpath}{Images/}

\input{Sections/introduction}
\input{Sections/momentmodels}

\input{Sections/standardscheme.tex}

\input{Sections/newscheme.tex}

\input{Sections/implementation.tex}

\input{Sections/numericalexperiments.tex}

\input{Sections/outlook}

\section*{Acknowledgments}
The authors thank Florian Schneider and Hendrik Ranocha for fruitful discussions.

\bibliographystyle{siamplain}
\bibliography{jabreflib}

\appendix
\gdef\thesection{\Alph{section}} 
\makeatletter
\renewcommand\@seccntformat[1]{Appendix \csname the#1\endcsname.\hspace{0.5em}}
\makeatother
\input{Sections/appendix.tex}

\pagebreak
\gdef\thesection{S} 
\makeatletter
\renewcommand\@seccntformat[1]{\csname the#1\endcsname.\hspace{0.5em}}
\makeatother
\input{Sections/supplement.tex}

\end{document}

%% file: Sections/definitions.tex
\newcommand*{\eulere}{\mathrm{e}}

\newcommand*{\regepsilon}{\epsilon}

\newcommand*{\Cpp}{C\nolinebreak[4]\hspace{-.05em}\raisebox{.4ex}{\relsize{-3}{\textbf{++}}}\ }

\newcommand{\quand}{\quad \mbox{and} \quad}

\theoremstyle{plain}
\newtheorem{theorem}{Theorem}[section]
\newtheorem{definition}[theorem]{Definition}
\newtheorem{remark}[theorem]{Remark}
\newtheorem{example}[theorem]{Example}

\newtheorem{corollary}[theorem]{Corollary}
\numberwithin{equation}{section}            

\DeclareMathOperator{\argmin}{argmin}

\DeclareMathOperator{\rangeop}{Range}

\newcommand*{\derivative}{\mathrm{d}}
\newcommand*{\dirac}{\delta}

\newcommand*{\floor}[1]{\left\lfloor #1 \right\rfloor}
\newcommand*{\interior}[1]{\operatorname{int}\left(#1\right)}
\newcommand*{\landauO}{\mathcal{O}}

\newcommand*{\outernormal}{\Vn}

\newcommand*{\volume}[1]{\abs{#1}}

\newcommand*{\indicator}[1]{\mathbbm{1}_{#1}}

\providecommand\given{}
\newcommand\SetSymbol[1][]{%
  \nonscript\:#1\vert{}
  \allowbreak{}
  \nonscript\:
  \mathopen{}}
\DeclarePairedDelimiterX\Set[1]\{\}{%
\renewcommand\given{\,\SetSymbol[\delimsize]\,}
\,#1\,
}

\newcommand*{\R}{\mathbb{R}}

\newcommand*{\Rpos}{\R^{>0}}

\newcommand*{\Lp}[1]{L^{#1}}
\newcommand*{\Lppos}[1]{\Lp{#1}_{+}}
\newcommand*{\LpError}[2][ ]{E_{#1}^{#2}}
\newcommand*{\AnsatzSpace}[1][\basis]{\mathcal{A}_{#1}}
\newcommand*{\sphere}[1][2]{\mathcal{S}^{#1}}

\newcommand*{\abs}[1]{\left\vert #1 \right\vert}
\newcommand*{\norm}[2]{\left\lVert #1 \right\rVert_{#2}}

\newcommand*{\generalnorm}[1]{\left\lVert #1 \right\rVert}

\newcommand*{\massmatrix}{\MM}

\newcommand*{\eyematrix}{\MI}

\newcommand*{\gridindexset}{I_{\fvgrid}}

\newcommand*{\gridindex}{i}
\newcommand*{\gridindexalt}{j}

\newcommand*{\cellindex}{\gridindex}

\newcommand*{\gridsize}{n_{\spatialvar}}
\newcommand*{\gridsizeoned}{n_{\spatialvarcomp}}

\newcommand*{\timeindex}{\kappa}
\newcommand*{\dimension}{d}
\newcommand*{\dimindex}{k}
\newcommand*{\momentindex}{l} 
\newcommand*{\momentindexvar}{m} 
\newcommand*{\rkindex}{p} 
\newcommand*{\rkindexalt}{q} 
\newcommand*{\basisindex}{\momentindex} 
\newcommand*{\basisindexvar}{\momentindexvar} 
\newcommand*{\momentorder}{N} 
\newcommand*{\momentnumber}{n} 
\newcommand*{\fvvecsize}{\fedimension} 
\newcommand*{\fvvecindex}{l} 
\newcommand*{\altvariable}[1]{\tilde{#1}} 

\newcommand*{\numtimesteps}{n_t}

\newcommand*{\dx}[1][]{\partial_{x_{#1}}}

\newcommand*{\dz}{\partial_{\z}}

\newcommand*{\dt}{\partial_\timevar}

\newcommand*{\spatialGradient}{\nabla_{\spatialvar}}
\newcommand*{\multipliersGradient}{\nabla_{\multipliers}}

\newcommand*{\jacobianop}{\MD}

\newcommand*{\x}{\ensuremath{x}}
\newcommand*{\y}{\ensuremath{y}}
\newcommand*{\z}{\ensuremath{z}}
\newcommand*{\domain}{X}
\newcommand*{\lipschitzdomain}{A}

\newcommand*{\spatialvar}{\Vx}
\newcommand*{\spatialvaronedimension}{x}
\newcommand*{\spatialvarcomp}{x}
\newcommand*{\velocityvar}{\SC}
\newcommand*{\velocityvaronedimension}{\SCheight}
\newcommand*{\gridwidth}{\Delta x}

\newcommand*{\fvgrid}{\mathcal{G}}
\newcommand*{\fvgridentity}{T}
\newcommand*{\fvgridface}{S}
\newcommand*{\cell}[1]{I_{#1}}
\newcommand*{\ccell}[1]{\cell{{#1}}}
\newcommand*{\cellmean}[2][\cellindex]{\overline{#2}_{#1}}

\newcommand*{\timevar}{t}
\newcommand*{\timeint}{T} 
\newcommand*{\timestep}{\Delta\timevar}
\newcommand*{\timestepnonadaptive}{\timestep_{\mathrm{max}}} 
\newcommand*{\timesteppingorder}{r}

\newcommand*{\rktol}{\uptau}
\newcommand*{\rkstages}{s}
\newcommand*{\multipliersrkstage}[1]{\Vbeta^{#1}}
\newcommand*{\rkabstol}{\rktol_{\mathrm{abs}}}
\newcommand*{\rkreltol}{\rktol_{\mathrm{rel}}}
\newcommand*{\rkerr}{\text{err}}
\newcommand*{\rkb}{\check{b}}
\newcommand*{\rkbalt}{\widetilde{b}}
\newcommand*{\rka}{\check{a}}
\newcommand*{\tf}{t_{\text{end}}} 

\makeatletter
\@tfor\next:=abcdefghijklmnopqrstuvwxyzPQ\do{%
  \def\command@factory#1{%
    \expandafter\def\csname V#1\endcsname{\mathbf{#1}}
  }
  \expandafter\command@factory\next{}
}
\makeatother

\makeatletter
\def\greekvectors#1{%
\@for\next:=#1\do{%
\def\X##1;{%
\expandafter\def\csname V##1\endcsname{\boldsymbol{\csname##1\endcsname}}
}
\expandafter\X\next;
}
}
\makeatother

\greekvectors{alpha,beta,delta,epsilon,iota,gamma,lambda,theta,mu,nu,eta,chi,phi,psi,Gamma,Chi,varsigma,varphi}

\makeatletter
\@tfor\next:=ABCDEFGHIJKLMOPQRSTUVWXYZ\do{%
  \def\command@factory#1{%
    \expandafter\def\csname M#1\endcsname{\mathbf{#1}}
  }
  \expandafter\command@factory\next{}
}
\makeatother

\makeatletter
\def\greekmatrices#1{%
\@for\next:=#1\do{%
\def\X##1;{%
\expandafter\def\csname M##1\endcsname{\boldsymbol{\csname##1\endcsname}}
}
\expandafter\X\next;
}
}
\makeatother

\greekmatrices{Gamma,Chi,Omega,Sigma,Psi,Phi,Theta}

\newcommand*{\fedimension}{M}

\newcommand*{\Source}{\Vs}
\DeclareMathOperator{\coll}{\mathcal{C}}
\newcommand*{\collisionop}{\coll}
\newcommand*{\collision}[1]{\collisionop\left(#1\right)}

\newcommand*{\SC}{\ensuremath{\boldsymbol{\Omega}}} 
\newcommand*{\SCx}{\ensuremath{\Omega_\x}} 
\newcommand*{\SCy}{\ensuremath{\Omega_\y}} 
\newcommand*{\SCz}{\ensuremath{\Omega_\z}} 
\newcommand*{\SCcomp}[1][i]{\ensuremath{\Omega_{#1}}}
\newcommand*{\SCheight}{\ensuremath{\mu}}
\newcommand*{\SCangle}{\ensuremath{\varphi}}

\newcommand*{\Flux}{\Vf}

\newcommand*{\Fluxmatrix}{\MF}

\newcommand*{\numericalflux}{\Vg}

\newcommand*{\kineticflux}{\numericalflux^{kin}}

\newcommand*{\gridneighbors}[1]{\mathcal{N}(#1)} 

\newcommand*{\scattering}{\sigma_s}

\newcommand*{\absorption}{\sigma_a}

\newcommand*{\crosssection}{\sigma_t}
\newcommand*{\crosssectionmax}{\sigma_{t}^{\mathrm{max}}} 

\newcommand*{\source}{Q}

\newcommand*{\moments}[1][ ]{\Vu_{#1}} 
\newcommand*{\momentsfv}[1][ ]{\overline{\Vu}_{#1}} 
\newcommand*{\multipliersfv}[1][ ]{\overline{\Valpha}_{#1}} 
\newcommand*{\multipliersembedded}[1][ ]{\widetilde{\Valpha}_{#1}} 

\newcommand*{\momentscellmean}[1]{\cellmean[#1]{\moments}} 

\newcommand*{\isotropicmomentbasis}[1]{\moments[#1]^{\mathrm{iso}}}
\newcommand*{\momentcomp}[1]{u_{#1}}

\newcommand*{\normalizedmoments}[1][ ]{\Vphi_{#1}}

\newcommand*{\uupdateterm}{\Vu^{\uparrow}}
\newcommand*{\uupdatetermexplicit}{\Vu^{\uparrow,\timeindex}}

\newcommand*{\uupdatetermcomp}{u^{\uparrow}}
\newcommand*{\alphaupdateterm}{\Valpha^{\uparrow}}

\newcommand*{\multipliers}[1][ ]{\Valpha_{#1}}
\newcommand*{\numericalmultipliers}{\widetilde{\Valpha}}
\newcommand*{\multipliersone}[1][\basis]{\Valpha_{#1}^{\mathbbm{1}}}

\newcommand*{\multipliersiso}{\Valpha^{\mathrm{iso}}}
\newcommand*{\multiplierscomp}[1]{\alpha_{#1}}
\newcommand*{\multiplierscomptilde}[1]{\widetilde{\alpha}_{#1}}

\newcommand*{\density}{\rho} 

\newcommand*{\densityvacuum}{{\rho}_{\mathrm{vac}}}

\newcommand*{\distribution}[1][ ]{\psi_{#1}}
\newcommand*{\distributioninitial}[1][ ]{\psi_0}

\newcommand*{\distributiontzero}{\distribution[\timevar=0]}
\newcommand*{\distributionboundary}{\distribution[b]}
\newcommand*{\distributionvacuum}{\distribution[\mathrm{vac}]}

\newcommand*{\basis}{\Vb} 
\DeclareMathOperator{\basisop}{\basis}
\newcommand*{\basiscomp}[1][\basisindex]{b_{#1}}
\newcommand*{\fmbasis}[1][\momentorder]{\Vf_{#1}}

\newcommand*{\pmbasis}[1][\momentnumber]{\Vp_{#1}}

\newcommand*{\pmbasiscomp}[1][\basisindex]{p_{#1}}
\newcommand*{\hfbasis}[1][\momentnumber]{\Vh_{#1}}

\newcommand*{\hfbasiscomp}[1][\basisindex]{h_{#1}}
\newcommand*{\PN}[1][\momentorder]{\mathrm{P}_{#1}}

\newcommand*{\MN}[1][\momentorder]{\mathrm{M}_{#1}}

\newcommand*{\HFMN}[1][\momentnumber]{\mathrm{HFM}_{#1}}

\newcommand*{\PMMN}[1][\momentnumber]{\mathrm{PMM}_{#1}}

\newcommand*{\RD}[2]{\mathcal{R}_{#1}^{#2}}

\newcommand*{\RDpos}[1]{\mathcal{R}^{+}_{#1}}
\newcommand*{\RQ}[1]{\RD{#1}{\mathcal{Q}}}

\newcommand*{\hankelhalfind}{k}

\newcommand*{\entropydomain}{D}
\newcommand*{\angularDomain}{V}
\newcommand*{\angularDomainp}[1]{V^{+,#1}}
\newcommand*{\angularDomainm}[1]{V^{-,#1}}
\newcommand*{\angularQuadrature}{\mathcal{Q}}

\newcommand*{\ints}[1]{\left<#1\right>} 
\newcommand*{\intA}[2]{\left<#1\right>_{#2}}

\newcommand*{\intp}[1]{\intA{#1}{+}}
\newcommand*{\intm}[1]{\intA{#1}{-}}

\newcommand*{\intpn}[2]{\intA{#1}{+,#2}}
\newcommand*{\intmn}[2]{\intA{#1}{-,#2}}

\newcommand*{\intvar}[1]{\,\mathrm{d}#1}

\newcommand*{\ld}[1]{{#1}_*} 
\newcommand*{\entropy}{\eta} 
\newcommand*{\gesamtentropy}{\hat{\mathcal{H}}} 
\newcommand*{\gesamtalpha}{\hat{\Valpha}} 
\newcommand*{\gesamtbeta}{\hat{\Vbeta}} 
\newcommand*{\gesamtrhs}{\hat{\Valpha}^{\uparrow}} 
\newcommand*{\boseeinsteinentropy}{\entropy_{{}_{\mathrm{BE}}}} 

\newcommand*{\entropyFunctional}{\mathcal{H}} 
\newcommand*{\ansatz}[1][ ]{\hat{\psi}_{#1}}
\newcommand*{\ansatziso}{\hat{\psi}^{\mathrm{iso}}}

\newcommand*{\opttol}{\tau}
\newcommand*{\opttolmod}{\altvariable{\tau}}
\newcommand*{\opttoleps}{{\varepsilon_{\gamma}}}

\newcommand*{\velocityHessian}{\MJ}

\newcommand*{\optObjective}{p}
\newcommand*{\optGradient}{\Vq}
\newcommand*{\optHessian}{\MH}

\newcommand*{\optHessianInvEntry}{H^{-1}}
\newcommand*{\newtonDirection}{\Vd}
\newcommand*{\optNewtonDirection}{\newtonDirection}
\newcommand*{\regularizationParameter}[1][ ]{r_{#1}}

\newcommand*{\regularizedmoments}[1][ ]{\moments[#1]^{\regularizationParameter}} 

\newcommand*{\multiplierstilde}{\Vbeta}
\newcommand*{\isomatrix}[1][\basis]{\MG_{#1}^{\mathrm{iso}}}

\newcommand*{\newschemeregparam}{\epsilon}

\makeatletter
\DeclareFontFamily{U}{tipa}{}
\DeclareFontShape{U}{tipa}{m}{n}{<->tipa10}{}
\newcommand*{\arc@char}{{\usefont{U}{tipa}{m}{n}\symbol{62}}}%
\newcommand*{\arc}[1]{\mathpalette\arc@arc{#1}}
\newcommand*{\arc@arc}[2]{%
  \sbox0{\(\m@th#1#2\)}%
  \vbox{
    \hbox{\resizebox{\wd0}{\height}{\arc@char}}
    \nointerlineskip{}
    \box0
  }%
}
\makeatother

\newcommand*{\sphericaltriangle}{\arc{K}}

\newcommand*{\refinementnumber}{r} 
\newcommand*{\nvertex}{n_v}
\newcommand*{\nentity}{n_e}

\newcommand*{\generalpartition}{\mathcal{P}}

\definecolor{greenyellow}   {cmyk}{0.15, 0, 0.69, 0}
\definecolor{yellow}        {cmyk}{0, 0, 1, 0}
\definecolor{goldenrod}     {cmyk}{0, 0.10, 0.84, 0}
\definecolor{dandelion}     {cmyk}{0, 0.29, 0.84, 0}
\definecolor{apricot}       {cmyk}{0, 0.32, 0.52, 0}
\definecolor{peach}         {cmyk}{0, 0.50, 0.70, 0}
\definecolor{melon}         {cmyk}{0, 0.46, 0.50, 0}
\definecolor{yelloworange}  {cmyk}{0, 0.42, 1, 0}
\definecolor{orange}        {cmyk}{0, 0.61, 0.87, 0}
\definecolor{burntorange}   {cmyk}{0, 0.51, 1, 0}
\definecolor{bittersweet}   {cmyk}{0, 0.75, 1, 0.24}
\definecolor{redorange}     {cmyk}{0, 0.77, 0.87, 0}
\definecolor{mahogany}      {cmyk}{0, 0.85, 0.87, 0.35}
\definecolor{maroon}        {cmyk}{0, 0.87, 0.68, 0.32}
\definecolor{brickred}      {cmyk}{0, 0.89, 0.94, 0.28}
\definecolor{red}           {cmyk}{0, 1, 1, 0}
\definecolor{orangered}     {cmyk}{0, 1, 0.50, 0}
\definecolor{rubinered}     {cmyk}{0, 1, 0.13, 0}
\definecolor{wildstrawberry}{cmyk}{0, 0.96, 0.39, 0}
\definecolor{salmon}        {cmyk}{0, 0.53, 0.38, 0}
\definecolor{carnationpink} {cmyk}{0, 0.63, 0, 0}
\definecolor{magenta}       {cmyk}{0, 1, 0, 0}
\definecolor{violetred}     {cmyk}{0, 0.81, 0, 0}
\definecolor{rhodamine}     {cmyk}{0, 0.82, 0, 0}
\definecolor{mulberry}      {cmyk}{0.34, 0.90, 0, 0.02}
\definecolor{redviolet}     {cmyk}{0.07, 0.90, 0, 0.34}
\definecolor{fuchsia}       {cmyk}{0.47, 0.91, 0, 0.08}
\definecolor{lavender}      {cmyk}{0, 0.48, 0, 0}
\definecolor{thistle}       {cmyk}{0.12, 0.59, 0, 0}
\definecolor{orchid}        {cmyk}{0.32, 0.64, 0, 0}
\definecolor{darkorchid}    {cmyk}{0.40, 0.80, 0.20, 0}
\definecolor{purple}        {cmyk}{0.45, 0.86, 0, 0}
\definecolor{plum}          {cmyk}{0.50, 1, 0, 0}
\definecolor{violet}        {cmyk}{0.79, 0.88, 0, 0}
\definecolor{royalpurple}   {cmyk}{0.75, 0.90, 0, 0}
\definecolor{blueviolet}    {cmyk}{0.86, 0.91, 0, 0.04}
\definecolor{periwinkle}    {cmyk}{0.57, 0.55, 0, 0}
\definecolor{cadetblue}     {cmyk}{0.62, 0.57, 0.23, 0}
\definecolor{cornflowerblue}{cmyk}{0.65, 0.13, 0, 0}
\definecolor{midnightblue}  {cmyk}{0.98, 0.13, 0, 0.43}
\definecolor{navyblue}      {cmyk}{0.94, 0.54, 0, 0}
\definecolor{royalblue}     {cmyk}{1, 0.50, 0, 0}
\definecolor{blue}          {cmyk}{1, 1, 0, 0}
\definecolor{cerulean}      {cmyk}{0.94, 0.11, 0, 0}
\definecolor{cyan}          {cmyk}{1, 0, 0, 0}
\definecolor{processblue}   {cmyk}{0.96, 0, 0, 0}
\definecolor{skyblue}       {cmyk}{0.62, 0, 0.12, 0}
\definecolor{turquoise}     {cmyk}{0.85, 0, 0.20, 0}
\definecolor{tealblue}      {cmyk}{0.86, 0, 0.34, 0.02}
\definecolor{aquamarine}    {cmyk}{0.82, 0, 0.30, 0}
\definecolor{bluegreen}     {cmyk}{0.85, 0, 0.33, 0}
\definecolor{emerald}       {cmyk}{1, 0, 0.50, 0}
\definecolor{junglegreen}   {cmyk}{0.99, 0, 0.52, 0}
\definecolor{seagreen}      {cmyk}{0.69, 0, 0.50, 0}
\definecolor{green}         {cmyk}{1, 0, 1, 0}
\definecolor{forestgreen}   {cmyk}{0.91, 0, 0.88, 0.12}
\definecolor{pinegreen}     {cmyk}{0.92, 0, 0.59, 0.25}
\definecolor{limegreen}     {cmyk}{0.50, 0, 1, 0}
\definecolor{yellowgreen}   {cmyk}{0.44, 0, 0.74, 0}
\definecolor{springgreen}   {cmyk}{0.26, 0, 0.76, 0}
\definecolor{olivegreen}    {cmyk}{0.64, 0, 0.95, 0.40}
\definecolor{rawsienna}     {cmyk}{0, 0.72, 1, 0.45}
\definecolor{sepia}         {cmyk}{0, 0.83, 1, 0.70}
\definecolor{brown}         {cmyk}{0, 0.81, 1, 0.60}
\definecolor{tan}           {cmyk}{0.14, 0.42, 0.56, 0}
\definecolor{gray}          {cmyk}{0, 0, 0, 0.50}
\definecolor{black}         {cmyk}{0, 0, 0, 1}
\definecolor{white}         {cmyk}{0, 0, 0, 0}

\pgfplotscreateplotcyclelist{color parula}{%
  royalblue!70,every mark/.append style={solid,line width = 0pt,fill=royalblue!60!black},mark=ball\\%
  goldenrod!50!yelloworange,every mark/.append style={solid,fill=goldenrod!30!black},mark=*\\%
  green,every mark/.append style={black,fill=yellowgreen},mark=diamond*\\%
  bluegreen,every mark/.append style={fill=black},mark=triangle*\\%
  royalblue!70,densely dashed,every mark/.append style={solid,line width = 0pt,fill=royalblue!60!black},mark=ball\\%
  goldenrod!50!yelloworange,densely dashed,every mark/.append style={solid,black,fill=goldenrod},mark=diamond*\\%
  yellowgreen,densely dashed,every mark/.append style={solid,fill=yellowgreen!60!royalblue},mark=square*, mark size= 1pt\\%
  bluegreen,densely dashed,mark size = 1.5pt,every mark/.append style={solid,fill=white},mark=otimes*\\%
  royalblue,densely dashed,mark size = 1.5pt,every mark/.append style={solid,fill=royalblue!30!black},mark=pentagon*\\%
  goldenrod!40!yelloworange,every mark/.append style={solid,fill=goldenrod!30!black},mark=|\\%
}

\pgfplotscreateplotcyclelist{color parula pairwise}{%
royalblue!70,solid,every mark/.append style={solid,line width = 0pt,fill=royalblue!60!black},mark=ball\\%
royalblue!70,densely dashed,every mark/.append style={solid,line width = 0pt,fill=royalblue!60!black},mark=*\\%
goldenrod!50!yelloworange,solid,every mark/.append
style={solid,fill=goldenrod!30!black},mark=diamond*\\%
goldenrod!50!yelloworange,densely dashed, every mark/.append
style={solid,fill=goldenrod!30!black},mark=triangle*\\%
green,solid,every mark/.append style={solid, black,fill=yellowgreen},mark=square*\\%
green,densely dashed,every mark/.append style={solid, black, fill=yellowgreen},mark=pentagon*\\%
bluegreen,every mark/.append style={fill=black},mark=triangle*\\%
blueviolet,densely dashed,every mark/.append style={solid,fill=blueviolet},mark=square*, mark size= 1pt\\%
bluegreen,densely dashed,mark size = 1.5pt,every mark/.append style={solid,fill=white},mark=otimes*\\%
royalblue,densely dashed,mark size = 1.5pt,every mark/.append style={solid,fill=royalblue!30!black},mark=pentagon*\\%
goldenrod!40!yelloworange,every mark/.append style={solid,fill=goldenrod!30!black},mark=|\\%
}

\pgfplotsset{
  discard if not/.style 2 args={
      x filter/.code={
          \edef\tempa{\thisrow{#1}}
          \edef\tempb{#2}
          \ifx\tempa\tempb{}
          \else
            \def\pgfmathresult{inf}
          \fi
        }
    }
}

\usepgfplotslibrary{groupplots}
\newlength{\figurewidth}
\newlength{\figureheight}
\newlength{\figurewidthfortwo}
\newlength{\figureheightforthree}
\pgfplotsset{%
  compat=1.11,
  colormap={parula}{%
      rgb(0pt)=(0.2081,0.1663,0.5292); 
      rgb(1pt)=(0.208355,0.16778,0.532238); 
      rgb(2pt)=(0.208611,0.169261,0.535275); 
      rgb(3pt)=(0.208866,0.170741,0.538313); 
      rgb(4pt)=(0.209121,0.172222,0.54135); 
      rgb(5pt)=(0.209376,0.173702,0.544388); 
      rgb(6pt)=(0.209632,0.175183,0.547425); 
      rgb(7pt)=(0.209887,0.176663,0.550463); 
      rgb(8pt)=(0.210134,0.178144,0.553505); 
      rgb(9pt)=(0.210338,0.179624,0.556568); 
      rgb(10pt)=(0.210542,0.181105,0.559631); 
      rgb(11pt)=(0.210746,0.182585,0.562694); 
      rgb(12pt)=(0.210944,0.184066,0.565763); 
      rgb(13pt)=(0.211123,0.185546,0.568852); 
      rgb(14pt)=(0.211302,0.187027,0.57194); 
      rgb(15pt)=(0.21148,0.188507,0.575029); 
      rgb(16pt)=(0.211642,0.189996,0.578117); 
      rgb(17pt)=(0.21177,0.191502,0.581206); 
      rgb(18pt)=(0.211897,0.193008,0.584295); 
      rgb(19pt)=(0.212025,0.194514,0.587383); 
      rgb(20pt)=(0.212132,0.19602,0.590472); 
      rgb(21pt)=(0.212208,0.197526,0.59356); 
      rgb(22pt)=(0.212285,0.199032,0.596649); 
      rgb(23pt)=(0.212361,0.200538,0.599738); 
      rgb(24pt)=(0.212413,0.202044,0.602839); 
      rgb(25pt)=(0.212438,0.20355,0.605953); 
      rgb(26pt)=(0.212464,0.205056,0.609067); 
      rgb(27pt)=(0.212489,0.206562,0.612181); 
      rgb(28pt)=(0.212471,0.208083,0.61531); 
      rgb(29pt)=(0.21242,0.209614,0.61845); 
      rgb(30pt)=(0.212368,0.211146,0.621589); 
      rgb(31pt)=(0.212317,0.212677,0.624729); 
      rgb(32pt)=(0.212216,0.214209,0.627868); 
      rgb(33pt)=(0.212088,0.215741,0.631008); 
      rgb(34pt)=(0.211961,0.217272,0.634148); 
      rgb(35pt)=(0.211833,0.218804,0.637287); 
      rgb(36pt)=(0.211668,0.220354,0.640446); 
      rgb(37pt)=(0.211489,0.221911,0.643611); 
      rgb(38pt)=(0.21131,0.223468,0.646776); 
      rgb(39pt)=(0.211132,0.225025,0.649941); 
      rgb(40pt)=(0.210848,0.226603,0.653107); 
      rgb(41pt)=(0.210541,0.228186,0.656272); 
      rgb(42pt)=(0.210235,0.229768,0.659437); 
      rgb(43pt)=(0.209929,0.231351,0.662602); 
      rgb(44pt)=(0.209553,0.232934,0.665767); 
      rgb(45pt)=(0.20917,0.234516,0.668932); 
      rgb(46pt)=(0.208787,0.236099,0.672098); 
      rgb(47pt)=(0.208405,0.237681,0.675263); 
      rgb(48pt)=(0.20787,0.239289,0.678453); 
      rgb(49pt)=(0.207334,0.240897,0.681644); 
      rgb(50pt)=(0.206798,0.242505,0.684835); 
      rgb(51pt)=(0.206255,0.244114,0.688025); 
      rgb(52pt)=(0.205617,0.245722,0.691216); 
      rgb(53pt)=(0.204979,0.24733,0.694407); 
      rgb(54pt)=(0.204341,0.248938,0.697597); 
      rgb(55pt)=(0.203675,0.250554,0.700792); 
      rgb(56pt)=(0.202858,0.252213,0.704008); 
      rgb(57pt)=(0.202041,0.253872,0.707224); 
      rgb(58pt)=(0.201225,0.255531,0.710441); 
      rgb(59pt)=(0.200372,0.257184,0.713657); 
      rgb(60pt)=(0.199402,0.258818,0.716873); 
      rgb(61pt)=(0.198432,0.260452,0.720089); 
      rgb(62pt)=(0.197462,0.262085,0.723305); 
      rgb(63pt)=(0.196419,0.263735,0.726522); 
      rgb(64pt)=(0.195219,0.26542,0.729738); 
      rgb(65pt)=(0.19402,0.267105,0.732954); 
      rgb(66pt)=(0.19282,0.268789,0.73617); 
      rgb(67pt)=(0.191549,0.270474,0.739386); 
      rgb(68pt)=(0.19017,0.272159,0.742603); 
      rgb(69pt)=(0.188792,0.273843,0.745819); 
      rgb(70pt)=(0.187414,0.275528,0.749035); 
      rgb(71pt)=(0.1859,0.277237,0.752264); 
      rgb(72pt)=(0.184241,0.278973,0.755505); 
      rgb(73pt)=(0.182581,0.280709,0.758747); 
      rgb(74pt)=(0.180922,0.282444,0.761989); 
      rgb(75pt)=(0.179133,0.284209,0.765245); 
      rgb(76pt)=(0.177244,0.285996,0.768512); 
      rgb(77pt)=(0.175356,0.287783,0.77178); 
      rgb(78pt)=(0.173467,0.289569,0.775047); 
      rgb(79pt)=(0.171363,0.291406,0.778314); 
      rgb(80pt)=(0.169142,0.293269,0.781581); 
      rgb(81pt)=(0.166922,0.295132,0.784849); 
      rgb(82pt)=(0.164701,0.296996,0.788116); 
      rgb(83pt)=(0.162238,0.298934,0.791365); 
      rgb(84pt)=(0.159686,0.300899,0.794606); 
      rgb(85pt)=(0.157133,0.302865,0.797848); 
      rgb(86pt)=(0.15458,0.30483,0.80109); 
      rgb(87pt)=(0.151738,0.306858,0.804352); 
      rgb(88pt)=(0.148828,0.3089,0.80762); 
      rgb(89pt)=(0.145918,0.310942,0.810887); 
      rgb(90pt)=(0.143008,0.312984,0.814154); 
      rgb(91pt)=(0.139687,0.31514,0.81733); 
      rgb(92pt)=(0.136318,0.31731,0.820495); 
      rgb(93pt)=(0.132949,0.319479,0.82366); 
      rgb(94pt)=(0.129579,0.321649,0.826826); 
      rgb(95pt)=(0.125811,0.323918,0.829841); 
      rgb(96pt)=(0.122033,0.32619,0.832853); 
      rgb(97pt)=(0.118256,0.328462,0.835865); 
      rgb(98pt)=(0.114458,0.330737,0.838862); 
      rgb(99pt)=(0.110349,0.333059,0.841619); 
      rgb(100pt)=(0.106239,0.335382,0.844376); 
      rgb(101pt)=(0.102129,0.337705,0.847132); 
      rgb(102pt)=(0.0979874,0.340021,0.849835); 
      rgb(103pt)=(0.093648,0.342292,0.852209); 
      rgb(104pt)=(0.0893087,0.344564,0.854583); 
      rgb(105pt)=(0.0849694,0.346836,0.856957); 
      rgb(106pt)=(0.08063,0.349091,0.859234); 
      rgb(107pt)=(0.0762907,0.351286,0.861174); 
      rgb(108pt)=(0.0719514,0.353481,0.863114); 
      rgb(109pt)=(0.067612,0.355676,0.865053); 
      rgb(110pt)=(0.0633195,0.357817,0.866853); 
      rgb(111pt)=(0.0591333,0.359833,0.868333); 
      rgb(112pt)=(0.0549471,0.36185,0.869814); 
      rgb(113pt)=(0.050761,0.363866,0.871294); 
      rgb(114pt)=(0.0466838,0.365823,0.872626); 
      rgb(115pt)=(0.0427784,0.367687,0.873724); 
      rgb(116pt)=(0.038873,0.36955,0.874821); 
      rgb(117pt)=(0.0349676,0.371414,0.875919); 
      rgb(118pt)=(0.0315066,0.373217,0.876872); 
      rgb(119pt)=(0.0285456,0.374953,0.877664); 
      rgb(120pt)=(0.0255847,0.376688,0.878455); 
      rgb(121pt)=(0.0226237,0.378424,0.879246); 
      rgb(122pt)=(0.0202132,0.380061,0.879868); 
      rgb(123pt)=(0.0182477,0.381618,0.880353); 
      rgb(124pt)=(0.0162823,0.383175,0.880838); 
      rgb(125pt)=(0.0143168,0.384732,0.881323); 
      rgb(126pt)=(0.0127892,0.386241,0.881695); 
      rgb(127pt)=(0.0115129,0.387721,0.882001); 
      rgb(128pt)=(0.0102366,0.389202,0.882307); 
      rgb(129pt)=(0.00896036,0.390682,0.882614); 
      rgb(130pt)=(0.00812372,0.392089,0.88281); 
      rgb(131pt)=(0.00746006,0.393468,0.882963); 
      rgb(132pt)=(0.0067964,0.394846,0.883116); 
      rgb(133pt)=(0.00613273,0.396224,0.883269); 
      rgb(134pt)=(0.00581622,0.397562,0.88332); 
      rgb(135pt)=(0.00558649,0.398889,0.883346); 
      rgb(136pt)=(0.00535676,0.400217,0.883371); 
      rgb(137pt)=(0.00512703,0.401544,0.883397); 
      rgb(138pt)=(0.00516757,0.402804,0.883332); 
      rgb(139pt)=(0.00524414,0.404054,0.883256); 
      rgb(140pt)=(0.00532072,0.405305,0.883179); 
      rgb(141pt)=(0.0053973,0.406556,0.883103); 
      rgb(142pt)=(0.00572012,0.407757,0.882952); 
      rgb(143pt)=(0.00605195,0.408957,0.882799); 
      rgb(144pt)=(0.00638378,0.410157,0.882646); 
      rgb(145pt)=(0.00672643,0.411355,0.882489); 
      rgb(146pt)=(0.00728799,0.412529,0.882259); 
      rgb(147pt)=(0.00784955,0.413704,0.88203); 
      rgb(148pt)=(0.00841111,0.414878,0.8818); 
      rgb(149pt)=(0.00898919,0.416045,0.881564); 
      rgb(150pt)=(0.00967838,0.417168,0.881283); 
      rgb(151pt)=(0.0103676,0.418292,0.881002); 
      rgb(152pt)=(0.0110568,0.419415,0.880721); 
      rgb(153pt)=(0.011773,0.420532,0.880435); 
      rgb(154pt)=(0.0125898,0.42163,0.880129); 
      rgb(155pt)=(0.0134066,0.422728,0.879823); 
      rgb(156pt)=(0.0142234,0.423825,0.879516); 
      rgb(157pt)=(0.0150703,0.424915,0.879195); 
      rgb(158pt)=(0.0159892,0.425987,0.878838); 
      rgb(159pt)=(0.0169081,0.427059,0.87848); 
      rgb(160pt)=(0.017827,0.428132,0.878123); 
      rgb(161pt)=(0.0187748,0.429194,0.877746); 
      rgb(162pt)=(0.0197703,0.430241,0.877338); 
      rgb(163pt)=(0.0207658,0.431287,0.876929); 
      rgb(164pt)=(0.0217613,0.432334,0.876521); 
      rgb(165pt)=(0.0227802,0.43338,0.876113); 
      rgb(166pt)=(0.0238267,0.434427,0.875704); 
      rgb(167pt)=(0.0248733,0.435473,0.875296); 
      rgb(168pt)=(0.0259198,0.43652,0.874887); 
      rgb(169pt)=(0.0269802,0.437553,0.874451); 
      rgb(170pt)=(0.0280523,0.438574,0.873992); 
      rgb(171pt)=(0.0291243,0.439595,0.873532); 
      rgb(172pt)=(0.0301964,0.440616,0.873073); 
      rgb(173pt)=(0.0312844,0.441621,0.872614); 
      rgb(174pt)=(0.032382,0.442616,0.872154); 
      rgb(175pt)=(0.0334796,0.443612,0.871695); 
      rgb(176pt)=(0.0345772,0.444607,0.871235); 
      rgb(177pt)=(0.0357108,0.445603,0.870758); 
      rgb(178pt)=(0.0368595,0.446598,0.870273); 
      rgb(179pt)=(0.0380081,0.447594,0.869788); 
      rgb(180pt)=(0.0391568,0.448589,0.869303); 
      rgb(181pt)=(0.0402652,0.449565,0.868798); 
      rgb(182pt)=(0.0413628,0.450535,0.868287); 
      rgb(183pt)=(0.0424604,0.451505,0.867777); 
      rgb(184pt)=(0.043558,0.452474,0.867266); 
      rgb(185pt)=(0.0445889,0.453444,0.866756); 
      rgb(186pt)=(0.0456099,0.454414,0.866245); 
      rgb(187pt)=(0.0466309,0.455384,0.865735); 
      rgb(188pt)=(0.047652,0.456354,0.865224); 
      rgb(189pt)=(0.0486,0.457324,0.864714); 
      rgb(190pt)=(0.0495444,0.458294,0.864203); 
      rgb(191pt)=(0.0504889,0.459264,0.863692); 
      rgb(192pt)=(0.0514315,0.460234,0.863181); 
      rgb(193pt)=(0.0523249,0.461204,0.862645); 
      rgb(194pt)=(0.0532183,0.462174,0.862109); 
      rgb(195pt)=(0.0541117,0.463144,0.861573); 
      rgb(196pt)=(0.0549991,0.464111,0.861034); 
      rgb(197pt)=(0.0558414,0.465056,0.860472); 
      rgb(198pt)=(0.0566838,0.466,0.859911); 
      rgb(199pt)=(0.0575261,0.466944,0.859349); 
      rgb(200pt)=(0.0583532,0.467889,0.858793); 
      rgb(201pt)=(0.0591189,0.468833,0.858257); 
      rgb(202pt)=(0.0598847,0.469778,0.857721); 
      rgb(203pt)=(0.0606505,0.470722,0.857185); 
      rgb(204pt)=(0.0614018,0.471667,0.856641); 
      rgb(205pt)=(0.0621165,0.472611,0.85608); 
      rgb(206pt)=(0.0628312,0.473556,0.855518); 
      rgb(207pt)=(0.0635459,0.4745,0.854957); 
      rgb(208pt)=(0.064242,0.475444,0.854405); 
      rgb(209pt)=(0.0649057,0.476389,0.853868); 
      rgb(210pt)=(0.0655694,0.477333,0.853332); 
      rgb(211pt)=(0.066233,0.478278,0.852796); 
      rgb(212pt)=(0.0668625,0.479222,0.852249); 
      rgb(213pt)=(0.0674495,0.480167,0.851687); 
      rgb(214pt)=(0.0680366,0.481111,0.851126); 
      rgb(215pt)=(0.0686237,0.482056,0.850564); 
      rgb(216pt)=(0.0691838,0.483,0.850003); 
      rgb(217pt)=(0.0697198,0.483944,0.849441); 
      rgb(218pt)=(0.0702559,0.484889,0.84888); 
      rgb(219pt)=(0.0707919,0.485833,0.848318); 
      rgb(220pt)=(0.0712967,0.486778,0.847772); 
      rgb(221pt)=(0.0717817,0.487722,0.847236); 
      rgb(222pt)=(0.0722667,0.488667,0.8467); 
      rgb(223pt)=(0.0727517,0.489611,0.846164); 
      rgb(224pt)=(0.0732012,0.490573,0.845628); 
      rgb(225pt)=(0.0736351,0.491543,0.845092); 
      rgb(226pt)=(0.0740691,0.492513,0.844556); 
      rgb(227pt)=(0.074503,0.493483,0.84402); 
      rgb(228pt)=(0.0748973,0.494433,0.843484); 
      rgb(229pt)=(0.0752802,0.495378,0.842948); 
      rgb(230pt)=(0.0756631,0.496322,0.842412); 
      rgb(231pt)=(0.0760459,0.497267,0.841876); 
      rgb(232pt)=(0.0763631,0.498233,0.841362); 
      rgb(233pt)=(0.0766694,0.499203,0.840851); 
      rgb(234pt)=(0.0769757,0.500173,0.840341); 
      rgb(235pt)=(0.077282,0.501143,0.83983); 
      rgb(236pt)=(0.0775162,0.502137,0.83932); 
      rgb(237pt)=(0.0777459,0.503132,0.838809); 
      rgb(238pt)=(0.0779757,0.504128,0.838298); 
      rgb(239pt)=(0.0782042,0.505123,0.837789); 
      rgb(240pt)=(0.0783829,0.506093,0.837304); 
      rgb(241pt)=(0.0785616,0.507063,0.836819); 
      rgb(242pt)=(0.0787402,0.508033,0.836334); 
      rgb(243pt)=(0.0789135,0.509008,0.835851); 
      rgb(244pt)=(0.0790411,0.510029,0.835392); 
      rgb(245pt)=(0.0791688,0.51105,0.834932); 
      rgb(246pt)=(0.0792964,0.512071,0.834473); 
      rgb(247pt)=(0.0794048,0.513092,0.834018); 
      rgb(248pt)=(0.0794303,0.514113,0.833584); 
      rgb(249pt)=(0.0794559,0.515134,0.83315); 
      rgb(250pt)=(0.0794814,0.516155,0.832717); 
      rgb(251pt)=(0.0794862,0.517183,0.832289); 
      rgb(252pt)=(0.0794351,0.51823,0.831881); 
      rgb(253pt)=(0.0793841,0.519276,0.831473); 
      rgb(254pt)=(0.079333,0.520323,0.831064); 
      rgb(255pt)=(0.079255,0.521369,0.830665); 
      rgb(256pt)=(0.0791273,0.522416,0.830282); 
      rgb(257pt)=(0.0789997,0.523462,0.829899); 
      rgb(258pt)=(0.0788721,0.524509,0.829516); 
      rgb(259pt)=(0.0786889,0.525589,0.829156); 
      rgb(260pt)=(0.0784336,0.526712,0.828824); 
      rgb(261pt)=(0.0781784,0.527835,0.828492); 
      rgb(262pt)=(0.0779231,0.528958,0.82816); 
      rgb(263pt)=(0.077615,0.530081,0.827868); 
      rgb(264pt)=(0.0772577,0.531205,0.827613); 
      rgb(265pt)=(0.0769003,0.532328,0.827357); 
      rgb(266pt)=(0.0765429,0.533451,0.827102); 
      rgb(267pt)=(0.0761243,0.534589,0.826862); 
      rgb(268pt)=(0.0756649,0.535738,0.826632); 
      rgb(269pt)=(0.0752054,0.536886,0.826403); 
      rgb(270pt)=(0.0747459,0.538035,0.826173); 
      rgb(271pt)=(0.0742168,0.539219,0.825961); 
      rgb(272pt)=(0.0736553,0.540418,0.825756); 
      rgb(273pt)=(0.0730937,0.541618,0.825552); 
      rgb(274pt)=(0.0725321,0.542818,0.825348); 
      rgb(275pt)=(0.0718925,0.544037,0.825183); 
      rgb(276pt)=(0.0712288,0.545262,0.82503); 
      rgb(277pt)=(0.0705652,0.546487,0.824877); 
      rgb(278pt)=(0.0699015,0.547713,0.824723); 
      rgb(279pt)=(0.0691514,0.548938,0.824614); 
      rgb(280pt)=(0.0683856,0.550163,0.824511); 
      rgb(281pt)=(0.0676198,0.551388,0.824409); 
      rgb(282pt)=(0.0668541,0.552614,0.824307); 
      rgb(283pt)=(0.0660408,0.553886,0.824205); 
      rgb(284pt)=(0.065224,0.555162,0.824103); 
      rgb(285pt)=(0.0644072,0.556439,0.824001); 
      rgb(286pt)=(0.0635892,0.557715,0.823899); 
      rgb(287pt)=(0.0626703,0.558991,0.823848); 
      rgb(288pt)=(0.0617514,0.560268,0.823797); 
      rgb(289pt)=(0.0608324,0.561544,0.823746); 
      rgb(290pt)=(0.0599087,0.56282,0.823693); 
      rgb(291pt)=(0.0589387,0.564096,0.823616); 
      rgb(292pt)=(0.0579688,0.565373,0.82354); 
      rgb(293pt)=(0.0569988,0.566649,0.823463); 
      rgb(294pt)=(0.0560243,0.567925,0.823386); 
      rgb(295pt)=(0.0550288,0.569202,0.82331); 
      rgb(296pt)=(0.0540333,0.570478,0.823233); 
      rgb(297pt)=(0.0530378,0.571754,0.823157); 
      rgb(298pt)=(0.0520423,0.57303,0.82308); 
      rgb(299pt)=(0.0510468,0.574307,0.823004); 
      rgb(300pt)=(0.0500514,0.575583,0.822927); 
      rgb(301pt)=(0.0490559,0.576859,0.82285); 
      rgb(302pt)=(0.0480604,0.578127,0.822756); 
      rgb(303pt)=(0.0470649,0.579377,0.822629); 
      rgb(304pt)=(0.0460694,0.580628,0.822501); 
      rgb(305pt)=(0.0450739,0.581879,0.822374); 
      rgb(306pt)=(0.0441,0.583119,0.822235); 
      rgb(307pt)=(0.0431556,0.584344,0.822082); 
      rgb(308pt)=(0.0422111,0.585569,0.821929); 
      rgb(309pt)=(0.0412667,0.586795,0.821776); 
      rgb(310pt)=(0.0403351,0.58802,0.821597); 
      rgb(311pt)=(0.0394162,0.589245,0.821392); 
      rgb(312pt)=(0.0384973,0.59047,0.821188); 
      rgb(313pt)=(0.0375784,0.591695,0.820984); 
      rgb(314pt)=(0.0367495,0.592891,0.820735); 
      rgb(315pt)=(0.0359838,0.594065,0.820454); 
      rgb(316pt)=(0.035218,0.595239,0.820173); 
      rgb(317pt)=(0.0344523,0.596413,0.819892); 
      rgb(318pt)=(0.0337721,0.597553,0.819595); 
      rgb(319pt)=(0.0331339,0.598676,0.819288); 
      rgb(320pt)=(0.0324958,0.599799,0.818982); 
      rgb(321pt)=(0.0318577,0.600923,0.818676); 
      rgb(322pt)=(0.0312964,0.602026,0.818312); 
      rgb(323pt)=(0.0307604,0.603124,0.817929); 
      rgb(324pt)=(0.0302243,0.604222,0.817546); 
      rgb(325pt)=(0.0296883,0.605319,0.817163); 
      rgb(326pt)=(0.0292375,0.606395,0.816738); 
      rgb(327pt)=(0.0288036,0.607468,0.816304); 
      rgb(328pt)=(0.0283697,0.60854,0.81587); 
      rgb(329pt)=(0.0279357,0.609612,0.815436); 
      rgb(330pt)=(0.0275721,0.610637,0.814955); 
      rgb(331pt)=(0.0272147,0.611658,0.81447); 
      rgb(332pt)=(0.0268574,0.612679,0.813985); 
      rgb(333pt)=(0.0265,0.6137,0.8135); 
      rgb(334pt)=(0.0262447,0.614695,0.812964); 
      rgb(335pt)=(0.0259895,0.615691,0.812428); 
      rgb(336pt)=(0.0257342,0.616686,0.811892); 
      rgb(337pt)=(0.0254853,0.61768,0.811352); 
      rgb(338pt)=(0.0253066,0.61865,0.810765); 
      rgb(339pt)=(0.0251279,0.61962,0.810177); 
      rgb(340pt)=(0.0249492,0.62059,0.80959); 
      rgb(341pt)=(0.024779,0.621551,0.808995); 
      rgb(342pt)=(0.0246514,0.62247,0.808357); 
      rgb(343pt)=(0.0245237,0.623389,0.807719); 
      rgb(344pt)=(0.0243961,0.624308,0.80708); 
      rgb(345pt)=(0.0242748,0.625221,0.80643); 
      rgb(346pt)=(0.0241727,0.626114,0.805741); 
      rgb(347pt)=(0.0240706,0.627008,0.805051); 
      rgb(348pt)=(0.0239685,0.627901,0.804362); 
      rgb(349pt)=(0.0238832,0.628786,0.803656); 
      rgb(350pt)=(0.0238321,0.629654,0.802916); 
      rgb(351pt)=(0.0237811,0.630522,0.802176); 
      rgb(352pt)=(0.02373,0.631389,0.801435); 
      rgb(353pt)=(0.023679,0.632247,0.800685); 
      rgb(354pt)=(0.0236279,0.633089,0.799919); 
      rgb(355pt)=(0.0235769,0.633932,0.799153); 
      rgb(356pt)=(0.0235258,0.634774,0.798387); 
      rgb(357pt)=(0.0234748,0.635604,0.797596); 
      rgb(358pt)=(0.0234237,0.63642,0.79678); 
      rgb(359pt)=(0.0233727,0.637237,0.795963); 
      rgb(360pt)=(0.0233216,0.638054,0.795146); 
      rgb(361pt)=(0.0232706,0.638856,0.794329); 
      rgb(362pt)=(0.0232195,0.639647,0.793512); 
      rgb(363pt)=(0.0231685,0.640439,0.792695); 
      rgb(364pt)=(0.0231174,0.64123,0.791879); 
      rgb(365pt)=(0.0230832,0.642005,0.791011); 
      rgb(366pt)=(0.0230577,0.64277,0.790118); 
      rgb(367pt)=(0.0230321,0.643536,0.789225); 
      rgb(368pt)=(0.0230066,0.644302,0.788331); 
      rgb(369pt)=(0.0229811,0.645049,0.787438); 
      rgb(370pt)=(0.0229556,0.645789,0.786544); 
      rgb(371pt)=(0.02293,0.646529,0.785651); 
      rgb(372pt)=(0.0229045,0.647269,0.784758); 
      rgb(373pt)=(0.022858,0.64801,0.783843); 
      rgb(374pt)=(0.0228069,0.64875,0.782924); 
      rgb(375pt)=(0.0227559,0.64949,0.782005); 
      rgb(376pt)=(0.0227048,0.65023,0.781086); 
      rgb(377pt)=(0.0227,0.650947,0.780144); 
      rgb(378pt)=(0.0227,0.651662,0.7792); 
      rgb(379pt)=(0.0227,0.652377,0.778256); 
      rgb(380pt)=(0.0227,0.653092,0.777311); 
      rgb(381pt)=(0.0228261,0.653781,0.776341); 
      rgb(382pt)=(0.0229538,0.65447,0.775371); 
      rgb(383pt)=(0.0230814,0.655159,0.774402); 
      rgb(384pt)=(0.0232108,0.655849,0.77343); 
      rgb(385pt)=(0.023364,0.656538,0.772434); 
      rgb(386pt)=(0.0235171,0.657227,0.771439); 
      rgb(387pt)=(0.0236703,0.657916,0.770443); 
      rgb(388pt)=(0.0238312,0.658602,0.769444); 
      rgb(389pt)=(0.0240354,0.659265,0.768423); 
      rgb(390pt)=(0.0242396,0.659929,0.767402); 
      rgb(391pt)=(0.0244438,0.660592,0.766381); 
      rgb(392pt)=(0.0247021,0.661256,0.765354); 
      rgb(393pt)=(0.025136,0.66192,0.764307); 
      rgb(394pt)=(0.02557,0.662583,0.763261); 
      rgb(395pt)=(0.0260039,0.663247,0.762214); 
      rgb(396pt)=(0.0264541,0.663911,0.761168); 
      rgb(397pt)=(0.026939,0.664574,0.760121); 
      rgb(398pt)=(0.027424,0.665238,0.759074); 
      rgb(399pt)=(0.027909,0.665902,0.758028); 
      rgb(400pt)=(0.028445,0.666555,0.756971); 
      rgb(401pt)=(0.0290577,0.667193,0.755899); 
      rgb(402pt)=(0.0296703,0.667832,0.754827); 
      rgb(403pt)=(0.0302829,0.66847,0.753755); 
      rgb(404pt)=(0.030994,0.669095,0.752683); 
      rgb(405pt)=(0.0318108,0.669708,0.751611); 
      rgb(406pt)=(0.0326276,0.670321,0.750539); 
      rgb(407pt)=(0.0334444,0.670933,0.749467); 
      rgb(408pt)=(0.0343045,0.67156,0.748366); 
      rgb(409pt)=(0.0351979,0.672198,0.747243); 
      rgb(410pt)=(0.0360913,0.672837,0.74612); 
      rgb(411pt)=(0.0369847,0.673475,0.744996); 
      rgb(412pt)=(0.0380432,0.674096,0.743873); 
      rgb(413pt)=(0.0391919,0.674709,0.74275); 
      rgb(414pt)=(0.0403405,0.675322,0.741627); 
      rgb(415pt)=(0.0414892,0.675934,0.740504); 
      rgb(416pt)=(0.0427123,0.676528,0.739381); 
      rgb(417pt)=(0.0439631,0.677115,0.738258); 
      rgb(418pt)=(0.0452138,0.677702,0.737135); 
      rgb(419pt)=(0.0464646,0.678289,0.736011); 
      rgb(420pt)=(0.0477153,0.678897,0.734868); 
      rgb(421pt)=(0.0489661,0.67951,0.733719); 
      rgb(422pt)=(0.0502168,0.680123,0.73257); 
      rgb(423pt)=(0.0514676,0.680735,0.731422); 
      rgb(424pt)=(0.0529237,0.681325,0.73025); 
      rgb(425pt)=(0.0544042,0.681912,0.729076); 
      rgb(426pt)=(0.0558847,0.682499,0.727902); 
      rgb(427pt)=(0.0573652,0.683086,0.726728); 
      rgb(428pt)=(0.0587709,0.683673,0.725553); 
      rgb(429pt)=(0.0601748,0.68426,0.724379); 
      rgb(430pt)=(0.0615787,0.684847,0.723205); 
      rgb(431pt)=(0.0629946,0.685435,0.722028); 
      rgb(432pt)=(0.0646027,0.686022,0.720803); 
      rgb(433pt)=(0.0662108,0.686609,0.719577); 
      rgb(434pt)=(0.0678189,0.687196,0.718352); 
      rgb(435pt)=(0.069427,0.687779,0.717131); 
      rgb(436pt)=(0.0710351,0.688341,0.715931); 
      rgb(437pt)=(0.0726432,0.688902,0.714731); 
      rgb(438pt)=(0.0742514,0.689464,0.713532); 
      rgb(439pt)=(0.0758709,0.690026,0.712326); 
      rgb(440pt)=(0.07753,0.690587,0.711101); 
      rgb(441pt)=(0.0791892,0.691149,0.709876); 
      rgb(442pt)=(0.0808483,0.69171,0.70865); 
      rgb(443pt)=(0.0825387,0.692272,0.707417); 
      rgb(444pt)=(0.0843,0.692833,0.706167); 
      rgb(445pt)=(0.0860613,0.693395,0.704916); 
      rgb(446pt)=(0.0878225,0.693956,0.703665); 
      rgb(447pt)=(0.089564,0.694518,0.702405); 
      rgb(448pt)=(0.0912742,0.69508,0.701128); 
      rgb(449pt)=(0.0929844,0.695641,0.699852); 
      rgb(450pt)=(0.0946946,0.696203,0.698576); 
      rgb(451pt)=(0.0965009,0.696752,0.697299); 
      rgb(452pt)=(0.0984153,0.697288,0.696023); 
      rgb(453pt)=(0.10033,0.697824,0.694747); 
      rgb(454pt)=(0.102244,0.69836,0.693471); 
      rgb(455pt)=(0.10413,0.698896,0.69218); 
      rgb(456pt)=(0.105994,0.699432,0.690878); 
      rgb(457pt)=(0.107857,0.699968,0.689577); 
      rgb(458pt)=(0.10972,0.700505,0.688275); 
      rgb(459pt)=(0.111632,0.701041,0.686973); 
      rgb(460pt)=(0.113572,0.701577,0.685671); 
      rgb(461pt)=(0.115512,0.702113,0.684369); 
      rgb(462pt)=(0.117452,0.702649,0.683068); 
      rgb(463pt)=(0.119429,0.703185,0.681747); 
      rgb(464pt)=(0.12142,0.703721,0.68042); 
      rgb(465pt)=(0.123411,0.704257,0.679093); 
      rgb(466pt)=(0.125402,0.704793,0.677765); 
      rgb(467pt)=(0.127372,0.705308,0.676438); 
      rgb(468pt)=(0.129338,0.705819,0.675111); 
      rgb(469pt)=(0.131303,0.706329,0.673783); 
      rgb(470pt)=(0.133269,0.70684,0.672456); 
      rgb(471pt)=(0.135369,0.70735,0.671084); 
      rgb(472pt)=(0.137488,0.707861,0.669705); 
      rgb(473pt)=(0.139607,0.708371,0.668327); 
      rgb(474pt)=(0.141725,0.708882,0.666949); 
      rgb(475pt)=(0.143795,0.709392,0.665595); 
      rgb(476pt)=(0.145862,0.709903,0.664242); 
      rgb(477pt)=(0.14793,0.710414,0.662889); 
      rgb(478pt)=(0.150003,0.710924,0.661534); 
      rgb(479pt)=(0.152198,0.711435,0.66013); 
      rgb(480pt)=(0.154394,0.711945,0.658726); 
      rgb(481pt)=(0.156589,0.712456,0.657322); 
      rgb(482pt)=(0.158784,0.712963,0.655922); 
      rgb(483pt)=(0.160979,0.713448,0.654543); 
      rgb(484pt)=(0.163174,0.713933,0.653165); 
      rgb(485pt)=(0.16537,0.714418,0.651786); 
      rgb(486pt)=(0.16757,0.714908,0.650397); 
      rgb(487pt)=(0.169791,0.715419,0.648968); 
      rgb(488pt)=(0.172012,0.715929,0.647538); 
      rgb(489pt)=(0.174232,0.71644,0.646109); 
      rgb(490pt)=(0.176483,0.716935,0.64468); 
      rgb(491pt)=(0.178806,0.717395,0.64325); 
      rgb(492pt)=(0.181129,0.717854,0.641821); 
      rgb(493pt)=(0.183452,0.718314,0.640391); 
      rgb(494pt)=(0.185755,0.718783,0.638952); 
      rgb(495pt)=(0.188027,0.719268,0.637497); 
      rgb(496pt)=(0.190299,0.719753,0.636042); 
      rgb(497pt)=(0.192571,0.720238,0.634587); 
      rgb(498pt)=(0.194913,0.720711,0.633132); 
      rgb(499pt)=(0.197338,0.72117,0.631677); 
      rgb(500pt)=(0.199762,0.72163,0.630223); 
      rgb(501pt)=(0.202187,0.722089,0.628768); 
      rgb(502pt)=(0.204612,0.722549,0.627299); 
      rgb(503pt)=(0.207037,0.723008,0.625818); 
      rgb(504pt)=(0.209462,0.723468,0.624338); 
      rgb(505pt)=(0.211887,0.723927,0.622857); 
      rgb(506pt)=(0.214328,0.724386,0.621377); 
      rgb(507pt)=(0.216778,0.724846,0.619896); 
      rgb(508pt)=(0.219229,0.725305,0.618416); 
      rgb(509pt)=(0.221679,0.725765,0.616935); 
      rgb(510pt)=(0.224202,0.726188,0.615455); 
      rgb(511pt)=(0.226754,0.726597,0.613974); 
      rgb(512pt)=(0.229307,0.727005,0.612494); 
      rgb(513pt)=(0.231859,0.727414,0.611014); 
      rgb(514pt)=(0.234392,0.727842,0.609513); 
      rgb(515pt)=(0.236919,0.728276,0.608007); 
      rgb(516pt)=(0.239446,0.72871,0.606501); 
      rgb(517pt)=(0.241973,0.729144,0.604995); 
      rgb(518pt)=(0.244611,0.729556,0.603467); 
      rgb(519pt)=(0.247266,0.729964,0.601935); 
      rgb(520pt)=(0.24992,0.730372,0.600404); 
      rgb(521pt)=(0.252575,0.730781,0.598872); 
      rgb(522pt)=(0.25523,0.731189,0.597365); 
      rgb(523pt)=(0.257884,0.731598,0.595859); 
      rgb(524pt)=(0.260539,0.732006,0.594353); 
      rgb(525pt)=(0.263194,0.732414,0.592846); 
      rgb(526pt)=(0.265848,0.732796,0.591314); 
      rgb(527pt)=(0.268503,0.733179,0.589783); 
      rgb(528pt)=(0.271158,0.733562,0.588251); 
      rgb(529pt)=(0.27383,0.733945,0.58672); 
      rgb(530pt)=(0.276638,0.734328,0.585188); 
      rgb(531pt)=(0.279446,0.734711,0.583657); 
      rgb(532pt)=(0.282254,0.735094,0.582125); 
      rgb(533pt)=(0.285051,0.735471,0.580594); 
      rgb(534pt)=(0.287808,0.735829,0.579062); 
      rgb(535pt)=(0.290565,0.736186,0.577531); 
      rgb(536pt)=(0.293322,0.736544,0.575999); 
      rgb(537pt)=(0.2961,0.736894,0.574468); 
      rgb(538pt)=(0.298933,0.737226,0.572936); 
      rgb(539pt)=(0.301767,0.737557,0.571405); 
      rgb(540pt)=(0.3046,0.737889,0.569873); 
      rgb(541pt)=(0.307452,0.738221,0.568351); 
      rgb(542pt)=(0.310336,0.738553,0.566845); 
      rgb(543pt)=(0.313221,0.738885,0.565339); 
      rgb(544pt)=(0.316105,0.739217,0.563833); 
      rgb(545pt)=(0.318978,0.739537,0.562315); 
      rgb(546pt)=(0.321837,0.739843,0.560784); 
      rgb(547pt)=(0.324696,0.74015,0.559252); 
      rgb(548pt)=(0.327555,0.740456,0.557721); 
      rgb(549pt)=(0.330468,0.740749,0.556216); 
      rgb(550pt)=(0.333429,0.741029,0.554736); 
      rgb(551pt)=(0.336389,0.74131,0.553255); 
      rgb(552pt)=(0.33935,0.741591,0.551775); 
      rgb(553pt)=(0.342296,0.741872,0.550279); 
      rgb(554pt)=(0.345231,0.742153,0.548773); 
      rgb(555pt)=(0.348167,0.742433,0.547267); 
      rgb(556pt)=(0.351102,0.742714,0.545761); 
      rgb(557pt)=(0.354038,0.742977,0.54429); 
      rgb(558pt)=(0.356973,0.743232,0.542835); 
      rgb(559pt)=(0.359908,0.743488,0.54138); 
      rgb(560pt)=(0.362844,0.743743,0.539925); 
      rgb(561pt)=(0.365839,0.743959,0.53847); 
      rgb(562pt)=(0.368851,0.744163,0.537015); 
      rgb(563pt)=(0.371863,0.744367,0.53556); 
      rgb(564pt)=(0.374875,0.744571,0.534105); 
      rgb(565pt)=(0.377843,0.744775,0.532672); 
      rgb(566pt)=(0.380804,0.74498,0.531243); 
      rgb(567pt)=(0.383765,0.745184,0.529814); 
      rgb(568pt)=(0.386726,0.745388,0.528384); 
      rgb(569pt)=(0.389711,0.745568,0.527003); 
      rgb(570pt)=(0.392697,0.745747,0.525624); 
      rgb(571pt)=(0.395684,0.745926,0.524246); 
      rgb(572pt)=(0.39867,0.746104,0.522868); 
      rgb(573pt)=(0.401657,0.746257,0.521489); 
      rgb(574pt)=(0.404643,0.74641,0.520111); 
      rgb(575pt)=(0.40763,0.746563,0.518732); 
      rgb(576pt)=(0.410611,0.746716,0.517359); 
      rgb(577pt)=(0.413546,0.746869,0.516032); 
      rgb(578pt)=(0.416482,0.747023,0.514705); 
      rgb(579pt)=(0.419417,0.747176,0.513377); 
      rgb(580pt)=(0.422357,0.747319,0.512055); 
      rgb(581pt)=(0.425318,0.747421,0.510753); 
      rgb(582pt)=(0.428279,0.747523,0.509451); 
      rgb(583pt)=(0.43124,0.747626,0.50815); 
      rgb(584pt)=(0.43418,0.747735,0.506848); 
      rgb(585pt)=(0.437065,0.747862,0.505546); 
      rgb(586pt)=(0.439949,0.74799,0.504244); 
      rgb(587pt)=(0.442834,0.748117,0.502942); 
      rgb(588pt)=(0.445727,0.748227,0.501659); 
      rgb(589pt)=(0.448637,0.748304,0.500408); 
      rgb(590pt)=(0.451547,0.74838,0.499157); 
      rgb(591pt)=(0.454457,0.748457,0.497906); 
      rgb(592pt)=(0.457333,0.748522,0.496667); 
      rgb(593pt)=(0.460167,0.748573,0.495441); 
      rgb(594pt)=(0.463,0.748624,0.494216); 
      rgb(595pt)=(0.465833,0.748675,0.492991); 
      rgb(596pt)=(0.468667,0.748726,0.491779); 
      rgb(597pt)=(0.4715,0.748777,0.490579); 
      rgb(598pt)=(0.474333,0.748829,0.48938); 
      rgb(599pt)=(0.477167,0.74888,0.48818); 
      rgb(600pt)=(0.479969,0.748931,0.486995); 
      rgb(601pt)=(0.482752,0.748982,0.485821); 
      rgb(602pt)=(0.485534,0.749033,0.484647); 
      rgb(603pt)=(0.488316,0.749084,0.483473); 
      rgb(604pt)=(0.491081,0.7491,0.482316); 
      rgb(605pt)=(0.493838,0.7491,0.481168); 
      rgb(606pt)=(0.496595,0.7491,0.480019); 
      rgb(607pt)=(0.499351,0.7491,0.47887); 
      rgb(608pt)=(0.502069,0.74912,0.477722); 
      rgb(609pt)=(0.504775,0.749145,0.476573); 
      rgb(610pt)=(0.50748,0.749171,0.475424); 
      rgb(611pt)=(0.510186,0.749196,0.474276); 
      rgb(612pt)=(0.512892,0.7492,0.47317); 
      rgb(613pt)=(0.515598,0.7492,0.472073); 
      rgb(614pt)=(0.518303,0.7492,0.470975); 
      rgb(615pt)=(0.521009,0.7492,0.469877); 
      rgb(616pt)=(0.523644,0.749176,0.46878); 
      rgb(617pt)=(0.526273,0.749151,0.467682); 
      rgb(618pt)=(0.528902,0.749125,0.466585); 
      rgb(619pt)=(0.531531,0.7491,0.465487); 
      rgb(620pt)=(0.53416,0.749074,0.464415); 
      rgb(621pt)=(0.536789,0.749049,0.463343); 
      rgb(622pt)=(0.539418,0.749023,0.462271); 
      rgb(623pt)=(0.542043,0.748998,0.461199); 
      rgb(624pt)=(0.544621,0.748972,0.460127); 
      rgb(625pt)=(0.547199,0.748947,0.459055); 
      rgb(626pt)=(0.549777,0.748921,0.457983); 
      rgb(627pt)=(0.55235,0.748891,0.45692); 
      rgb(628pt)=(0.554903,0.74884,0.455899); 
      rgb(629pt)=(0.557456,0.748789,0.454878); 
      rgb(630pt)=(0.560008,0.748738,0.453857); 
      rgb(631pt)=(0.562554,0.748687,0.452829); 
      rgb(632pt)=(0.565081,0.748636,0.451783); 
      rgb(633pt)=(0.567608,0.748585,0.450736); 
      rgb(634pt)=(0.570135,0.748534,0.449689); 
      rgb(635pt)=(0.572653,0.748474,0.44866); 
      rgb(636pt)=(0.575155,0.748397,0.447665); 
      rgb(637pt)=(0.577656,0.748321,0.446669); 
      rgb(638pt)=(0.580158,0.748244,0.445674); 
      rgb(639pt)=(0.582649,0.748168,0.444678); 
      rgb(640pt)=(0.585125,0.748091,0.443683); 
      rgb(641pt)=(0.587601,0.748014,0.442687); 
      rgb(642pt)=(0.590077,0.747938,0.441692); 
      rgb(643pt)=(0.59254,0.747861,0.440709); 
      rgb(644pt)=(0.59499,0.747785,0.439739); 
      rgb(645pt)=(0.597441,0.747708,0.438769); 
      rgb(646pt)=(0.599891,0.747632,0.437799); 
      rgb(647pt)=(0.602311,0.747555,0.436814); 
      rgb(648pt)=(0.604711,0.747478,0.435819); 
      rgb(649pt)=(0.60711,0.747402,0.434823); 
      rgb(650pt)=(0.60951,0.747325,0.433828); 
      rgb(651pt)=(0.611909,0.747232,0.432867); 
      rgb(652pt)=(0.614308,0.747129,0.431922); 
      rgb(653pt)=(0.616708,0.747027,0.430978); 
      rgb(654pt)=(0.619107,0.746925,0.430033); 
      rgb(655pt)=(0.621487,0.746823,0.429089); 
      rgb(656pt)=(0.623861,0.746721,0.428144); 
      rgb(657pt)=(0.626235,0.746619,0.4272); 
      rgb(658pt)=(0.628609,0.746517,0.426256); 
      rgb(659pt)=(0.630962,0.746393,0.425311); 
      rgb(660pt)=(0.63331,0.746266,0.424367); 
      rgb(661pt)=(0.635658,0.746138,0.423422); 
      rgb(662pt)=(0.638007,0.746011,0.422478); 
      rgb(663pt)=(0.640332,0.745906,0.421557); 
      rgb(664pt)=(0.642654,0.745804,0.420638); 
      rgb(665pt)=(0.644977,0.745702,0.419719); 
      rgb(666pt)=(0.6473,0.7456,0.4188); 
      rgb(667pt)=(0.649623,0.745472,0.417881); 
      rgb(668pt)=(0.651946,0.745345,0.416962); 
      rgb(669pt)=(0.654268,0.745217,0.416043); 
      rgb(670pt)=(0.656587,0.745089,0.415124); 
      rgb(671pt)=(0.658859,0.744962,0.414205); 
      rgb(672pt)=(0.661131,0.744834,0.413286); 
      rgb(673pt)=(0.663402,0.744707,0.412368); 
      rgb(674pt)=(0.665674,0.744579,0.411453); 
      rgb(675pt)=(0.667946,0.744451,0.410559); 
      rgb(676pt)=(0.670218,0.744324,0.409666); 
      rgb(677pt)=(0.672489,0.744196,0.408773); 
      rgb(678pt)=(0.674755,0.744062,0.407879); 
      rgb(679pt)=(0.677001,0.743909,0.406986); 
      rgb(680pt)=(0.679247,0.743756,0.406092); 
      rgb(681pt)=(0.681494,0.743603,0.405199); 
      rgb(682pt)=(0.68374,0.743458,0.404306); 
      rgb(683pt)=(0.685986,0.74333,0.403412); 
      rgb(684pt)=(0.688232,0.743203,0.402519); 
      rgb(685pt)=(0.690479,0.743075,0.401626); 
      rgb(686pt)=(0.692704,0.742937,0.400732); 
      rgb(687pt)=(0.694899,0.742784,0.399839); 
      rgb(688pt)=(0.697094,0.742631,0.398945); 
      rgb(689pt)=(0.699289,0.742477,0.398052); 
      rgb(690pt)=(0.701497,0.742324,0.397171); 
      rgb(691pt)=(0.703718,0.742171,0.396303); 
      rgb(692pt)=(0.705939,0.742018,0.395435); 
      rgb(693pt)=(0.708159,0.741865,0.394568); 
      rgb(694pt)=(0.710351,0.741712,0.3937); 
      rgb(695pt)=(0.71252,0.741559,0.392832); 
      rgb(696pt)=(0.71469,0.741405,0.391964); 
      rgb(697pt)=(0.71686,0.741252,0.391096); 
      rgb(698pt)=(0.719029,0.741082,0.390228); 
      rgb(699pt)=(0.721199,0.740904,0.38936); 
      rgb(700pt)=(0.723369,0.740725,0.388492); 
      rgb(701pt)=(0.725538,0.740546,0.387625); 
      rgb(702pt)=(0.727708,0.740386,0.386757); 
      rgb(703pt)=(0.729878,0.740233,0.385889); 
      rgb(704pt)=(0.732047,0.74008,0.385021); 
      rgb(705pt)=(0.734217,0.739927,0.384153); 
      rgb(706pt)=(0.736366,0.739753,0.383285); 
      rgb(707pt)=(0.73851,0.739574,0.382417); 
      rgb(708pt)=(0.740654,0.739395,0.38155); 
      rgb(709pt)=(0.742798,0.739217,0.380682); 
      rgb(710pt)=(0.744919,0.739038,0.379837); 
      rgb(711pt)=(0.747038,0.738859,0.378995); 
      rgb(712pt)=(0.749156,0.738681,0.378152); 
      rgb(713pt)=(0.751275,0.738502,0.37731); 
      rgb(714pt)=(0.753394,0.738323,0.376442); 
      rgb(715pt)=(0.755512,0.738145,0.375574); 
      rgb(716pt)=(0.757631,0.737966,0.374707); 
      rgb(717pt)=(0.75975,0.737789,0.373841); 
      rgb(718pt)=(0.761868,0.737636,0.372998); 
      rgb(719pt)=(0.763987,0.737483,0.372156); 
      rgb(720pt)=(0.766105,0.73733,0.371314); 
      rgb(721pt)=(0.76822,0.737169,0.370471); 
      rgb(722pt)=(0.770313,0.736965,0.369629); 
      rgb(723pt)=(0.772406,0.73676,0.368786); 
      rgb(724pt)=(0.774499,0.736556,0.367944); 
      rgb(725pt)=(0.776592,0.736358,0.367096); 
      rgb(726pt)=(0.778686,0.736179,0.366228); 
      rgb(727pt)=(0.780779,0.736001,0.36536); 
      rgb(728pt)=(0.782872,0.735822,0.364492); 
      rgb(729pt)=(0.784957,0.735643,0.363632); 
      rgb(730pt)=(0.787024,0.735465,0.36279); 
      rgb(731pt)=(0.789092,0.735286,0.361948); 
      rgb(732pt)=(0.791159,0.735107,0.361105); 
      rgb(733pt)=(0.793227,0.734929,0.360263); 
      rgb(734pt)=(0.795295,0.73475,0.359421); 
      rgb(735pt)=(0.797362,0.734571,0.358578); 
      rgb(736pt)=(0.79943,0.734392,0.357736); 
      rgb(737pt)=(0.801485,0.734214,0.356881); 
      rgb(738pt)=(0.803527,0.734035,0.356014); 
      rgb(739pt)=(0.805569,0.733856,0.355146); 
      rgb(740pt)=(0.807611,0.733678,0.354278); 
      rgb(741pt)=(0.809668,0.733499,0.353424); 
      rgb(742pt)=(0.811735,0.73332,0.352582); 
      rgb(743pt)=(0.813803,0.733142,0.35174); 
      rgb(744pt)=(0.81587,0.732963,0.350897); 
      rgb(745pt)=(0.817921,0.732784,0.350038); 
      rgb(746pt)=(0.819963,0.732606,0.349171); 
      rgb(747pt)=(0.822005,0.732427,0.348303); 
      rgb(748pt)=(0.824047,0.732248,0.347435); 
      rgb(749pt)=(0.826071,0.73207,0.346567); 
      rgb(750pt)=(0.828087,0.731891,0.345699); 
      rgb(751pt)=(0.830104,0.731712,0.344831); 
      rgb(752pt)=(0.83212,0.731534,0.343963); 
      rgb(753pt)=(0.834158,0.731355,0.343095); 
      rgb(754pt)=(0.8362,0.731176,0.342228); 
      rgb(755pt)=(0.838242,0.730998,0.34136); 
      rgb(756pt)=(0.840284,0.730819,0.340492); 
      rgb(757pt)=(0.842303,0.73064,0.339624); 
      rgb(758pt)=(0.84432,0.730462,0.338756); 
      rgb(759pt)=(0.846336,0.730283,0.337888); 
      rgb(760pt)=(0.848353,0.730104,0.33702); 
      rgb(761pt)=(0.850369,0.729926,0.336153); 
      rgb(762pt)=(0.852386,0.729747,0.335285); 
      rgb(763pt)=(0.854402,0.729568,0.334417); 
      rgb(764pt)=(0.856419,0.729391,0.333546); 
      rgb(765pt)=(0.858435,0.729238,0.332627); 
      rgb(766pt)=(0.860452,0.729085,0.331708); 
      rgb(767pt)=(0.862468,0.728932,0.330789); 
      rgb(768pt)=(0.864481,0.728778,0.329874); 
      rgb(769pt)=(0.866472,0.728625,0.32898); 
      rgb(770pt)=(0.868463,0.728472,0.328087); 
      rgb(771pt)=(0.870454,0.728319,0.327194); 
      rgb(772pt)=(0.872445,0.728166,0.326295); 
      rgb(773pt)=(0.874436,0.728013,0.325376); 
      rgb(774pt)=(0.876427,0.727859,0.324457); 
      rgb(775pt)=(0.878418,0.727706,0.323538); 
      rgb(776pt)=(0.880417,0.727561,0.322619); 
      rgb(777pt)=(0.882433,0.727433,0.3217); 
      rgb(778pt)=(0.88445,0.727306,0.320781); 
      rgb(779pt)=(0.886466,0.727178,0.319862); 
      rgb(780pt)=(0.888463,0.72705,0.318933); 
      rgb(781pt)=(0.890429,0.726923,0.317989); 
      rgb(782pt)=(0.892394,0.726795,0.317044); 
      rgb(783pt)=(0.894359,0.726668,0.3161); 
      rgb(784pt)=(0.896337,0.726552,0.315132); 
      rgb(785pt)=(0.898328,0.72645,0.314136); 
      rgb(786pt)=(0.900319,0.726348,0.313141); 
      rgb(787pt)=(0.90231,0.726246,0.312145); 
      rgb(788pt)=(0.904301,0.726158,0.31115); 
      rgb(789pt)=(0.906292,0.726081,0.310154); 
      rgb(790pt)=(0.908283,0.726005,0.309159); 
      rgb(791pt)=(0.910274,0.725928,0.308163); 
      rgb(792pt)=(0.912249,0.725851,0.307151); 
      rgb(793pt)=(0.914214,0.725775,0.30613); 
      rgb(794pt)=(0.91618,0.725698,0.305109); 
      rgb(795pt)=(0.918145,0.725622,0.304088); 
      rgb(796pt)=(0.920111,0.7256,0.303031); 
      rgb(797pt)=(0.922076,0.7256,0.301959); 
      rgb(798pt)=(0.924041,0.7256,0.300886); 
      rgb(799pt)=(0.926007,0.7256,0.299814); 
      rgb(800pt)=(0.927972,0.7256,0.298722); 
      rgb(801pt)=(0.929938,0.7256,0.297624); 
      rgb(802pt)=(0.931903,0.7256,0.296527); 
      rgb(803pt)=(0.933869,0.7256,0.295429); 
      rgb(804pt)=(0.935812,0.725668,0.294264); 
      rgb(805pt)=(0.937752,0.725744,0.29309); 
      rgb(806pt)=(0.939692,0.725821,0.291916); 
      rgb(807pt)=(0.941632,0.725897,0.290741); 
      rgb(808pt)=(0.943571,0.726023,0.289518); 
      rgb(809pt)=(0.945511,0.726151,0.288293); 
      rgb(810pt)=(0.947451,0.726278,0.287068); 
      rgb(811pt)=(0.949389,0.726411,0.285839); 
      rgb(812pt)=(0.951278,0.726641,0.284537); 
      rgb(813pt)=(0.953167,0.72687,0.283235); 
      rgb(814pt)=(0.955056,0.7271,0.281933); 
      rgb(815pt)=(0.956938,0.72734,0.280622); 
      rgb(816pt)=(0.958776,0.727646,0.279243); 
      rgb(817pt)=(0.960614,0.727952,0.277865); 
      rgb(818pt)=(0.962451,0.728259,0.276486); 
      rgb(819pt)=(0.964273,0.728597,0.275086); 
      rgb(820pt)=(0.966034,0.729057,0.273606); 
      rgb(821pt)=(0.967795,0.729516,0.272126); 
      rgb(822pt)=(0.969557,0.729976,0.270645); 
      rgb(823pt)=(0.971288,0.730473,0.269135); 
      rgb(824pt)=(0.972947,0.73106,0.267552); 
      rgb(825pt)=(0.974606,0.731647,0.265969); 
      rgb(826pt)=(0.976265,0.732234,0.264387); 
      rgb(827pt)=(0.977857,0.732879,0.262785); 
      rgb(828pt)=(0.979338,0.733619,0.261151); 
      rgb(829pt)=(0.980818,0.734359,0.259518); 
      rgb(830pt)=(0.982299,0.735099,0.257884); 
      rgb(831pt)=(0.983697,0.73591,0.256227); 
      rgb(832pt)=(0.984999,0.736803,0.254542); 
      rgb(833pt)=(0.986301,0.737697,0.252858); 
      rgb(834pt)=(0.987603,0.73859,0.251173); 
      rgb(835pt)=(0.988753,0.739566,0.249474); 
      rgb(836pt)=(0.989774,0.740613,0.247764); 
      rgb(837pt)=(0.990795,0.741659,0.246054); 
      rgb(838pt)=(0.991816,0.742706,0.244344); 
      rgb(839pt)=(0.992677,0.743816,0.242681); 
      rgb(840pt)=(0.993443,0.744965,0.241048); 
      rgb(841pt)=(0.994209,0.746114,0.239414); 
      rgb(842pt)=(0.994975,0.747262,0.23778); 
      rgb(843pt)=(0.995578,0.748465,0.236165); 
      rgb(844pt)=(0.996114,0.74969,0.234557); 
      rgb(845pt)=(0.99665,0.750915,0.232949); 
      rgb(846pt)=(0.997186,0.752141,0.231341); 
      rgb(847pt)=(0.997562,0.753386,0.229813); 
      rgb(848pt)=(0.997893,0.754637,0.228307); 
      rgb(849pt)=(0.998225,0.755887,0.226801); 
      rgb(850pt)=(0.998557,0.757138,0.225295); 
      rgb(851pt)=(0.998711,0.758433,0.223856); 
      rgb(852pt)=(0.998839,0.759735,0.222426); 
      rgb(853pt)=(0.998966,0.761037,0.220997); 
      rgb(854pt)=(0.999094,0.762339,0.219567); 
      rgb(855pt)=(0.999076,0.763641,0.218186); 
      rgb(856pt)=(0.99905,0.764942,0.216808); 
      rgb(857pt)=(0.999025,0.766244,0.21543); 
      rgb(858pt)=(0.998995,0.767546,0.214054); 
      rgb(859pt)=(0.998868,0.768848,0.212752); 
      rgb(860pt)=(0.99874,0.77015,0.21145); 
      rgb(861pt)=(0.998613,0.771451,0.210149); 
      rgb(862pt)=(0.998473,0.772756,0.208856); 
      rgb(863pt)=(0.998243,0.774083,0.207631); 
      rgb(864pt)=(0.998014,0.775411,0.206405); 
      rgb(865pt)=(0.997784,0.776738,0.20518); 
      rgb(866pt)=(0.997539,0.77806,0.20396); 
      rgb(867pt)=(0.997232,0.779362,0.20276); 
      rgb(868pt)=(0.996926,0.780664,0.201561); 
      rgb(869pt)=(0.99662,0.781966,0.200361); 
      rgb(870pt)=(0.996299,0.783268,0.199168); 
      rgb(871pt)=(0.995942,0.784569,0.197994); 
      rgb(872pt)=(0.995584,0.785871,0.19682); 
      rgb(873pt)=(0.995227,0.787173,0.195646); 
      rgb(874pt)=(0.994842,0.788475,0.19449); 
      rgb(875pt)=(0.994408,0.789777,0.193367); 
      rgb(876pt)=(0.993974,0.791078,0.192244); 
      rgb(877pt)=(0.99354,0.79238,0.191121); 
      rgb(878pt)=(0.993083,0.793671,0.190021); 
      rgb(879pt)=(0.992598,0.794947,0.188949); 
      rgb(880pt)=(0.992113,0.796223,0.187877); 
      rgb(881pt)=(0.991628,0.797499,0.186805); 
      rgb(882pt)=(0.99113,0.798789,0.185732); 
      rgb(883pt)=(0.990619,0.800091,0.18466); 
      rgb(884pt)=(0.990109,0.801393,0.183588); 
      rgb(885pt)=(0.989598,0.802695,0.182516); 
      rgb(886pt)=(0.989072,0.803996,0.18146); 
      rgb(887pt)=(0.988536,0.805298,0.180413); 
      rgb(888pt)=(0.988,0.8066,0.179367); 
      rgb(889pt)=(0.987464,0.807902,0.17832); 
      rgb(890pt)=(0.98691,0.809186,0.177291); 
      rgb(891pt)=(0.986349,0.810462,0.17627); 
      rgb(892pt)=(0.985787,0.811738,0.175249); 
      rgb(893pt)=(0.985226,0.813015,0.174228); 
      rgb(894pt)=(0.984644,0.814311,0.173207); 
      rgb(895pt)=(0.984057,0.815613,0.172186); 
      rgb(896pt)=(0.98347,0.816914,0.171165); 
      rgb(897pt)=(0.982883,0.818216,0.170144); 
      rgb(898pt)=(0.982296,0.819518,0.169145); 
      rgb(899pt)=(0.981709,0.82082,0.16815); 
      rgb(900pt)=(0.981122,0.822122,0.167154); 
      rgb(901pt)=(0.980535,0.823423,0.166159); 
      rgb(902pt)=(0.979947,0.824725,0.165163); 
      rgb(903pt)=(0.97936,0.826027,0.164168); 
      rgb(904pt)=(0.978773,0.827329,0.163172); 
      rgb(905pt)=(0.978186,0.828631,0.162177); 
      rgb(906pt)=(0.977599,0.829932,0.161181); 
      rgb(907pt)=(0.977012,0.831234,0.160186); 
      rgb(908pt)=(0.976425,0.832536,0.15919); 
      rgb(909pt)=(0.975838,0.833841,0.158195); 
      rgb(910pt)=(0.975251,0.835168,0.157199); 
      rgb(911pt)=(0.974664,0.836495,0.156204); 
      rgb(912pt)=(0.974077,0.837823,0.155208); 
      rgb(913pt)=(0.973489,0.83915,0.154213); 
      rgb(914pt)=(0.972902,0.840477,0.153217); 
      rgb(915pt)=(0.972315,0.841805,0.152222); 
      rgb(916pt)=(0.971728,0.843132,0.151226); 
      rgb(917pt)=(0.971155,0.844466,0.150224); 
      rgb(918pt)=(0.970619,0.845819,0.149203); 
      rgb(919pt)=(0.970083,0.847172,0.148182); 
      rgb(920pt)=(0.969547,0.848525,0.147161); 
      rgb(921pt)=(0.96902,0.849886,0.14614); 
      rgb(922pt)=(0.968509,0.851265,0.145119); 
      rgb(923pt)=(0.967999,0.852643,0.144098); 
      rgb(924pt)=(0.967488,0.854022,0.143077); 
      rgb(925pt)=(0.967,0.855411,0.142056); 
      rgb(926pt)=(0.966541,0.856815,0.141035); 
      rgb(927pt)=(0.966081,0.858219,0.140014); 
      rgb(928pt)=(0.965622,0.859623,0.138992); 
      rgb(929pt)=(0.965189,0.86104,0.137945); 
      rgb(930pt)=(0.96478,0.862469,0.136873); 
      rgb(931pt)=(0.964372,0.863899,0.135801); 
      rgb(932pt)=(0.963963,0.865328,0.134729); 
      rgb(933pt)=(0.96357,0.866773,0.133657); 
      rgb(934pt)=(0.963187,0.868228,0.132585); 
      rgb(935pt)=(0.962805,0.869683,0.131513); 
      rgb(936pt)=(0.962422,0.871138,0.130441); 
      rgb(937pt)=(0.962091,0.87261,0.129351); 
      rgb(938pt)=(0.961785,0.874091,0.128253); 
      rgb(939pt)=(0.961478,0.875571,0.127156); 
      rgb(940pt)=(0.961172,0.877052,0.126058); 
      rgb(941pt)=(0.960885,0.878571,0.124961); 
      rgb(942pt)=(0.960605,0.880103,0.123863); 
      rgb(943pt)=(0.960324,0.881634,0.122765); 
      rgb(944pt)=(0.960043,0.883166,0.121668); 
      rgb(945pt)=(0.959849,0.884719,0.120549); 
      rgb(946pt)=(0.95967,0.886276,0.119426); 
      rgb(947pt)=(0.959491,0.887833,0.118302); 
      rgb(948pt)=(0.959313,0.88939,0.117179); 
      rgb(949pt)=(0.959181,0.890995,0.116032); 
      rgb(950pt)=(0.959054,0.892603,0.114884); 
      rgb(951pt)=(0.958926,0.894211,0.113735); 
      rgb(952pt)=(0.958799,0.895819,0.112587); 
      rgb(953pt)=(0.958748,0.897453,0.111464); 
      rgb(954pt)=(0.958697,0.899086,0.110341); 
      rgb(955pt)=(0.958646,0.90072,0.109217); 
      rgb(956pt)=(0.958602,0.902359,0.108089); 
      rgb(957pt)=(0.958628,0.904043,0.106915); 
      rgb(958pt)=(0.958653,0.905728,0.105741); 
      rgb(959pt)=(0.958679,0.907413,0.104567); 
      rgb(960pt)=(0.958718,0.909102,0.103393); 
      rgb(961pt)=(0.95882,0.910812,0.102219); 
      rgb(962pt)=(0.958922,0.912522,0.101044); 
      rgb(963pt)=(0.959024,0.914232,0.0998703); 
      rgb(964pt)=(0.959153,0.915962,0.0986961); 
      rgb(965pt)=(0.959357,0.917749,0.0975219); 
      rgb(966pt)=(0.959561,0.919536,0.0963477); 
      rgb(967pt)=(0.959765,0.921323,0.0951736); 
      rgb(968pt)=(0.959996,0.923118,0.0939907); 
      rgb(969pt)=(0.960277,0.924931,0.092791); 
      rgb(970pt)=(0.960557,0.926743,0.0915913); 
      rgb(971pt)=(0.960838,0.928555,0.0903916); 
      rgb(972pt)=(0.961151,0.930378,0.0891919); 
      rgb(973pt)=(0.961509,0.932216,0.0879922); 
      rgb(974pt)=(0.961866,0.934054,0.0867925); 
      rgb(975pt)=(0.962223,0.935892,0.0855928); 
      rgb(976pt)=(0.96262,0.937768,0.0843802); 
      rgb(977pt)=(0.963053,0.939683,0.083155); 
      rgb(978pt)=(0.963487,0.941597,0.0819297); 
      rgb(979pt)=(0.963921,0.943512,0.0807045); 
      rgb(980pt)=(0.964415,0.945426,0.0794643); 
      rgb(981pt)=(0.964951,0.947341,0.0782135); 
      rgb(982pt)=(0.965487,0.949255,0.0769628); 
      rgb(983pt)=(0.966023,0.951169,0.075712); 
      rgb(984pt)=(0.966594,0.953118,0.0744441); 
      rgb(985pt)=(0.967181,0.955083,0.0731679); 
      rgb(986pt)=(0.967768,0.957049,0.0718916); 
      rgb(987pt)=(0.968355,0.959014,0.0706153); 
      rgb(988pt)=(0.96898,0.960999,0.0693006); 
      rgb(989pt)=(0.969619,0.96299,0.0679733); 
      rgb(990pt)=(0.970257,0.964981,0.0666459); 
      rgb(991pt)=(0.970895,0.966972,0.0653186); 
      rgb(992pt)=(0.971554,0.968984,0.0639486); 
      rgb(993pt)=(0.972218,0.971001,0.0625703); 
      rgb(994pt)=(0.972882,0.973017,0.0611919);  
      rgb(995pt)=(0.973545,0.975034,0.0598135); 
      rgb(996pt)=(0.974232,0.97705,0.058318); 
      rgb(997pt)=(0.974922,0.979067,0.056812); 
      rgb(998pt)=(0.975611,0.981083,0.055306); 
      rgb(999pt)=(0.9763,0.9831,0.0538) 
    }
}

\usetikzlibrary{arrows,shapes,positioning}
\usetikzlibrary{decorations.markings}
\tikzstyle arrowstyle=[scale=1]
\tikzstyle directed=[postaction={decorate,decoration={markings,
        mark=at position .65 with {\arrow[arrowstyle]{stealth}}}}]
\tikzstyle reverse directed=[postaction={decorate,decoration={markings,
        mark=at position .65 with {\arrowreversed[arrowstyle]{stealth};}}}]

 \newcounter{tikzsubfigcounter}[figure]
\renewcommand{\thetikzsubfigcounter}{\thesection.\the\numexpr\value{figure}+1\relax\alph{tikzsubfigcounter}}
 \newcommand*{\tikztitle}[1]{ %
   \refstepcounter{tikzsubfigcounter}
   \textbf{(\alph{tikzsubfigcounter})}\space\space #1
 }

%% file: Sections/introduction.tex
\section{Introduction}
Kinetic equations play an important role in many physical applications. One of the earliest and most prominent examples is the Boltzmann equation which was derived by the
Austrian physicist Ludwig Boltzmann in 1872~\cite{Boltzmann.1872} and still forms the basis for the kinetic theory of rarefied gases. The Boltzmann equation or similar kinetic equations
proved to be applicable not only to classical gases but also to electron transport in solids and plasmas, neutron transport in nuclear reactors, photon transport in superfluids
and radiative transfer, among others~\cite{Lewis1984, Cercignani1988, Lanckau1989, Mihalas1999, Markowich1990}. More recently, kinetic equations were also derived in the context of
biological modelling, e.g., for studying cell movement or wolf migration~\cite{Capasso.1999, Hillen.2013, Borsche2019}.

While analytic solutions can be derived in some special cases~\cite{Ganapol2001}, usually kinetic equations have to be solved numerically.
Due to their high dimensionality, directly solving kinetic equations with standard discretizations (e.g., finite difference methods)
is often infeasible or restricted to very small grid sizes. For that reason, a variety of specialized approximate methods
have been developed, many of which belong to the class of moment methods.
Instead of computing the whole kinetic density function, moment approximations choose a set of weight functions (usually polynomials up to some order)
on the velocity space and only track the weighted velocity averages (called moments) of the kinetic density with respect to these functions.
This is usually done by performing a Galerkin projection of the original kinetic equation
to the linear span of the weight functions. In general, the resulting moment equations are not closed and
thus an ansatz for the velocity distribution has to be made. Choosing a linear combination of the weight functions gives
the widely used \(\PN\) closure~\cite{Lewis1984}, where \(\momentorder\) is the degree of the highest-order
moments in the model. The \(\PN\) closure results in linear equations, is simple to implement and often gives reasonable results.
However, it does not guarantee non-negativity of the approximated kinetic density. This sometimes leads to physically meaningless
solutions, as the \(\PN\) solutions can, e.g., contain negative values for the local
particle density.

The so-called minimum-entropy moment models \(\MN\)~\cite{Min78,DubFeu99} avoid these problems by choosing the ansatz function such
that it minimizes an entropy functional which usually models the
(negative) physical entropy. The resulting closed system of equations is
hyperbolic and dissipates the chosen entropy~\cite{Levermore1996}. However, numerically solving the \(\MN\) equations
requires the solution of a non-linear optimization problem at every point on the
space-time grid. Although the optimization problems can be solved in parallel~\cite{Hauck2010,Alldredge2012,KristopherGarrett2015,Schaerer2017},
the computational cost for high moment orders still is prohibitively high in practical applications.
Another drawback of the entropy-based moment closures is that the optimization problem
is solvable only for so-called realizable moment vectors, i.e., vectors that actually are
moments of a positive density function. As explicit descriptions of the set of realizable moment vectors are usually not available,
discretizations (especially of higher order) often struggle to keep the
approximate solutions
realizable~\cite{Zhang2010,Schar1996,Alldredge2015,Schneider2015b,Schneider2016a,Chidyagwai2017,Olbrant2012}.

A partial remedy for the high computational cost of the minimum entropy models could be additional model reduction,
for example via reduced basis methods~\cite{Quarteroni2016}. These methods
generate a reduced description of the (discretized) equations first
and then use this reduced model to perform the actual computations.
In some cases, e.g., if a given kinetic equation has to be solved many times for different parameters,
this reduces overall computation time by several orders of magnitude.
Generating the reduced model is usually done by constructing a low-dimensional linear subspace
from solution trajectories and then projecting the problem to this subspace.
This has been successfully done for the \(\PN\) models~\cite{Himpe2018}.
In the context of minimum-entropy moment models, however, this procedure is
problematic as it does not preserve realizability, which may render the reduced
model useless as it does not admit a solution.

Checking realizability is much easier when using piecewise linear bases instead
of the standard polynomial basis on the whole velocity
space~\cite{Frank2006,DubKla02,Schneider2014,Ritter2016,Schneider2017,SchneiderLeibner2020,SchneiderLeibner2019b}.
In addition, the computational cost is significantly lower for these models.
However, solving the optimization problems is still costly compared to linear models and
maintaining realizability still requires additional limiters~\cite{SchneiderLeibner2019b}.

Another approach to fix the realizability issues is to introduce a regularization
of the optimization problem~\cite{Alldredge2019}. The regularized problem
admits a solution also for moments vectors that are not realizable and maintains most of
the desirable properties of the original problem, at the cost
of an additional approximation error
(which, however, can be controlled by the regularization parameter).
However, this approach still requires the solution of the (regularized)
minimum entropy problem in each cell of the space-time grid.

In this paper, we will
present a new discretization scheme for the minimum-entropy moment equations based on a
transformation of the semi-discretized equations to entropy variables.
The new scheme replaces the non-linear optimization problems by matrix inversions and
inherently guarantees realizability. As a consequence, it avoids many
of the problems described above.
In addition, the new scheme is often significantly faster
than the untransformed scheme and shows improved parallel scaling.
Moreover, a discrete entropy law can be enforced
for the new scheme by using a relaxed Runge-Kutta method.
On the downside,
adaptive timestepping is strictly needed for the transformed scheme. Moreover,
numerically singular Hessian matrices will result in a failure of the scheme
if no additional regularization is employed. However, we did not encounter
such a situation during our extensive numerical experiments (despite the fact that
the untransformed reference scheme had to use regularization in several of the tests).

This paper is organized as follows. First, in \cref{sec:background} we shortly
recall the necessary background on minimum entropy moment models.
In \cref{sec:newscheme}, the new scheme is presented and analysed.
In \cref{sec:implementationdetails}, we give an outline
of our implementation which is then used for the
extensive numerical investigations in \cref{sec:numericalexperiments}.

%% file: Sections/momentmodels.tex
\section{Minimum-entropy moment models}\label{sec:background}
\subsection{Kinetic transport equation}

We consider the linear transport equation
\begin{subequations}\label{eq:FokkerPlanckEasy2}
  \begin{equation}
    \dt\distribution+\SC\cdot\spatialGradient\distribution + \absorption\distribution = \scattering\collision{\distribution}+\source,
  \end{equation}
  which describes the density of particles with speed \(\SC\in\sphere\) at position
  \(\spatialvar\in\domain\subseteq\R^3\) and time \(\timevar \in \timeint = [0, \tf]\) under the
  events of scattering
  (proportional to \(\scattering\left(\timevar,\spatialvar\right) \geq 0\)), absorption (proportional to \(\absorption\left(\timevar,\spatialvar\right) \geq 0\))
  and emission (proportional to \(\source\left(\timevar,\spatialvar,\SC\right) \geq 0\)).
  The equation is supplemented with initial condition and Dirichlet boundary conditions:
  \begin{align}
    \distribution(0,\spatialvar,\SC)
                                                    & =
    \distributiontzero(\spatialvar,\SC)             & \text{for } \spatialvar\in\domain, \SC\in\sphere     \label{eq:kineticequationinitial},                                    \\
    \distribution(\timevar,\spatialvar,\SC)
                                                    & =
    \distributionboundary(\timevar,\spatialvar,\SC) & \text{for } \timevar\in\timeint, \spatialvar\in\partial\domain, \outernormal\cdot\SC<0, \label{eq:kineticequationboundary}
  \end{align}
  where \(\distributiontzero\) and \(\distributionboundary\) are given functions and \(\outernormal\) is the outward unit normal vector in \(\spatialvar\in\partial\domain\).
\end{subequations}
For simplicity, we will consider isotropic scattering
\begin{equation}\label{eq:isotropicscatteringop}
  \collision{\distribution}(\timevar, \spatialvar, \SC) =  \frac{1}{\abs{\sphere}}
  \int\limits_{\sphere} \distribution(\timevar, \spatialvar, \tilde{\SC})~d\tilde{\SC} - \distribution(\timevar, \spatialvar, \SC),
\end{equation}
isotropic time-independent source \(\source\left(\timevar,\spatialvar,\SC\right) = \source(\spatialvar)\) and
time-independent scattering \(\scattering\left(\timevar,\spatialvar\right) = \scattering\left(\spatialvar\right)\)
and absorption \(\absorption\left(\timevar,\spatialvar\right) = \absorption\left(\spatialvar\right)\).

Parameterizing \(\SC\) in spherical coordinates we obtain
\begin{equation}
  \SC = \left(\sqrt{1-\SCheight^2}\cos(\SCangle),\sqrt{1-\SCheight^2}\sin(\SCangle), \SCheight\right)^T \eqqcolon \left(\SCx,\SCy,\SCz\right)^T,
\end{equation}
where \(\SCangle\in[0,2\pi]\) is the azimuthal and \(\SCheight\in[-1,1]\) the cosine of the polar angle.

As a one-dimensional simplification, we will also consider the models in slab geometry, which is a projection of the sphere onto
the \(\z\)-axis~\cite{Seibold2014}. The transport equation under consideration then has the form
\begin{equation}
  \dt\distribution+\SCheight\dz\distribution + \absorption\distribution = \scattering\collision{\distribution}+\source, \qquad \timevar\in\timeint,\z\in\domain,\SCheight\in[-1,1].
\end{equation}

\subsection{The moment approximation}

In the following, \(\angularDomain\) will always denote the angular domain, i.e., \(\angularDomain = [-1,1]\) in slab geometry and \(\angularDomain = \sphere\) in the three-dimensional case,
and \(\SC\) will denote the corresponding angular variable.
Moreover, we will use angle brackets to denote integration over \(\angularDomain\), i.e.,
\begin{equation*}
  \ints{f} \coloneqq \int\limits_{\angularDomain} f(\SC)~d\SC \quad  {\rm for  \ all  \ } f \in L^1(\angularDomain).
\end{equation*}

Due to the high-dimensionality, directly discretizing and solving~\eqref{eq:FokkerPlanckEasy2} via standard numerical schemes is usually not viable.
We will thus consider moment approximations of~\eqref{eq:FokkerPlanckEasy2}.
These models transfer the kinetic equation to a coupled system
of PDEs for weighted velocity averages (moments) of the solution.
\begin{definition}
  The vector of functions \(\basis\colon\angularDomain\to\R^{\momentnumber}\) consisting of \(\momentnumber\)
  linearly-independent
  \emph{basis functions} \(\basiscomp[\basisindex] \in \Lp{1}(\angularDomain)\), \(\basisindex \in \Set{0,\ldots\momentnumber-1}\), is called a
  \emph{moment basis}.
  The \emph{moments} \(\moments[\basis,\distribution] = {\left(\momentcomp{0},\ldots,\momentcomp{\momentnumber-1}\right)}^T \in \R^\momentnumber\)
  with respect to the basis \(\basis\)
  of a given density function \(\distribution \in \Lp{1}(\angularDomain)\) are then defined by
  \begin{equation}\label{eq:moments}
    \moments[\basis,\distribution] \coloneqq \ints{\basis\distribution},
  \end{equation}
  where the integration is performed component-wise.
  Furthermore, the vector \(\isotropicmomentbasis{\basis} \coloneqq \ints{\basis}\) is called the \emph{isotropic moment}.
\end{definition}

\begin{definition}\label{def:density}
  The quantity \(\density_{\distribution} \coloneqq \ints{\distribution} \in \R\) is called the \emph{local particle density} of the
  function \(\distribution\).
  If we assume that there exists a vector \(\multipliersone \in \R^{\momentnumber}\) such
  that
  \begin{equation}
    \multipliersone \cdot \basis \equiv 1,
  \end{equation}
  we have
  \begin{equation*}
    \multipliersone \cdot \moments[\basis,\distribution] = \ints{ \multipliersone \cdot \basis \distribution} = \ints{\distribution} = \density_{\distribution}.
  \end{equation*}
  Hence, we define the local particle density of the moment vector \(\moments \in \R^{\momentnumber}\)
  with respect to the basis \(\basis\) as
  \begin{equation}\label{eq:densityfrommultipliersone}
    \density_{\basis}(\moments) \coloneqq \multipliersone \cdot \moments.
  \end{equation}
\end{definition}
\begin{remark}
  The vector \(\multipliersone\) exists for all bases regarded in this paper (see \cref{sec:basisfunctions}).
\end{remark}
In the following, if basis \(\basis\) or density function \(\distribution\) are clear from the context,
we will usually omit the corresponding subscripts.

Equations for the moments \(\moments\) can be
obtained by multiplying~\eqref{eq:FokkerPlanckEasy2} with \(\basis\) and integrating over \(\angularDomain\), yielding
\begin{equation*}
  \ints{\basis\dt\distribution}+\ints{\basisop\left(\SC\cdot\spatialGradient\distribution\right)} + \absorption\ints{\basis\distribution} = \scattering\ints{\basis\collision{\distribution}}+\ints{\basis\source}.
\end{equation*}
Collecting known terms, and interchanging integration and differentiation where possible, the moment system has the form
\begin{equation}
  \label{eq:MomentSystemUnclosed}
  \dt\moments+\ints{\basisop\left(\SC\cdot\spatialGradient\distribution\right)} + \absorption\moments =
  \scattering\ints{\basis\collision{\distribution}}+\ints{\basis\source}.
\end{equation}
For isotropic collision operator~\eqref{eq:isotropicscatteringop}, the scattering term becomes
\begin{equation}\label{eq:MomentSystemIsotropicScatteringTerm}
  \ints{\basis\collision{\distribution}}
  = \ints{\basisop\left(\frac{\ints{\distribution}}{\volume{\angularDomain}} - \distribution\right)}
  = \frac{\ints{\distribution}}{\volume{\angularDomain}}\ints{\basis}- \moments
  = \frac{\multipliersone \cdot \moments}{\volume{\angularDomain}}\ints{\basis} - \moments
  = \left(\isomatrix - \eyematrix\right) \moments
\end{equation}
where
\(\eyematrix \in \R^{\momentnumber\times\momentnumber}\) is the unit matrix and
\begin{equation}\label{eq:isomatrix}
  \isomatrix = \frac{1}{\volume{\angularDomain}}
  \ints{\basis}{\left(\multipliersone\right)}^T \in \R^{\momentnumber\times\momentnumber}
\end{equation}
is the matrix
mapping the moment vector \(\moments\)
to the isotropic moment vector with the same density
\begin{equation*}
  \isomatrix \moments = \isotropicmomentbasis{\basis} \cdot \frac{\density_{\basis}(\moments)}{\abs{\angularDomain}}.
\end{equation*}
Consequently, for isotropic scattering,~\eqref{eq:MomentSystemUnclosed} simplifies to
\begin{equation}\label{eq:MomentSystemUnclosedIsotropic}
  \dt\moments+\ints{\basisop\left(\SC\cdot\spatialGradient\distribution\right)}
  = \left(\scattering \isomatrix - \crosssection\eyematrix\right) \moments  + \ints{\basis\source}
\end{equation}
where \(\crosssection = \absorption + \scattering\) is the \emph{total cross section}.

However, even in the isotropic case,
the transport term \(\ints{\basisop\left(\SC\cdot\spatialGradient\distribution\right)}\) usually
cannot be given explicitly in terms of \(\moments\).
For non-isotropic scattering operator, the same applies to the scattering term.
Therefore, additional assumptions have to be made to close the
unknown terms. A common approach is to replace \(\distribution\) in~\eqref{eq:MomentSystemUnclosedIsotropic} by a moment-dependent ansatz
\(\ansatz[\moments,\basis]\),
resulting in a
closed system of non-linear equations for \(\moments\):
\begin{equation}\label{eq:MomentSystemClosedIsotropic}
  \dt\moments+\sum_{\dimindex=1}^{\dimension}\dx[\dimindex]\Flux_\dimindex\left(\moments\right) = \Source\left(\spatialvar,\moments\right),
\end{equation}
where
\begin{equation}\label{eq:fluxfuncdef}
  \Flux_\dimindex\left(\moments\right) =  \ints{\SCcomp[\dimindex]\basis\ansatz[\moments,\basis]}
\end{equation}
and
\begin{equation}\label{eq:sourcedef}
  \Source\left(\spatialvar,\moments\right)
  = \left(\scattering \isomatrix - \crosssection\eyematrix\right) \moments  + \ints{\basis\source}.
\end{equation}

\begin{remark}
  We will always assume that the ansatz exactly reproduces the moments, i.e., \(\ints{\basis\ansatz[\moments,\basis]} = \moments\).
  Note that this may not be fulfilled by regularized moment approximations as regarded, e.g., in \cite{Alldredge2019}.
\end{remark}

It remains to specify the basis functions and the ansatz density \(\ansatz[\moments,\basis]\).
In the following, we will often omit the \(\basis\)-dependency of the ansatz function and only write \(\ansatz[\moments]\) if the basis is
clear from the context.

\subsection{Minimum-entropy closure}
\label{subsec:mn_and_basis}
For the minimum entropy closure~\cite{Levermore1996,Min78,minerbo1979ment,Friedrichs1971}, we choose a strictly convex and twice continuously
differentiable \emph{entropy density function} \(\entropy\colon\entropydomain \subset \R\to\R\) and demand that the ansatz function minimizes the
entropy functional
\begin{equation}
  \entropyFunctional(\distribution) = \ints{\entropy(\distribution)}
\end{equation}
under the moment constraints
\begin{equation}\label{eq:MomentConstraints}
  \ints{\basis\distribution} = \moments.
\end{equation}
Here, the minimum is simply taken over all functions
\(\distribution = \distribution(\SC)\) such that \(\entropyFunctional(\distribution)\) is well-defined, i.e.\
\begin{equation}\label{eq:primal}
  \ansatz[\moments] = \ansatz[\moments,\basis,\entropy] = \underset{\Set{\distribution \given \rangeop(\distribution)\subset\entropydomain,
      \, \entropy(\distribution)\in\Lp{1}(\angularDomain),
      \, \ints{\basis \distribution} = \moments}}{\argmin} \ints{\entropy(\distribution)}.
\end{equation}
This problem, which must be solved over the space-time mesh, is typically solved through its strictly convex finite-dimensional dual,
\begin{equation}\label{eq:dual}
  \multipliers[\basis,\entropy](\moments) \coloneqq \underset{\tilde{\multipliers} \in \R^{\momentnumber}}{\argmin}
  \ints{\ld{\entropy}(\basis\cdot\tilde{\multipliers})} - \moments\cdot\tilde{\multipliers},
\end{equation}
where \(\ld{\entropy}\) is the Legendre dual of \(\entropy\). The first-order necessary
conditions for the
multipliers \(\multipliers[\basis,\entropy](\moments)\) show that
the solution to~\eqref{eq:primal}, if it exists, has the form
\begin{equation}\label{eq:psiME}
  \ansatz[\moments,\basis,\entropy] = \ld{\entropy}' \left(\basis \cdot \multipliers[\basis,\entropy](\moments) \right),
\end{equation}
where \(\ld{\entropy}'\) is the derivative of \(\ld{\entropy}\).

As in~\cite{Levermore1996,Hauck2010,SchneiderLeibner2020}, for sake of simplicity, we focus on \emph{Maxwell-Boltzmann entropy}%
\begin{equation}\label{eq:maxwellboltzmannentropy}
  \entropy(\distribution) = \distribution \log(\distribution) - \distribution,
\end{equation}
which is used for non-interacting, classical particles as in an ideal gas. Thus, \(\entropydomain = (0,\infty)\) and~\eqref{eq:primal} becomes
\begin{equation}\label{eq:primal2}
  \ansatz[\moments] = \underset{\Set{\distribution \given \distribution \in \Lppos{1}(\angularDomain),\, \ints{\basis \distribution} = \moments}}{\argmin} \ints{\entropy(\distribution)},
\end{equation}
where
\begin{equation}
  \Lppos{1} \coloneqq \Set{ \distribution \in \Lp{1}(\angularDomain) \given \distribution > 0 \text{ almost everywhere. } }
\end{equation}
is the space of positive integrable functions.
Further, we have \(\ld{\entropy}(p) = \ld{\entropy}'(p) = \ld{\entropy}''(p) = \exp(p)\) and thus the minimum entropy
ansatz~\eqref{eq:psiME} becomes \(\ansatz[\moments] = \exp\left(\basis \cdot \multipliers[\basis,\entropy](\moments) \right)\).

\begin{remark}
  In principle, the new scheme described in \cref{sec:newscheme} could be used
  in the same way with other physically relevant entropies, e.g.\ the Bose-Einstein entropy
  \begin{equation*}
    \boseeinsteinentropy(\distribution) = \distribution \log(\distribution) -
    \left(1+\distribution\right)\log\left(1+\distribution\right).
  \end{equation*}
  In this case, the ansatz distribution is given by
  \begin{equation*}
    \ansatz[\moments] = \ld{(\boseeinsteinentropy)}'(\multipliers[{}_{\mathrm{BE}}](\moments)\cdot\basis)
    = \frac{\exp(\multipliers[{}_{\mathrm{BE}}](\moments) \cdot \basis)}{1-\exp(\multipliers[{}_{\mathrm{BE}}](\moments) \cdot \basis)}
  \end{equation*}
  with \(\multipliers[{}_{\mathrm{BE}}] = \multipliers[\basis,\boseeinsteinentropy]\).
  To ensure positivity, we thus have to keep the multipliers \(\multipliers[{}_{\mathrm{BE}}]\) in the basis-dependent set
  \begin{equation*}
    \Set{\multipliers \in \R^{\momentnumber} \given \multipliers \cdot \basis < 0},
  \end{equation*}
  i.e., other than in the Maxwell-Boltzmann case, we again have to
  deal with realizability (compare \cref{sec:realizability}), also in the transformed scheme.
  Depending on the basis, this might significantly complicate the implementation.
\end{remark}

Using the entropy-based closure, the moment system \eqref{eq:MomentSystemClosedIsotropic}
is hyperbolic and dissipates the chosen entropy~\cite{Schneider2016} (for \(\absorption = \source = 0\))
\begin{equation}\label{eq:entropyeq}
  \partial_t \entropyFunctional(\ansatz[\moments]) + \sum_{\dimindex=0}^{\dimension-1}
  \partial_{\spatialvarcomp_{\dimindex}} \ints{\SCcomp[\dimindex] \entropy(\ansatz[\moments])} \leq 0,
\end{equation}
i.e.\
\(\entropyFunctional(\ansatz[\moments]) = \ints{\entropy(\ansatz[\moments])} = \ints{\entropy \circ \ld{\entropy}'(\multipliers(\moments) \cdot \basis)} \)
and
\(\ints{\SC \entropy(\ansatz[\moments])}
= \ints{\SC\, \entropy \circ \ld{\entropy}'(\multipliers(\moments) \cdot \basis)}\)
form an entropy--entropy flux pair in the sense of hyperbolic systems.

\subsection{Basis functions}\label{sec:basisfunctions}
We will consider three options for the basis functions \(\basis\): the full moment basis \(\fmbasis\),
the hat function basis \(\hfbasis\) and the partial moment basis \(\pmbasis\).

\subsubsection{Full moment basis}
The full moment basis \(\fmbasis\) is the standard choice and consists of polynomials of up to
order \(\momentorder\), resulting in \(\momentnumber = \momentorder + 1\) and \(\momentnumber = (\momentorder+1)^2\) basis
functions in one and three dimensions, respectively.
We will use Legendre polynomials in slab geometry and real spherical harmonics in the full three-dimensional setting.
In one dimension, the isotropic moment is \(\isotropicmomentbasis{\fmbasis} = \ints{\fmbasis} = {(2, 0, 0, \ldots, 0)}^T\) and the multiplier
\(\multipliersone[\fmbasis]\) can be chosen as \(\multipliersone[\fmbasis] = {(1, 0, 0, \ldots, 0)}^T\).
In three dimensions, we have \(\isotropicmomentbasis{\fmbasis} = {(\sqrt{4\pi}, 0, 0, \ldots, 0)}^T =
\multipliersone[\fmbasis]\).
\begin{definition}
  The minimum-entropy moment models using the \(\fmbasis\)
  will be called \(\MN\) models, where \(\momentorder\)
  is the maximal polynomial order of the basis functions.
\end{definition}

\subsubsection{First-order finite-element bases}\label{sec:FirstOrderFEMBases}
Models using the full moment basis show optimal (spectral) convergence for
smooth problems. For non-smooth problems, however, instead of increasing the
polynomial order \(\momentorder\), it might be
better to keep \(\momentorder\) fixed and
regard piecewise polynomials on increasingly refined partitions of the domain.
We will here restrict ourselves to piecewise linear bases
(\(\momentorder = 1\)) which avoid many of the performance
and realizability problems of the classical polynomial
models~\cite{SchneiderLeibner2020,SchneiderLeibner2019b}.
In the following, we will shortly state the definitions of the first-order bases.
For a more detailed introduction see \cite{SchneiderLeibner2020}.

To define the
first-order bases, we choose a partition \(\generalpartition\)
dividing the
velocity domain \(\angularDomain\) into intervals (slab geometry) or spherical
triangles (three dimensions).  Let \(\nvertex\) and \(\nentity\) be the number
of nodes (vertices) and elements (intervals or spherical triangles)
of this partition, respectively.

The first basis of interest,
the hat function basis \(\hfbasis[\momentnumber]\), consists of \(\momentnumber =
\nvertex\) continuous basis functions \(\hfbasiscomp\) which, similar to the
linear basis typically used in the continuous finite element method, fulfill the
partition of unity property, i.e.\
\(\sum\limits_{\basisindex=0}^{\momentnumber-1} \hfbasiscomp[\basisindex]
\equiv 1\), and the Lagrange property, i.e.\ each basis function evaluates to
\(1\) at one node of the partition and to \(0\) at all other nodes.

The partial moment basis \(\pmbasis\), on the other hand, is
defined in analogy to the discontinuous finite element method and consists of
the \(\momentnumber = 2 \nentity\) or \(\momentnumber = 4 \nentity\) (in one
and three dimensions, respectively) basis functions for the space of piecewise
linear functions on \(\generalpartition\) that may be discontinuous between
elements of the partition.

More precisely, in slab geometry, we will always choose the partition \(\generalpartition\) as the equidistant partition
of \(\angularDomain = [-1,1]\) into \(\hankelhalfind\) intervals \(\cell{\cellindex} = [\SCheight_\cellindex,\SCheight_{\cellindex+1}]\)
given by the set of \(\hankelhalfind+1\)
angular ``grid'' points
\(-1 = \SCheight_0 < \SCheight_1 = -1+\frac{2}{\hankelhalfind}  < \cdots < \SCheight_{\hankelhalfind-1} = -1+(\hankelhalfind-1)\cdot\frac{2}{\hankelhalfind}
< \SCheight_{\hankelhalfind}=1\).
Given this partition, the continuous
piecewise linear basis
functions \(\hfbasis[\momentnumber] = {\left(\hfbasiscomp[0],\ldots,\hfbasiscomp[\momentnumber-1]\right)}^T\) (hat functions) are
defined as
\begin{equation}\label{eq:hfbasis}
  \hfbasiscomp[\basisindex](\SCheight) =
  \indicator{\ccell{\basisindex-1}}\cfrac{\SCheight-\SCheight_{\basisindex-1}}{\SCheight_\basisindex-\SCheight_{\basisindex-1}}
  +\indicator{\ccell{\basisindex}}\cfrac{\SCheight-\SCheight_{\basisindex+1}}{\SCheight_\basisindex-\SCheight_{\basisindex+1}},
\end{equation}
where \(\indicator{\ccell{\cellindex}}(\SCheight)\) is the indicator function on the interval
\(\ccell{\cellindex}\) (with \(\ccell{-1} \equiv \ccell{\hankelhalfind} \equiv 0\)) and \(\momentnumber = \hankelhalfind + 1\) is the number of basis functions.
The isotropic moment is \(\isotropicmomentbasis{\hfbasis}
= {\left(\frac{1}{\hankelhalfind}, \frac{2}{\hankelhalfind}, \ldots,
\frac{2}{\hankelhalfind}, \frac{1}{\hankelhalfind}\right)}^T\) and we have
\(\multipliersone[\hfbasis] = {(1, \ldots, 1)}^T\) due to the partition of unity property.

The partial moment basis in slab geometry is given by
\(\pmbasis = \left( \pmbasiscomp[0], \ldots, \pmbasiscomp[\momentnumber-1] \right)
= {\left(\pmbasis[\cell{0}]^T,\ldots,\pmbasis[\cell{\hankelhalfind-1}]^T\right)}^T\)
with
\begin{equation*}
  \pmbasis[\cell{\cellindex}](\SCheight) = \begin{cases} {\left(1,\SCheight\right)}^T & \text{ if } \SCheight \in \interior{\cell{\cellindex}},      \\
              {(0, 0)}^T                   & \text{ if } \SCheight \in [-1,1]\setminus\ccell{\cellindex}.
  \end{cases}
\end{equation*}
Here, \(\momentnumber=2\hankelhalfind\) is again the number of basis functions and \(\interior{\cell{\cellindex}}\) is the interior of \(\cell{\cellindex}\).
The isotropic moment is
\(\isotropicmomentbasis{\pmbasis}
= \frac{1}{2}{\left( \frac{1}{\hankelhalfind}, \SCheight_1^2 - \SCheight_0^2, \frac{1}{\hankelhalfind}, \SCheight_2^2 - \SCheight_1^2,
  \ldots\right)}^T\) and \(\multipliersone[\pmbasis] = {(1, 0, 1, 0, \ldots)}^T\).

In three dimensions, the triangulation \(\generalpartition\) will be obtained by dyadic refinement
of the octants of the sphere \(\angularDomain = \sphere\),
i.e.\ the coarsest triangulation contains the eight spherical triangles obtained by projecting the octahedron with vertices
\(\{{(\pm1, 0, 0)}^T\), \({(0, \pm1, 0)}^T\), \({(0, 0, \pm1)}^T\}\) to the sphere
and finer partitions are obtained by iteratively subdividing each spherical triangle into four new ones,
adding vertices at the midpoints of the triangle edges. After
\(\refinementnumber\) refinements, we thus obtain \(\nvertex(\refinementnumber) =
4^{\refinementnumber+1}+2\) vertices and \(\nentity(\refinementnumber) = 2 \cdot
4^{\refinementnumber+1}\) spherical triangles.

To get a three-dimensional equivalent of the continuous hat function basis~\eqref{eq:hfbasis},
we consider basis functions defined using spherical barycentric coordinates~\cite{Buss2001,Langer2006,Rustamov2010}.
On each spherical triangle \(\sphericaltriangle \in \generalpartition\)
all elements of \(\hfbasis\) are defined to be zero except for the three basis functions
associated with the vertices of \(\sphericaltriangle\). Denoting these vertices as
\(\SC_1, \SC_2, \SC_3 \in \sphere\), the values of the corresponding basis functions
\(\hfbasiscomp[1], \hfbasiscomp[2], \hfbasiscomp[3]\) are defined by requiring that
\begin{equation*}
  \hfbasiscomp[\basisindex](\SC_\basisindexvar) = \dirac_{\basisindex\basisindexvar}
  \ \ \text{ (Lagrange property) }
\end{equation*}
for \(\basisindex, \basisindexvar \in \Set{1, 2, 3}\), and that,
for every point \(\SC\in\sphericaltriangle\),
\begin{equation*}
  \hfbasiscomp[1]\left(\SC\right)+\hfbasiscomp[2]\left(\SC\right)+\hfbasiscomp[3]\left(\SC\right) = 1 \ \ \text{ (partition of unity) }
\end{equation*}
and
\begin{equation*}
  \SC\in\sphericaltriangle \text{ is the Riemannian center of mass with weights }
  \hfbasiscomp[\basisindex]\left(\SC\right)
  \text{ and nodes } \SC_\basisindex.
\end{equation*}
As in one dimension, the resulting basis functions are non-negative.
Due to the partition of unity property we again have \(\multipliersone[\hfbasis] = {(1, \ldots, 1)}^T\).

The three-dimensional discontinuous partial moment basis \(\pmbasis\) is chosen
analogously to the one-dimensional case as
\begin{equation*}
  \pmbasis = {({(\pmbasis[\sphericaltriangle_0])}^T, \ldots, {(\pmbasis[\sphericaltriangle_{\nentity-1}])}^T)}^T
\end{equation*}
with
\begin{equation*}
  \pmbasis[\sphericaltriangle](\SC)
  = \begin{cases} {\left(1,\SCx,\SCy,\SCz\right)}^T & \text{if } \SC \in \interior{\sphericaltriangle},      \\
              {(0, 0)}^T                        & \text{if } \SC \in \sphere\setminus\sphericaltriangle,
  \end{cases}
\end{equation*}
for any spherical triangle
\(\sphericaltriangle \in \generalpartition = \Set{ \sphericaltriangle_\cellindex \given \cellindex = 0, \ldots, \nentity-1 }\)
Here, \(\interior{\sphericaltriangle}\) is the interior of \(\sphericaltriangle\).
The unit multiplier is given by
\(\multipliersone[\pmbasis] = {(1, 0, 0, 0, 1, 0, 0, 0, \ldots)}^T.\)
\begin{definition}
  The minimum-entropy moment models using the \(\hfbasis\) and \(\pmbasis\) basis
  will be called \(\HFMN\) and \(\PMMN\) models, respectively,
  where \(\momentnumber\) is the number of moments
  (or equivalently, the number of basis functions, i.e.\ the length of the vectors of functions \(\hfbasis\) and \(\pmbasis\)).
\end{definition}

\subsection{Realizability}\label{sec:realizability}
Since the dual problem \eqref{eq:dual} is strictly convex, a solution exists
if and only if the first-order necessary conditions (compare \eqref{eq:psiME}) are fulfilled,
i.e.~\eqref{eq:primal} is
solvable for moment vectors in the ansatz set
\begin{equation}\label{eq:ansatzset}
  \AnsatzSpace[\basis,\entropy] \coloneqq \Set{\ints{\basis \ld{\entropy}'\left(\multipliers \cdot \basis\right) } \given \multipliers \in \R^{\momentnumber}}.
\end{equation}
For Maxwell-Boltzmann entropy, it can be shown~\cite{Junk2000, SchneiderLeibner2020}
that this set is equal to the \emph{positively realizable set}
\begin{equation}\label{eq:posrealizableset}
  \RDpos{\basis}{} \coloneqq \Set{\moments  \in \R^{\momentnumber} \given \exists \distribution \in \Lppos{1}(\angularDomain)
    \text{ such that } \moments =\ints{\basis\distribution} }.
\end{equation}
and that the map \(\multipliers[\basis,\entropy]\colon \RDpos{\basis}{} \to \R^{\momentnumber}\) given by \(\eqref{eq:dual}\)
is a diffeomorphism with inverse map
\begin{equation}\label{eq:alpha_to_u_mapping}
  \moments[\basis,\entropy]\colon \R^{\momentnumber} \to \RDpos{\basis}{}, \quad \moments[\basis,\entropy](\multipliers)
  \coloneqq \ints{\basis \ld{\entropy}'(\multipliers \cdot \basis)}.
\end{equation}
Vectors \(\moments \in \RDpos{\basis}{}\) will be called \emph{realizable}.
Similar to the ansatz function, we will often omit one or all of the subscripts of \(\moments\) and \(\multipliers\)
if the corresponding dependencies are clear from the context.
Motivated by the mapping \eqref{eq:alpha_to_u_mapping}, in the following, we will also refer to the multipliers
\(\multipliers\) as the \emph{entropy variables} (or transformed variables) and to the realizable moments \(\moments\)
as standard or original variables.

The realizable set is a convex cone that, depending on the choice of basis \(\basis\), may have a complicated structure.
For example, a moment vector \(\moments\) is realizable with respect to the full-moment basis \(\fmbasis\)
if some (\(\moments\)-dependent) Hankel matrices
are positive definite~\cite{Curto1991}. This criterion is hard to test in practice, especially for large polynomial order \(\momentorder\).
As a consequence, given a moment vector \(\moments \in \R^{\momentnumber}\), it can be very difficult to check whether
\(\moments\) is realizable and even more difficult to compute a projection to the realizable set.
This is a major problem for numerical solvers
which have to ensure that the approximate solutions stay realizable during the whole solution process
since otherwise the minimum entropy optimization problems are ill-posed.

In contrast, the realizability conditions for the piecewise linear bases are quite simple~\cite{SchneiderLeibner2020}.
In particular, a moment vector is realizable with respect to \(\hfbasis\)
if and only if all its entries are positive~\cite{Schneider2016,SchneiderLeibner2020}:
\begin{equation}\label{eq:hfbasisrealizability}
  \RDpos{\hfbasis}{} \coloneqq \Set{\moments  \in \R^{\momentnumber} \given \momentcomp{i} > 0 \text{ for all } i \in \Set{0, \ldots, \momentnumber-1}}.
\end{equation}
In this case, distinguishing realizable from non-realizable vectors is easy.
Still, as we will see in the next section, also for the piecewise linear bases we have to take some extra measures
(in particular, restrict the time step size) to ensure that the numerical solutions are always realizable.
Moreover, for higher-order numerical schemes or reduced order-models, maintaining realizability at all times (without introducing large errors)
is still challenging, also for the hat function models.

\begin{remark}
Realizability is further complicated by the fact that we usually cannot solve the velocity integrals
analytically and have to approximate them by a numerical quadrature
\(\angularQuadrature\). For the full moments, this can have a severe impact on the realizable set~\cite{Alldredge2012,Alldredge2015},
i.e.\ the numerically realizable set \(\RQ{\basis}\) obtained by replacing the integral in \eqref{eq:posrealizableset} by its quadrature approximation
significantly differs from the realizable set \(\RDpos{\basis}{}\).
In the following, for notational simplicity, we will neglect the quadrature-related complications
and assume that the integrals are evaluated exactly.
\end{remark}

%% file: Sections/standardscheme.tex
\subsection{Standard finite volume discretization}\label{sec:finitevolumescheme}
We consider two discretization schemes for the moment equations~\eqref{eq:MomentSystemClosedIsotropic}, a
standard finite volume scheme presented in this section and a new scheme
based on the identification~\eqref{eq:alpha_to_u_mapping} between realizable set
and \(\R^\momentorder\) (see \cref{sec:newscheme}).
For simplicity, we will restrict ourselves to first-order schemes.

The reference scheme is a standard first-order finite volume scheme.
Let \(\fvgrid = \Set{\fvgridentity_{\gridindex} \given \gridindex \in \gridindexset=\Set{0,\ldots,\gridsize-1}}\) be
a numerical grid for the spatial domain \(\domain\) with \(\gridsize\) elements such that
\begin{equation*}
  \domain = \bigcup_{\gridindex} \fvgridentity_{\gridindex}
\end{equation*}
and define
\begin{equation*}
  \momentscellmean{\gridindex}(t) = \frac{1}{\abs{\fvgridentity_{\gridindex}}} \int_{\fvgridentity_{\gridindex}} \moments(t,\spatialvar) d\spatialvar.
\end{equation*}
Integrating~\eqref{eq:MomentSystemClosedIsotropic} over a grid cell \(\fvgridentity_{\gridindex}\) and dividing by \(\abs{{\fvgridentity}_{\gridindex}}\) gives
\begin{equation*}
  \partial_t \momentscellmean{\gridindex} + \frac{1}{\abs{\fvgridentity_{\gridindex}}} \int_{\fvgridentity_{\gridindex}} \sum_{\dimindex=0}^{\dimension-1}
  \dx[\dimindex]\Flux_{\dimindex}\left(\moments\right) \mathrm{d}\spatialvar
  = \frac{1}{\abs{\fvgridentity_{\gridindex}}} \int_{\fvgridentity_{\gridindex}} \Source\left(\spatialvar,\moments\right).
\end{equation*}
Using the midpoint rule to approximate the source term
\begin{equation*}
  \frac{1}{\abs{\fvgridentity_{\gridindex}}} \int_{\fvgridentity_{\gridindex}} \Source\left(\spatialvar,\moments\right) d\spatialvar
  = \Source\left(\spatialvar_\cellindex,\momentscellmean{\cellindex}\right)+\landauO(\gridwidth^2),
\end{equation*}
where \(\spatialvar_{\cellindex}\) is the centre of grid cell \(\fvgridentity_{\cellindex}\), we arrive at
\begin{equation*}
  \partial_t \momentscellmean{\gridindex} + \frac{1}{\abs{\fvgridentity_{\gridindex}}} \int_{\fvgridentity_{\gridindex}} \sum_{\dimindex=0}^{\dimension-1}
  \dx[\dimindex]\Flux_{\dimindex}\left(\moments\right) \mathrm{d}\spatialvar = \Source\left(\spatialvar_{\gridindex},\momentscellmean{\gridindex}\right).
\end{equation*}
By applying the divergence theorem, we obtain
\begin{equation*}
  \partial_t \momentsfv[\gridindex]
  +  \frac{1}{\abs{\fvgridentity_\gridindex}} \sum_{\gridindexalt\in\gridneighbors{\gridindex}} \int_{\fvgridface_{\gridindex \gridindexalt}} \Fluxmatrix(\moments) \outernormal_{\gridindex \gridindexalt}
  =
  \Source(\spatialvar_{\gridindex},\momentsfv[\gridindex]),
\end{equation*}
where \(\gridneighbors{\gridindex}\) is the set of all indices of neighbors of \(\fvgridentity_\gridindex\),
\(\fvgridface_{\gridindex \gridindexalt} = \fvgridentity_\gridindex \cap \fvgridentity_\gridindexalt\) is the interface between grid cells \(\fvgridentity_{\gridindex}\) and
\(\fvgridentity_{\gridindexalt}\), \(\outernormal_{\gridindex \gridindexalt}\) is the unit outer normal of \(\fvgridentity_\gridindex\)
on \(\fvgridface_{\gridindex \gridindexalt}\) and the
flux matrix is given as \(\Fluxmatrix(\moments) = \left(\Flux_0(\moments), \ldots, \Flux_{\dimension-1}(\moments)\right) \in \R^{\momentnumber \times \dimension}\).

Replacing the flux term by a numerical flux \(\numericalflux_{\gridindex\gridindexalt}\) on \(\fvgridface_{\gridindex\gridindexalt}\), we get the
semidiscrete form
\begin{equation}\label{eq:finitevolumesemidiscrete}
  \partial_t \momentsfv[\gridindex] +  \frac{1}{\abs{\fvgridentity_{\gridindex}}} \sum_{\gridindexalt} \numericalflux_{\gridindex\gridindexalt}(\momentsfv[\gridindex], \momentsfv[\gridindexalt])
  = \Source(\momentsfv[\gridindex]).
\end{equation}
In principle, we could use any numerical flux for hyperbolic equations, e.g.\ the Lax-Friedrichs flux. We will, however, use a numerical flux
which is specifically designed for the equations under consideration. Define the two half integrals
\begin{equation*}
  \intpn{\cdot}{\outernormal} = \int\limits_{\angularDomainp{\outernormal}} \cdot\ \intvar{\SC} \hspace{1cm} \text{ and } \hspace{1cm}
  \intmn{\cdot}{\outernormal} = \int\limits_{\angularDomainm{\outernormal}} \cdot\ \intvar{\SC},
\end{equation*}
where
\begin{equation*}
  \angularDomainp{\outernormal} = \Set{\SC \in \angularDomain \given \SC \cdot \outernormal > 0}, \quad
  \angularDomainm{\outernormal} = \Set{\SC \in \angularDomain \given \SC \cdot \outernormal < 0}.
\end{equation*}
In the following, we will omit the normal vector if it is clear from the context and write, e.g., \(\intp{\cdot}\) instead of
\(\intpn{\cdot}{\outernormal}\) in these cases.
The kinetic flux is defined as~\cite{Schneider2016c,Hauck2010,Garrett2013,Frank2006}
\begin{equation}\label{eq:kineticFlux}
  \kineticflux_{\gridindex\gridindexalt}(\moments[\gridindex], \moments[\gridindexalt])
  = \left( \intpn{ (\SC \cdot \outernormal_{\gridindex\gridindexalt}) \ansatz[{\moments[\gridindex]}] \basis }{\outernormal_{\gridindex\gridindexalt}} +
  \intmn{(\SC \cdot \outernormal_{\gridindex\gridindexalt}) \ansatz[{\moments[\gridindexalt]}] \basis}{\outernormal_{\gridindex\gridindexalt}}\right)
  \abs{\fvgridface_{\gridindex\gridindexalt}}.
\end{equation}
Using the kinetic flux, the semidiscrete form~\eqref{eq:finitevolumesemidiscrete} becomes
\begin{equation}\label{eq:semidiscretekinetic}
  \partial_t \momentsfv[\gridindex] +  \sum_{\gridindexalt\in\gridneighbors{\gridindex}} \frac{\abs{\fvgridface_{\gridindex\gridindexalt}}}{\abs{\fvgridentity_{\gridindex}}}
  \left(\intp{(\SC \cdot \outernormal_{\gridindex\gridindexalt}) \ansatz[{\momentsfv[\gridindex]}] \basis }
  + \intm{(\SC \cdot \outernormal_{\gridindex\gridindexalt}) \ansatz[{\momentsfv[\gridindexalt]}] \basis }\right)
  = \Source(\spatialvar_{\gridindex},\momentsfv[\gridindex]).
\end{equation}
We will then use an explicit one-step scheme for the time discretization. For example, an explicit Euler discretization gives the fully discrete form
\begin{align}\label{eq:fvscheme}
  \begin{split}
    \momentsfv[\gridindex]^{\timeindex+1}
    &= \momentsfv[\gridindex]^{\timeindex} 
    - \timestep \left(\sum_{\gridindexalt\in\gridneighbors{\gridindex}} \frac{\abs{\fvgridface_{\gridindex\gridindexalt}}}{\abs{\fvgridentity_{\gridindex}}}
    \left(\intp{(\SC \cdot \outernormal_{\gridindex\gridindexalt}) \ansatz[{\momentsfv[\gridindex]^{\timeindex}}] \basis } 
      + \intm{(\SC \cdot \outernormal_{\gridindex\gridindexalt}) \ansatz[{\momentsfv[\gridindexalt]^{\timeindex}}] \basis }\right) - 
    \Source(\spatialvar_{\gridindex},\momentsfv[\gridindex]^{\timeindex})\right) \\ 
    &= \momentsfv[\gridindex]^{\timeindex}
    - \timestep \uupdateterm_{\gridindex}(\momentsfv[0]^{\timeindex}, \momentsfv[1]^{\timeindex}, \ldots, \momentsfv[\gridsize-1]^{\timeindex}), \\
  \end{split}
\end{align}
where \(\momentsfv[\gridindex]^{\timeindex}\) is the approximation of \(\momentsfv[\gridindex]\) at time step \(\timeindex\)
and we defined
\begin{equation}\label{eq:uupdateterm}
  \uupdateterm_{\gridindex}(\moments[0]^{\timeindex}, \moments[1]^{\timeindex}, \ldots, \moments[\gridsize-1]^{\timeindex})
  \coloneqq \sum_{\gridindexalt\in\gridneighbors{\gridindex}} \frac{\abs{\fvgridface_{\gridindex\gridindexalt}}}{\abs{\fvgridentity_{\gridindex}}}
  \left(\intp{(\SC \cdot \outernormal_{\gridindex\gridindexalt}) \ansatz[{\moments[\gridindex]^{\timeindex}}] \basis }
  + \intm{(\SC \cdot \outernormal_{\gridindex\gridindexalt}) \ansatz[{\moments[\gridindexalt]^{\timeindex}}] \basis }\right) -
  \Source(\spatialvar_{\gridindex}, \moments[\gridindex]^{\timeindex}).
\end{equation}

The scheme~\eqref{eq:fvscheme} requires the solution of the minimization problem~\eqref{eq:primal2} in every time step on each grid cell.
The initial values thus have to be realizable and we have to limit the time step \(\timestep\) to ensure that the scheme yields realizable moments.

\begin{theorem}\label{thm:timestepping}
  The numerical scheme~\eqref{eq:fvscheme} using a structured cubic grid with equally-sized grid cells with edge length \(\gridwidth\)
  is realizability-preserving under the CFL-like condition
  \begin{equation}\label{eq:fvschemetimestep}
    \timestep < \frac{1}{\crosssectionmax + \frac{\sqrt{\dimension}}{\gridwidth}}
  \end{equation}
  where \(\crosssectionmax = \max\limits_{\spatialvar \in \domain} \crosssection(\spatialvar)\).
\end{theorem}
\begin{proof}
  We will generalize the proof of~\cite[Corollary 3.17]{Schneider2016} to several dimensions.
  Let \(\momentsfv[\gridindex]^{\timeindex}\) be realizable. 
  By~\eqref{eq:semidiscretekinetic}, \eqref{eq:sourcedef}
  and the definition of \(\isomatrix\) (compare~\eqref{eq:MomentSystemIsotropicScatteringTerm}), we have 
  \begin{align*}
    \momentsfv[\gridindex]^{\timeindex+1} 
     & = \bigg\langle\basis \bigg(\ansatz[{\momentsfv[\gridindex]^{\timeindex}}] 
    - \timestep\bigg(\sum_{\gridindexalt\in\gridneighbors{\gridindex}} \frac{\abs{\fvgridface_{\gridindex\gridindexalt}}}{\abs{\fvgridentity_{\gridindex}}}
    \left({\left(\SC \cdot \outernormal_{\gridindex\gridindexalt}\right)}^{+} \ansatz[{\momentsfv[\gridindex]^{\timeindex}}] 
    + {\left(\SC \cdot \outernormal_{\gridindex\gridindexalt}\right)}^{-} \ansatz[{\momentsfv[\gridindexalt]^{\timeindex}}]\right)  
    -
    \frac{\scattering}{\abs{\angularDomain}} \ints{\ansatz[{\momentsfv[\gridindex]^{\timeindex}}]} - \source + \crosssection \ansatz[{\momentsfv[\gridindex]^{\timeindex}}] 
    \bigg)\bigg)\bigg\rangle                                                     \\
     & \eqqcolon \ints{\basis \distribution[\gridindex]^{\timeindex+1}},         
  \end{align*}
  where \({\left(\SC \cdot \outernormal_{\gridindex\gridindexalt}\right)}^{+} = \max(\SC \cdot \outernormal_{\gridindex\gridindexalt}, 0)\) and
  \({\left(\SC \cdot \outernormal_{\gridindex\gridindexalt}\right)}^{-} = \min(\SC \cdot \outernormal_{\gridindex\gridindexalt}, 0)\).
  We have to show that
  \(\distribution[\gridindex]^{\timeindex+1}\) 
  is positive for all \(\SC \in \angularDomain\) under the
  time step restriction~\eqref{eq:fvschemetimestep}.  Neglecting non-negative terms (remember that \(\scattering\), \(\crosssection\)
  and \(\source\) are assumed to be non-negative), we arrive at
  \begin{equation}\label{eq:CFLproof1}
    \distribution[\gridindex]^{\timeindex+1} 
    \geq \left(1 - \timestep\left(\crosssection + \sum_{\gridindexalt\in\gridneighbors{\gridindex}} \frac{\abs{\fvgridface_{\gridindex\gridindexalt}}}{\abs{\fvgridentity_{\gridindex}}}
      {\left(\SC \cdot \outernormal_{\gridindex\gridindexalt}\right)}^{+} \right)\right) \ansatz[{\momentsfv[\gridindex]^{\timeindex}}]. 
  \end{equation}
  For our uniform equidistant grid, we have
  \(\frac{\abs{\fvgridface_{\gridindex\gridindexalt}}}{\abs{\fvgridentity_{\gridindex}}} = \frac{1}{\gridwidth}\) for all \((\gridindex, \gridindexalt)\)
  and
  \begin{equation*}
    \sum_{\gridindexalt\in\gridneighbors{\gridindex}}
    {\left(\SC \cdot \outernormal_{\gridindex\gridindexalt}\right)}^{+} = \norm{\SC}{1} \leq \sqrt{\dimension}\norm{\SC}{2} \leq \sqrt{\dimension},
  \end{equation*}
  where we used that \(\norm{\SC}{2} \leq 1\) for all \(\SC \in \angularDomain\) holds true both for \(\angularDomain = [-1,1]\) and for \(\angularDomain = \sphere\).
  Consequently,~\eqref{eq:CFLproof1} becomes
  \begin{equation*}
    \distribution[\gridindex]^{\timeindex+1} 
    \geq \left(1 - \timestep\left(\crosssection + \frac{\sqrt{\dimension}}{\gridwidth} \right)\right) \ansatz[{\momentsfv[\gridindex]^{\timeindex}}], 
  \end{equation*}
  which is positive if \(\eqref{eq:fvschemetimestep}\) holds.
\end{proof}
The time step restriction due to the cross-section \(\crosssection\) can be avoided, e.g.\ by using implicit-explicit methods
where the source term is treated implicitly~\cite{Schneider2016a,Schneider2016aCodeIMEX,schneider2016implicit,SchneiderLeibner2019b}.
As in~\cite{SchneiderLeibner2019b}, we will use a second-order Strang splitting scheme for the split system
\begin{subequations}\label{eq:splittedsystem}
  \begin{align}
    \label{eq:SplittedSystema}
    \dt\momentsfv[\gridindex] & = -\frac{1}{\abs{\fvgridentity_{\gridindex}}}
    \sum_{\gridindexalt\in\gridneighbors{\gridindex}} \kineticflux_{\gridindex\gridindexalt}(\momentsfv[\gridindex], \momentsfv[\gridindexalt]), \\
    \label{eq:SplittedSystemb}
    \dt\momentsfv[\gridindex] & = \Source\left(\spatialvar,\momentsfv[\gridindex]\right),
  \end{align}
\end{subequations}
i.e., in each time step from \(\timevar\) to \(\timevar + \timestep\) we first solve the (linear) source system~\eqref{eq:SplittedSystemb} analytically
(see \cref{sec:implementationdetails})
up to the time \(\timevar + \frac{\timestep}{2}\), then we use the result as input to solve
the hyperbolic part~\eqref{eq:SplittedSystema} with a full timestep \(\timestep\),
then we again advance~\eqref{eq:SplittedSystemb} analytically for a half time step \(\frac{\timestep}{2}\). As the source system is solved analytically (and thus preserves realizability), we only have
to ensure that realizability is preserved when solving the hyperbolic part. We can thus avoid the time step restriction due to the cross-section \(\crosssection\).
\begin{corollary}\label{cor:CFLfirstorderscheme}
  The splitting scheme based on~\eqref{eq:splittedsystem} is realizability-preserving under the CFL-like condition
  \begin{equation}\label{eq:splittingschemetimestep}
    \timestep < \frac{\gridwidth}{\sqrt{\dimension}}.
  \end{equation}
\end{corollary}
Note that we assumed in the proof of \cref{thm:timestepping} that the optimization problems are solved exactly (by using
the exact ansatz functions \(\ansatz[{\momentsfv[\gridindex]^{\timeindex}}]\)).
In practice, we have to solve these problems numerically
which inevitably introduces some numerical errors. Fortunately,
we can account for these inexact solutions by using a slightly tighter time step restriction
as long as we can control the relative error in the ansatz function (compare~\cite{Alldredge2012,Alldredge2014,SchneiderLeibner2019b}).
\begin{corollary}\label{cor:inexactsolutions}
  Let \(\opttoleps\in(0,1)\).
  If the numerical solution \(\numericalmultipliers(\moments)\) of the
  minimum-entropy problem~\eqref{eq:dual}
  fulfils
  \begin{equation}\label{eq:CFLcondcorollaryassumption}
    \gamma(\velocityvar) \coloneqq \frac{\ansatz[\moments](\velocityvar)}{\ld{\entropy}'(\basis(\velocityvar) \cdot \numericalmultipliers(\moments))} \geq 1 - \opttoleps
  \end{equation}
  for all \(\moments\) for which the problem has to be solved during the finite volume scheme,
  then \cref{cor:CFLfirstorderscheme}
  still holds if the time step restrictions are tightened to
  \begin{equation}\label{eq:CFLfirstorderschemeinexact}
    \timestep < \frac{1-\opttoleps}{\sqrt{\dimension}}\gridwidth.
  \end{equation}
\end{corollary}
We can ensure that \eqref{eq:CFLcondcorollaryassumption} holds
by appropriately choosing the stopping criterion for the Newton scheme that is used to solve the
minimum entropy optimization problems (see \cref{sec:implementationdetails}).
We will always use \(\opttoleps = 0.1\) which means that
we scale the time step restrictions \eqref{eq:splittingschemetimestep} by a safety factor of 0.9.

If we advance the hyperbolic system~\eqref{eq:SplittedSystema} in time by strong-stability
preserving Runge-Kutta schemes~\cite{Gottlieb2005} the CFL condition~\eqref{eq:splittingschemetimestep} still holds as these schemes consist
of convex combinations of forward Euler steps. In particular, we will use Heun's method in all tests which is of second order.

%% file: Sections/newscheme.tex
\section{New scheme in transformed variables}\label{sec:newscheme}
We will now present the new scheme in transformed variables which uses the identification between the realizable set \(\RDpos{\basis}\) and \(\R^{\momentnumber}\)
given by the diffeomorphism~\eqref{eq:alpha_to_u_mapping}.

\subsection{Semidiscrete formulation}
To derive the transformed scheme, note that
\begin{equation}\label{eq:du_dalpha}
  \frac{\derivative\moments}{\derivative\multipliers}
  = \frac{\derivative}{\derivative\multipliers} \left( \ints{\basis \ansatz[\moments]}  \right)
  \overset{\eqref{eq:psiME}}{=} \frac{\derivative}{\derivative\multipliers} \left( \ints{\basis \ld{\entropy}'(\multipliers \cdot  \basis)}	\right)
  = \ints{ \basis \basis^T \ld{\entropy}''(\multipliers \cdot  \basis) }  = \optHessian(\multipliers)
\end{equation}
is the (positive definite) Hessian of the objective function in the dual problem \eqref{eq:dual} (compare \cref{sec:optimizationproblem}).
Further, for the flux
\begin{equation}
  \Flux_{\dimindex}(\multipliers) = \ints{\SCcomp[\dimindex] \basis \ld{\entropy}'(\multipliers \cdot \basis)}
\end{equation}
(compare~\eqref{eq:fluxfuncdef}) we have
\begin{equation}\label{eq:velocityHessianDef}
  \frac{\derivative\Vf_{\dimindex}}{\derivative\multipliers} = \ints{ \SCcomp[\dimindex] \basis \basis^T \ld{\entropy}''(\multipliers \cdot  \basis) }
  \eqqcolon \velocityHessian_{\dimindex}(\multipliers).
\end{equation}
In transformed variables, assuming \(\moments\) is sufficiently smooth, the hyperbolic system of equations~\eqref{eq:MomentSystemClosedIsotropic} thus
becomes~\cite{Levermore1996}
\begin{align}\label{eq:symmetrichyperbolicform}
  \begin{split}
    \Source\left(\spatialvar,\moments(\multipliers)\right) &=
    \dt\moments(\multipliers)+\sum_{\dimindex=0}^{\dimension-1}\dx[\dimindex]\Flux_\dimindex\left(\moments(\multipliers)\right)
    = \frac{\derivative\moments}{\derivative\multipliers}(\multipliers) \dt \multipliers
    + \sum_{\dimindex=0}^{\dimension-1}\frac{\derivative\Flux_\dimindex}{\derivative\multipliers}(\multipliers) \dx[\dimindex] \multipliers\\
    &= \optHessian(\multipliers) \dt \multipliers
    + \sum_{\dimindex=0}^{\dimension-1}\velocityHessian_{\dimindex}(\multipliers) \dx[\dimindex] \multipliers.
  \end{split}
\end{align}
A numerical scheme based on the form~\eqref{eq:symmetrichyperbolicform} could potentially be much faster than the standard
finite volume scheme~\eqref{eq:fvscheme}
as it avoids solving the non-linear optimization problem and only needs inversion of the positive definite symmetric matrix
\(\optHessian(\multipliers)\).
However,~\eqref{eq:symmetrichyperbolicform} is not in conservation form which makes it hard to guarantee that numerical schemes
converge to the correct weak solution of~\eqref{eq:MomentSystemClosedIsotropic}.

On the other hand, if we perform the space discretization first and then
transform the semi-discrete equation~\eqref{eq:finitevolumesemidiscrete} to the new variables, we arrive at
\begin{equation}\label{eq:transformedschemesemidiscrete}
  \optHessian(\multipliersfv[\gridindex]) \partial_t \multipliersfv[\gridindex] + \frac{1}{\abs{\fvgridentity_{\gridindex}}} \sum_{\gridindexalt\in\gridneighbors{\gridindex}}
  \numericalflux_{\gridindex\gridindexalt}(\moments(\multipliersfv[\gridindex]), \moments(\multipliersfv[\gridindexalt]))
  = \Source(\spatialvar_{\gridindex},\moments(\multipliersfv[\gridindex])),
\end{equation}
where \(\multipliersfv[\gridindex] = \multipliers(\momentsfv[\gridindex])\) are the multipliers corresponding to the
finite volume averaged moment \(\momentsfv[\gridindex]\) via the inverse diffeomorphism~\eqref{eq:dual}.

Since the space discretization is inherited from the conservative form \eqref{eq:finitevolumesemidiscrete},
we now only have to discretize in time and can expect that the moments
of the solution converge to the corresponding solution of the non-transformed scheme.
In the following, we will show that this is indeed the case (see \cref{sec:timediscretization}).

\begin{remark}
  A similar idea has been used in~\cite{Schaerer2017} to efficiently solve an semi-implicit
  version of the standard finite volume
  scheme~\eqref{eq:fvscheme} with Lax-Friedrichs flux.
  In fact, the scheme presented in the following is equivalent to using the approach in~\cite{Schaerer2017}
  with explicit time discretization and only performing a single step of the Newton iteration.
  The authors in \cite{Schaerer2017} restrict their investigation to slab geometry
  and note that further research is needed to examine the efficiency of their scheme
  in higher dimensions where the
  Jacobians of the coupled system are not block-tridiagonal anymore.
  In contrast, using the fully explicit approach considered here
  the grid cells decouple and
  the equations can be solved independently
  for each grid cell, also in several dimensions.
\end{remark}

\subsection{Entropy stability on the semi-discrete level}\label{sec:entropystabilitysemidiscrete}
On the semidiscrete level, the new scheme is just a variable transformation of the standard scheme. We can thus show that the
solutions of \eqref{eq:transformedschemesemidiscrete} also inherit a semidiscrete version of the
entropy-dissipation property \eqref{eq:entropyeq}.
\begin{theorem}\label{thm:entropydissipationsemidiscrete}
  Let \(\Set{ \multipliersfv[\gridindex](\timevar) \given \gridindex \in \gridindexset}\) be a solution
  of the transformed semidiscrete equation~\eqref{eq:transformedschemesemidiscrete} on a (hyper)rectangular grid.
  Then we have
  \begin{equation}\begin{split}\label{eq:entropyeqsemidiscrete}
      \begin{multlined}[t][.8\displaywidth]
        \partial_t \entropyFunctional(\ansatz[{\moments(\multipliersfv[\gridindex])}])
        + \sum_{\gridindexalt\in\gridneighbors{\gridindex}} \frac{\abs{\fvgridface_{\gridindex\gridindexalt}}}{\abs{\fvgridentity_{\gridindex}}}
        \left(\intp{(\SC \cdot \outernormal_{\gridindex\gridindexalt})
            \entropy(\ansatz[{\moments(\multipliersfv[\gridindex])}])}
        + \intm{(\SC \cdot \outernormal_{\gridindex\gridindexalt})
            \entropy(\ansatz[{\moments(\multipliersfv[\gridindexalt])}])}\right) \\
        \leq -\absorption \ints{ (\multipliersfv[\gridindex] \cdot \basis) \ld{\entropy}'(\multipliersfv[\gridindex] \cdot \basis)}
        + \ints{ (\multipliersfv[\gridindex] \cdot \basis) \source}.
      \end{multlined}
    \end{split}
  \end{equation}
\end{theorem}
\begin{proof}
  The negativity of the entropy contribution by the scattering term
  \(
  \scattering (\isomatrix - \eyematrix) \moments
  \)
  can be shown
  exactly as in the continuous case, see, e.g., \cite{Schneider2016}.
  The remaining part \(-\absorption \moments + \ints{\basis \source}\) of the source term
  directly gives the right-hand side
  of~\eqref{eq:entropyeqsemidiscrete} after multiplication by \(\multipliersfv[\gridindex]\) (as done below).
  In the following, for notational simplicity, we will thus only regard the flux term. To that end, first note that
  \begin{align}\begin{split}\label{eq:entropytimederivative}
      \partial_t \entropyFunctional(\ansatz[{\moments(\multipliers)}])
      &= \partial_t \ints{\entropy(\ansatz[{\moments(\multipliers)}])}
      = \partial_t \ints{\entropy \circ \ld{\entropy}'(\multipliers \cdot \basis)} \\
      &= \ints{ \entropy' \circ \ld{\entropy}'(\multipliers \cdot \basis) \ld{\entropy}''(\multipliers \cdot  \basis) \basis } \cdot \partial_t \multipliers
      = \ints{ (\multipliers \cdot \basis) \ld{\entropy}''(\multipliers \cdot  \basis) \basis } \cdot \partial_{\timevar} \multipliers, \end{split}
  \end{align}
  where we used that \(\entropy' \circ \ld{\entropy}'\) is the identity by definition of the Legendre transform (see, e.g.,~\cite{Rockafellar1970}).
  Further,
  the Legendre transform fulfills the relation
  \begin{equation}\label{eq:legendrerelation}
    p \ld{\entropy}'(p) - \ld{\entropy}(p) = \entropy(\ld{\entropy}'(p)).
  \end{equation}
  Multiplying
  the semi-discrete equation \eqref{eq:transformedschemesemidiscrete}
  by \(\multipliersfv[\gridindex]\), we thus obtain (using the kinetic flux \eqref{eq:kineticFlux} and neglecting the source term)
  \begingroup
  \allowdisplaybreaks
  \begin{align*}
    0
                                               & =
    \multipliersfv[\gridindex] \cdot \optHessian(\multipliersfv[\gridindex]) \partial_t \multipliersfv[\gridindex]
    + \multipliersfv[\gridindex] \cdot \frac{1}{\abs{\fvgridentity_{\gridindex}}} \sum_{\gridindexalt\in\gridneighbors{\gridindex}}
    \numericalflux_{\gridindex\gridindexalt}(\moments(\multipliersfv[\gridindex]), \moments(\multipliersfv[\gridindexalt])) \\
                                               & = \begin{multlined}[t][.9\displaywidth]
      \multipliersfv[\gridindex] \cdot \ints{ \basis \basis^T \ld{\entropy}''(\multipliersfv[\gridindex]\cdot  \basis) }
      \partial_t \multipliersfv[\gridindex] \\
      + \multipliersfv[\gridindex] \cdot
      \left(\sum_{\gridindexalt\in\gridneighbors{\gridindex}} \frac{\abs{\fvgridface_{\gridindex\gridindexalt}}}{\abs{\fvgridentity_{\gridindex}}}
      \left(\intp{(\SC \cdot \outernormal_{\gridindex\gridindexalt}) \ld{\entropy}'(\multipliersfv[\gridindex]\cdot \basis) \basis }
        + \intm{(\SC \cdot \outernormal_{\gridindex\gridindexalt}) \ld{\entropy}'(\multipliersfv[\gridindexalt]\cdot \basis) \basis }\right)\right)
    \end{multlined}                                               \\
                                               & = \begin{multlined}[t][.9\displaywidth]
      \ints{ (\multipliersfv[\gridindex] \cdot \basis) \ld{\entropy}''(\multipliersfv[\gridindex]\cdot  \basis) \basis }
      \cdot \partial_t \multipliersfv[\gridindex] \\
      + \left(\sum_{\gridindexalt\in\gridneighbors{\gridindex}} \frac{\abs{\fvgridface_{\gridindex\gridindexalt}}}{\abs{\fvgridentity_{\gridindex}}}
      \left(\intp{(\SC \cdot \outernormal_{\gridindex\gridindexalt}) (\multipliersfv[\gridindex] \cdot \basis)
            \ld{\entropy}'(\multipliersfv[\gridindex]\cdot \basis)}
        + \intm{(\SC \cdot \outernormal_{\gridindex\gridindexalt})
            (\multipliersfv[\gridindex] \cdot \basis) \ld{\entropy}'(\multipliersfv[\gridindexalt]\cdot \basis)
          }\right)\right)
    \end{multlined}                                               \\
    \overset{\eqref{eq:entropytimederivative}} & = \begin{multlined}[t][.9\displaywidth]
      \partial_t \entropyFunctional(\ansatz[{\moments(\multipliersfv[\gridindex])}]) \\
      + \Bigg(\sum_{\gridindexalt\in\gridneighbors{\gridindex}} \frac{\abs{\fvgridface_{\gridindex\gridindexalt}}}{\abs{\fvgridentity_{\gridindex}}}
      \Big(\intp{(\SC \cdot \outernormal_{\gridindex\gridindexalt}) (\multipliersfv[\gridindex] \cdot \basis)
            \ld{\entropy}'(\multipliersfv[\gridindex]\cdot \basis)}
        + \intm{(\SC \cdot \outernormal_{\gridindex\gridindexalt})
            (\multipliersfv[\gridindexalt] \cdot \basis) \ld{\entropy}'(\multipliersfv[\gridindexalt]\cdot \basis)} \\
        +
        \intm{(\SC \cdot \outernormal_{\gridindex\gridindexalt})
            ((\multipliersfv[\gridindex] -  \multipliersfv[\gridindexalt])\cdot \basis) \ld{\entropy}'(\multipliersfv[\gridindexalt]\cdot \basis)
          }
        \Big)\Bigg)
    \end{multlined}                                               \\
    \overset{\eqref{eq:legendrerelation}}      & = \begin{multlined}[t][.9\displaywidth]
      \partial_t \entropyFunctional(\ansatz[{\moments(\multipliersfv[\gridindex])}])
      + \sum_{\gridindexalt\in\gridneighbors{\gridindex}} \frac{\abs{\fvgridface_{\gridindex\gridindexalt}}}{\abs{\fvgridentity_{\gridindex}}}
      \Bigg(\intp{(\SC \cdot \outernormal_{\gridindex\gridindexalt})
          \entropy \circ \ld{\entropy}'(\multipliersfv[\gridindex]\cdot \basis)}
      + \intp{(\SC \cdot \outernormal_{\gridindex\gridindexalt})
          \ld{\entropy}(\multipliersfv[\gridindex]\cdot \basis)} \\
      \hspace{1cm}+ \intm{(\SC \cdot \outernormal_{\gridindex\gridindexalt})
          \entropy \circ \ld{\entropy}'(\multipliersfv[\gridindexalt]\cdot \basis)}
      + \intm{(\SC \cdot \outernormal_{\gridindex\gridindexalt})
          \ld{\entropy}(\multipliersfv[\gridindexalt]\cdot \basis)} \\
      + \intm{(\SC \cdot \outernormal_{\gridindex\gridindexalt})
          ((\multipliersfv[\gridindex] -  \multipliersfv[\gridindexalt])\cdot \basis) \ld{\entropy}'(\multipliersfv[\gridindexalt]\cdot \basis)
        }
      \Bigg)
    \end{multlined}                                               \\
                                               & = \begin{multlined}[t][.9\displaywidth]
      \partial_t \entropyFunctional(\ansatz[{\moments(\multipliersfv[\gridindex])}])
      + \sum_{\gridindexalt\in\gridneighbors{\gridindex}} \frac{\abs{\fvgridface_{\gridindex\gridindexalt}}}{\abs{\fvgridentity_{\gridindex}}}
      \Bigg(\intp{(\SC \cdot \outernormal_{\gridindex\gridindexalt})
          \entropy(\ansatz[{\moments(\multipliersfv[\gridindex])}])}
      + \intm{(\SC \cdot \outernormal_{\gridindex\gridindexalt})
          \entropy(\ansatz[{\moments(\multipliersfv[\gridindexalt])}])} \\
      + \intp{(\SC \cdot \outernormal_{\gridindex\gridindexalt})
          \ld{\entropy}(\multipliersfv[\gridindex]\cdot \basis)} \\
      + \intm{(\SC \cdot \outernormal_{\gridindex\gridindexalt})
          \left(\ld{\entropy}(\multipliersfv[\gridindexalt]\cdot \basis) +
          ((\multipliersfv[\gridindex] -  \multipliersfv[\gridindexalt])\cdot \basis) \ld{\entropy}'(\multipliersfv[\gridindexalt]\cdot \basis)\right)
        }
      \Bigg)
    \end{multlined}                                               \\
                                               & \geq \begin{multlined}[t][.9\displaywidth]
      \partial_t \entropyFunctional(\ansatz[{\moments(\multipliersfv[\gridindex])}])
      + \sum_{\gridindexalt\in\gridneighbors{\gridindex}} \frac{\abs{\fvgridface_{\gridindex\gridindexalt}}}{\abs{\fvgridentity_{\gridindex}}}
      \Bigg(\intp{(\SC \cdot \outernormal_{\gridindex\gridindexalt})
          \entropy(\ansatz[{\moments(\multipliersfv[\gridindex])}])}
      + \intm{(\SC \cdot \outernormal_{\gridindex\gridindexalt})
          \entropy(\ansatz[{\moments(\multipliersfv[\gridindexalt])}])} \\
      + \ints{(\SC \cdot \outernormal_{\gridindex\gridindexalt})
          \ld{\entropy}(\multipliersfv[\gridindex]\cdot \basis)}\Bigg),
    \end{multlined}
  \end{align*}
  \endgroup
  where we used for the estimate
  that the integral \(\intm{\cdot}\) is defined such that
  \(\SC \cdot \outernormal_{\gridindex\gridindexalt}\) is negative
  and
  that \(\multipliers \mapsto \ld{\entropy}(\multipliers \cdot \basis)\) is convex
  (and thus
  \(\ld{\entropy}(\multipliersfv[\gridindexalt]\cdot \basis) +
  ((\multipliersfv[\gridindex] -  \multipliersfv[\gridindexalt])\cdot \basis) \ld{\entropy}'(\multipliersfv[\gridindexalt]\cdot \basis)
  \leq \ld{\entropy}(\multipliersfv[\gridindex]\cdot \basis)\)).
  For a hyperrectangular grid, for each interface of grid cell \(\fvgridentity_\gridindex\) with outer normal \(\outernormal_{\gridindex\gridindexalt}\)
  the opposite interface has the same area and outer normal \(-\outernormal_{\gridindex\gridindexalt}\). Thus, the
  \(\ints{(\SC \cdot \outernormal_{\gridindex\gridindexalt})
    \ld{\entropy}(\multipliersfv[\gridindex]\cdot \basis)}\)
  terms cancel out which finally gives \eqref{eq:entropyeqsemidiscrete}.
\end{proof}

\begin{remark}
  The entropy density \(\entropy\) is strictly convex and thus attains its
  minimum at \(\distribution\) with \(\entropy'(\distribution) = 0\).
  Consequently, for \(\distribution = \ansatz[\moments(\multipliers)]\),
  the entropy density is minimal if
  \(0 = \entropy'(\ansatz[\moments(\multipliers)]) = \entropy' \circ \ld{\entropy}'(\multipliers \cdot \basis) = \multipliers \cdot \basis\).
  This explains why absorption decreases the entropy for \(\multipliers \cdot \basis > 0\) and increases the
  entropy for \(\multipliers \cdot \basis < 0\) (see the right-hand side of~\eqref{eq:entropyeqsemidiscrete}).
  Similarly, a source of particles \(\source(\velocityvar)\) with velocity \(\velocityvar\) increases the entropy if
  \(\multipliers \cdot \basis(\velocityvar) > 0\) and decreases it otherwise.
\end{remark}

\begin{remark}
  Since the two schemes are related by the transformation \eqref{eq:alpha_to_u_mapping} on the semidiscrete level,
  the entropy dissipation law~\eqref{eq:entropyeqsemidiscrete} also holds for the standard unsplit finite volume scheme
  if we replace \(\moments(\multipliersfv[\gridindex])\) by \(\momentsfv[\gridindex]\) and \(\multipliersfv[\gridindex]\)
  by \(\multipliers(\momentsfv[\gridindex])\).
\end{remark}

\subsection{Time discretization}\label{sec:timediscretization}
To get a fully discrete numerical scheme, we still have to choose a time discretization for \eqref{eq:transformedschemesemidiscrete}.
To avoid having to solve a large coupled non-linear system of equations in each time step, we will only consider explicit schemes.
Using, for example, the explicit Euler scheme and the kinetic flux~\eqref{eq:kineticFlux}, the fully discrete form
of~\eqref{eq:transformedschemesemidiscrete} becomes
\begin{align}
  \begin{split}\label{eq:newschemeeuler}
    \multipliersfv[\gridindex]^{\timeindex+1} 
    &=
    \begin{multlined}[t][.8\displaywidth]
      \multipliersfv[\gridindex]^{\timeindex} + \timestep 
      {\optHessian(\multipliersfv[\gridindex]^{\timeindex})}^{-1} 
      \bigg(\Source(\spatialvar_{\gridindex},\moments(\multipliersfv[\gridindex]^{\timeindex})) \\ 
      -  \sum_{\gridindexalt\in\gridneighbors{\gridindex}}
      \frac{\abs{\fvgridface_{\gridindex\gridindexalt}}}{\abs{\fvgridentity_{\gridindex}}}
      \left(\intp{(\SC \cdot \outernormal_{\gridindex\gridindexalt})
            \ld{\entropy}'(\multipliersfv[\gridindex]^{\timeindex} \cdot \basis) \basis } 
        + \intm{(\SC \cdot \outernormal_{\gridindex\gridindexalt})
            \ld{\entropy}'(\multipliersfv[\gridindexalt]^{\timeindex} \cdot \basis) \basis }\right)\bigg) 
    \end{multlined} \\
    &= \multipliersfv[\gridindex]^{\timeindex} + \timestep 
    {\optHessian(\multipliersfv[\gridindex]^{\timeindex})}^{-1} 
    \uupdateterm_{\gridindex}(\moments(\multipliersfv[0]^{\timeindex}),
    \moments(\multipliersfv[1]^{\timeindex}), \ldots,
    \moments(\multipliersfv[\gridsize-1]^{\timeindex}))
    \eqqcolon \multipliersfv[\gridindex]^{\timeindex} + \timestep 
    {\optHessian(\multipliersfv[\gridindex]^{\timeindex})}^{-1} 
    \uupdatetermexplicit_{\gridindex},
  \end{split}
\end{align}
where the update term \(\uupdateterm_{\gridindex}\) has been defined in~\eqref{eq:uupdateterm}.
However, a time discretization using fixed time steps might not be a suitable choice for the transformed scheme.

\begin{example}\label{ex:initialtimesteps}
  Consider exemplarily the plane-source test (compare \cref{sec:numericalexperiments})
  using an equidistant grid with \(\gridsizeoned = 240\) elements for the domain \(\domain = [-1.2, 1.2]\).
  We thus have \(\gridwidth = \frac{2.4}{240} = 10^{-2}\).
  The initial value in this test is a small isotropic vacuum density \(\distributionvacuum = 5 \cdot 10^{-7}\) plus a Dirac delta at \(\spatialvaronedimension=0\)
  which is split into the grid cells
  \(\fvgridentity_{\text{left}} = \fvgridentity_{119}\) and \(\fvgridentity_{\text{right}} = \fvgridentity_{120}\) adjacent to \(0\).
  The initial values for the finite volume scheme thus are given as
  \begin{equation*}
    \momentsfv[119]^0 = \momentsfv[120]^0 
    = \frac{1}{\gridwidth}
    \int\limits_{0}^{\gridwidth} \ints{\basis(\velocityvaronedimension) \distributiontzero(\spatialvaronedimension,\velocityvaronedimension)} \mathrm{d}\spatialvaronedimension
    = \frac{1}{\gridwidth}
    \int\limits_{0}^{\gridwidth} \ints{\basis(\velocityvaronedimension) \big(\distributionvacuum + \dirac_0(\spatialvaronedimension)\big)} \mathrm{d}\spatialvaronedimension
    =
    (\distributionvacuum + \frac{1}{2\gridwidth})\ints{\basis} \approx 50 \ints{\basis}
  \end{equation*}
  and
  \(\momentsfv[\gridindex]^0 
  = \distributionvacuum \ints{\basis} = 5 \cdot 10^{-7} \ints{\basis}\)
  for \(\gridindex \notin \{119,120\}\).
  The initial values are thus isotropic with local
  particle densities
  \(\density_{\gridindex} \coloneqq \density(\momentsfv[\gridindex]^0)\) 
  given as \(\density_{119} = \density_{120} = \ints{50} = 100\)
  and \(\density_{\gridindexalt} = \ints{5 \cdot 10^{-7}} = 10^{-6}\) for \(\gridindexalt \notin \{119,120\}\).
  The multipliers corresponding to the isotropic moment
  with local particle density \(\density\) are
  \begin{equation}\label{eq:isotropicmultipliers}
    \multipliersiso_{\basis}(\density) = \mathop{\entropy'}\left(\frac{\density}{\ints{1}}\right) \multipliersone
  \end{equation}
  with corresponding ansatz function
  \begin{equation*}
    \ansatziso_{\basis}(\density) = \mathop{\ld{\entropy}'}\left(\multipliersiso_{\basis}(\density) \cdot \basis\right)
    = \mathop{\ld{\entropy}'}\left(\mathop{\entropy'}\left(\frac{\density}{\ints{1}}\right) \multipliersone \cdot \basis\right)
    = \mathop{\ld{\entropy}'}\left(\mathop{\entropy'}\left(\frac{\density}{\ints{1}}\right)\right)
    = \frac{\density}{\ints{1}},
  \end{equation*}
  where we used again that \(\entropy'\) and \(\ld{\entropy}'\) are inverse functions.
  If we now focus on cell \(\fvgridentity_{121}\) and ignore the source term \(\Source(\spatialvar_{121},\moments(\multipliersfv[121]))\),
  the update~\eqref{eq:newschemeeuler} takes the form
  \begin{align*}
    \multipliersfv[121]^{1} 
     & = \multipliersfv[121]^{0}                                                                                             
    - \timestep {\optHessian(\multipliersfv[121])}^{-1}\left(
    \sum_{\gridindexalt\in \{120,122\}}
    \frac{1}{\gridwidth}
    \left(\intp{\abs{\velocityvaronedimension}\basis \ld{\entropy}'(\multipliersfv[121] \cdot \basis) } 
    + \intm{-\abs{\velocityvaronedimension}\basis \ld{\entropy}'(\multipliersfv[\gridindexalt] \cdot \basis) }\right)\right) \\ 
     & = \multipliersfv[121]^{0}                                                                                             
    - \timestep {\optHessian(\multipliersfv[121])}^{-1}\left(
    \sum_{\gridindexalt\in \{120,122\}}
    \frac{1}{\gridwidth}
    \left(\intp{\abs{\velocityvaronedimension}\basis \frac{\density_{121}}{\ints{1}} }
    + \intm{-\abs{\velocityvaronedimension}\basis \frac{\density_{\gridindexalt}}{\ints{1}}}\right)\right)                   \\
     & \approx
    \multipliersfv[121]^{0} 
    + \timestep {\optHessian(\multipliersfv[121])}^{-1}\left(
    \frac{1}{\gridwidth}
    \intm{\abs{\velocityvaronedimension}\basis \frac{\density_{120}}{\ints{1}}}\right),
  \end{align*}
  where the approximation in the last step is based on the observation that
  \(\density_{121} = \density_{122} = 10^{-6}\) are much smaller than \(\density_{120} = 100\) (see
  above).
  Inserting the definitions of the Hessian matrix~\eqref{eq:optHessian}
  and the values of \(\gridwidth, \density_{120}\) and \(\density_{121}\) we finally obtain
  \begin{equation*}
    \multipliersfv[121]^{1} 
    \approx
    \multipliersfv[121]^{0} 
    + \timestep {\ints{\basis\basis^T \, 5 \cdot 10^{-7}}}^{-1}\left(
    100 \intm{\abs{\velocityvaronedimension}\basis\, 50}\right)
    =
    \multipliersfv[121]^{0} 
    +  10^{10} \, \timestep \massmatrix^{-1}\intm{\abs{\velocityvaronedimension}\basis}.
  \end{equation*}
  For the full-moment \(\MN\) models using Legendre polynomials as a basis,
  the mass matrix \(\massmatrix\) is the unit matrix and
  the basis functions \(\basis\) are approximately of unit order.
  We thus see that the update term is in the order of
  \(10^{10}\) which is why a very small
  time step \(\timestep\) has to be chosen
  initially to limit the time stepping error.
\end{example}
The example shows that the transformed scheme is expected to require
very small time steps in some instances, e.g.\
whenever there are large differences in the particle density between adjacent cells
(which is initially the case for all our
numerical tests, see \cref{sec:numericalexperiments}).
On the other hand,
the time step does not have to be restricted
to ensure realizability which may allow for time steps
that are even larger than those used in the
standard scheme in some situations.
A time stepping scheme using a fixed time
step \(\timestep\) would thus be very inefficient for the new scheme.

Instead, we will use
the Runge-Kutta method by
Bogacki and Shampine~\cite{Bogacki1989}
which adaptively chooses the time step according to an embedded error estimate (see \cref{sec:adaptivetimestepping} for details).
This way,
we can use large time steps where possible without introducing
uncontrollable errors in time regions where a small time step is required.
If we define
\begin{equation}\label{eq:alphaupdateterm}
  \alphaupdateterm_{\gridindex}(\multipliersfv[0]^{\timeindex},
  \multipliersfv[1]^{\timeindex}, \ldots,
  \multipliersfv[\gridsize-1]^{\timeindex})
  \coloneqq {\optHessian(\multipliersfv[\gridindex]^{\timeindex})}^{-1} \uupdateterm_{\gridindex}(\moments(\multipliersfv[0]^{\timeindex}),
  \ldots, \moments(\multipliersfv[\gridsize-1]^{\timeindex})),
\end{equation}
the Runge-Kutta update takes the form
\begin{subequations}\label{eq:rkscheme}
  \begin{equation}
    \multipliersfv[\gridindex]^{\timeindex+1}
    = \multipliersfv[\gridindex]^{\timeindex} + \timestep \sum_{\rkindex=0}^{\rkstages-1} \rkb_{\rkindex}
    \alphaupdateterm_{\gridindex}(\multipliersrkstage{\rkindex}_{0}, \ldots, \multipliersrkstage{\rkindex}_{\gridsize-1}) \label{eq:rkschemeupdate}
  \end{equation}
  with stages \(\multipliersrkstage{\rkindex}_{\gridindex}\) given by
  \begin{equation}
    \multipliersrkstage{\rkindex}_{\gridindex}
    = \multipliersfv[\gridindex]^{\timeindex}
    + \timestep \sum_{\rkindexalt=0}^{\rkindex-1} \rka_{\rkindex\rkindexalt}
    \alphaupdateterm_{\gridindex}(\multipliersrkstage{\rkindexalt}_{0}, \ldots, \multipliersrkstage{\rkindexalt}_{\gridsize-1}).
  \end{equation}
\end{subequations}
Here, \(\rkstages\) is the number of stages in the Runge-Kutta scheme
\(\rka_{\rkindex\rkindexalt}, \rkb_{\rkindex}\) are the
Runge-Kutta coefficients and we neglected the varying time step for notational simplicity.

\subsection{Convergence properties}\label{sec:convergenceproperties}
The new scheme calculates approximate solutions in transformed (\(\multipliers\)) variables.
However, we are usually interested in the solution in original (\(\moments\)) variables.
We can easily obtain such a solution by applying the
diffeomorphism~\eqref{eq:alpha_to_u_mapping} to the transformed solution. However,
a priori, it might be possible that this
transformation amplifies the discretization errors and destroys the accuracy of the scheme.
Luckily, we can show that the order of convergence of the time stepping scheme is preserved by the transformation.
\begin{theorem}\label{thm:convergenceorder}
  Let \(\Set{\multipliersfv[\gridindex](\timevar) \given \gridindex \in \gridindexset}\) be a solution of
  the transformed semidiscrete equation \eqref{eq:transformedschemesemidiscrete} and let
  \(\Set{\momentsfv[\gridindex](t) \coloneqq \moments(\multipliersfv[\gridindex](\timevar)) \given \gridindex \in \gridindexset}\)
  be the corresponding solution of the untransformed equation \eqref{eq:finitevolumesemidiscrete}.
  Let further \(\Set{\multipliersfv[\gridindex]^{\timeindex}}\) be approximations of \(\multipliersfv[\gridindex]\) at
  discrete time points \(\timevar_{\timeindex}\)
  obtained by a time stepping scheme of order \(\timesteppingorder\), i.e.
  \begin{equation*}
    \multipliersfv[\gridindex]^\timeindex 
    = \multipliersfv[\gridindex](\timevar_\timeindex) + \Vr_{\gridindex}^{\timeindex}
  \end{equation*}
  with \( \Vr_{\gridindex}^{\timeindex} = \landauO({\left(\timestep\right)}^{\timesteppingorder})\).
  If the spectral norm \(\generalnorm{\optHessian}\) of the Hessian is bounded,
  then the corresponding moments converge with the same order, i.e.\
  \begin{equation*}
    \mathop{\moments}\left(\multipliersfv[\gridindex]^\timeindex\right) 
    = \momentsfv[\gridindex](\timevar_\timeindex) + \landauO({\left(\timestep\right)}^{\timesteppingorder}).
  \end{equation*}
\end{theorem}
\begin{proof}
  Using a zeroth order Taylor approximation with Lagrange form of the remainder, we have
  \begin{equation*}
    \mathop{\moments}\left(\multipliersfv[\gridindex]^\timeindex\right) 
    = \mathop{\moments}\left(\multipliersfv[\gridindex](\timevar_\timeindex) + \Vr_{\gridindex}^{\timeindex}\right)
    = \mathop{\moments}\left(\multipliersfv[\gridindex](\timevar_\timeindex)\right) + \optHessian(\multipliersfv[\gridindex]^{\xi})\Vr_{\gridindex}^{\timeindex}
    = \momentsfv[\gridindex](\timevar_\timeindex) + \optHessian(\multipliersfv[\gridindex]^{\xi})\Vr_{\gridindex}^{\timeindex}
  \end{equation*}
  with \(\multipliersfv[\gridindex]^{\xi} = (1-\xi)\multipliersfv[\gridindex](\timevar_{\timeindex}) + \xi \Vr_{\gridindex}^{\timeindex}\) for some
  \(\xi \in [0,1]\). Due to the boundedness of \(\optHessian\), we further have
  \begin{equation*}
    \generalnorm{\optHessian(\multipliersfv[\gridindex]^{\xi})\Vr_{\gridindex}^{\timeindex}}
    \leq \generalnorm{\optHessian(\multipliersfv[\gridindex]^{\xi})} \generalnorm{\Vr_{\gridindex}^{\timeindex}} =
    \landauO(\timestep^{\timesteppingorder}). \qedhere
  \end{equation*}
\end{proof}
For arbitrary \(\multipliers\), we cannot expect the upper bound on \(\generalnorm{\optHessian}\) required by \cref{thm:convergenceorder} to hold.
However, note that
\begin{equation}\label{eq:optHessiandoubleproduct}
  \Vw^T\optHessian(\multipliers)\Vw
  = \ints{{(\basis \cdot \Vw)}^2 \ld{\entropy}''(\multipliers \cdot \basis)} \leq \left(\max_{\SC \in \angularDomain} \generalnorm{\basis(\SC)}^2\right)
  \generalnorm{\Vw}^2 \ints{\ld{\entropy}''(\multipliers \cdot \basis)}.
\end{equation}
For Maxwell-Boltzmann entropy, we have \(\ld{\entropy}'' = \ld{\entropy}'\) and thus
\begin{equation*}
  \ints{\ld{\entropy}''(\multipliers \cdot \basis)} = \ints{\ld{\entropy}'(\multipliers \cdot \basis)} = \density(\moments(\multipliers))
\end{equation*}
is the local particle density corresponding to the multipliers \(\multipliers\).
Due to the conservation properties, the local particle density remains bounded for the solutions of the moment equations.
As a consequence, \(\generalnorm{\optHessian(\multipliers)}\) should be bounded
for \(\multipliers\) close enough to a solution of \(\eqref{eq:transformedschemesemidiscrete}\).

Once we have a \(r\)-th order time discretization for \eqref{eq:transformedschemesemidiscrete},
we thus can expect the corresponding moments to converge with the same order in practice,
at least for time steps \(\timestep\) small enough.
However, the order of convergence of Runge-Kutta schemes is usually shown under the assumption
of Lipschitz continuity of the ordinary differential
equation's right-hand side (see, e.g., \cite[Theorem II.3.4]{Hairer1993}).
In our case, we have to show Lipschitz continuity of the
functions \(\alphaupdateterm_{\gridindex}\). Under some additional assumptions,
we indeed obtain Lipschitz continuity for multipliers \(\multipliers\) in a domain \(\lipschitzdomain\).

\begin{theorem}\label{thm:lipschitzcontinuity}
  The function \(\alphaupdateterm_{\gridindex}\) is Lipschitz-continuous on \(\lipschitzdomain^{\gridsize}\)
  if there exist constants
  \(c_1, c_2, c_3, c_4, c_5 \in \Rpos\) such that for all \(\multipliers \in \lipschitzdomain \subset \R^{\momentnumber}\) we have
  \begin{subequations}\label{eq:proofconvergenceassumptions}
    \begin{align}
      c_1 \leq \generalnorm{\optHessian(\multipliers)}                                    & \leq c_2, \label{eq:Hessianbounded}            \\
      \density(\moments(\multipliers)) = \ints{\ld{\entropy}'(\multipliers \cdot \basis)} & \leq c_3, \label{eq:densitybounded}            \\
      \ints{\ld{\entropy}''(\multipliers \cdot \basis)}                                   & \leq c_4, \label{eq:entropysecondderivbounded} \\
      \ints{\ld{\entropy}'''(\multipliers \cdot \basis)}                                  & \leq c_5 \label{eq:entropythirdderivbounded}.
    \end{align}
  \end{subequations}
\end{theorem}

The proof is technical and can be found in \cref{sec:appendixlipschitzproof}.

For Maxwell-Boltzmann entropy, since \(\ld{\entropy}' = \ld{\entropy}'' = \ld{\entropy}''' = \exp\) and by definition of the Hessian
\eqref{eq:optHessian}, we see that \eqref{eq:densitybounded} implies \eqref{eq:entropysecondderivbounded}, \eqref{eq:entropythirdderivbounded}
and the upper bound in \eqref{eq:Hessianbounded}. Moreover, all the assumptions \eqref{eq:Hessianbounded}-\eqref{eq:entropythirdderivbounded} are fulfilled
if \(\generalnorm{\multipliers}\) remains bounded.

\begin{corollary}\label{cor:lipschitzmaxwellboltzmann}
  For Maxwell-Boltzmann entropy, \(\alphaupdateterm_{\gridindex}\) is Lipschitz-continuous on \(\lipschitzdomain^{\gridsize}\)
  if \(\lipschitzdomain\) is bounded.
\end{corollary}

\subsection{Regularization}\label{subsec:regularization}
As discussed above, the bound \eqref{eq:densitybounded} on the local particle density \(\density\)
can be expected to hold in practice.
Unfortunately, this is not true for the lower bound on \(\generalnorm{\optHessian}\).
Considering again \eqref{eq:optHessiandoubleproduct}, we see that we cannot expect the lower bound
to hold as long as \(\ld{\entropy}''\) is not bounded from below. In particular,
for Maxwell-Boltzmann entropy, \eqref{eq:proofconvergenceassumptions}
would require a lower bound on the
ansatz density \(\ansatz[{\moments(\multipliers[\gridindex])}]\).
Even if such a lower bound exists for the exact solution, it might be extremely small.
Factoring in numerical errors, it is clear that in practice this bound will not always hold.
In addition, \(\optHessian\) may be very badly conditioned, up to the point
that it might be numerically singular.

For the standard scheme, where \(\optHessian\) shows
up as the Hessian of the objective function in the
minimum entropy optimization problem \eqref{eq:dual}, we thus
use several regularization techniques to ensure that the Newton
scheme always converges (see \cref{sec:fvschemeregularization}).
While most of these techniques are not easily carried over to the transformed scheme,
we here investigate two approaches to regularize the new scheme.

\subsubsection{Isotropic regularization of the Hessian}
As a first approach, we regularize the Hessian matrix by adding a small multiple of the mass matrix
\(\massmatrix = \ints{\basis\basis^T}\).
Exemplarily, for the explicit Euler scheme, the update formula
\eqref{eq:newschemeeuler} then becomes
\begin{equation}\label{eq:alpha_rk_scheme_regularized}
  \multipliersfv[\gridindex]^{\timeindex+1}
  = \multipliersfv[\gridindex]^\timeindex + \timestep {\left(\optHessian(\multipliersfv[\gridindex]^{\timeindex})
    + \newschemeregparam \massmatrix\right)}^{-1}
  \uupdatetermexplicit_{\gridindex}.
\end{equation}
This corresponds to adding a small isotropic particle density to the derivative of the ansatz function in the Hessian
\begin{equation*}
  \optHessian(\multipliers)
  + \newschemeregparam \massmatrix = \ints{\basis\basis^T \ld{\entropy}''(\multipliers \cdot \basis)} + \newschemeregparam \ints{\basis\basis^T}
  = \ints{\basis\basis^T (\ld{\entropy}''(\multipliers \cdot \basis) + \newschemeregparam)}.
\end{equation*}
and ensures that the lower bound in \eqref{eq:Hessianbounded} holds for the regularized matrix.

To analyze the introduced error, note that without regularization the update in the Euler scheme is
\begin{align}
  \begin{split}\label{eq:localtruncationerror}
    \moments\left(\multipliersfv[\gridindex]^{\timeindex+1}\right)
    & =
    \moments\left(\multipliersfv[\gridindex]^\timeindex
    + \timestep
    {\optHessian(\multipliersfv[\gridindex]^{\timeindex})}^{-1}
    \uupdatetermexplicit_{\gridindex}\right)
    =
    \moments\left(\multipliersfv[\gridindex]^\timeindex\right)
    +
    \frac{\derivative\moments}{\derivative\multipliers}(\multipliersfv[\gridindex]^{\timeindex})
    \timestep
    {\optHessian(\multipliersfv[\gridindex]^{\timeindex})}^{-1}
    \uupdatetermexplicit_{\gridindex} + \Vr \\
    \overset{\eqref{eq:du_dalpha}} &=
    \moments\left(\multipliersfv[\gridindex]^\timeindex\right)
    + \timestep
    \uupdatetermexplicit_{\gridindex} +
    \Vr
    = \momentsfv[\gridindex]^{\timeindex+1} + \Vr
  \end{split}
\end{align}
with remainder of order \(\landauO({(\timestep)}^2)\).
Using the regularization \eqref{eq:alpha_rk_scheme_regularized}, the local truncation error becomes
\begin{equation*}
  \moments\left(\multipliersfv[\gridindex]^{\timeindex+1}\right) =
  \momentsfv[\gridindex]^{\timeindex+1} - \timestep \regepsilon \massmatrix
  {\left(\optHessian(\multipliersfv[\gridindex]^{\timeindex})+\regepsilon \massmatrix\right)}^{-1}
  \uupdatetermexplicit_{\gridindex} + \Vr.
\end{equation*}
Remember that the terms in \(\uupdatetermexplicit_{\gridindex}\), except for the constant source term \(\ints{\basis\source}\), all
scale with either \(\ansatz[{\moments(\multipliersfv[\gridindex]^{\timeindex})}]\)
or \(\ansatz[{\moments(\multipliersfv[\gridindexalt]^{\timeindex})}]\), \(\gridindexalt \in \gridneighbors{\gridindex}\).
Thus, if \(\multipliersfv[\gridindex]^{\timeindex}\) corresponds to an ansatz density
\(\ansatz[{\moments(\multipliersfv[\gridindex]^{\timeindex})}]\)
which is significantly greater than \(\newschemeregparam\),
the additional error term
\begin{equation*}
  \timestep \regepsilon \massmatrix
  {\left(\optHessian(\multipliersfv[\gridindex]^{\timeindex})+\regepsilon \massmatrix\right)}^{-1}
  \uupdatetermexplicit_{\gridindex} \approx
  \timestep \regepsilon \massmatrix
  {\left(\optHessian(\multipliersfv[\gridindex]^{\timeindex})\right)}^{-1}
  \uupdatetermexplicit_{\gridindex} =
  \landauO(\timestep \regepsilon)
\end{equation*}
does not negatively impact the rate of convergence as long as we
choose \(\newschemeregparam\) as \(\landauO(\timestep)\).
Similarly, if \(\ansatz[{\moments(\multipliersfv[\gridindex]^{\timeindex})}]\),
\(\ansatz[{\moments(\multipliersfv[\gridindexalt]^{\timeindex})}]\) and \(\source\)
are small (more precisely, \(\landauO(\newschemeregparam)\)), the additional error is of
order \(\landauO(\timestep \newschemeregparam \frac{1}{\newschemeregparam} \newschemeregparam)
= \landauO(\timestep \newschemeregparam)\). However, if
\(\ansatz[{\moments(\multipliersfv[\gridindex]^{\timeindex})}]\) is small, but at least one of
\(\ansatz[{\moments(\multipliersfv[\gridindexalt]^{\timeindex})}]\) and \(\source\) is not,
we obtain an additional error of
\(\landauO(\timestep \newschemeregparam \frac{1}{\newschemeregparam} 1) = \landauO(\timestep)\).
In this case, i.e.\ if the particle density in
grid cell \(\fvgridentity_{\gridindex}\) is small
and
there is a strong source or an influx
from an neighbor entity \(\fvgridentity_{\gridindex}\),
the regularization \eqref{eq:alpha_rk_scheme_regularized} might negatively affect the solution.
For higher-order time stepping schemes, the error analysis gets much more involved, but we
expect similar results, i.e.\ significant regularization errors only in regions
with large differences in particle densities between adjacent cells.

\subsubsection{Direct constraints for the entropy variables}\label{sec:directconstraintsregularization}
Instead of modifying the Hessian, we might try to directly enforce the boundedness of
the entropy variables required in \cref{cor:lipschitzmaxwellboltzmann}. For example, we could replace all multipliers \(\multipliers\) with
\(\generalnorm{\multipliers} > \alpha_{\text{max}}\) for some \(\alpha_{\text{max}} \in \Rpos\)
by the closest multiplier \(\multipliers[\text{reg}]\) with
\(\generalnorm{\multipliers[\text{reg}]} \leq \alpha_{\text{max}}\).

We tested this approach in several numerical tests and found that, for
general bases \(\basis\), the regularization either introduced large errors
into the solution (for \(\alpha_{\text{max}}\) small) or
did not improve the numerical stability (for large \(\alpha_{\text{max}}\)).

A notable exception is the hat function basis \(\hfbasis\).
Since the hat functions are non-negative,
remembering
that the ansatz density has the form \(\exp(\multipliers \cdot \hfbasis\)) (for Maxwell-Boltzmann entropy),
we see that
large positive entries of \(\multipliers\) correspond to large ansatz densities
and large (in absolute values) negative
entries correspond to small ansatz densities.
Thus, if we enforce a lower bound on the entries of \(\multipliers\)
by replacing all entries \(\multiplierscomp{\basisindex} < \alpha_{\text{min}}\)
by \(\alpha_{\text{min}}\) for some negative \(\alpha_{\text{min}} \in \R\),
we would expect in general that the introduced error is small
since we are only replacing entries corresponding
to very small ansatz densities by entries corresponding
to slightly larger but still very small densities.

Indeed, in our numerical tests, the error is very small (see \cref{sec:hfmbasisregularization}).
Note, however, that there might be cases where also this regularization technique leads to significant errors.
Consider, e.g., a highly anisotropic particle distributions in one dimension where the density
is very low at one boundary of an interval in the velocity partition and very high at the other boundary.
Here, replacing the entry corresponding to the very low density might significantly
alter the ansatz density in that interval.

For the other bases (\(\pmbasis\) and \(\fmbasis\)), a similar regularization technique which
replaces only multipliers corresponding to small ansatz densities is
not straightforward since for these bases there is no clear correspondence between the
sign of \(\multipliers\) entries and the magnitude of the ansatz density.

\subsection{Entropy stability on the fully discrete level}\label{sec:entropystability}
In general, we cannot expect that an explicit time discretization preserves the entropy-stability of the semidiscrete scheme.
However, we can enforce this property by using a relaxation of the standard Runge-Kutta scheme~\cite{Ranocha2020}.
To that end, we collect the multipliers for each grid cell in a single vector, i.e.\ we define
\begin{equation*}
  \gesamtalpha \coloneqq \begin{pmatrix}
    \multipliersfv[0](\timevar) \\
    \vdots                      \\
    \multipliersfv[\gridsize - 1](\timevar)
  \end{pmatrix} \quand
  \gesamtrhs(\gesamtalpha) = \begin{pmatrix}
    \alphaupdateterm_{0}(\multipliersfv[0], \ldots, \multipliersfv[\gridsize-1]) \\
    \vdots                                                                       \\
    \alphaupdateterm_{\gridsize-1}(\multipliersfv[0], \ldots, \multipliersfv[\gridsize-1])
  \end{pmatrix}
  = \begin{pmatrix}
    {\optHessian(\multipliersfv[0])}^{-1} \uupdateterm_{0}(\multipliersfv[0], \ldots, \multipliersfv[\gridsize-1]) \\
    \vdots                                                                                                         \\
    {\optHessian(\multipliersfv[\gridsize-1])}^{-1} \uupdateterm_{\gridsize-1}(\multipliersfv[0], \ldots, \multipliersfv[\gridsize-1])
  \end{pmatrix}.
\end{equation*}
Further, we define the total entropy as
\begin{equation*}
  \gesamtentropy(\gesamtalpha) \coloneqq
  \sum_{\gridindex \in \gridindexset}
  \entropyFunctional(\ansatz[{\moments(\multipliersfv[\gridindex])}])
  = \sum_{\gridindex \in \gridindexset} \ints{\entropy(\ansatz[{\moments(\multipliersfv[\gridindex])}])}
  = \sum_{\gridindex \in \gridindexset} \ints{\entropy(\ld{\entropy}'(\multipliersfv[\gridindex] \cdot \basis)}.
\end{equation*}
Then
\begin{equation*}
  \gesamtentropy'(\gesamtalpha) = \begin{pmatrix}
    \optHessian(\multipliersfv[0]) \multipliersfv[0]                        \\
    \vdots                                                                  \\
    \optHessian(\multipliersfv[\gridsize -1]) \multipliersfv[\gridsize -1], \\
  \end{pmatrix}
\end{equation*}
(where we again used that \(\entropy' \circ \ld{\entropy}'\) is the identity)
and thus
\begin{align}
  \begin{split}\label{eq:discreteentropydings2}
    \frac{\mathrm{d}}{\mathrm{d}\timevar} \gesamtentropy(\gesamtalpha(t)) &= \gesamtentropy'(\gesamtalpha(t)) \cdot \gesamtrhs(\gesamtalpha(t)) \\
    & = \sum_{\gridindex \in \gridindexset}
    \multipliersfv[\gridindex] \cdot \uupdateterm_{\gridindex}(\multipliersfv[0], \ldots, \multipliersfv[\gridsize-1]) \\
    \overset{\text{(Proof of)} \eqref{eq:entropyeqsemidiscrete}} & {\leq}
    \begin{multlined}[t][.6\displaywidth]
      \sum_{\gridindex\in\gridindexset} \Bigg(
      \ints{(\multipliersfv[\gridindex] \cdot \basis) \source}
      - \absorption \ints{(\multipliersfv[\gridindex] \cdot \basis) \ld{\entropy}'(\multipliersfv[\gridindex] \cdot \basis)} \\
      -\sum_{\gridindexalt\in\gridneighbors{\gridindex}} \frac{\abs{\fvgridface_{\gridindex\gridindexalt}}}{\abs{\fvgridentity_{\gridindex}}}
      \left(\intp{(\SC \cdot \outernormal_{\gridindex\gridindexalt})
            \entropy(\ansatz[{\moments(\multipliersfv[\gridindex])}])}
        + \intm{(\SC \cdot \outernormal_{\gridindex\gridindexalt})
            \entropy(\ansatz[{\moments(\multipliersfv[\gridindexalt])}])}\right)\Bigg)
    \end{multlined}                                                                                         \\
    & =
    \begin{multlined}[t][.6\displaywidth]
      \sum_{\gridindex\in\gridindexset} \left(
      \ints{(\multipliersfv[\gridindex] \cdot \basis) \source}
      - \absorption \ints{(\multipliersfv[\gridindex] \cdot \basis) \ld{\entropy}'(\multipliersfv[\gridindex] \cdot \basis)}\right) \\
      - \sum_{\fvgridface_{\gridindex\gridindexalt} \text{ boundary interface}}
      \frac{\abs{\fvgridface_{\gridindex\gridindexalt}}}{\abs{\fvgridentity_{\gridindex}}}
      \left(\intp{(\SC \cdot \outernormal_{\gridindex\gridindexalt})
          \entropy(\ansatz[{\moments(\multipliersfv[\gridindex])}])}
      + \intm{(\SC \cdot \outernormal_{\gridindex\gridindexalt})
          \entropy(\distributionboundary(\spatialvar_{\gridindex\gridindexalt})}\right),
    \end{multlined}
  \end{split}
\end{align}
i.e.\ the change in total entropy is bounded by entropy fluxes over the domain boundary and entropy production
via particle absorption or creation. Here,
\(\spatialvar_{\gridindex\gridindexalt}\) is
the center of interface \(\fvgridface_{\gridindex\gridindexalt}\) and \(\distributionboundary\)
is the boundary value of the kinetic equation \eqref{eq:kineticequationboundary}.
If we use an explicit Runge-Kutta scheme of the form \eqref{eq:rkscheme}
for time discretization of \eqref{eq:transformedschemesemidiscrete},
we would expect the total entropy to approximately evolve as
\begin{equation}\label{eq:entropyupdate2}
  \gesamtentropy(\gesamtalpha^{\timeindex+1}) = \gesamtentropy(\gesamtalpha^{\timeindex})
  + \timestep \sum_{\rkindex=0}^{\rkstages-1} \rkb_{\rkindex} \gesamtentropy'(\gesamtbeta^{\rkindex}) \cdot
  \gesamtrhs(\gesamtbeta^{\rkindex}),                                                        \\
\end{equation}
where \(\gesamtbeta^{\rkindex} =
{(\multipliersrkstage{\rkindex}_{0}, \ldots, \multipliersrkstage{\rkindex}_{\gridsize-1})}^T\) contains the collected
Runge-Kutta stages (compare the definition \eqref{eq:rkscheme} of the Runge-Kutta scheme).
Though \eqref{eq:entropyupdate2} does not hold exactly, we can enforce this property by introducing
a relaxation into the Runge-Kutta scheme \eqref{eq:rkscheme}
and replace \eqref{eq:rkschemeupdate} by the relaxed version \cite{Ranocha2020}
\begin{equation}\label{eq:modifiedrkupdate}
  \multipliers[\gridindex,\gamma]^{\timeindex+1}
  = \multipliers[\gridindex]^{\timeindex} + \gamma_{\timeindex} \timestep \sum_{\rkindex=0}^{\rkstages-1} \rkb_{\rkindex}
  \alphaupdateterm_{\gridindex}(\gesamtbeta^{\rkindex}),
\end{equation}
where \(\gamma_{\timeindex}\) is calculated by finding a root of
\begin{align}
  \begin{split}\label{eq:r_gamma}
    r(\gamma) & \coloneqq
    \gesamtentropy(\gesamtalpha_{\gamma}^{\timeindex+1}) - \gesamtentropy(\gesamtalpha^{\timeindex})
    - \gamma \timestep \sum_{\rkindex=0}^{\rkstages-1} \rkb_{\rkindex}
    \gesamtentropy'(\gesamtbeta^{\rkindex}) \cdot
    \gesamtrhs(\gesamtbeta^{\rkindex})                                                                          \\
    & = \sum_{\gridindex \in \gridindexset} \left(\ints{\entropy \circ \ld{\entropy}'(\multipliers[\gridindex,\gamma]^{\timeindex+1} \cdot \basis)}
    - \ints{\entropy \circ \ld{\entropy}'(\multipliers[\gridindex]^{\timeindex} \cdot \basis)}
    - \gamma \timestep \sum_{\rkindex=0}^{\rkstages-1} \rkb_{\rkindex}
    \multipliersrkstage{\rkindex}_{\gridindex} \cdot \uupdateterm_{\gridindex}(\gesamtbeta_{\rkindex})\right).
  \end{split}
\end{align}
If \(\gesamtentropy\) is convex, \(r(\gamma)\) is also convex and has exactly two roots, one at zero
and one close to one \cite{Ranocha2020} (for \(\timestep\) small enough).  Unfortunately,
in our case, \(\gesamtentropy\) is not necessarily convex since \(\entropy \circ \ld{\entropy}'\) is not necessarily convex.
In the Maxwell-Boltzmann case, for example, we have
\(\entropy \circ \ld{\entropy}'(p) =
\exp(p)(p - 1)\)
which is convex only for \(p > -1\) (and concave for \(p < -1\)).
However, in our numerical experiments, we always found a root of \(r(\gamma)\) close to one.

\begin{remark}
  We could use the same approach to
  obtain entropy-stability for the standard finite volume scheme.
  However, in that case, each time we want to evaluate \(r(\gamma)\) for a new \(\gamma\)
  we would have to solve the minimum-entropy optimization problem on each grid cell
  (to evaluate the equivalent of \(\gesamtentropy(\gesamtalpha_{\gamma}^{\timeindex+1})\)),
  which would make the root finding procedure prohibitively expensive.
\end{remark}

Additional details on our implementation of the relaxed Runge-Kutta scheme can be found in \cref{sec:adaptivetimestepping}.

%% file: Sections/implementation.tex
\section{Implementation details}\label{sec:implementationdetails}
We implemented both schemes in the generic \Cpp framework 
\texttt{DUNE}~\cite{Bastian2008,Bastian2008a}, more specifically in the \texttt{DUNE} generic
discretization toolbox \texttt{dune-gdt}~\cite{Schindler2017} and the \texttt{dune-xt}-modules~\cite{Milk2017, Milk2017a}.
The implementation is available in \cite{Leibner2021}.

\subsection{Standard finite volume scheme}
The implementation for the standard finite volume scheme is taken from \cite{SchneiderLeibner2019b}.
Here, we only shortly recall the relevant parts. Further, we will restrict ourselves
to Maxwell-Boltzmann entropy~\eqref{eq:maxwellboltzmannentropy} such that
the ansatz becomes \(\ansatz[\moments] =
\ld{\entropy}'(\multipliers(\moments) \cdot \basis) = \exp(\multipliers(\moments) \cdot \basis)\).

\subsubsection{Analytic solution of the source system}
As mentioned above, we use a second-order splitting approach (see~\eqref{eq:splittedsystem}) to
handle the source term independently of the flux term.
Remember that we assume that the scattering is isotropic and
that the parameters \(\scattering\), \(\absorption\), \(\source\) are time-independent.
The source system~\eqref{eq:SplittedSystemb} on grid cell \(\fvgridentity_{\gridindex}\) thus takes the form
\begin{equation}\label{eq:SourceTermODE}
  \dt\momentsfv[\gridindex](\timevar)
  = \left(\scattering(\spatialvar_{\gridindex}) \isomatrix - \crosssection(\spatialvar_{\gridindex})\eyematrix\right) \momentsfv[\gridindex](\timevar)  + \ints{\basis\source}.
\end{equation}
Since~\eqref{eq:SourceTermODE} is
linear in \(\momentsfv[\gridindex]\) with time-independent system matrix, we can solve it explicitly
using matrix exponentials and the variation of constants formula. Under the additional assumption that
the source is also isotropic, i.e.\ that \(\source\) does not depend on \(\SC\),
the solution to~\eqref{eq:SourceTermODE} is (see~\cite{SchneiderLeibner2019b} for a detailed derivation)
\begin{equation*}
  \momentsfv[\gridindex](\timevar)
  = \eulere^{-\absorption \timevar} \left(\eulere^{-\scattering \timevar} \momentsfv[\gridindex](0) +
  \left(1 - \eulere^{-\scattering \timevar}\right) \isomatrix \momentsfv[\gridindex](0) \right)
  +
  \frac{1-\eulere^{-\absorption \timevar}}{\absorption} \ints{\basis} \source.
  \label{eq:matexpsolutionisotropicsource}
\end{equation*}
Note that~\eqref{eq:matexpsolutionisotropicsource} can easily be calculated without any matrix
operations due to the rank one structure of \(\isomatrix\) (see~\eqref{eq:isomatrix}).

\subsubsection{Solving the optimization problem}\label{sec:optimizationproblem}
The second part of the standard splitting scheme, the flux system~\eqref{eq:SplittedSystema}, is advanced
in time using Heun's method,
which is a second-order strong-stability preserving Runge-Kutta scheme~\cite{Gottlieb2005}.
In each stage of the time stepping scheme, we
have to solve the optimization problem~\eqref{eq:primal} once in each cell.
This usually accounts for
the majority of computation time which makes it mandatory to pay special
attention to the implementation of the optimization algorithm.

Recall that the objective function in the dual problem~\eqref{eq:dual} is
\begin{equation*}
  \optObjective(\multipliers) = \optObjective_{\moments,\basis,\entropy}(\multipliers) = \ints{\ld{\entropy}(\basis \cdot \multipliers)} - \moments \cdot \multipliers.
\end{equation*}
The gradient and the Hessian of \(\optObjective\) are given by
\begin{equation}\label{eq:optGradient}
  \optGradient(\multipliers) = \optGradient_{\moments,\basis,\entropy}(\multipliers) = \multipliersGradient \optObjective(\multipliers) = \ints{\basis \ld{\entropy}'(\basis \cdot \multipliers)} - \moments
\end{equation}
and
\begin{equation}\label{eq:optHessian}
  \optHessian(\multipliers) = \optHessian_{\basis,\entropy}(\multipliers) =
  \jacobianop_{\multipliers} \optGradient(\multipliers)
  = \ints{\basis \basis^T \ld{\entropy}''\left(\basis \cdot \multipliers\right)},
\end{equation}
respectively. Note that \(\ld{\entropy}'' > 0\) since we assumed that \(\entropy\) (and thus also \(\ld{\entropy}\)) is
strictly convex, and remember that
the basis functions contained in \(\basis\) are linearly independent. As a consequence, the Hessian \(\optHessian\)
is positive definite (and thus invertible).

To find a minimizer of \(\optObjective\), we are searching for a root of the gradient \(\optGradient\) using Newton's method.
The Newton direction \(\optNewtonDirection(\multipliers)\) solves
\begin{equation*}
  \optHessian(\multipliers)\optNewtonDirection(\multipliers) = -\optGradient(\multipliers).
\end{equation*}
The update takes the form
\begin{equation*}
  \multipliers[k+1] = \multipliers[k] + \zeta_k \optNewtonDirection(\multipliers[k])
\end{equation*}
where \(\zeta_k\) is determined by a backtracking line search such that
\begin{equation*}
  \optObjective(\multipliers[k+1]) < \optObjective(\multipliers[k]) + \xi \zeta_k {\optGradient(\multipliers[k])} \cdot \optNewtonDirection(\multipliers[k])
\end{equation*}
with \(\xi \in (0,1)\). In our implementation, we always use \(\xi = 10^{-3}\).

To avoid numerical problems for moments corresponding to a
very small local particle density,
before entering the Newton algorithm
for the moment vector \(\moments\) we rescale it to
\begin{equation*}
  \normalizedmoments \coloneqq \frac{\moments}{\density(\moments)}
\end{equation*}
such that
\(\density(\normalizedmoments) = 1\). If the optimization algorithm
for \(\normalizedmoments\) stops at an iterate \(\multiplierstilde\), we return
\begin{equation}\label{eq:beta_to_alpha}
  \numericalmultipliers = \multiplierstilde + \multipliersone
  \log\left(\frac{\density(\moments)}{\density(\moments(\multiplierstilde))}\right).
\end{equation}
(remember that \(\multipliersone\) is the multiplier with the
property that \(\multipliersone \cdot \basis \equiv 1\), see \cref{def:density}).
This ensures that the local particle density is preserved exactly:
\begin{equation}\label{eq:densityadjustmultipliers}
  \density\left(\moments(\numericalmultipliers)\right) =
  \ints{\exp(\basis \cdot \numericalmultipliers)} = \ints{\exp(\basis \cdot \multiplierstilde)}
  \frac{\density(\moments)}{\density(\moments(\multiplierstilde))} =
  \density(\moments).
\end{equation}
Given tolerances \(\opttol\in\Rpos\), \(\opttoleps\in(0,1)\),
we will stop the Newton iteration
at iterate \(\multiplierstilde\) if
\begin{align}\label{eq:firststoppingcriterion}
  (1) & \ \ \norm{\optGradient_{\normalizedmoments}(\multiplierstilde)}{2} < \opttolmod \coloneqq
  \begin{cases}
    \frac{\opttol}{\left(1+\norm{\normalizedmoments}{2}
    \right)\density(\moments) + \opttol}                     & \text{ if }\quad  \basis = \fmbasis                \\
    \frac{\opttol}{\left(1+\sqrt{\momentnumber}\norm{\normalizedmoments}{2}
    \right)\density(\moments) + \sqrt{\momentnumber}\opttol} & \text{ if }\quad  \basis\in\{\hfbasis, \pmbasis\},
  \end{cases} \quad \text{and }                                                         \\
  (2) & \ \ \moments - (1-\opttoleps) \moments(\numericalmultipliers)               \in\RDpos{\basis},
  \label{eq:thirdstoppingcriterion}
\end{align}
where \(\numericalmultipliers\) is obtained from \(\multiplierstilde\) by~\eqref{eq:beta_to_alpha} and, as always, \(\momentnumber\) is the number of moments.
As shown in \cite{SchneiderLeibner2019b}, the first criterion guarantees that
the gradient of the objective function is sufficiently small, i.e.,
\(\norm{\optGradient_{\moments}(\numericalmultipliers)}{2} \leq \opttol\).
The second criterion~\eqref{eq:thirdstoppingcriterion} ensures that~\eqref{eq:CFLcondcorollaryassumption}
holds and thus that the whole scheme is realizability-preserving although
the optimization problems are only solved approximately.
In our implementation, we choose the tolerances as \(\opttol = 10^{-9}\) and \(\opttoleps = 0.1\).

Checking the second stopping criterion~\eqref{eq:thirdstoppingcriterion}
might be quite expensive (depending
on the basis \(\basis\)). We therefore check this criterion only if additionally
\begin{equation}
  1-\opttoleps < \exp(-\left(\norm{\newtonDirection(\multiplierstilde)}{1} +
    \abs{\log{\density(\moments(\multiplierstilde))}} \right))
  \label{eq:secondstoppingcriterion}
\end{equation}
holds. This criterion approximately ensures~\eqref{eq:CFLcondcorollaryassumption}
(see~\cite{Schneider2016, Alldredge2014}) but, in general, is much easier to
evaluate than~\eqref{eq:thirdstoppingcriterion}. For the \(\HFMN\) models, however,
checking realizability is just checking positivity, so in that case we do not
need to check~\eqref{eq:secondstoppingcriterion} first.

\subsubsection{Caching}\label{subsec:caching}
We use two types of caching for the standard scheme.
First, for each grid cell we
store the moment vector \(\momentsfv[\gridindex]^{\timeindex-1}\) 
from the last time step and
the corresponding multiplier \(\numericalmultipliers_{\gridindex}^{\timeindex-1}\)  
obtained by entropy minimization. In this way we do not have to
solve the optimization problem again if the moment vector in that grid cell
did not change during the last time step. In addition, we store the
last few solutions of the minimization problem with corresponding input moment vectors per thread of execution, so if several grid cells contain the
same values,
we only have to perform the optimization once and then use the cached values. If we encounter a moment
vector that can not
be found in the caches, we take the moment vector that is closest to the input vector (in one-norm) and use the corresponding multiplier
\(\multipliers\) as an initial guess.

\subsubsection{Linear solvers}\label{sec:linearsolvers}
In each iteration of the Newton scheme described above and
in each time step of the new scheme,
we have to apply the inverse
of a positive definite Hessian matrix. We assemble the matrices
using the quadratures described in \cref{sec:quadratures}.
Inversion is then done by computing a Cholesky factorization
of the assembled matrix. For the full moment models,
the Hessian matrices are dense, so we use
the \texttt{LAPACK}~\cite{Anderson1999} routine \texttt{dpotrf}
to compute the factorization and then use
\texttt{dtrsv} to actually invert the linear systems.
For the \(\PMMN\) models,
the Hessian is block-diagonal (each
block corresponds to one interval/triangle of the partition)
such that we can perform the Cholesky decomposition
independently for each block. For the \(\HFMN\)
models in one dimension, the Hessian matrices are
tridiagonal, so we can use the specialized \texttt{LAPACK}
algorithms \texttt{dpttrf} and \texttt{dpttrs}.
In three dimensions, the \(\HFMN\)
Hessians are not tridiagonal anymore
but still sparse, so we use the sparse
\texttt{SimplicialLDLT} solver from
the \texttt{Eigen} library~\cite{Guennebaud2010}.

\subsubsection{Regularization}\label{sec:fvschemeregularization}
Though the Hessian \(\optHessian(\multipliers)\) is positive definite and thus invertible, it may be very badly conditioned,
especially for multipliers \(\multipliers\) corresponding to moments \(\moments(\multipliers)\) close to the boundary
of the realizable set. Moreover, in general, the integral in the definition~\eqref{eq:optHessian} of \(\optHessian\)
can only be calculated approximately using a numerical quadrature (see \cref{sec:quadratures}). If the quadrature is not sufficiently accurate,
the approximate Hessian may have a significantly worse condition or may even be numerically singular.

To improve this situation, a change of
basis can be performed after each Newton iteration such that
the Hessian at the current iterate becomes the unit matrix in the new basis \cite{Alldredge2014}.
We use this procedure in our implementation for all bases except for the
hat function bases \(\hfbasis\) where the change of basis would destroy the sparsity
of the Hessian~\cite{SchneiderLeibner2019b}.

For tests with strong absorption,
the local particle density may become very small in parts of the domain.
As a consequence, also the entries
of the Hessian \(\optHessian\) become very small
which may cause numerical problems.
We thus choose a ``vacuum'' density
\(\distributionvacuum\) with corresponding local particle density
\(\densityvacuum = \ints{\distributionvacuum}\). We then enforce
a minimum local particle density of \(\densityvacuum\)
by replacing moments \(\moments\) with local particle density
\(\density(\moments) < \densityvacuum\) by the
isotropic moment with vacuum density \(\densityvacuum\).
Obviously, this approach leads to a violation of the conservation properties of the scheme.
However, since we only replace moments with
very small local particle densities by moments with slightly larger
but still very small densities, the effect should be negligible in practice.

Finally,
if the optimizer fails for a moment vector \(\moments\)
(for example, by reaching a maximum number of iterations or being
unable to solve for the Newton direction)
we use the isotropic-regularization
technique from~\cite{Alldredge2014}, i.e.\
we replace \(\moments\) by the regularized moment vector
\begin{equation}\label{eq:RegularizedMoments}
  \regularizedmoments \coloneqq (1 - \regularizationParameter) \moments + \regularizationParameter \isomatrix \moments.
\end{equation}
and retry the optimization.
If the optimizer still fails,
we increase \(\regularizationParameter\) until the optimizer succeeds,
which is guaranteed at least for \(\regularizationParameter = 1\) where
\(\regularizedmoments\) is isotropic.
In our implementation, the sequence of regularization parameters
\(\regularizationParameter\) is
chosen as
\(\Set{10^{-8}, 10^{-6}, 10^{-4}, 10^{-3}, 0.01, 0.05, 0.1, 0.5, 1}\).
As the regularized moment vector \(\regularizedmoments\) always has the same local particle
density as the original moment vector \(\moments\),
this technique does not violate the mass conservation of the scheme
but it may potentially completely alter the solution.
In practice, regularization is only used rarely and if it is used, a
small regularization parameters is usually sufficient.

\subsubsection{Quadrature rules}\label{sec:quadratures}
We have to approximate the same integrals for both schemes,
so we use the quadratures and quadrature orders
that have been determined in~\cite{SchneiderLeibner2019b}
for the standard scheme. Using these quadratures,
the quadrature error should usually
be negligible compared to the moment
approximation error~\cite{SchneiderLeibner2019b}.

In one dimension, we use Gauss-Lobatto quadratures. These quadratures include
the endpoints of the interval in the set of quadrature points
which ensures that the numerically realizable set and
the analytically realizable set agree for the \(\HFMN\) and \(\PMMN\) models (see~\cite{SchneiderLeibner2019b}
for details).
For these models, we use a quadrature of order \(15\)
per interval of the partition \(\generalpartition\).
For the full moment \(\MN\) models, we split
the domain in the two intervals \([-1,0]\) and \([0,1]\) that
are needed for calculation of the kinetic flux and
use a quadrature of order \(2\momentorder+40\) on each interval.

In three dimensions, for the \(\HFMN\) and \(\PMMN\) models,
we use Fekete quadratures
on each spherical triangle of the triangulation \(\generalpartition\).
Similar to the one-dimensional situation,
these quadratures include
the vertices of the triangle
which is again important for realizability considerations~\cite{SchneiderLeibner2019b}.
We use a quadrature order of \(15\) for the \(\HFMN[6]\) and \(\PMMN[32]]\)
models and a quadrature order of \(9\) for the other \(\HFMN\) and \(\PMMN\)
models. For the \(\MN\) models, we use tensor-product
quadrature rules of order \(2\momentorder+8\) on the octants of the sphere.

\subsubsection{Implementation of initial and boundary conditions}\label{sec:initialandboundaryfvscheme}
The initial values for the finite volume scheme
are computed by integration of the kinetic equation's initial values \eqref{eq:kineticequationinitial}:
\begin{equation*}
  \moments[\gridindex]^{0} = \frac{1}{\abs{\fvgridentity_{\gridindex}}} \int_{\fvgridentity_{\gridindex}} \ints{\distributiontzero(\spatialvar,\SC) \basis} \mathrm{d}\spatialvar
\end{equation*}
Since the initial values in our test cases are isotropic (see \cref{sec:numericalexperiments}), i.e.\
\(\distributiontzero(\spatialvar,\SC) = \distributiontzero(\spatialvar)\), we only have
to compute the velocity integral of the basis \(\ints{\basis}\). For this integral, we use the
same quadratures as in \cref{sec:quadratures} to ensure that the result is numerically realizable.
Except for the plane-source and point-source tests, the initial values are constant in each grid cell
\(\fvgridentity_{\gridindex}\), so we use the midpoint quadrature to evaluate the spatial integral.
For the plane-source test, we always use an even number of grid cells and distribute
the Dirac delta at \(\spatialvaronedimension=0\) into the two adjacent grid cells, i.e.\ the initial
value in these grid cells is set
to the constant \(\left.\distributiontzero\right|_{\fvgridentity_{\gridindex}}(\spatialvar) = \distributionvacuum + \frac{1}{2\gridwidth}\).
For the point-source test, we use a Gauss-Legendre tensor product quadrature of order 20 to evaluate the spatial integrals for the initial values.

Boundary conditions for the moment equations are implemented
by replacing the ansatz function
\(\ansatz[{\moments[\gridindexalt]}]\) belonging to a grid cell \(\fvgridentity_{\gridindexalt}\) outside of the computational domain
(such cells often called ``ghost cells'') by the boundary
condition \(\distributionboundary\) of the kinetic equation \eqref{eq:kineticequationboundary}
in the computation of the kinetic flux \(\eqref{eq:kineticFlux}\).

\subsection{New scheme}
For the new scheme, evaluation of quadrature rules and boundary conditions and
assembly and inversion of the Hessian matrices is performed exactly in
the same way as for the standard scheme (see \cref{sec:quadratures,sec:linearsolvers,sec:initialandboundaryfvscheme}.
To get the initial values \(\Set{\multipliersfv[\gridindex]^0}\) for the new scheme,
we solve the minimum entropy problems for the initial moments
\(\Set{\momentsfv[\gridindex]^0}\) which
are computed as described in \cref{sec:initialandboundaryfvscheme}
using the Newton scheme from \cref{sec:optimizationproblem}.

\subsubsection{Embedded and relaxed Runge-Kutta schemes}\label{sec:adaptivetimestepping}
Our time stepping scheme is outlined in \cref{alg:embeddedRK}.
As mentioned above (see \cref{sec:timediscretization}), we use embedded Runge-Kutta methods
(see, e.g., \cite[Chapter II.4]{Hairer1993})
to adaptively choose the time step for the new scheme.
These schemes include a second set
of coefficients \(\{\rkbalt_{\rkindex}\}_{\rkindex=0,\ldots,\rkstages-1}\)
to obtain a different approximation (usually of lower order) of the solution
at the next time step which is used for error estimation.
More precisely, in addition to the approximation \(\multipliersfv[\gridindex]^{\timeindex+1}\)
given by \eqref{eq:rkschemeupdate} we compute a second approximation
\begin{equation}\label{eq:secondapproximationRK}
  \multipliersembedded[\gridindex]^{\timeindex+1}
  = \multipliersfv[\gridindex]^{\timeindex} + \timestep \sum_{\rkindex=0}^{\rkstages-1} \rkbalt_{\rkindex}
  \alphaupdateterm_{\gridindex}(\multipliersrkstage{\rkindex}_{0}, \ldots, \multipliersrkstage{\rkindex}_{\gridsize-1})
\end{equation}
in each grid cell \(\fvgridentity_{\gridindex}\)
and regard the error between these two approximations to decide if the time step is appropriate.
As an error measure, we use the mixed error (see~\cite[Chapter II.4, Equation (4.11)]{Hairer1993})
\begin{equation}\label{eq:embeddedrkerror}
  \rkerr(\gesamtalpha^{\timeindex+1}, \multipliersembedded^{\timeindex+1})
  = \max\limits_{\fvvecindex=0,\ldots,\fvvecsize-1}{\frac{\abs{\multiplierscomp{\fvvecindex}-\multiplierscomptilde{\fvvecindex}}}{\rkabstol
      + \max(\multiplierscomp{\fvvecindex},\multiplierscomptilde{\fvvecindex})\rkreltol}}.
\end{equation}
Here, \(\rkabstol\) and \(\rkreltol\) are absolute and relative error tolerances and \(\multiplierscomp{\fvvecindex}\) and \(\multiplierscomptilde{\fvvecindex}\)
are the \(\fvvecsize = \momentnumber\cdot\gridsize\) components of the coefficient vectors
\begin{equation}\label{eq:definitiongesamtalphas}
  \gesamtalpha^{\timeindex+1}         =
  {\big({(\multipliersfv[0]^{\timeindex+1})}^T, 
  \ldots,
  {(\multipliersfv[\gridsize-1]^{\timeindex+1})}^T\big)}^T \quad \text{and} \quad
  \multipliersembedded^{\timeindex+1}=
  {\big({(\multipliersembedded[0]^{\timeindex+1})}^T, 
  \ldots,
  {(\multipliersembedded[\gridsize-1]^{\timeindex+1})}^T\big)}^T 
\end{equation}
respectively.
For simplicity, we will always use \(\rkabstol = \rkreltol = \rktol\) in the following.
We use an automatic step size control which tries to select the time step as large
as possible while still maintaining an error \eqref{eq:embeddedrkerror} below one, i.e.\
at \(\timevar = 0\) we start with a very small time step of \(\timestep_0 = 10^{-15}\)
and then compute the
new time step \(\timestep_{\mathrm{new}}\) from the previous time step \(\timestep\) by (compare~\cite[Chapter II.4, Equation (4.13)]{Hairer1993})
\begin{equation}\label{eq:timestepnew}
\timestep_{\mathrm{new}}(\timestep, \rkerr) = \timestep \cdot \min\left(\max\bigg(0.8 \cdot {\left(\frac{1}{\rkerr}\right)}^{\frac{1}{q + 1.}}, \frac{1}{5}\bigg), 5\right)
\end{equation}
where \(q\) is the order of the lower-order scheme in the adaptive Runge-Kutta method, \(0.8\) is a safety factor and the minimum and maximum
ensure that the time step does not change too fast.
If \(\rkerr > 1\)
or if an exception is thrown during the computation (e.g.\ if a
matrix inversion fails or \texttt{inf}s or \texttt{NaN}s are
detected in the results)
we recompute the solutions with halved time step \(\frac{\timestep}{2}\).
\begin{algorithm}
  \caption{Adaptive time stepping scheme}\label{alg:embeddedRK}
  \SetKwProg{try}{try}{:}{}
  \SetKwProg{catch}{catch}{:}{end}
  $\timestep \gets 10^{-15}, \ \timevar \gets 0, \ \gesamtalpha^{\timeindex} \gets \gesamtalpha^{0}$\;
    \While{$\timevar < \tf$}{
      $e = 1000$\tcp*[l]{Make sure the following while-loop is entered}
    \While{$e > 1$}{
      \try{}{Compute \(\gesamtalpha^{\timeindex+1}\), \(\multipliersembedded^{\timeindex+1}\) from \(\gesamtalpha^{\timeindex}\) by applying \eqref{eq:rkschemeupdate}, \eqref{eq:secondapproximationRK} on each grid cell.}
      \catch{(Linear solver failures or invalid values (\texttt{inf}s or \texttt{NaN}s) in \(\gesamtalpha^{\timeindex+1}\), \(\multipliersembedded^{\timeindex+1}\))}{
      $\timestep \gets \timestep / 2$\;
      continue\;
    }
    \(\timestep_{\mathrm{accepted}} \gets \timestep\)\;
    \(e \gets \rkerr(\gesamtalpha^{\timeindex+1}, \multipliersembedded^{\timeindex+1})\)\tcp*[l]{see \eqref{eq:embeddedrkerror}}
    \(\timestep \gets \timestep_{\mathrm{new}}(\timestep, e)\)\tcp*[l]{see \eqref{eq:timestepnew}}
  }
  \(\timevar \gets \timevar + \timestep_{\mathrm{accepted}}\)\;
  \(\gesamtalpha^{\timeindex} \gets \gesamtalpha^{\timeindex+1}\)\;
}
\end{algorithm}

For implementation of the relaxed Runge-Kutta scheme, we simply compute (compare \eqref{eq:r_gamma})
\begin{equation*}
  \sum_{\gridindex \in \gridindexset} \sum_{\rkindex=0}^{\rkstages-1} \rkb_{\rkindex}
  \multipliersrkstage{\rkindex}_{\gridindex} \cdot \uupdateterm_{\gridindex}(\multipliersrkstage{\rkindex}_{0}, \ldots, \multipliersrkstage{\rkindex}_{\gridsize-1})
\end{equation*}
on the fly while computing our Runge-Kutta scheme.
Once the time step is accepted by the adaptive control, we compute
\begin{equation*}
  \timestep \sum_{\rkindex=0}^{\rkstages-1} \rkb_{\rkindex}
  \alphaupdateterm_{\gridindex}(\multipliersrkstage{\rkindex}_{0}, \ldots, \multipliersrkstage{\rkindex}_{\gridsize-1})
\end{equation*}
(which is needed to compute \(\multipliers[\gridindex,\gamma]^{\timeindex+1}\), see \eqref{eq:modifiedrkupdate}))
and use a simple bisection
algorithm to find the root \(\gamma_{\timeindex}\) of \(r(\gamma)\) that is close to 1.
We then compute \(\multipliersfv[\gridindex]^{\timeindex+1} = \multipliersfv[\gridindex,\gamma_{\timeindex}]^{\timeindex+1}\) according to \eqref{eq:modifiedrkupdate}.


%% file: Sections/numericalexperiments.tex
\section{Numerical Experiments}\label{sec:numericalexperiments}
We want to investigate the behaviour of the new scheme
in several benchmarks. For that purpose, we use the same test cases
as in~\cite{SchneiderLeibner2019b}. Our \Cpp implementation and the generated data can be found in \cite{Leibner2021}.

In the following, we will briefly restate the test cases. For a more detailed description and plots of (numerical) solutions
see~\cite{SchneiderLeibner2019b} and references therein. As the minimum entropy models cannot handle zero densities, we use
the small isotropic distribution \(\distributionvacuum = 5 \cdot 10^{-7}\) to approximate a vacuum.
Note that we increased the vacuum density slightly compared to~\cite{SchneiderLeibner2019b} to avoid numerical difficulties with very low densities.
We use the following test cases:
\begin{itemize}
  \item \emph{Plane-source}. In this test case, all mass is concentrated in the middle of the computational domain \(\domain = [-1.2, 1.2]\), i.e.,
        we use the isotropic initial distribution
        \begin{equation*}
          \distributiontzero(\z, \SCheight) = \distributionvacuum + \delta(\z) \text{ for } \z \in \domain.
        \end{equation*}
        See \cref{sec:initialandboundaryfvscheme} for details on the implementation of this initial condition.
        The physical coefficients are set to \(\scattering \equiv 1\), \(\absorption \equiv 0\) and \(\source \equiv 0\). Vacuum boundary conditions are used.
  \item \emph{Source-beam}. In this test case, a strongly anisotropic beam enters the computational domain \(\domain = [0,3]\) from the left.
        In addition, a source is present in the interval \([1, 1.5]\). More precisely, the approximate vacuum is used as initial
        condition and boundary condition on the right-hand side, and the left boundary distribution is
        \begin{equation*}
          \distributionboundary(\timevar,0,\SCheight) = \cfrac{e^{-10^5{(\SCheight-1)}^2}}{\ints{e^{-10^5{(\SCheight-1)}^2}}}
        \end{equation*}
        The parameters are set to
        \begin{gather*}
          \absorption(\z) = \begin{cases}
            1 & \text{ if } \z\leq 2, \\
            0 & \text{ else},
          \end{cases} \quad
          \scattering(\z) = \begin{cases}
            0  & \text{ if } \z\leq 1,   \\
            2  & \text{ if } 1<\z\leq 2, \\
            10 & \text{ else}
          \end{cases} \quad
          \source(\z) = \begin{cases}
            \frac12 & \text{ if } 1\leq \z\leq 1.5, \\
            0       & \text{ else},
          \end{cases}
        \end{gather*}
  \item \emph{Point-source}. The point-source test is a smoothed three-dimensional analogue of the plane-source test.
        The initial condition in the domain \(\domain = {[-1,1]}^3\) is
        \begin{equation*}
          \distributiontzero(\spatialvar, \SC) = \distributionvacuum +
          \frac{1}{4\pi^4\sigma^3}\exp\left(-\frac{\abs{\spatialvar}^2}{\pi\sigma^2}
          \right),
        \end{equation*}
        where \(\sigma = 0.03\). The parameters are the same as in the plane-source test.
  \item \emph{Checkerboard}.
        The checkerboard test case is loosely based on a part of a reactor core~\cite{Brunner2005}.
        The domain \(\domain = {[0,7]}^3\) is split into scattering and absorbing regions, \(\domain = \domain_s \cup \domain_a\), where
        \begin{equation*}
          \domain_a = \Set*{\spatialvar = {(\x,\y,\z)}^T \in {[1,6]}^3 \given \begin{aligned}               & (\floor{\x}+\floor{\y}+\floor{\z})\bmod 2=1,                \\
                              & \spatialvar\notin {[3,4]}^3\cup [3,4]\times[5,6]\times[3,4]\end{aligned}}
        \end{equation*}
        The parameters are
        \begin{equation*}
          \scattering(\spatialvar) = \begin{cases}
            1 & \text{ if } \spatialvar\in\domain_s, \\
            0 & \text{ else},
          \end{cases},~
          \absorption(\spatialvar) = \begin{cases}
            0  & \text{ if } \spatialvar\in\domain_s, \\
            10 & \text{ else},
          \end{cases},~
          \source(\spatialvar) = \begin{cases}
            \frac{1}{4\pi} & \text{ if } \spatialvar\in {[3,4]}^3, \\
            0              & \text{ else}.
          \end{cases}
        \end{equation*}
        Vacuum initial and boundary conditions are used.
  \item \emph{Shadow}.
        The shadow test case represents an isotropic particle stream that is partially blocked by
        an absorber, resulting in a shadowed region behind the absorber.
        The particle stream is given by an isotropic boundary condition with density \(\rho = 2\)
        at \(\x = 0\). On the other boundaries of the domain \(\domain = [0,12]\times[0,4]\times[0,3]\) and as an initial condition,
        the approximate vacuum is prescribed. The parameters are as follows:
        \begin{align*}
          \scattering(\spatialvar) & = \source(\spatialvar) = 0   \\
          \absorption(\spatialvar) & = \begin{cases}
            50 & \text{ if } \spatialvar\in[2,3]\times[1,3]\times[0,2] \\
            0  & \text{ else},
          \end{cases}
        \end{align*}
\end{itemize}

Whenever we need a reference scheme, we use the splitting scheme based on~\eqref{eq:splittedsystem}.
The obvious choice might be the unsplit scheme~\eqref{eq:fvscheme},
as the new scheme is just a coordinate transformation of this scheme.
However, the splitting scheme is easy to implement and avoids the time step
restriction due to the physical parameters.
For the new scheme, a similar splitting approach is not straightforward.
Thus, using~\eqref{eq:fvscheme} as
a reference would arguably give the new scheme an unfair advantage.

\subsection{Convergence}\label{subsec:numericalconvergence}
Under some assumed bounds on the Hessian \(\optHessian\) and the local particle density \(\density\), the transformed scheme will always converge to the same solution
as the splitting scheme~\eqref{eq:splittedsystem} (see \cref{sec:convergenceproperties}).
However, in practice, taking numerical errors into account, these bounds may not always hold (in particular the lower bound on the norm of \(\optHessian\)).

In our first experiment, we thus want to validate that the two schemes converge to the same solutions also in our numerical test cases.
For that purpose, we compute numerical solutions with both schemes for varying tolerance
and time step parameter, respectively, and calculate the errors with respect to a reference solution (new scheme with \(\rktol =
10^{-9}\)).
\begin{remark}
  It might be more intuitive to compute the reference solution using the standard finite volume scheme with a very small time step \(\timestep\).
  However, since we only want to show that both schemes converge to the same solution, it does not matter whether we use
  the standard scheme or the new transformed scheme as a reference, and computing the new scheme for a
  small tolerance \(\rktol\) is significantly faster.
\end{remark}
As an error measure, we choose the \(\Lp{1}\)-error of the piecewise constant finite volume approximations
\(\LpError{1}(\moments) = \norm{\moments - \moments[\mathrm{ref}]}{\Lp{1}(\domain)}\) at the final time \(\tf\).
We use a relatively coarse grid for all test cases to be able to compute the results for very small tolerance parameter \(\rktol\) or time step
\(\timestep\)
in reasonable time.
However, we confirmed at least for large parameters (\(\rktol\in\{10^{-2},10^{-3}\}\), \(\timestep = \timestepnonadaptive\))
that the results are similar for the grid sizes and
final times used in~\cref{subsec:numericalperformance} (see \cref{tab:AdditionalErrors1d,tab:AdditionalErrors3d} in the supplementary materials). 

As can be seen in \cref{fig:Convergence1d} (for the plane source test)
and \cref{fig:AppendixConvergenceSourcebeam,fig:AppendixConvergencePointsource} in the supplementary materials (source-beam and point-source),
both schemes nicely converge to the same solution.
For the standard scheme, the
error is basically independent of the model which is not true for the new scheme. This is probably due to the fact that for the new scheme, the
error estimate
during the time stepping is calculated in transformed (\(\multipliers\)-)variables 
while the final error is
plotted in the original (\(\moments\)-)variables. 
The \(\Lp{\infty}\)-errors behave similarly (data not shown).
\begin{figure}
  \centering
  \externaltikz{PlanesourceConvergence}{\input{Images/PlanesourceConvergence}}
  \caption[Convergence of tested schemes in the plane-source test]{\(\Lp{1}\)-error against reference solution (new scheme with \(\rktol = 10^{-9}\))
    in the plane-source test (\(\gridsizeoned =
    240, \tf = 0.5\)).
    (a) New scheme for decreasing tolerance parameter \(\rktol\).
    (b) Standard scheme for decreasing time step \(\timestep\).}\label{fig:Convergence1d}
\end{figure}

\subsection{Time stepping behavior}\label{sec:timesteppingtests}
Now that we have confirmed also numerically that the new scheme indeed yields the same solutions as the standard scheme,
we would like to investigate the properties of the new scheme. We will focus on the time stepping behavior first.
We will only present the results for some exemplary models in the one-dimensional test cases here.
Results for additional models and for the three-dimensional tests are similar and can be found in the supplementary materials
(\cref{fig:Timesteps3d,fig:Timesteptimes3d,fig:AdditionalTimestepsSourceBeamM1,fig:AdditionalTimestepsSourceBeamM2,fig:AdditionalTimestepsShadowM,fig:AdditionalTimestepsShadowOtherModels,fig:AdditionalTimestepsConvergencePlanesource,fig:AdditionalTimestepsConvergencePointsource}).

\cref{fig:Timesteps1d} shows the time steps chosen by the adaptive time stepping scheme (with a tolerance of \(\rktol = 10^{-3}\))
in the one-dimensional test cases
for some exemplary models.
As expected, for all models, the new scheme takes very small time steps initially.
This can be explained by the large time derivatives
of the solution in \(\multipliers\)-variables at the beginning of the test (see \cref{ex:initialtimesteps}).

In the plane-source test (\cref{fig:Timesteps1d_a}),
the time steps are rapidly and almost monotonically increasing for all models and finally reach a time step that is
even above the maximal realizability-preserving time step~\eqref{eq:splittingschemetimestep}
used in the standard scheme (except for the \(\PMMN[2]\) model). In the source-beam test (\cref{fig:Timesteps1d_b}),
the time steps are also increasing initially but are less stable afterwards.
In particular for the \(\MN\) models there are strong oscillations in the time step sizes.

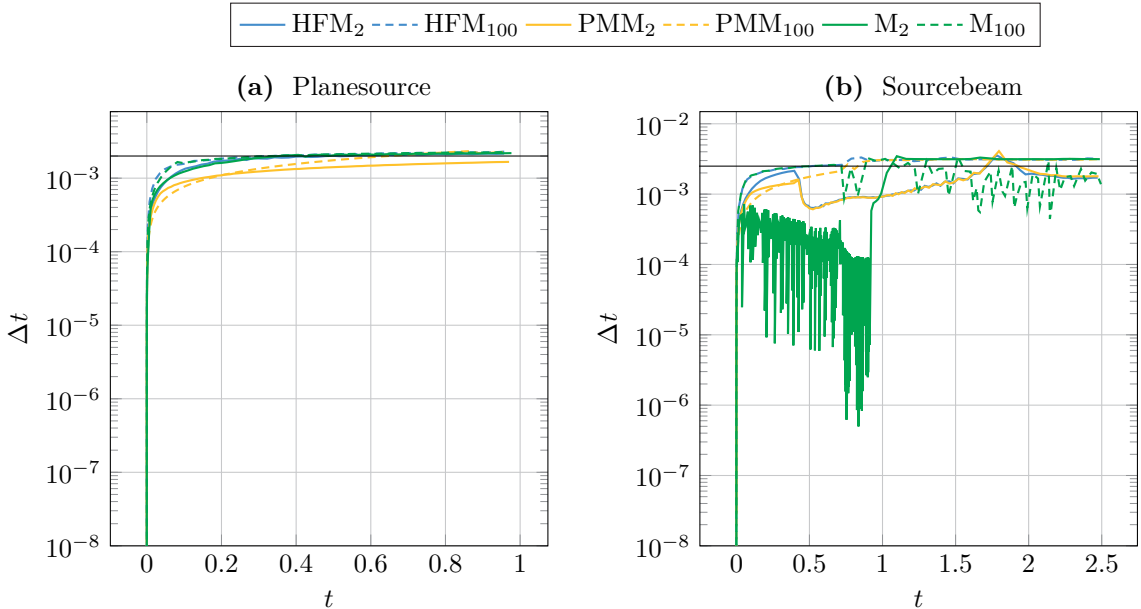
\begin{figure}
  \centering \externaltikz{Timesteps1d}{\input{Images/Timesteps1d}}
  \caption[Time steps taken in the adaptive Runge-Kutta scheme for one-dimensional tests]{
    Time steps
    taken in the adaptive Runge-Kutta scheme. For clarity, instead of plotting each time point we
    plot the mean of 10 time steps each to avoid small oscillations.
    The solid horizontal line represents the upper bound~\eqref{eq:splittingschemetimestep} on the time
    step used in the standard splitting scheme
    (which approximately agrees with the bound~\eqref{eq:fvschemetimestep} because
    \(\frac{1}{\gridwidth} \gg \crosssection\) for these test cases).
    (a)~Plane-source test, \(\gridsizeoned = 1200\), \(\tf = 1\), \(\rktol=0.001\).
    (b)~Source-beam test, \(\gridsizeoned = 1200\), \(\tf = 2.5\), \(\rktol=0.001\).}\label{fig:Timesteps1d}
\end{figure}

To test the influence of the tolerance parameter \(\rktol\) on the time steps,
we compute the test cases for the \(\MN[10]\) model again for varying \(\rktol\).
The \(\MN[10]\) model was chosen since it is more complex than the \(\MN[2]\)
model, not as expensive to compute as the \(\MN[100]\) model and still shows the
time step oscillations in the source-beam test. However, we also tested
several other models and found that they all show a similar behavior
with respect to the \(\rktol\) tolerance.
As can be seen in \cref{fig:TimestepsConvergence}, for small
tolerances (\(\rktol \leq 10^{-3}\)), increasing \(\rktol\)
basically just scales the time step curve by a constant factor
which is due to the third-order time-stepping scheme
(increasing \(\rktol\) by a factor of \(10\)
results in an increase of the time step by a factor of
approximately \(\sqrt[3]{10}\)).
However, this scaling does not extend to large tolerance parameters (\(\rktol=10^{-1},10^{-2}\)).
In particular, for the plane-source test, increasing the tolerance above \(\rktol = 10^{-3}\) does not
result in larger time steps (see \cref{fig:TimestepsConvergence_a}).
In addition, the time steps oscillate much more. This is due to the fact that the
time step predicted from the standard error estimate for these tolerances is
often too large, leading to {\ttfamily inf}s
and {\ttfamily NaN}s during the computations and a subsequent reduction in the time step
(compare \cref{sec:adaptivetimestepping}).

\begin{figure}
  \centering
  \externaltikz{TimestepsConvergence}{\input{Images/TimestepsConvergence}}
  \caption[Time steps for different models in the plane-source and source-beam tests]{Time steps for the \(\MN[10]\) model in the
    plane-source (\(\gridsize=1200\), \(\tf = 1\))
    and the source-beam test (\(\gridsize=600\), \(\tf = 2.5\)) for different tolerance parameters~\(\rktol\).
    Again, the solid horizontal line represents the upper bound~\eqref{eq:splittingschemetimestep} and the time steps are plotted
    as the mean of 10 time steps each.}\label{fig:TimestepsConvergence}
\end{figure}
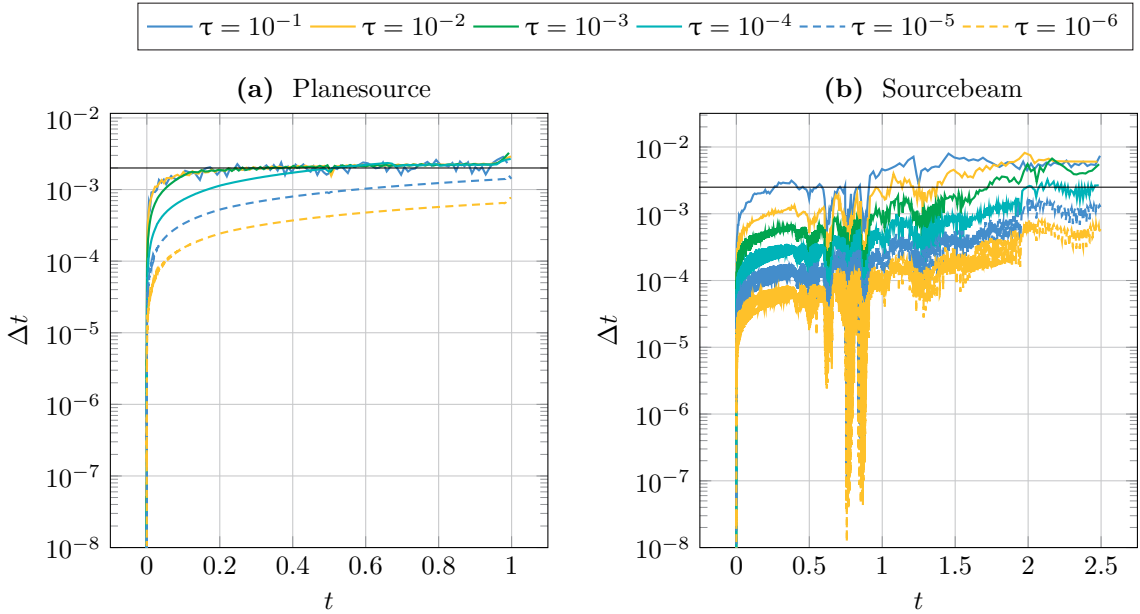

Finally, we measured the times needed to compute a single time step of the new scheme or the standard scheme
for several models (see \cref{fig:Timesteptimes1d}).
For the new scheme, the time needed to compute a time step is basically constant, except for time steps which have to be recomputed because
the error estimate is above the tolerance. In contrast, time step computation times of the standard scheme are increasing over time.
This is probably mostly due to the caching used in the implementation of the standard scheme (see \cref{subsec:caching})
which is particularly effective during the first time
steps where most grid cells still contain the initial approximate vacuum.
The new scheme does not use any caching. As a consequence, the standard scheme is faster for the first few time steps
but after a short time the new scheme's time steps are computed significantly faster.

As can also be seen from \cref{fig:Timesteptimes1d},
recomputations of time steps in the adaptive time stepping scheme (which show up in \cref{fig:Timesteptimes1d}
as spikes in the computation times of the new scheme) occur rarely, except for the \(\MN\)
models in the source-beam test. This is in line with the erratic time stepping behavior
that we observed for these models (compare \cref{fig:Timesteps1d_b,fig:TimestepsConvergence_b})
and can be improved using regularization (see next section).

\begin{figure}
  \centering \externaltikz{Timesteptimes1d}{\input{Images/Timesteptimes1d}}
  \caption[Wall times for computing a single time step of transformed
    and standard schemes in one-dimensional tests]{
    Wall times for computing a single time step of new
    and standard schemes (using \(\multipliers\) and \(\moments\) variables, respectively) in one-dimensional tests
    (a)~Plane-source test, \(\gridsizeoned = 1200\), \(\tf = 1\), \(\rktol=0.001\).
    (b)~Source-beam test, \(\gridsizeoned = 1200\), \(\tf = 2.5\), \(\rktol=0.001\).}\label{fig:Timesteptimes1d}
\end{figure}

\subsection{Regularization}
\subsubsection{Isotropic regularization of the Hessian}
In the previous section,
we saw that the time steps for the \(\MN[10]\) model in the source-beam test (\cref{fig:TimestepsConvergence_b})
oscillate a lot between \(\timevar = 0.5\) and \(\timevar = 1\) and
sometimes even get as small as \(10^{-8}\). These oscillations are mostly caused
by ill-conditioned Hessian matrices which arise from the interaction between
the highly anisotropic particle beam from the left boundary
and the particles from the source \(\source\) in the absorbing
but non-scattering region \([0, 1]\).

A possible workaround for this problem is the regularization
\eqref{eq:alpha_rk_scheme_regularized} which adds a small
isotropic particle density during computation of the Hessian matrix. As can be seen when
comparing \cref{fig:SourcebeamRegularization_a} to \cref{fig:TimestepsConvergence_b},
this indeed improves the time step sizes and removes the very small time steps.
As a consequence, the regularized scheme needs significantly less time
steps (see \cref{fig:SourcebeamRegularization_b}). For example, the regularized scheme
with regularization parameter \(\newschemeregparam = 10^{-7}\) only needs \(\numtimesteps = 1495\)
time steps instead of \(3918\) for the non-regularized scheme.

The price to pay for the regularization is an additional error
in the order of \(10^{-3}\) to \(10^{-4}\) at the final time \(\tf\)
(see \cref{fig:SourcebeamRegularization_b} for the \(\Lp{1}\) error,
the \(\Lp{\infty}\) error is of the same order (data not shown)).
This seems to be below the typical error introduced by
the temporal discretization in the standard scheme
(see \cref{sec:timesteppingparamchoice,fig:AppendixConvergenceSourcebeam_b}).

Increasing the regularization parameter \(\newschemeregparam\) above \(10^{-7}\)
increases the regularization error but does not further decrease
the number of time steps.

\begin{figure}
  \centering
  \externaltikz{SourcebeamRegularization}{\input{Images/SourcebeamRegularization}}
  \caption[Effect of regularization on timesteps and error]{Effect of the regularization \eqref{eq:alpha_rk_scheme_regularized}
    on time steps and errors for the \(\MN[10]\) model in the source-beam test.
    (a)~Time steps for different tolerance parameters \(\rktol\) using a regularization parameter of \(\newschemeregparam = 10^{-7}\).
    (b)~Errors (\(\Lp{1}\)-error with respect to the non-regularized solution) and number of time steps \(\numtimesteps\) for \(\rktol = 10^{-9}\) and
    varying regularization parameter \(\newschemeregparam\). The non-regularized version uses \(\numtimesteps = 3918\) time steps. }\label{fig:SourcebeamRegularization}
\end{figure}

\subsubsection{Regularization for the hat function basis}\label{sec:hfmbasisregularization}
For the hat function (\(\HFMN\)) models, we can use a different regularization technique
that might result in smaller regularization errors.
Although, according to our testing, time step declines are much less common for the \(\HFMN\) models than for the \(\MN\) models,
we observed such declines for the \(\HFMN[6]\) model in the shadow test
around \(\timevar = 8\) (see \cref{fig:TimestepsHFM6Shadow_a}).
In this test, (almost) all particles enter the domain with positive \(\x\)-velocity.
Since there is no scattering and a strongly absorbing region,
the density of particles with negative \(\x\)-velocity becomes very low in parts of the domain
which also leads to large (in absolute values) negative entries
in the \(\gesamtalpha\) coefficient vectors. To improve the situation, we tested the
regularization technique introduced in \cref{sec:directconstraintsregularization}:
Whenever the time step falls below \(\timestep_{\text{min}} = 0.01\),
we replace all entries in the \(\gesamtalpha\) vector that are smaller
than \(\multiplierscomp{\text{min}}=-1000\) by \(\multiplierscomp{\text{min}}\).

As
can be seen in \cref{fig:TimestepsHFM6Shadow_b},
this simple regularization technique removes the very small time steps. The overall number of time steps is reduced
from \(1152\) to \(547\). The \(\Lp{1}\)-error at the final time \(\tf\) between the results with
and without regularization is only about \(10^{-6}\) in this case.

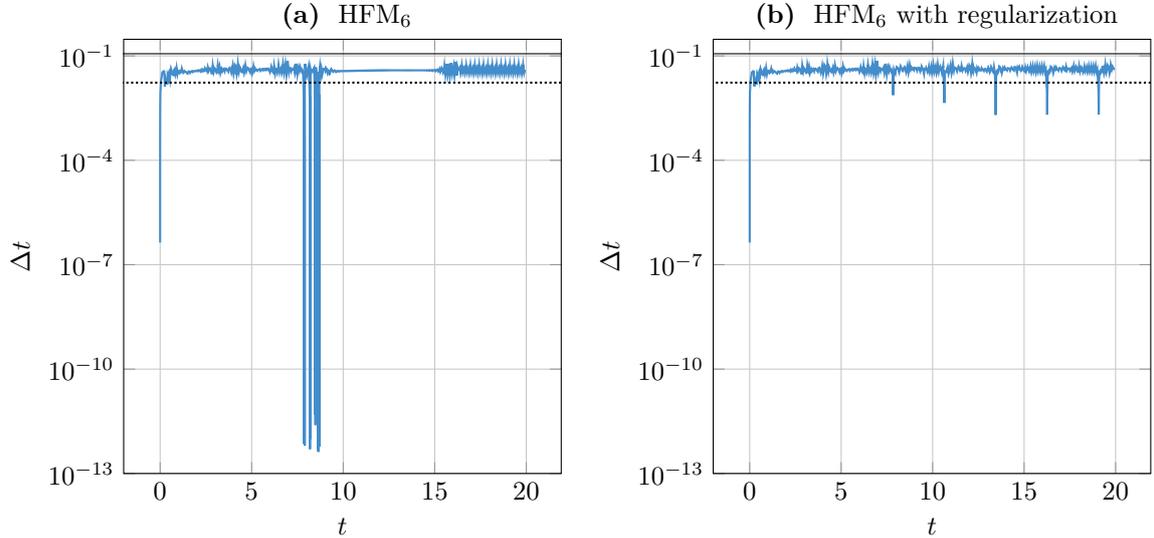
\begin{figure}
  \centering
  \externaltikz{TimestepsHFMShadow}{\input{Images/TimestepsHFM6Shadow.tex}}
  \caption[Effect of regularization on timesteps and error]{
    Time steps taken by the adaptive Runge-Kutta scheme in the \(\HFMN[6]\) shadow
    test case (\(\gridsize=60\times20\times15\), \(\tf=20\)) for a tolerance of \(\rktol = 0.01\).
    The solid and dotted horizontal line represent the time step restrictions~\eqref{eq:splittingschemetimestep} and~\eqref{eq:fvschemetimestep}, respectively.
    (a)~No regularization. The time step sharply declines around \(\timevar=8\).
    (b)~When enforcing a lower bound of \(-1000\) on the
    entries of the multipliers \(\multipliers\), the time steps stay above \(10^{-3}\).
  }\label{fig:TimestepsHFM6Shadow}
\end{figure}

\subsection{Entropy stability}
To test the entropy stability properties of the different schemes, we
calculate the entropy \(\gesamtentropy\) (see \cref{sec:entropystability})
at each time step and compute the difference between the actual entropy
\(\gesamtentropy(\gesamtalpha^{\timeindex+1})\)
and the entropy given by the discrete entropy law.
To that end, let
\(\gesamtentropy_{\text{est}}(\gesamtalpha^{\timeindex})\)
be the entropy estimated for the next time point \(\timevar_{\timeindex+1}\)
from \(\gesamtalpha^{\timeindex}\) using the entropy law \eqref{eq:entropyupdate2}.
To get an compact error measure, we use the cumulated difference between actual entropy and
estimated entropy, i.e.
\begin{equation*}
  \Delta\gesamtentropy \coloneqq
  \sum_{\timeindex=0}^{\numtimesteps-1} \abs{\gesamtentropy(\gesamtalpha^{\timeindex+1})
    - \gesamtentropy_{\text{est}}(\gesamtalpha^{\timeindex})} \timestep_{\timeindex}
\end{equation*}
We tested several representative moment models in the one-dimensional test cases using the
new transformed scheme with either
the non-modified Runge-Kutta scheme (RK) or the relaxed Runge-Kutta scheme (RRK).
We restrict our investigation to
one dimension here since we do not expect a qualitatively
different behavior in three dimensions.

As can be seen in \cref{tab:EntropyErrors}, with the non-relaxed Runge-Kutta scheme
the difference \(\Delta\gesamtentropy\) between the entropy of the solutions
and the discrete entropy law \(\eqref{eq:entropyupdate2}\)
is in the order of \(10^{-4}\) to \(10^{-7}\). The entropy
\(\gesamtentropy(\gesamtalpha^{\timeindex})\) is in the order of \(100\) to \(1000\),
so the relative error is quite low even without relaxation.
If relaxation is used,
the error vanishes for all test cases (up to a remainder in the order of \(10^{-12}\) which is the tolerance
parameter for our root finding algorithm). However, the relaxation comes at a price,
as can be seen from the computation times. With relaxation, the computations take at least
twice as long (for the \(\MN[50]\) model in the source-beam test),
up to a factor of about 20 for the \(\HFMN[50]\) model in the plane-source test. Note that we use a simple
custom bisection root finding algorithm to compute the relaxation parameters
and that we did not optimize the relaxed version for performance,
so it should be possible to significantly reduce the performance impact
of the relaxation. Still, given that the error is already quite low without relaxation,
it seems advisable to simply use the non-modified Runge-Kutta scheme
unless exact entropy stability is required in the application.
\begin{table}[htbp]
  \centering
  \begin{tabular}{l c c c c c c c}
                 &               &                                  & \multicolumn{3}{c}{Wall time} & \multicolumn{2}{c}{\(\Delta\gesamtentropy\)} \\
    \cmidrule(r){4-6} \cmidrule(r){7-8}
    Test case    & Model         & \(|\gesamtentropy|_{\text{av}}\) & RK (s) & RRK (s) & Factor     & RK      & RRK                                \\
    \midrule
    Plane-source & \(\HFMN[10]\) & 877                              & 22     & 170     & 7.7        & 3.9e-04 & 4.4e-13                            \\
    Plane-source & \(\HFMN[50]\) & 878                              & 81     & 1689    & 20.8       & 4.7e-04 & 4.5e-13                            \\
    Plane-source & \(\PMMN[10]\) & 881                              & 24     & 135     & 5.6        & 3.7e-05 & 4.7e-13                            \\
    Plane-source & \(\PMMN[50]\) & 872                              & 83     & 1181    & 14.2       & 5.9e-05 & 4.6e-13                            \\
    Plane-source & \(\MN[10]\)   & 865                              & 45     & 189     & 4.2        & 2.5e-04 & 4.4e-13                            \\
    Plane-source & \(\MN[50]\)   & 876                              & 253    & 821     & 3.3        & 3.2e-04 & 4.6e-13                            \\
    \midrule
    Source-beam  & \(\HFMN[10]\) & 411                              & 41     & 442     & 10.8       & 1.3e-04 & 1.1e-12                            \\
    Source-beam  & \(\HFMN[50]\) & 500                              & 140    & 2073    & 14.8       & 2.2e-05 & 1.1e-12                            \\
    Source-beam  & \(\PMMN[10]\) & 336                              & 38     & 474     & 12.5       & 3.7e-04 & 2.5e-12                            \\
    Source-beam  & \(\PMMN[50]\) & 492                              & 122    & 1413    & 11.6       & 5.2e-06 & 1.1e-11                            \\
    Source-beam  & \(\MN[10]\)   & 507                              & 955    & 3519    & 3.7        & 3.8e-07 & 1.4e-11                            \\
    Source-beam  & \(\MN[50]\)   & 518                              & 1827   & 3398    & 1.9        & 1.4e-05 & 1.3e-12                            \\
  \end{tabular}
  \caption[Entropy stability results]{Entropy stability results. The average absolute value of the entropy is computed as
  \(|\gesamtentropy|_{\text{av}} = \frac{1}{\tf} \sum_{\timeindex=0}^{\numtimesteps-1} |\gesamtentropy(\gesamtalpha^{\timeindex})| \timestep_{\timeindex}\).
  RK: Transformed scheme using Runge-Kutta method. RRK: Transformed scheme using relaxed Runge-Kutta method.
  }\label{tab:EntropyErrors}
\end{table}


\subsection{Performance}\label{subsec:numericalperformance}
We now want to compare the performance of the new scheme to the standard splitting scheme.
Since the standard scheme uses a fixed time step \(\timestep\) and the new scheme
uses adaptive time steps controlled by the tolerance parameter \(\rktol\),
we first have to decide on how to choose these parameters to have a fair comparison.

\subsubsection{Choice of time stepping parameters}\label{sec:timesteppingparamchoice}
For the standard scheme, we have an
upper bound \(\timestepnonadaptive \coloneqq \frac{1-\opttoleps}{\sqrt{\dimension}}\gridwidth\)
(see \cref{eq:CFLfirstorderschemeinexact})
on the time step due to realizability considerations.
If we again consider the convergence results for the standard scheme
(compare \cref{fig:Convergence1d_b,fig:AppendixConvergenceSourcebeam_b,fig:AppendixConvergencePointsource_b,tab:AdditionalErrors1d,tab:AdditionalErrors3d})
we see that choosing \(\timestep = \timestepnonadaptive\) results in a time stepping error
in the order of about \(10^{-2}\) to \(10^{-3}\).
Note that the errors regarded here are solely due to the time stepping, i.e.\ the solution both schemes
are converging to in the convergence tests
is the exact solution of the semidiscrete moment equation \eqref{eq:finitevolumesemidiscrete}.
Since we are usually interested in an approximation of the solution to the kinetic
equation~\eqref{eq:FokkerPlanckEasy2}, we also have to take errors into account that arise
due to the space discretization and due to the moment approximation.
Choosing a time step smaller than \(\timestepnonadaptive\) thus would only be reasonable
if the time stepping error is of the same order or even larger than the errors due to the spatial discretization and the
moment approximation. This seems to be the case, e.g., for the checkerboard test where the moment approximation errors are relatively small
and for the shadow test which has a large final time \(\tf\) such
that time discretization errors accumulate over time \cite{SchneiderLeibner2019b}.
However, for most of the regarded test cases, the errors introduced by the moment approximation (compare~\cite{SchneiderLeibner2019b})
are much larger than the error of up to \(10^{-2}\) we observed due to the time stepping.
We thus always use \(\timestep = \timestepnonadaptive\) with \(\opttoleps = 0.1\) for the standard scheme.

By the same arguments, we could use a tolerance parameter of \(\rktol = 0.1\) for the new scheme which also yields time stepping errors
in the order of \(10^{-3}\) to \(10^{-2}\) in our tests. However, we observed in \cref{sec:timesteppingtests} that, with
with the current standard error estimate,
increasing the tolerance above about \(10^{-3}\) to \(10^{-2}\)
does not necessarily improve performance since the time steps do not increase accordingly.
In some cases, we even observed increased
computation
times for larger tolerances as time steps had to be recomputed more frequently. We thus
choose a tolerance of \(\rktol = 10^{-3}\) in slab geometry and \(\rktol =
10^{-2}\) in three dimensions.

Note that with this choice of parameters, the time stepping error is
probably considerably lower for the new scheme. For the test cases
where this error is insignificant, it might be possible to
choose larger time steps for the new scheme with an improved
error estimate that is specifically adapted to the moment equations.
However, we did not find such an error estimate yet.
On the other hand, for test cases (e.g.\ checkerboard)
where the time stepping error is relevant, a smaller
time step would have to be
used for the standard scheme which then would
be significantly slower than the following results indicate.
Alternatively, adaptive time stepping could also for the standard scheme.
However, finding an adaptive time stepping scheme of the desired order that (provenly) preserves
realizability may be difficult.
The Bogacki-Shampine method used here is realizability-preserving (with the same
times step restriction~\eqref{eq:fvschemetimestep} as the forward Euler method) since the intermediate stages
are just convex combinations of forward Euler steps and the zero vector. Similarly,
the strong-stability-preserving embedded methods from the preprint~\cite{Conde2018} could be used.
For the Dormand-Prince method~\cite{Dormand1980}, on the other hand,
it is not clear under which conditions realizability is preserved
(due to the negative coefficients in its Butcher tableau).
In contrast, any adaptive time
stepping scheme can be used for the new scheme.

For the standard scheme, we use the regularization techniques described in \cref{sec:optimizationproblem}
to ensure that we are always able to solve the optimization problems. For the new scheme,
we do not use any regularization.

\subsubsection{Timings}
Computational times for the one-dimensional test cases can be found in \cref{fig:Timings1d}.
As expected (see~\cite{SchneiderLeibner2019b}), computational times are increasing linearly (\(\HFMN\), \(\PMMN\)) or quadratically (\(\MN\)).

In the plane-source test (see \cref{fig:Timings1d_a}), the new scheme is several times faster
than the standard scheme for all models except for the low-order \(\HFMN\) and
\(\PMMN\) models.

For the source-beam test, we saw in \cref{sec:timesteppingtests} that
the time steps significantly vary over time for some models. In particular,
the \(\MN\) models show strongly oscillating time steps.
These oscillations seem to be much less pronounced for the higher-order models
(see Supplementary \cref{fig:AdditionalTimestepsSourceBeamM1,fig:AdditionalTimestepsSourceBeamM2}).
As a consequence, computational times for the new scheme are not increasing
monotonically with the moment order (see \cref{fig:Timings1d_b}), e.g.\ computing the \(\MN[20]\) model takes longer than
computing the \(\MN[60]\) model.
Thus, the standard scheme is faster for the low-order models and again significantly slower for the high-order models.
Except for the \(\HFMN[2]\) and \(\PMMN[2]\) model, the \(\HFMN\) and \(\PMMN\) models do not show these oscillations
and again reach time steps that are
larger than \(\timestepnonadaptive\) after some time. Consequently, overall computation times are faster with the new scheme.

For all models, the time steps are initially very small (see analysis in \cref{subsec:numericalconvergence})
but are rapidly increasing (see \cref{fig:Timesteps1d_a}).
After some time, the time steps are even larger than the time steps taken by the standard scheme for most models. In addition,
the time to compute a time step is (on average) much smaller for the new scheme (see~\cref{fig:Timesteptimes1d_a}). This is especially
true for the \(\MN\) models which is why the speed-up for these models is significantly larger than for the \(\HFMN\) and \(\PMMN\) models where
solving the optimization
problem is already quite fast.

In the three-dimensional point-source test, the final time \(\tf=0.75\) is relatively small and none of the models reaches
\(\timestepnonadaptive\)
during the test (compare Supplementary \cref{fig:Timesteps3d}). Consequently, the computation times are
only slightly faster for most models with the new scheme (see \cref{fig:Timings3d_a}). The
higher-order
\(\PMMN\) models show considerably smaller time steps than the other models and thus overall
computation times are even higher than with the standard
scheme.

The checkerboard test case has strongly absorbing regions and thus is the first test where the time step restrictions~\eqref{eq:splittingschemetimestep}
and~\eqref{eq:fvschemetimestep} significantly differ. After some time, the time steps are mostly between these two bounds and even exceed
the upper bound several times (compare Supplementary \cref{fig:Timesteps3d_b}).
The higher order \(\MN\) models show some oscillations in the beginning but much less than in
the source-beam test and the time steps always stay relatively large.
Thus, overall computation times are greatly improved and up to ten times as
fast as with the standard splitting scheme (see \cref{fig:Timings3d_b}).
In addition, as mentioned earlier, for this test case the increased accuracy of the new scheme
might be important, as the error due to the
timestepping with the standard scheme are of the same order as the error due to the spatial and moment
approximation.
Again, the speed-up is smaller for the \(\HFMN\) models and non-existent for the \(\PMMN\) models.

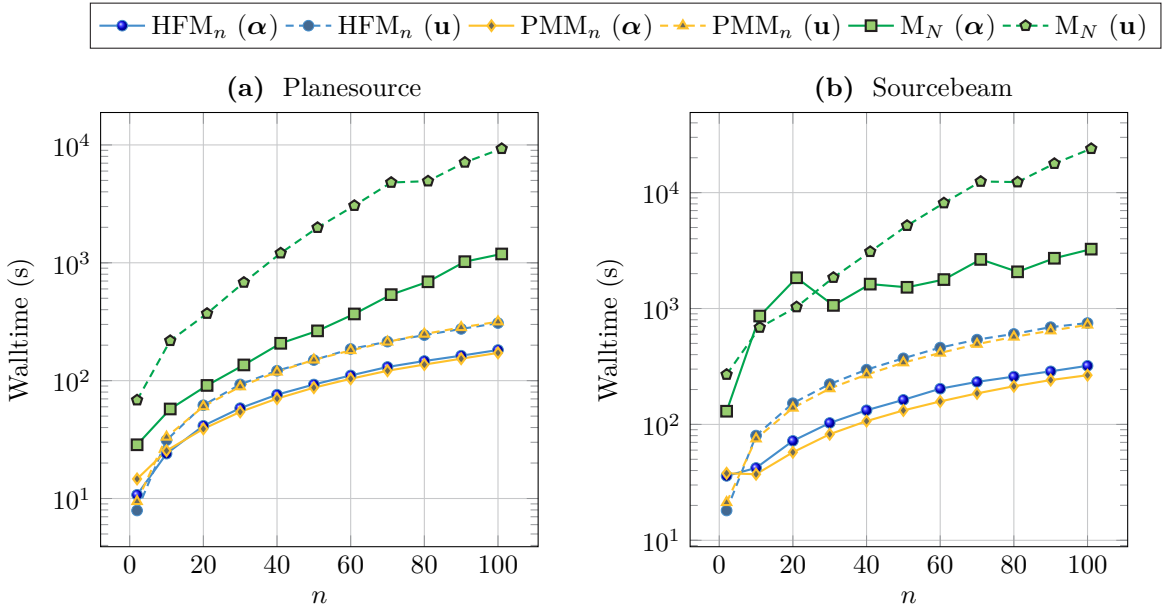
\begin{figure}
  \centering \externaltikz{Timings1d}{\input{Images/Timings1d}}
  \caption[Timings for transformed and standard finite volume schemes in one-dimensional tests]{Computational
    times versus moment number \(\momentnumber\) for the two schemes in the one-dimensional tests
    (\(\gridsizeoned = 1200\) grid elements, \(\rktol = 0.001\), no parallelization). }\label{fig:Timings1d}
\end{figure}

\begin{figure}
  \centering \externaltikz{Timings3d}{\input{Images/Timings3d}}
  \caption[Timings for transformed and standard finite volume schemes in three-dimensional tests]{Computational times versus moment number \(\momentnumber\)
    for the two schemes in the three-dimensional tests (\(\rktol = 0.01\), 32 threads, 1000 tasks per thread)
    and parallel scaling.
    (a)~Point-source problem (\(\gridsize=50^3\), \(\tf=0.75\)). The new scheme is slightly faster for the \(\MN\) and \(\HFMN\)
    models and slower for the \(\PMMN\) models.
    (b)~Checkerboard problem (\(\gridsize=70^3\), \(\tf=3.2\)). Here, the new scheme is several times faster for the \(\MN\) models and
    slightly faster for the \(\HFMN\) models. In the \(\PMMN\) tests, the new model is on par or slightly faster in the low-order tests and
    slightly slower for \(\PMMN[512]\).
    (c)~Shadow problem (\(\gridsize=60\times20\times15\), \(\tf=20\)). The new scheme is significantly slower for the
    \(\HFMN\) and \(\PMMN\) models and about as fast as the standard scheme for most \(\MN\) models.
    (d)~Computational times for 10 time steps in the point-source
    test (\(\gridsize = 100^3\), \(\MN[3]\) model) against number of threads. The dotted
    lines represent perfect scaling.\ t/t: tasks per thread}\label{fig:Timings3d}
\end{figure}

The shadow test case is highly challenging for the numerical solvers. In the absorbing domain, very small local particle densities occur which lead
to numerical problems when inverting the Hessians (whose entries scale with the density).
In addition, as only right-going particles are entering the
domain,
densities for particles with negative \(\x\)-velocity decline much faster than positive \(\x\)-velocity densities, resulting in very anisotropic
distributions
and ill-conditioned Hessians.
For the standard scheme,
we are dealing with these problems by enforcing a minimum density and using an
isotropic regularization technique to replace ill-conditioned
moment vectors (see \cref{sec:optimizationproblem}). These techniques,
in particular the isotropic regularization, introduce additional errors
which in theory might completely alter the solution. In practice,
regularization is usually mainly applied to moment vectors with very low densities and thus does not destroy accuracy.
However, it should be noted that this
is not guaranteed automatically and has to be verified for every new application of the scheme.
In our case, regularization is massively used by the
\(\MN\) and \(\PMMN\) models. The \(\HFMN\) models do not use regularization.

For the new scheme, for this comparison, we did not use any regularization. Thus, it is particularly remarkable
that, for the \(\MN\) models, the new scheme is about as fast as the standard splitting scheme,
although frequent recomputations can be observed (see Supplementary \cref{fig:Timesteptimes3d_d}) and the time steps
are considerably smaller and highly varying (compare Supplementary \cref{fig:AdditionalTimestepsShadowM})

For the \(\PMMN\) and higher-order \(\HFMN\) models,
computational times are several times higher for the new scheme than
for the standard splitting scheme (see \cref{fig:Timings3d_c}). For these models,
the time steps are converging to a value well above the maximum time step of the unsplit scheme but also significantly below the time step of the
splitting scheme (compare Supplementary \cref{fig:AdditionalTimestepsShadowOtherModels}).
In addition, computing a time step for these models is already quite fast with the standard scheme
and the speed-up of the new scheme is not large
enough to compensate the smaller time step (compare Supplementary \cref{fig:Timesteptimes3d_c}).
Note however
that the \(\PMMN\) models do not show significant oscillations or other problems though the standard scheme has to use regularization.
To be competitive with respect to computation times in this test case,
the new scheme probably also needs to use a splitting technique. Though there is no formal limit on the time step for the new scheme due to the
strong absorption, we would expect the approximation error (and thus the time step) to be dominated by this term.

\subsubsection{Parallel scaling}\label{subsec:scaling}
As already mentioned, one of the major drawbacks of the minimum-entropy-based moment models are their computational costs. As we have seen
above,
the new scheme often is several times faster than the scheme in standard variables. However, even with this speed-up computations without
parallelization
still take excessively long. In addition, there are cases where the new scheme is not faster or even slower than the standard scheme.

As the minimization problems or matrix inversions on different grid cells are
independent, parallelization is easily possible for both schemes. However, for the standard scheme, load balancing may be a serious
issue~\cite{KristopherGarrett2015}.
Usually, some minimization problems are harder to solve than others, resulting in different numbers of iterations in the Newton scheme. The new
scheme does not have this problem,
as it only needs the inversion of a relatively small positive definite matrix in each grid cell. For these matrix sizes, direct solvers usually
perform at least as good as iterative methods and take an approximately constant time per inversion.

To investigate the scaling behaviour, for both schemes we computed ten time steps of the point-source test case with a varying number of threads. We
use a work-stealing task-based
parallelization (implemented using Intel \texttt{TBB}~\cite{Intel2020}). To see the impact of load balancing we perform all test
both with 1 task per
thread (no load balancing) and
1000 tasks per thread. The results are shown in \cref{fig:ScalingPointsource}. If load balancing is used, both schemes scale almost perfectly to 16
threads.
When going to 32 threads the scaling is slightly worse which may be due to the used dual-socket system with \(2\times16\) CPU cores.

Removing the load balancing has a large impact on the standard scheme while the new scheme is much less affected. We would expect that this
difference is emphasized if even more
threads (or processes) are used. The new scheme thus should be better suited to massively (MPI)-parallel computations.

\subsubsection{Masslumping for the transformed scheme with hat function basis}
The basis functions used by the \(\HFMN\) models are basically the Lagrange \(\mathbb{P}_1\) nodal basis functions used in the
(continuous) finite element method, i.e.\
each basis function evaluates to 1 on exactly one node of the triangulation and to 0 on all other nodes (see \cref{sec:FirstOrderFEMBases}). As a
consequence, the
Hessian matrix~\eqref{eq:du_dalpha} is tridiagonal (in one dimension) or sparse (three dimensions). Compared to the \(\MN\) models where the
Hessian is dense, this significantly
reduces the computational effort required for assembly and inversion. However, especially in three dimensions, assembling and inverting the Hessian
matrix still account
for the vast majority of the new scheme's computational time.

We can significantly speed up these computations by using a quadrature that only contains the nodes of the triangulation. With such a quadrature,
the basis functions always evaluate to either zero or one and the Hessian matrix becomes diagonal. The downside is, of course,
that an additional quadrature error is introduced as the nodal quadrature is only of first order. However,
this additional error is of the same order as introduced by the linear finite element discretization.
This approach is sometimes called masslumping as using such a quadrature diagonalizes the mass matrix in the finite element method (``all mass is
lumped together on the diagonal'').

\begin{remark}
  Masslumping could also be used for the standard scheme and should lead to similar speed-ups (assuming that the masslumping does not negatively affect the
  number of iterations needed for the solution of the optimization problems). Since our focus in this work is on the new scheme,
  we did not yet test masslumping for the standard scheme.
\end{remark}

For the one-dimensional tests, we use the two-point Gauss-Lobatto quadrature in each interval (containing only the end-points of interval) for the
masslumped version. As
quadrature points that are on the same vertex of the partition can be merged, we only have one quadrature point per vertex. The reference
quadrature uses
24 quadrature points per interval.
In addition, we only have to evaluate one component of the integrand per quadrature point (the one corresponding to the non-zero basis function)
instead of two.
Overall, this reduces the number of integrand evaluations by a factor of about 48.
The results can be found in \cref{fig:TimingsMasslumped1d}.
For both test cases, almost independently of the number of moments \(\momentnumber\), the computations are
about 40 times as fast using masslumping (see \cref{fig:TimingsMasslumped1d_a}). This is in line with the reduction
in the number of evaluations. The \(\Lp{1}\) errors (compared to the non-masslumped result) are
decreasing with second order (see \cref{fig:TimingsMasslumped1d_b}). The \(\Lp{\infty}\) errors in the
source-beam test are decreasing with order about 1.3,
while the \(\Lp{\infty}\) errors in the plane-source tests are converging with very low order.

\begin{figure}
  \centering \externaltikz{TimingsMasslumped1d}{\input{Images/TimingsMasslumped1d}}
  \caption[Effect of masslumping on computational times and accuracy of the transformed finite volume scheme in one-dimensional tests]{Effect of masslumping on computational times and accuracy of the \(\HFMN\) models in the one-dimensional tests (1200 grid
    elements, \(\rktol = 0.001\), no parallelization).
    Ps: Plane-source, Sb: Source-beam, ml: masslumped.
    (a) Computational times versus moment number.
    (b) Errors introduced by masslumping (reference is the non-masslumped solution).}\label{fig:TimingsMasslumped1d}
\end{figure}
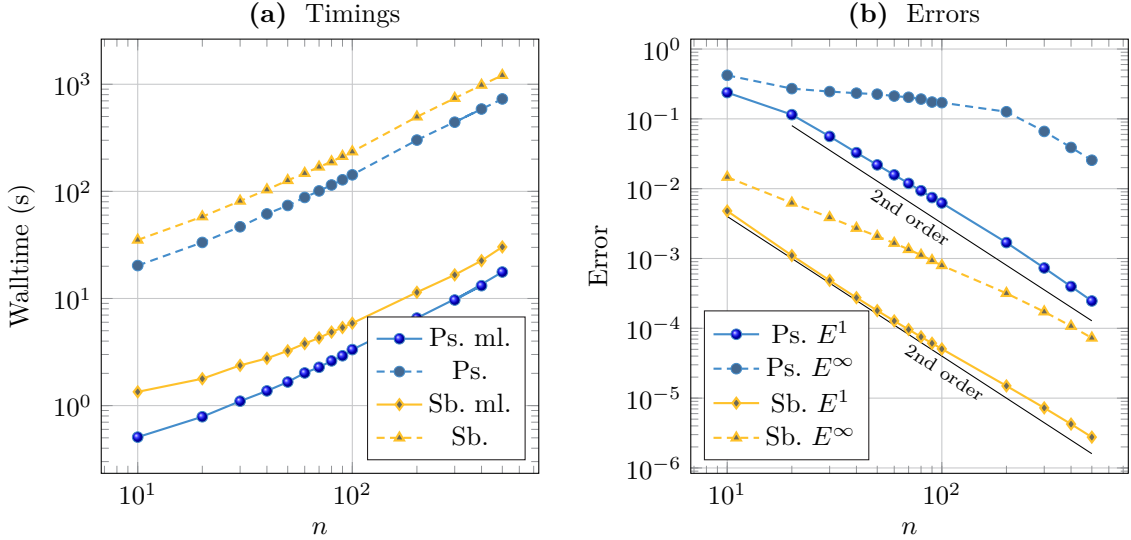

For the three-dimensional tests, we use the vertex quadrature on the
reference triangle (with the vertexes \((0, 0)\), \((1, 0)\), \((0, 1)\) as
quadrature points and
weight \(\frac{1}{6}\) each) transferred to each spherical triangle. This results in one
quadrature point per vertex of the
triangulation.
The triangulation consists of \(2\cdot 4^{\refinementnumber+1}\) triangles and \(2 + 4^{\refinementnumber+1}\)
vertices (where \(\refinementnumber\) is the number of refinements of the initial octants, see~\cref{sec:FirstOrderFEMBases}),
and the
reference quadrature has 55 quadrature points per spherical triangle.
The standard implementation thus uses about 110 times as many evaluations.
The results can be found in \cref{fig:TimingsMasslumped3d}. For all test cases, the masslumped version is more than
two orders of magnitude faster than the version using the reference quadrature. The maximum speed-up is 216 times (point-source,
\(\HFMN[1026]\)),
which is considerably higher than the reduction in quadrature points. The additional speed-up is due to the more efficient implementation,
as the masslumped version does not need to use sparse matrices and the associated indirect indexing. When looking at the profiler results,
we see that the time for assembling and inverting the (diagonal) Hessian is negligible in the masslumped version.
Overall, the time needed for the operator evaluation (which consists of calculating the kinetic fluxes and the source term and applying the inverse
Hessian matrix)
has been reduced to a point where the vector operations in the adaptive Runge-Kutta scheme now make up a major part of the computation time.

For all three-dimensional test cases, the errors compared to the non-masslumped version are
quite large (see \cref{fig:TimingsMasslumped3d_b}). However, both the \(\Lp{1}\) and
the \(\Lp{\infty}\) error
converge with first order in \(\momentnumber\) for all test cases which corresponds
to second-order convergence in the grid width of the (velocity space) triangulation,
as each refinement halves the grid width but increases the
number of vertices (approximately)
by a factor of 4. The convergence rate thus is similar (for the checkerboard test) to or
even higher (point-source, shadow) than the convergence
of the (second-order discretization of the) moment approximation (compare~\cite{SchneiderLeibner2019b}).
Thus, for high-order models, the additional
quadrature error introduced by
masslumping might be acceptable given the massive speed-up.
In any case, it might be preferable to replace a lower-order moment model with high-order quadrature by
a higher-order masslumped model.\enlargethispage{1\baselineskip}

\begin{figure}
  \centering \externaltikz{TimingsMasslumped3d}{\input{Images/TimingsMasslumped3d}}
  \caption[Effect of masslumping on computational times and accuracy of the transformed finite volume scheme in three-dimensional tests]{Effect of masslumping on computational times and accuracy of the \(\HFMN\) models in the three-dimensional tests
    (\(\rktol = 0.01\), no parallelization).
    Ps: Point-source (\(\gridsize = 30^3\), \(\tf = 0.75\)), Cb: Checkerboard (\(\gridsize = 35^3\), \(\tf=3.2\)), S\@: Shadow (\(\gridsize = 36\times12\times9\), \(\tf=20\)), ml: masslumped.
    (a) Computational times versus moment number.
    (b) Errors introduced by masslumping (reference is the non-masslumped solution).}\label{fig:TimingsMasslumped3d}
\end{figure}

%% file: Images/PlanesourceConvergence.tex
\begin{tikzpicture}
  \def\pfad{Images/Results/}
  \begin{groupplot}[group style={group size=2 by 1, horizontal sep = 2cm,  vertical sep = 2cm},
      width=\figurewidth,
      height=\figureheight,
      scale only axis,
      cycle list name=color parula pairwise,
      xmode = log,
      ymode = log,
      grid = major,
      title style = {yshift = -0.2cm},
      x dir = reverse,
      ymin = 4e-9,
      ymax = 1e-2,
    ]
    \nextgroupplot[
      title = \tikztitle{New scheme},
      xlabel = {\(\rktol\)},
      ylabel = {\(\LpError{1}\)},
      legend style={at={(1.2,1.2)},xshift=-0cm,yshift=0cm,anchor=center,nodes=right},
      legend columns = 6,
    ]
    \node [text width=5cm,anchor=south]  {{\phantomsubcaption\normalfont\label{fig:Convergence1d_a}}};
    \addplot+ [thick] plot table[x=tau,y=E1] {\pfad planesource_convergence_m10.txt};
    \addplot+ [thick] plot table[x=tau,y=E1] {\pfad planesource_convergence_m50.txt};
    \addplot+ [thick] plot table[x=tau,y=E1] {\pfad planesource_convergence_hfm10.txt};
    \addplot+ [thick] plot table[x=tau,y=E1] {\pfad planesource_convergence_hfm50.txt};
    \addplot+ [thick] plot table[x=tau,y=E1] {\pfad planesource_convergence_pmm10.txt};
    \addplot+ [thick] plot table[x=tau,y=E1] {\pfad planesource_convergence_pmm50.txt};
    \addlegendentry{\(\MN[10]\)};
    \addlegendentry{\(\MN[50]\)};
    \addlegendentry{\(\HFMN[10]\)};
    \addlegendentry{\(\HFMN[50]\)};
    \addlegendentry{\(\PMMN[10]\)};
    \addlegendentry{\(\PMMN[50]\)};

    \nextgroupplot[
      title = \tikztitle{Standard scheme},
      xlabel = {\(\timestep\)},
      ylabel = {\(\LpError{1}\)},
    ]
    \node [text width=5cm,anchor=south]  {{\phantomsubcaption\normalfont\label{fig:Convergence1d_b}}};
    \addplot+ [thick] plot table[x=dt,y=E1] {\pfad planesource_fv_convergence_m10.txt};
    \addplot+ [thick] plot table[x=dt,y=E1] {\pfad planesource_fv_convergence_m50.txt};
    \addplot+ [thick] plot table[x=dt,y=E1] {\pfad planesource_fv_convergence_hfm10.txt};
    \addplot+ [thick] plot table[x=dt,y=E1] {\pfad planesource_fv_convergence_hfm50.txt};
    \addplot+ [thick] plot table[x=dt,y=E1] {\pfad planesource_fv_convergence_pmm10.txt};
    \addplot+ [thick] plot table[x=dt,y=E1] {\pfad planesource_fv_convergence_pmm50.txt};
  \end{groupplot}
\end{tikzpicture}%

%% file: Images/Timesteps1d.tex
\begin{tikzpicture}
  \def\pfad{Images/Results/}
  \begin{groupplot}[group style={group size=2 by 1, horizontal sep = 2cm,  vertical sep = 2cm},
      width=\figurewidth,
      height=\figureheight,
      scale only axis,
      cycle list name=color parula pairwise,
      ymode = log,
      grid = major,
      title style = {yshift = -0.2cm},
    ]
    \nextgroupplot[
      title = \tikztitle{Planesource},
      xlabel = {\(\timevar\)},
      ylabel = {\(\timestep\)},
      legend style={at={(1.2,1.2)},xshift=-0cm,yshift=0cm,anchor=center,nodes=right},
      legend columns = 6,
      ymin = 1e-8,
    ]
    \addplot+ [thick,mark=none] plot table[x=t,y=dt] {\pfad planesource_timesteps_tol_0.001_hfm2.txt};
    \addplot+ [thick,mark=none] plot table[x=t,y=dt] {\pfad planesource_timesteps_tol_0.001_hfm100.txt};
    \addplot+ [thick,mark=none] plot table[x=t,y=dt] {\pfad planesource_timesteps_tol_0.001_pmm2.txt};
    \addplot+ [thick,mark=none] plot table[x=t,y=dt] {\pfad planesource_timesteps_tol_0.001_pmm100.txt};
    \addplot+ [thick,mark=none] plot table[x=t,y=dt] {\pfad planesource_timesteps_tol_0.001_m2.txt};
    \addplot+ [thick,mark=none] plot table[x=t,y=dt] {\pfad planesource_timesteps_tol_0.001_m100.txt};
    \node [text width=5cm,anchor=south]  {{\phantomsubcaption\normalfont\label{fig:Timesteps1d_a}}};
    \draw (-1,0.002) -- (2.,0.002);
    \addlegendentry{\(\HFMN[2]\)};
    \addlegendentry{\(\HFMN[100]\)};
    \addlegendentry{\(\PMMN[2]\)};
    \addlegendentry{\(\PMMN[100]\)};
    \addlegendentry{\(\MN[2]\)};
    \addlegendentry{\(\MN[100]\)};

    \nextgroupplot[
      title = \tikztitle{Sourcebeam},
      xlabel = {\(\timevar\)},
      ylabel = {\(\timestep\)},
      ymin = 1e-8,
    ]
    \addplot+ [thick, mark=none] plot table[x=t,y=dt] {\pfad sourcebeam_timesteps_tol_0.001_hfm2.txt};
    \addplot+ [thick, mark=none] plot table[x=t,y=dt] {\pfad sourcebeam_timesteps_tol_0.001_hfm100.txt};
    \addplot+ [thick, mark=none] plot table[x=t,y=dt] {\pfad sourcebeam_timesteps_tol_0.001_pmm2.txt};
    \addplot+ [thick, mark=none] plot table[x=t,y=dt] {\pfad sourcebeam_timesteps_tol_0.001_pmm100.txt};
    \addplot+ [thick, mark=none] plot table[x=t,y=dt] {\pfad sourcebeam_timesteps_tol_0.001_m2.txt};
    \addplot+ [thick, mark=none] plot table[x=t,y=dt] {\pfad sourcebeam_timesteps_tol_0.001_m100.txt};
    \node [text width=5cm,anchor=south]  {{\phantomsubcaption\normalfont\label{fig:Timesteps1d_b}}};
    \draw (-5,0.0025) -- (5,0.0025);
  \end{groupplot}
\end{tikzpicture}%

%% file: Images/TimestepsConvergence.tex
\begin{tikzpicture}
  \def\pfad{Images/Results/}
  \begin{groupplot}[group style={group size=2 by 2, horizontal sep = 2cm,  vertical sep = 2cm},
      width=\figurewidth,
      height=\figureheight,
      scale only axis,
      cycle list name=color parula,
      ymode = log,
      grid = major,
      title style = {yshift = -0.2cm},
    ]
    \nextgroupplot[
      title = \tikztitle{Planesource},
      xlabel = {\(\timevar\)},
      ylabel = {\(\timestep\)},
      legend style={at={(1.2,1.2)},xshift=-0cm,yshift=0cm,anchor=center,nodes=right},
      legend columns = 6,
      ymin = 1e-8,
    ]
    \addplot+ [thick,mark=none] plot table[x=t,y=dt] {\pfad planesource_timesteps_convergence_averaged_tol_0.1_m10.txt};
    \addplot+ [thick,mark=none] plot table[x=t,y=dt] {\pfad planesource_timesteps_convergence_averaged_tol_0.01_m10.txt};
    \addplot+ [thick,mark=none] plot table[x=t,y=dt] {\pfad planesource_timesteps_convergence_averaged_tol_0.001_m10.txt};
    \addplot+ [thick,mark=none] plot table[x=t,y=dt] {\pfad planesource_timesteps_convergence_averaged_tol_0.0001_m10.txt};
    \addplot+ [thick,mark=none] plot table[x=t,y=dt] {\pfad planesource_timesteps_convergence_averaged_tol_1e-05_m10.txt};
    \addplot+ [thick,mark=none] plot table[x=t,y=dt] {\pfad planesource_timesteps_convergence_averaged_tol_1e-06_m10.txt};
    \node [text width=5cm,anchor=south]  {{\phantomsubcaption\normalfont\label{fig:TimestepsConvergence_a}}};
    \draw (-1,0.002) -- (2.,0.002);
    \addlegendentry{\(\rktol=10^{-1}\)};
    \addlegendentry{\(\rktol=10^{-2}\)};
    \addlegendentry{\(\rktol=10^{-3}\)};
    \addlegendentry{\(\rktol=10^{-4}\)};
    \addlegendentry{\(\rktol=10^{-5}\)};
    \addlegendentry{\(\rktol=10^{-6}\)};

    \nextgroupplot[
      title = \tikztitle{Sourcebeam},
      xlabel = {\(\timevar\)},
      ylabel = {\(\timestep\)},
      ymin = 1e-8,
    ]
    \addplot+ [thick,mark=none] plot table[x=t,y=dt] {\pfad sourcebeam_timesteps_convergence_averaged_tol_0.1_m10.txt};
    \addplot+ [thick,mark=none] plot table[x=t,y=dt] {\pfad sourcebeam_timesteps_convergence_averaged_tol_0.01_m10.txt};
    \addplot+ [thick,mark=none] plot table[x=t,y=dt] {\pfad sourcebeam_timesteps_convergence_averaged_tol_0.001_m10.txt};
    \addplot+ [thick,mark=none] plot table[x=t,y=dt] {\pfad sourcebeam_timesteps_convergence_averaged_tol_0.0001_m10.txt};
    \addplot+ [thick,mark=none] plot table[x=t,y=dt] {\pfad sourcebeam_timesteps_convergence_averaged_tol_1e-05_m10.txt};
    \addplot+ [thick,mark=none] plot table[x=t,y=dt] {\pfad sourcebeam_timesteps_convergence_averaged_tol_1e-06_m10.txt};
    \node [text width=5cm,anchor=south]  {{\phantomsubcaption\normalfont\label{fig:TimestepsConvergence_b}}};
    \draw (-5,0.0025) -- (5,0.0025);
  \end{groupplot}
\end{tikzpicture}%

%% file: Images/Timesteptimes1d.tex
\begin{tikzpicture}
  \def\pfad{Images/Results/}
  \begin{groupplot}[group style={group size=2 by 1, horizontal sep = 2cm,  vertical sep = 2cm},
      width=\figurewidth,
      height=\figureheight,
      scale only axis,
      cycle list name=color parula pairwise,
      ymode = log,
      grid = major,
      title style = {yshift = -0.2cm},
    ]
    \nextgroupplot[
      title = \tikztitle{Planesource},
      xlabel = {\(\timevar\)},
      ylabel = {\(\timestep\)},
      legend style={at={(1.2,1.2)},xshift=-0cm,yshift=0cm,anchor=center,nodes=right},
      legend columns = 6,
    ]
    \addplot+ [thick] plot table[x=t,y=steptime,mark=none] {\pfad planesource_timesteptimes_tol_0.001_hfm50.txt};
    \addplot+ [thick] plot table[x=t,y=steptime,mark=none] {\pfad planesource_timesteptimes_ref_hfm50.txt};
    \addplot+ [thick] plot table[x=t,y=steptime,mark=none] {\pfad planesource_timesteptimes_tol_0.001_pmm50.txt};
    \addplot+ [thick] plot table[x=t,y=steptime,mark=none] {\pfad planesource_timesteptimes_ref_pmm50.txt};
    \addplot+ [thick] plot table[x=t,y=steptime,mark=none] {\pfad planesource_timesteptimes_tol_0.001_m50.txt};
    \addplot+ [thick] plot table[x=t,y=steptime,mark=none] {\pfad planesource_timesteptimes_ref_m50.txt};
    \addlegendentry{\(\HFMN[50]\) (\(\multipliers\))};
    \addlegendentry{\(\HFMN[50]\) (\(\moments\))};
    \addlegendentry{\(\PMMN[50]\) (\(\multipliers\))};
    \addlegendentry{\(\PMMN[50]\) (\(\moments\))};
    \addlegendentry{\(\MN[50]\) (\(\multipliers\))};
    \addlegendentry{\(\MN[50]\) (\(\moments\))};
    \node [text width=5cm,anchor=south]  {{\phantomsubcaption\normalfont\label{fig:Timesteptimes1d_a}}};

    \nextgroupplot[
      title = \tikztitle{Sourcebeam},
      xlabel = {\(\timevar\)},
      ylabel = {\(\timestep\)},
    ]
    \addplot+ [thick] plot table[x=t,y=steptime,mark=none] {\pfad sourcebeam_timesteptimes_tol_0.001_hfm50.txt};
    \addplot+ [thick] plot table[x=t,y=steptime,mark=none] {\pfad sourcebeam_timesteptimes_ref_hfm50.txt};
    \addplot+ [thick] plot table[x=t,y=steptime,mark=none] {\pfad sourcebeam_timesteptimes_tol_0.001_pmm50.txt};
    \addplot+ [thick] plot table[x=t,y=steptime,mark=none] {\pfad sourcebeam_timesteptimes_ref_pmm50.txt};
    \addplot+ [thick] plot table[x=t,y=steptime,mark=none] {\pfad sourcebeam_timesteptimes_tol_0.001_m50.txt};
    \addplot+ [thick] plot table[x=t,y=steptime,mark=none] {\pfad sourcebeam_timesteptimes_ref_m50.txt};
    \node [text width=5cm,anchor=south]  {{\phantomsubcaption\normalfont\label{fig:Timesteptimes1d_b}}};

  \end{groupplot}
\end{tikzpicture}%

%% file: Images/SourcebeamRegularization.tex
\begin{tikzpicture}
  \def\pfad{Images/Results/}
  \begin{groupplot}[group style={group size=2 by 1, horizontal sep = 2cm,  vertical sep = 2cm},
      width=\figurewidth,
      height=\figureheight,
      scale only axis,
      cycle list name=color parula,
      ymode = log,
      grid = major,
      title style = {yshift = -0.2cm},
      legend pos = south east,
    ]

    \pgfplotsset{select coords between index/.style 2 args={
        x filter/.code={
          \ifnum\coordindex<#1\def\pgfmathresult{}\fi
          \ifnum\coordindex>#2\def\pgfmathresult{}\fi
        }
    }}

    \nextgroupplot[
    title = \tikztitle{Timesteps},
      xlabel = {\(\timevar\)},
      ylabel = {\(\timestep\)},
      ymin = 1e-8,
    ]
    \addplot+ [thick,mark=none] plot table[x=t,y=dt] {\pfad sourcebeam_timesteps_convergence_regularized_averaged_tol_0.1_m10.txt};
    \addplot+ [thick,mark=none] plot table[x=t,y=dt] {\pfad sourcebeam_timesteps_convergence_regularized_averaged_tol_0.01_m10.txt};
    \addplot+ [thick,mark=none] plot table[x=t,y=dt] {\pfad sourcebeam_timesteps_convergence_regularized_averaged_tol_0.001_m10.txt};
    \addplot+ [thick,mark=none] plot table[x=t,y=dt] {\pfad sourcebeam_timesteps_convergence_regularized_averaged_tol_0.0001_m10.txt};
    \addplot+ [thick,mark=none] plot table[x=t,y=dt] {\pfad sourcebeam_timesteps_convergence_regularized_averaged_tol_1e-05_m10.txt};
    \addplot+ [thick,mark=none] plot table[x=t,y=dt] {\pfad sourcebeam_timesteps_convergence_regularized_averaged_tol_1e-06_m10.txt};
    \node [text width=5cm,anchor=south]  {{\phantomsubcaption\normalfont\label{fig:SourcebeamRegularization_a}}};

    \addlegendentry{\(\rktol=10^{-1}\)};
    \addlegendentry{\(\rktol=10^{-2}\)};
    \addlegendentry{\(\rktol=10^{-3}\)};
    \addlegendentry{\(\rktol=10^{-4}\)};
    \addlegendentry{\(\rktol=10^{-5}\)};
    \addlegendentry{\(\rktol=10^{-6}\)};

    \nextgroupplot[
      title = \tikztitle{Regularization error},
      xlabel = {\(\numtimesteps\)},
      ylabel = {\(\LpError{1}\)},
      ymode = log,
      grid = both,
      minor x tick num = 1,
      legend pos = north east,
      xtick distance = 400,
    ]
  \addplot+ [thick] plot table[x=nt,y=E1_1e-9,select coords between index={0}{0}] {\pfad sourcebeam_reg_error_m10.txt};
  \addplot+ [thick] plot table[x=nt,y=E1_1e-9,select coords between index={1}{1}] {\pfad sourcebeam_reg_error_m10.txt};
  \addplot+ [thick] plot table[x=nt,y=E1_1e-9,select coords between index={3}{3}] {\pfad sourcebeam_reg_error_m10.txt};
  \addplot+ [thick] plot table[x=nt,y=E1_1e-9,select coords between index={5}{5}] {\pfad sourcebeam_reg_error_m10.txt};
  \addplot+ [thick] plot table[x=nt,y=E1_1e-9,select coords between index={7}{7}] {\pfad sourcebeam_reg_error_m10.txt};
  \addplot+ [thick] plot table[x=nt,y=E1_1e-9,select coords between index={9}{9}] {\pfad sourcebeam_reg_error_m10.txt};
  \addlegendentry{\(\regepsilon = 10^{-5}\)};
  \addlegendentry{\(\regepsilon = 10^{-6}\)};
  \addlegendentry{\(\regepsilon = 10^{-7}\)};
  \addlegendentry{\(\regepsilon = 10^{-8}\)};
  \addlegendentry{\(\regepsilon = 10^{-9}\)};
  \addlegendentry{\(\regepsilon = 10^{-10}\)};
  \node [text width=5cm,anchor=south]  {{\phantomsubcaption\normalfont\label{fig:SourcebeamRegularization_b}}};
  \end{groupplot}
\end{tikzpicture}%

%% file: Images/TimestepsHFM6Shadow.tex
\begin{tikzpicture}
  \def\pfad{Images/Results/}
  \begin{groupplot}[group style={group size=2 by 1, horizontal sep = 2cm,  vertical sep = 2cm},
      width=\figurewidth,
      height=\figureheight,
      scale only axis,
      cycle list name=color parula pairwise,
      ymode = log,
      grid = major,
      title style = {yshift = -0.2cm},
      legend style={at={(1.2,1.2)},xshift=-0cm,yshift=0cm,anchor=center,nodes=right},
      legend columns = 6,
      ymin = 1e-13,
      ymax = 3e-1,
    ]
    \nextgroupplot[
      title = \tikztitle{\(\HFMN[6]\)},
      xlabel = {\(\timevar\)},
      ylabel = {\(\timestep\)},
    ]
    \addplot+ [thick] plot table[x=t,y=dt,mark=none] {\pfad shadow_timesteps_tol_0.01_adjust_0_hfm6.txt};
    \draw (-5,0.11547005383) -- (25,0.11547005383);
    \draw[thick,densely dotted] (-5,0.01704731792) -- (25,0.01704731792);
    \node [text width=5cm,anchor=south]  {{\phantomsubcaption\normalfont\label{fig:TimestepsHFM6Shadow_a}}};

    \nextgroupplot[
      title = \tikztitle{\(\HFMN[6]\) with regularization},
      xlabel = {\(\timevar\)},
      ylabel = {\(\timestep\)},
    ]
    \addplot+ [thick] plot table[x=t,y=dt,mark=none] {\pfad shadow_timesteps_tol_0.01_adjust_0.01_hfm6.txt};
    \draw (-5,0.11547005383) -- (25,0.11547005383);
    \draw[thick,densely dotted] (-5,0.01704731792) -- (25,0.01704731792);
    \node [text width=5cm,anchor=south]  {{\phantomsubcaption\normalfont\label{fig:TimestepsHFM6Shadow_b}}};
  \end{groupplot}
\end{tikzpicture}%

%% file: Images/Timings1d.tex
\begin{tikzpicture}
  \def\pfad{Images/Results/}
  \begin{groupplot}[group style={group size=2 by 1, horizontal sep = 2cm,  vertical sep = 2cm},
      width=\figurewidth,
      height=\figureheight,
      scale only axis,
      cycle list name=color parula pairwise,
      ymode = log,
      grid = major,
      title style = {yshift = -0.2cm},
    ]
    \nextgroupplot[
      title = \tikztitle{Planesource},
      xlabel = {\(\momentnumber\)},
      ylabel = {Walltime (s)},
      legend style={at={(1.2,1.2)},xshift=-0cm,yshift=0cm,anchor=center,nodes=right},
      legend columns = 6,
    ]
    \addplot+ [thick] plot table[x=n,y=new] {\pfad planesource_timings_hfm.txt};
    \addplot+ [thick] plot table[x=n,y=ref] {\pfad planesource_timings_hfm.txt};
    \addplot+ [thick] plot table[x=n,y=new] {\pfad planesource_timings_pmm.txt};
    \addplot+ [thick] plot table[x=n,y=ref] {\pfad planesource_timings_pmm.txt};
    \addplot+ [thick] plot table[x=n,y=new] {\pfad planesource_timings_m.txt};
    \addplot+ [thick] plot table[x=n,y=ref] {\pfad planesource_timings_m.txt};
    \addlegendentry{\(\HFMN\) \((\multipliers)\)};
    \addlegendentry{\(\HFMN\) \((\moments)\)};
    \addlegendentry{\(\PMMN\) \((\multipliers)\)};
    \addlegendentry{\(\PMMN\) \((\moments)\)};
    \addlegendentry{\(\MN\) \((\multipliers)\)};
    \addlegendentry{\(\MN\) \((\moments)\)};
    \node [text width=5cm,anchor=south]  {{\phantomsubcaption\normalfont\label{fig:Timings1d_a}}};

    \nextgroupplot[
      title = \tikztitle{Sourcebeam},
      xlabel = {\(\momentnumber\)},
      ylabel = {Walltime (s)},
      legend style={at={(1.2,1.2)},xshift=-0cm,yshift=0cm,anchor=center,nodes=right},
    ]
    \addplot+ [thick] plot table[x=n,y=new] {\pfad sourcebeam_timings_tol_0.001_hfm.txt};
    \addplot+ [thick] plot table[x=n,y=ref] {\pfad sourcebeam_timings_tol_0.001_hfm.txt};
    \addplot+ [thick] plot table[x=n,y=new] {\pfad sourcebeam_timings_tol_0.001_pmm.txt};
    \addplot+ [thick] plot table[x=n,y=ref] {\pfad sourcebeam_timings_tol_0.001_pmm.txt};
    \addplot+ [thick] plot table[x=n,y=new] {\pfad sourcebeam_timings_tol_0.001_m.txt};
    \addplot+ [thick] plot table[x=n,y=ref] {\pfad sourcebeam_timings_tol_0.001_m.txt};
    \draw (0,0.01) -- (2.5,0.01);
    \node [text width=5cm,anchor=south]  {{\phantomsubcaption\normalfont\label{fig:Timings1d_b}}};
  \end{groupplot}
\end{tikzpicture}%

%% file: Images/Timings3d.tex
\begin{tikzpicture}
  \def\pfad{Images/Results/}
  \begin{groupplot}[group style={group size=2 by 2, horizontal sep = 2cm,  vertical sep = 2cm},
      width=1.11\figurewidth,
      height=1.11\figureheight,
      scale only axis,
      cycle list name=color parula pairwise,
      xmode = log,
      ymode = log,
      grid = major,
      title style = {yshift = -0.2cm},
    ]
    \nextgroupplot[
      title = \tikztitle{Point-source},
      xlabel = {\(\momentnumber\)},
      ylabel = {Walltime (seconds)},
      legend style={at={(1.2,1.2)},xshift=-0cm,yshift=0.1cm,anchor=center,nodes=right},
      legend columns = 6,
    ]
    \node [text width=5cm,anchor=south]  {\subcaption{\label{fig:Timings3d_a}}};
    \addplot+ [thick] plot table[x=n,y=new] {\pfad pointsource_timings_tol_0.01_hfm.txt};
    \addplot+ [thick] plot table[x=n,y=ref] {\pfad pointsource_timings_tol_0.01_hfm.txt};
    \addplot+ [thick] plot table[x=n,y=new] {\pfad pointsource_timings_tol_0.01_pmm.txt};
    \addplot+ [thick] plot table[x=n,y=ref] {\pfad pointsource_timings_tol_0.01_pmm.txt};
    \addplot+ [thick] plot table[x=n,y=new] {\pfad pointsource_timings_tol_0.01_m.txt};
    \addplot+ [thick] plot table[x=n,y=ref] {\pfad pointsource_timings_tol_0.01_m.txt};
    \addlegendentry{\(\HFMN\) (\(\multipliers\))};
    \addlegendentry{\(\HFMN\) (\(\moments\))};
    \addlegendentry{\(\PMMN\) (\(\multipliers\))};
    \addlegendentry{\(\PMMN\) (\(\moments\))};
    \addlegendentry{\(\MN\) (\(\multipliers\))};
    \addlegendentry{\(\MN\) (\(\moments\))};

    \nextgroupplot[
      title = \tikztitle{Checkerboard},
      xlabel = {\(\momentnumber\)},
      ylabel = {Walltime (seconds)},
      legend style={at={(1.2,1.2)},xshift=-0cm,yshift=0cm,anchor=center,nodes=right},
    ]
    \node [text width=5cm,anchor=south]  {\subcaption{\label{fig:Timings3d_b}}};
    \addplot+ [thick] plot table[x=n,y=new] {\pfad checkerboard_timings_tol_0.01_hfm.txt};
    \addplot+ [thick] plot table[x=n,y=ref] {\pfad checkerboard_timings_tol_0.01_hfm.txt};
    \addplot+ [thick] plot table[x=n,y=new] {\pfad checkerboard_timings_tol_0.01_pmm.txt};
    \addplot+ [thick] plot table[x=n,y=ref] {\pfad checkerboard_timings_tol_0.01_pmm.txt};
    \addplot+ [thick] plot table[x=n,y=new] {\pfad checkerboard_timings_tol_0.01_m.txt};
    \addplot+ [thick] plot table[x=n,y=ref] {\pfad checkerboard_timings_tol_0.01_m.txt};

    \nextgroupplot[
      title = \tikztitle{Shadow},
      xlabel = {\(\momentnumber\)},
      ylabel = {Walltime (seconds)},
      legend style={at={(1.2,1.2)},xshift=-0cm,yshift=0cm,anchor=center,nodes=right},
    ]
    \node [text width=5cm,anchor=south]  {\subcaption{\label{fig:Timings3d_c}}};
    \addplot+ [thick] plot table[x=n,y=new] {\pfad shadow_timings_tol_0.01_hfm.txt};
    \addplot+ [thick] plot table[x=n,y=ref] {\pfad shadow_timings_tol_0.01_hfm.txt};
    \addplot+ [thick] plot table[x=n,y=new] {\pfad shadow_timings_tol_0.01_pmm.txt};
    \addplot+ [thick] plot table[x=n,y=ref] {\pfad shadow_timings_tol_0.01_pmm.txt};
    \addplot+ [thick] plot table[x=n,y=new] {\pfad shadow_timings_tol_0.01_m.txt};
    \addplot+ [thick] plot table[x=n,y=ref] {\pfad shadow_timings_tol_0.01_m.txt};

    \nextgroupplot[
      scale only axis,
      cycle list name=color parula pairwise,
      xmode=log,
      ymode=log,
      grid = major,
      title = \tikztitle{Parallel scaling},
      xlabel = {Threads},
      ylabel = {Walltime (seconds)},
    ]
    \node [text width=5cm,anchor=south]  {\subcaption{\label{fig:ScalingPointsource}}};
    \addplot+ [thick] plot table[x=threads,y=new] {\pfad pointsource_scaling_means_factor_1_m3.txt};
    \addplot+ [thick] plot table[x=threads,y=new] {\pfad pointsource_scaling_means_factor_1000_m3.txt};
    \addplot+ [thick] plot table[x=threads,y=ref] {\pfad pointsource_scaling_means_factor_1_m3.txt};
    \addplot+ [thick] plot table[x=threads,y=ref] {\pfad pointsource_scaling_means_factor_1000_m3.txt};
    \draw[thick, densely dotted] (1,5142.797609640666) -- (32,160.712425301);
    \draw[thick, densely dotted] (1,3288.837291278) -- (32,102.776165352);
    \addlegendentry{New \((\multipliers)\), \phantom{000}1 t/t};
    \addlegendentry{New \((\multipliers)\), 1000 t/t};
    \addlegendentry{Std. \((\moments)\), \phantom{000}1 t/t};
    \addlegendentry{Std. \((\moments)\), 1000 t/t};

  \end{groupplot}
\end{tikzpicture}%

%% file: Images/TimingsMasslumped1d.tex
\begin{tikzpicture}
  \def\pfad{Images/Results/}
  \begin{groupplot}[group style={group size=2 by 1, horizontal sep = 2cm,  vertical sep = 2cm},
      width=\figurewidth,
      height=\figureheight,
      scale only axis,
      cycle list name=color parula pairwise,
      xmode = log,
      ymode = log,
      grid = major,
      title style = {yshift = -0.2cm},
    ]
    \nextgroupplot[
      title = \tikztitle{Timings},
      xlabel = {\(\momentnumber\)},
      ylabel = {Walltime (s)},
      legend pos = south east,
    ]
    \addplot+ [thick] plot table[x=n,y=new] {\pfad planesource_masslumped_timings_tol_0.001.txt};
    \addplot+ [thick] plot table[x=n,y=ref] {\pfad planesource_masslumped_timings_tol_0.001.txt};
    \addplot+ [thick] plot table[x=n,y=new] {\pfad sourcebeam_masslumped_timings_tol_0.001.txt};
    \addplot+ [thick] plot table[x=n,y=ref] {\pfad sourcebeam_masslumped_timings_tol_0.001.txt};
    \addlegendentry{Ps.\ ml.};
    \addlegendentry{Ps.};
    \addlegendentry{Sb.\ ml.};
    \addlegendentry{Sb.};
    \node [text width=5cm,anchor=south]  {{\phantomsubcaption\normalfont\label{fig:TimingsMasslumped1d_a}}};

    \nextgroupplot[
      title = \tikztitle{Errors},
      xlabel = {\(\momentnumber\)},
      ylabel = {Error},
      legend pos = south west,
    ]
    \addplot+ [thick] plot table[x=n,y=E1] {\pfad planesource_masslumped_error.txt};
    \addplot+ [thick] plot table[x=n,y=Einf] {\pfad planesource_masslumped_error.txt};
    \addplot+ [thick] plot table[x=n,y=E1] {\pfad sourcebeam_masslumped_error.txt};
    \addplot+ [thick] plot table[x=n,y=Einf] {\pfad sourcebeam_masslumped_error.txt};
    \addlegendentry{Ps.\ \(\LpError{1}\)};
    \addlegendentry{Ps.\ \(\LpError{\infty}\)};
    \addlegendentry{Sb.\ \(\LpError{1}\)};
    \addlegendentry{Sb.\ \(\LpError{\infty}\)};
    \draw (10, 4e-3) -- (500, 0.0000016)
    node[draw=none,fill=none,font=\scriptsize,pos=0.6,xshift=0.1cm,below,rotate=-33] {2nd order};
    \draw (20, 8e-2) -- (500, 0.000128)
    node[draw=none,fill=none,font=\scriptsize,pos=0.4,xshift=0.1cm,below,rotate=-33] {2nd order};
    \node [text width=5cm,anchor=south]  {{\phantomsubcaption\normalfont\label{fig:TimingsMasslumped1d_b}}};
  \end{groupplot}
\end{tikzpicture}%

%% file: Images/TimingsMasslumped3d.tex
\begin{tikzpicture}
  \def\pfad{Images/Results/}
  \begin{groupplot}[group style={group size=2 by 2, horizontal sep = 2cm,  vertical sep = 2cm},
      width=\figurewidth,
      height=\figureheight,
      scale only axis,
      cycle list name=color parula pairwise,
      xmode = log,
      ymode = log,
      grid = major,
      title style = {yshift = -0.2cm},
    ]
    \nextgroupplot[
      title = \tikztitle{Timings},
      xlabel = {\(\momentnumber\)},
      ylabel = {Walltime (s)},
      legend pos = north west,
      legend columns = 2,
    ]
    \addplot+ [thick] plot table[x=n,y=new] {\pfad pointsource_masslumped_timings_tol_0.01.txt};
    \addplot+ [thick] plot table[x=n,y=ref] {\pfad pointsource_masslumped_timings_tol_0.01.txt};
    \addplot+ [thick] plot table[x=n,y=new] {\pfad checkerboard_masslumped_timings_tol_0.01.txt};
    \addplot+ [thick] plot table[x=n,y=ref] {\pfad checkerboard_masslumped_timings_tol_0.01.txt};
    \addplot+ [thick] plot table[x=n,y=new] {\pfad shadow_masslumped_timings_tol_0.01.txt};
    \addplot+ [thick] plot table[x=n,y=ref] {\pfad shadow_masslumped_timings_tol_0.01.txt};
    \addlegendentry{Ps.\ ml.};
    \addlegendentry{Ps.};
    \addlegendentry{Cb.\ ml};
    \addlegendentry{Cb.};
    \addlegendentry{S.\ ml};
    \addlegendentry{S.};
    \node [text width=5cm,anchor=south]  {{\phantomsubcaption\normalfont\label{fig:TimingsMasslumped3d_a}}};

    \nextgroupplot[
      title = \tikztitle{Errors},
      xlabel = {\(\momentnumber\)},
      ylabel = {Error},
      legend columns = 2,
    ]
    \addplot+ [thick] plot table[x=n,y=E1] {\pfad pointsource_masslumped_error.txt};
    \addplot+ [thick] plot table[x=n,y=Einf] {\pfad pointsource_masslumped_error.txt};
    \addplot+ [thick] plot table[x=n,y=E1] {\pfad checkerboard_masslumped_error.txt};
    \addplot+ [thick] plot table[x=n,y=Einf] {\pfad checkerboard_masslumped_error.txt};
    \addplot+ [thick] plot table[x=n,y=E1] {\pfad shadow_masslumped_error.txt};
    \addplot+ [thick] plot table[x=n,y=Einf] {\pfad shadow_masslumped_error.txt};
    \addlegendentry{Ps.\ \(\LpError{1}\)};
    \addlegendentry{Ps.\ \(\LpError{\infty}\)};
    \addlegendentry{Cb.\ \(\LpError{1}\)};
    \addlegendentry{Cb.\ \(\LpError{\infty}\)};
    \addlegendentry{S.\ \(\LpError{1}\)};
    \addlegendentry{S.\ \(\LpError{\infty}\)};
    \draw (50, 0.15) -- (1026, 0.00730994152)
    node[draw=none,fill=none,font=\scriptsize,midway,below,rotate=-35] {1st order};
    \node [text width=5cm,anchor=south]  {{\phantomsubcaption\normalfont\label{fig:TimingsMasslumped3d_b}}};
  \end{groupplot}
\end{tikzpicture}%

%% file: Sections/outlook.tex
\section{Conclusion and outlook}\label{sec:outlook}
In this paper, we introduced a new numerical scheme for entropy-based moment equations
that is based on a variable transformation of the semi-discretized equations and
gets rid of the minimum-entropy optimization problems (except for the initial values). We have shown analytically and numerically that the new
scheme converges to the correct solutions, and that it follow a discrete entropy law if a relaxed Runge-Kutta method is used for time stepping.
In addition, we investigated the performance of the new scheme in several numerical benchmarks and
showed that it is often several times faster than the untransformed scheme, at the same or
even higher accuracy in time. In addition, for the hat function basis, we showed that a massive speed-up can be obtained by using a
quadrature that contains only the vertices of the triangulation (at the cost of additional quadrature error), making very high-order models
computable in reasonable time.
Finally, we did some tests on parallel scaling of the schemes which suggest that the new scheme does not
have the same load-balancing problems as the untransformed scheme.

To improve the scheme, better error estimates for the adaptive timesteppers should be investigated to get rid of the erratic time step behaviour observed in the
source-beam and shadow tests. Here, larger errors could be allowed for multipliers that correspond to small densities and thus only have a minor effect
on the solution in original variables. In addition, regularization techniques could be used to replace such multipliers if they limit the time step.
These regularization techniques might also be needed to be able to solve problems where some Hessians are numerically singular. For applications involving strong scattering or absorption,
splitting methods for the new scheme might be of interest to remove the time step restriction induced by the corresponding terms.

While we restricted ourselves to a first-order scheme, the same variable transformation can also be applied to higher-order kinetic schemes as regarded, e.g.,
in~\cite{Schneider2015b, Schneider2016}.

In future work, we will investigate further model reduction by POD-based reduced basis methods~\cite{Quarteroni2016,Himpe2018} which should be much easier with
the new scheme as it is well-defined on the whole \(\R^{\momentnumber}\) and not only on the realizable set (which is a convex cone in \(\R^{\momentnumber}\)).

%% file: Sections/appendix.tex
\section{\texorpdfstring{Proof of \cref{thm:lipschitzcontinuity}}{Proof of Lipschitz continuity of the update functions}}\label{sec:appendixlipschitzproof}

\begin{proof}
  We want to show that the derivatives of \(\alphaupdateterm_{\gridindex}\) are bounded,
  i.e.\ \(\generalnorm{\frac{\partial\alphaupdateterm_{\gridindex}}{\partial\multipliers[\gridindexalt]}} \leq C\).
  Here, since all matrix norms are equivalent, \(\generalnorm{\cdot}\) can be an arbitrary matrix norm.
  We will show that the rows of \(\frac{\partial\alphaupdateterm_{\gridindex}}{\partial\multipliers[\gridindexalt]}\)
  are bounded in Euclidean norm which bounds the Frobenius matrix norm.
  By the
  definition of \(\alphaupdateterm_{\gridindex}\) (see \eqref{eq:alphaupdateterm}),
  we have
  \begin{align}
    \frac{\partial\alphaupdateterm_{\gridindex}}{\partial\multipliers[\gridindexalt]}
     & = \frac{\partial}{\partial\multipliers[\gridindexalt]} \left(
    {\optHessian(\multipliers[\gridindex])}^{-1}
    \uupdateterm_{\gridindex}\left(\moments(\multipliers[0]), \ldots, \moments(\multipliers[\gridsize-1])\right)\right)
    = \frac{\partial}{\partial\multipliers[\gridindexalt]} \begin{pmatrix}
      \sum_{\basisindex} \optHessianInvEntry_{0,\basisindex} \uupdatetermcomp_{\gridindex,\basisindex} \nonumber   \\
      \vdots                                                                                                       \\
      \sum_\basisindex \optHessianInvEntry_{\momentnumber-1,\basisindex} \uupdatetermcomp_{\gridindex,\basisindex} \\
    \end{pmatrix} \\
     & = {\left(
    \sum_\basisindex \frac{\partial\optHessianInvEntry_{k,\basisindex}}{\partial\multiplierscomp{\gridindexalt\basisindexvar}}
    \uupdatetermcomp_{\gridindex,\basisindex}
    + \optHessianInvEntry_{k,\basisindex}
    \frac{\partial\uupdatetermcomp_{\gridindex,\basisindex}}{\partial\multiplierscomp{\gridindexalt,\basisindexvar}}
    \right)}_{k,\basisindexvar=0,\ldots,\momentnumber-1} \nonumber                   \\
     & = \left(
    \frac{\partial\optHessian^{-1}}{\partial\multiplierscomp{\gridindexalt,0}}
    \uupdateterm_{\gridindex}, \ldots,
    \frac{\partial\optHessian^{-1}}{\partial\multiplierscomp{\gridindexalt,\momentnumber-1}}
    \uupdateterm_{\gridindex}\right)
    + \optHessian^{-1}\left(
    \frac{\partial\uupdateterm_{\gridindex}}{\partial\multiplierscomp{\gridindexalt,1}},
    \ldots,
    \frac{\partial\uupdateterm_{\gridindex}}{\partial\multiplierscomp{\gridindexalt,\momentnumber-1}} \right) \label{eq:lipschitzproof1}
  \end{align}
  where \(\optHessianInvEntry_{k,\basisindex}\) is the \((k,\basisindex)\)-th entry of \(\optHessian^{-1}\) and
  \(\uupdatetermcomp_{\gridindex\basisindex}\) and \(\multiplierscomp{\gridindexalt\basisindex}\)  are the \(\basisindex\)-th component of
  \(\uupdateterm_{\gridindex}\) and \(\multipliers[\gridindexalt]\), respectively.

  For \(\gridindex \neq \gridindexalt\), the first term in \eqref{eq:lipschitzproof1} vanishes.
  For \(\gridindex = \gridindexalt\), since (see, e.g., \cite{Giles2008})
  \begin{equation}
    \frac{\partial\optHessian^{-1}}{\partial\alpha} = -\optHessian^{-1}\frac{\partial\optHessian}{\partial\alpha}\optHessian^{-1},
  \end{equation}
  the first term in \eqref{eq:lipschitzproof1} depends on \(\optHessian^{-1}(\multipliers[\gridindex])\),
  \(\uupdateterm_{\gridindex}\) and derivatives of \(\optHessian(\multipliers[\gridindex])\).
  The terms in \(\uupdateterm_{\gridindex}\), except for the constant source term \(\ints{\basis\source}\), all
  scale with either \(\density(\moments(\multipliersfv[\gridindex]))\)
  or \(\density(\moments(\multipliersfv[\gridindexalt]))\), \(\gridindexalt \in \gridneighbors{\gridindex}\)
  (compare \eqref{eq:uupdateterm} and \eqref{eq:sourcedef}).
  By \eqref{eq:densitybounded}, \(\uupdateterm_{\gridindex}\) is thus bounded.
  By \eqref{eq:Hessianbounded}, after multiplication with \(\optHessian^{-1}\)
  the term is still bounded (as \(\generalnorm{\optHessian^{-1}} = \generalnorm{\optHessian}^{-1}\)).
  For the first term of \eqref{eq:lipschitzproof1}, it remains to show that
  \(\generalnorm{\frac{\partial\optHessian(\multipliers[\gridindex])}{\partial\multiplierscomp{\gridindex\basisindex}}}\)
  is bounded for \(\basisindex \in \{0, \ldots, \momentnumber-1\}\). We have
  \begin{align*}
    \Vw^T \frac{\partial\optHessian(\multipliers[\gridindex])}{\partial\multiplierscomp{\gridindex\basisindex}} \Vw
     & = \Vw^T \frac{\partial}{\partial\multiplierscomp{\gridindex\basisindex}} \ints{\basis\basis^T\ld{\entropy}''(\multipliers[\gridindex] \cdot \basis)} \Vw
    = \ints{{(\basis \cdot \Vw)}^2 \basiscomp[\basisindex] \ld{\entropy}'''(\multipliers[\gridindex] \cdot \basis)}                                             \\
     & \leq \left(\max_{\SC \in \angularDomain} \abs{\basiscomp[\basisindex](\SC)} \generalnorm{\basis(\SC)}^2\right)
    \generalnorm{\Vw}^2 \ints{\ld{\entropy}'''(\multipliers[\gridindex] \cdot \basis)}.
  \end{align*}
  Thus, by \eqref{eq:entropythirdderivbounded},
  \(\generalnorm{\frac{\partial\optHessian(\multipliers[\gridindex])}{\partial\multiplierscomp{\gridindex\basisindex}}}\) is bounded.

  To show that the second summand in \eqref{eq:lipschitzproof1} is bounded, by \eqref{eq:Hessianbounded}, we only have to show boundedness of
  \(\frac{\partial\uupdateterm_{\gridindex}}{\partial\multiplierscomp{\gridindexalt\basisindex}}\)
  for
  \(\basisindex \in \{0, \ldots, \momentnumber-1\}\). We have (compare \eqref{eq:uupdateterm})
  \begin{equation}\label{eq:appendixproof4}
    \frac{\partial\uupdateterm_{\gridindex}}{\partial\multiplierscomp{\gridindexalt\basisindex}} =
    \frac{\partial}{\partial\multiplierscomp{\gridindexalt\basisindex}}
    \left(
    \sum_{k\in\gridneighbors{\gridindex}}
    \frac{\abs{\fvgridface_{\gridindex k}}}{\abs{\fvgridentity_{\gridindex}}}
    \left(
      \intp{(\SC \cdot \outernormal_{\gridindex k}) \ld{\entropy}'(\multipliers[\gridindex]^{\timeindex} \cdot \basis) \basis }
      + \intm{(\SC \cdot \outernormal_{\gridindex k}) \ld{\entropy}'(\multipliers[k]^{\timeindex} \cdot \basis) \basis }
      \right)
    - \Source(\spatialvar_{\gridindex}, \moments[\gridindex]^{\timeindex})
    \right).
  \end{equation}
  Since \(\Source\) is linear in \(\moments\), and
  \(
  \frac{\derivative\moments[\gridindex]}{\derivative\multipliers[\gridindexalt]} = \optHessian(\moments[\gridindex]) \dirac_{\gridindex\gridindexalt}
  \) (compare \eqref{eq:du_dalpha}), the term in \eqref{eq:appendixproof4} involving \(\Source\) results in columns of
  the Hessian (multiplied by some constant matrices, see \eqref{eq:sourcedef}) and thus is bounded by \eqref{eq:Hessianbounded}.
  Computing the derivative of the flux terms in \eqref{eq:appendixproof4}, we obtain terms like
  \begin{equation*}
    \frac{\partial}{\partial\multiplierscomp{\gridindexalt\basisindex}}
    \intp{(\SC \cdot \outernormal_{\gridindex k}) \ld{\entropy}'(\multipliers[\gridindex]^{\timeindex} \cdot \basis) \basis }
    =
    \intp{(\SC \cdot \outernormal_{\gridindex k}) \ld{\entropy}''(\multipliers[\gridindex]^{\timeindex}
      \cdot \basis) \basiscomp[\basisindex] \dirac_{\gridindex\basisindex} \basis }
  \end{equation*}
  which are bounded by \eqref{eq:entropysecondderivbounded}.
\end{proof}

%% file: Sections/supplement.tex
\section{Supplementary data}\label{sec:supplementarydata}
\beginsupplement
This supplement contains additional figures and tabular data from the numerical experiments.
\vspace{0.5cm}

\begin{figure}[!hbp]
  \centering
  \externaltikz{SourcebeamConvergence}{\input{Images/SourcebeamConvergence}}
  \caption[Convergence of tested schemes in the source-beam test]{\(\Lp{1}\)-error against reference solution (new scheme with \(\rktol = 10^{-9}\)) in
    the source-beam test (\(\gridsizeoned = 600, \tf = 2.5\)).
    (a) New scheme for decreasing tolerance parameter \(\rktol\).
    (b) Standard scheme for decreasing time step \(\timestep\).}\label{fig:AppendixConvergenceSourcebeam}
\end{figure}

\begin{figure}[!hbp]
  \centering
  \externaltikz{PointsourceConvergence}{\input{Images/PointsourceConvergence}}
  \caption[Convergence of tested schemes in the point-source test]{\(\Lp{1}\)-error against reference solution (new scheme with \(\rktol = 10^{-9}\)) in
    the point-source test (\(\gridsize = 30^3, \tf = 0.25\)).
    (a) New scheme for decreasing tolerance parameter \(\rktol\).
    (b) Standard scheme for decreasing time step \(\timestep\).}\label{fig:AppendixConvergencePointsource}
\end{figure}

\begin{figure}
  \centering \externaltikz{Timesteptimes3d}{\input{Images/Timesteptimes3d}}
  \caption[Wall times for computing a single time step of transformed and standard schemes in three-dimensional tests]{
    Wall times for computing a single time step.
    (a)~Point-source problem, \(\gridsize = 50^3\), \(\tf = 0.75\).
    (b)~Checkerboard problem, \(\gridsize = 70^3\), \(\tf = 3.2\).
    (c)~Shadow problem, \(\gridsize = 60\times20\times15\), \(\tf = 20\), \(\HFMN[258]\) and \(\PMMN[512]\).
    (d)~Shadow problem, \(\MN[6]\).}\label{fig:Timesteptimes3d}
\end{figure}

\begin{figure}
  \centering
  \externaltikz{AppendixTimestepsSourcebeamM}{\input{Images/AppendixTimestepsSourcebeamM}}
  \caption[Time steps taken in the source-beam test case for different \texorpdfstring{\(\MN\)}{MN} models]{Time steps taken in the source-beam test
    case (\(\gridsizeoned=1200\), \(\tf=2.5\), \(\rktol=10^{-3}\)) for different models. The solid line
  represents the maximum realizability preserving time step \(\timestepnonadaptive\) for the standard splitting scheme.}\label{fig:AdditionalTimestepsSourceBeamM1}
\end{figure}

\begin{figure}
  \centering
  \externaltikz{AppendixTimestepsSourcebeamM2}{\input{Images/AppendixTimestepsSourcebeamM2}}
  \caption[Time steps taken in the source-beam test case for different \texorpdfstring{\(\MN\)}{MN} models (continued)]{Time steps taken in the source-beam
    test case (\(\gridsizeoned=1200\), \(\tf=2.5\), \(\rktol=10^{-3}\)) for different models (continued). The solid line
  represents the maximum realizability preserving time step \(\timestepnonadaptive\) for the standard splitting scheme.}\label{fig:AdditionalTimestepsSourceBeamM2}
\end{figure}

\begin{figure}
  \centering \externaltikz{Timesteps3d}{\input{Images/Timesteps3d}}
  \caption[Time steps taken in the adaptive Runge-Kutta scheme for point-source and checkerboard tests]{Time steps taken in the point-source and checkerboard tests (\(\rktol = 0.01\)).
    The solid and dotted horizontal line represent the time step restrictions~\eqref{eq:splittingschemetimestep} and~\eqref{eq:fvschemetimestep},
    respectively (which almost agree for the point-source test).
    (a)~Point-source problem with \(50^3\) grid cells, \(\tf = 0.75\).
    (b)~Checkerboard problem with \(70^3\) grid cells, \(\tf = 3.2\).}\label{fig:Timesteps3d}
\end{figure}
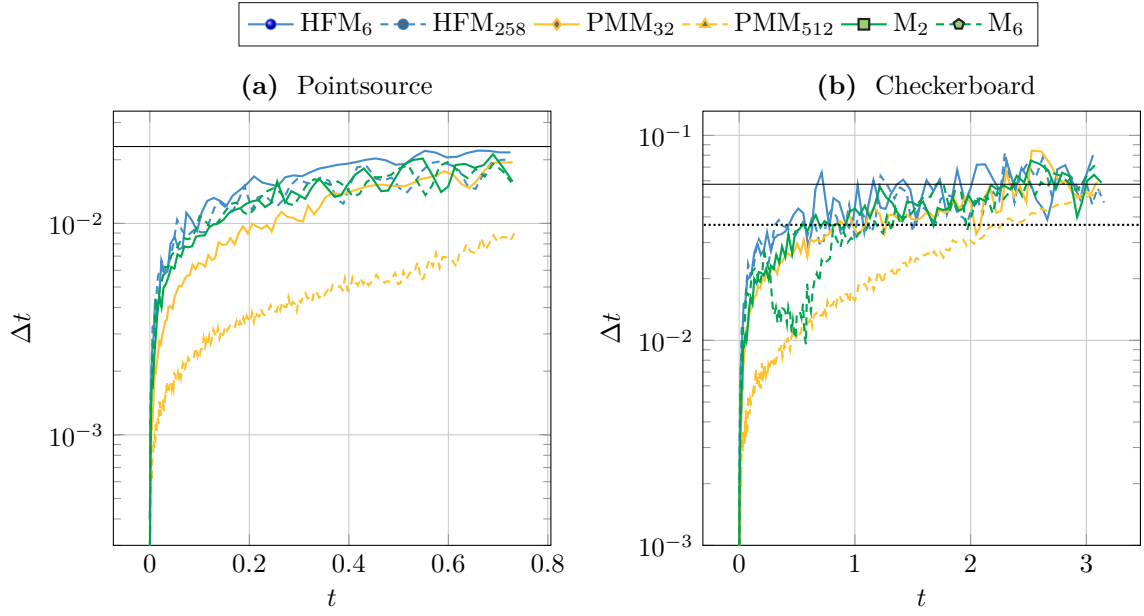

\begin{figure}
  \centering
  \externaltikz{AdditionalTimesteps3dShadowOtherModels}{\input{Images/AdditionalTimesteps3dShadowOtherModels}}
  \caption[Time steps taken in the shadow test case for \texorpdfstring{\(\HFMN\) and \(\PMMN\)}{HFMN and PMMN} models]{Time steps taken in the shadow test case (\(\gridsize=60\times20\times15\), \(\tf=20\), \(\rktol=10^{-2}\)) for the \(\HFMN\) and \(\PMMN\) models.
  The solid and dotted horizontal line represent the time step restrictions~\eqref{eq:splittingschemetimestep} and~\eqref{eq:fvschemetimestep},
  respectively.}\label{fig:AdditionalTimestepsShadowOtherModels}
\end{figure}

\begin{figure}
  \centering
  \externaltikz{AdditionalTimesteps3dShadowM}{\input{Images/AdditionalTimesteps3dShadowM}}
  \caption[Time steps taken in the shadow test case for \texorpdfstring{\(\MN\)}{MN} models]{Time steps taken in the shadow test case (\(\gridsize=60\times20\times15\), \(\tf=20\), \(\rktol=10^{-2}\)) for the \(\MN\) models.
  The solid and dotted horizontal line represent the time step restrictions~\eqref{eq:splittingschemetimestep} and~\eqref{eq:fvschemetimestep},
  respectively.}\label{fig:AdditionalTimestepsShadowM}
\end{figure}

\begin{figure}
  \centering
  \externaltikz{AppendixTimestepsConvergencePlanesource}{\input{Images/AppendixTimestepsConvergencePlanesource}}
  \caption[Time steps taken in the plane-source test case for different tolerances]{Time steps taken in the
    plane-source test case (\(\gridsizeoned=240\), \(\tf = 0.5\)) for different tolerance parameters \(\rktol\).
    The last step has been omitted for all models as
    it was chosen to reach \(\tf\) exactly
    and thus may be artificially small.}\label{fig:AdditionalTimestepsConvergencePlanesource}
\end{figure}

\begin{figure}
  \centering
  \externaltikz{AppendixTimestepsConvergencePointsource}{\input{Images/AppendixTimestepsConvergencePointsource}}
  \caption[Time steps taken in the point-source test case for different tolerances]{Time steps taken in the point-source test case (\(\gridsize=30^3\), \(\tf = 0.25\)) for different tolerance parameters \(\rktol\).
    The last step has been omitted for all models as it was chosen to reach \(\tf\) exactly and thus may be artificially small.}\label{fig:AdditionalTimestepsConvergencePointsource}
\end{figure}

\begin{table}[htbp]
  \centering
  \small
  \begin{tabular}{l c c l c c c r@{.}l r@{.}l}
    Test case    & \(\gridsizeoned\) & \(\tf\) & Model          & Scheme   & \(\timestep\) & \(\rktol\) & \multicolumn{2}{c}{\(\LpError{1}\)} & \multicolumn{2}{c}{\(\LpError{\infty}\)} \\
    \midrule
    Plane-source & 1200              & 1       & \(\HFMN[2]\)   & new      & ---           & 1e-03      & 1 & 47e-04                          & 1 & 63e-04                               \\
    Plane-source & 1200              & 1       & \(\HFMN[2]\)   & standard & 0.001800      & ---        & 4 & 24e-03                          & 9 & 33e-03                               \\
    Plane-source & 1200              & 1       & \(\HFMN[10]\)  & new      & ---           & 1e-03      & 2 & 89e-04                          & 1 & 97e-04                               \\
    Plane-source & 1200              & 1       & \(\HFMN[10]\)  & standard & 0.001800      & ---        & 4 & 27e-03                          & 5 & 69e-03                               \\
    Plane-source & 1200              & 1       & \(\HFMN[50]\)  & new      & ---           & 1e-03      & 2 & 40e-04                          & 1 & 07e-04                               \\
    Plane-source & 1200              & 1       & \(\HFMN[50]\)  & standard & 0.001800      & ---        & 2 & 79e-04                          & 9 & 40e-04                               \\
    Plane-source & 1200              & 1       & \(\HFMN[100]\) & new      & ---           & 1e-03      & 2 & 43e-04                          & 1 & 07e-04                               \\
    Plane-source & 1200              & 1       & \(\HFMN[100]\) & standard & 0.001800      & ---        & 2 & 61e-04                          & 9 & 42e-04                               \\
    Plane-source & 1200              & 1       & \(\PMMN[2]\)   & new      & ---           & 1e-03      & 8 & 80e-05                          & 9 & 68e-05                               \\
    Plane-source & 1200              & 1       & \(\PMMN[2]\)   & standard & 0.001800      & ---        & 4 & 24e-03                          & 9 & 33e-03                               \\
    Plane-source & 1200              & 1       & \(\PMMN[10]\)  & new      & ---           & 1e-03      & 2 & 83e-05                          & 5 & 52e-05                               \\
    Plane-source & 1200              & 1       & \(\PMMN[10]\)  & standard & 0.001800      & ---        & 3 & 94e-03                          & 5 & 53e-03                               \\
    Plane-source & 1200              & 1       & \(\PMMN[50]\)  & new      & ---           & 1e-03      & 1 & 66e-05                          & 1 & 38e-05                               \\
    Plane-source & 1200              & 1       & \(\PMMN[50]\)  & standard & 0.001800      & ---        & 4 & 42e-04                          & 8 & 76e-04                               \\
    Plane-source & 1200              & 1       & \(\PMMN[100]\) & new      & ---           & 1e-03      & 1 & 55e-05                          & 1 & 23e-05                               \\
    Plane-source & 1200              & 1       & \(\PMMN[100]\) & standard & 0.001800      & ---        & 2 & 65e-04                          & 9 & 40e-04                               \\
    Plane-source & 1200              & 1       & \(\MN[1]\)     & new      & ---           & 1e-03      & 8 & 79e-05                          & 9 & 67e-05                               \\
    Plane-source & 1200              & 1       & \(\MN[1]\)     & standard & 0.001800      & ---        & 4 & 24e-03                          & 9 & 33e-03                               \\
    Plane-source & 1200              & 1       & \(\MN[10]\)    & new      & ---           & 1e-03      & 3 & 00e-05                          & 3 & 77e-05                               \\
    Plane-source & 1200              & 1       & \(\MN[10]\)    & standard & 0.001800      & ---        & 4 & 61e-03                          & 4 & 58e-03                               \\
    Plane-source & 1200              & 1       & \(\MN[50]\)    & new      & ---           & 1e-03      & 4 & 61e-05                          & 3 & 63e-05                               \\
    Plane-source & 1200              & 1       & \(\MN[50]\)    & standard & 0.001800      & ---        & 3 & 22e-04                          & 9 & 43e-04                               \\
    Plane-source & 1200              & 1       & \(\MN[100]\)   & new      & ---           & 1e-03      & 5 & 13e-05                          & 3 & 13e-05                               \\
    Plane-source & 1200              & 1       & \(\MN[100]\)   & standard & 0.001800      & ---        & 2 & 66e-04                          & 9 & 43e-04                               \\
    \midrule
    Source-beam  & 1200              & 2.5     & \(\HFMN[2]\)   & new      & ---           & 1e-03      & 1 & 83e-04                          & 5 & 52e-04                               \\
    Source-beam  & 1200              & 2.5     & \(\HFMN[2]\)   & standard & 0.002250      & ---        & 2 & 19e-05                          & 8 & 27e-04                               \\
    Source-beam  & 1200              & 2.5     & \(\HFMN[10]\)  & new      & ---           & 1e-03      & 2 & 57e-05                          & 8 & 08e-05                               \\
    Source-beam  & 1200              & 2.5     & \(\HFMN[10]\)  & standard & 0.002250      & ---        & 2 & 49e-05                          & 8 & 31e-04                               \\
    Source-beam  & 1200              & 2.5     & \(\HFMN[50]\)  & new      & ---           & 1e-03      & 7 & 60e-07                          & 6 & 42e-07                               \\
    Source-beam  & 1200              & 2.5     & \(\HFMN[50]\)  & standard & 0.002250      & ---        & 2 & 23e-05                          & 8 & 35e-04                               \\
    Source-beam  & 1200              & 2.5     & \(\HFMN[100]\) & new      & ---           & 1e-03      & 1 & 14e-06                          & 1 & 01e-06                               \\
    Source-beam  & 1200              & 2.5     & \(\HFMN[100]\) & standard & 0.002250      & ---        & 2 & 23e-05                          & 8 & 35e-04                               \\
    Source-beam  & 1200              & 2.5     & \(\PMMN[2]\)   & new      & ---           & 1e-03      & 2 & 43e-04                          & 2 & 24e-03                               \\
    Source-beam  & 1200              & 2.5     & \(\PMMN[2]\)   & standard & 0.002250      & ---        & 2 & 21e-05                          & 8 & 27e-04                               \\
    Source-beam  & 1200              & 2.5     & \(\PMMN[10]\)  & new      & ---           & 1e-03      & 5 & 60e-05                          & 3 & 17e-04                               \\
    Source-beam  & 1200              & 2.5     & \(\PMMN[10]\)  & standard & 0.002250      & ---        & 2 & 72e-05                          & 8 & 26e-04                               \\
    Source-beam  & 1200              & 2.5     & \(\PMMN[50]\)  & new      & ---           & 1e-03      & 1 & 59e-07                          & 1 & 26e-07                               \\
    Source-beam  & 1200              & 2.5     & \(\PMMN[50]\)  & standard & 0.002250      & ---        & 2 & 25e-05                          & 8 & 26e-04                               \\
    Source-beam  & 1200              & 2.5     & \(\PMMN[100]\) & new      & ---           & 1e-03      & 1 & 74e-07                          & 1 & 51e-07                               \\
    Source-beam  & 1200              & 2.5     & \(\PMMN[100]\) & standard & 0.002250      & ---        & 2 & 24e-05                          & 8 & 26e-04                               \\
    Source-beam  & 1200              & 2.5     & \(\MN[1]\)     & new      & ---           & 1e-03      & 3 & 29e-04                          & 6 & 00e-03                               \\
    Source-beam  & 1200              & 2.5     & \(\MN[1]\)     & standard & 0.002250      & ---        & 2 & 21e-05                          & 8 & 27e-04                               \\
    Source-beam  & 1200              & 2.5     & \(\MN[10]\)    & new      & ---           & 1e-03      & 2 & 21e-07                          & 1 & 54e-06                               \\
    Source-beam  & 1200              & 2.5     & \(\MN[10]\)    & standard & 0.002250      & ---        & 2 & 95e-05                          & 8 & 20e-04                               \\
    Source-beam  & 1200              & 2.5     & \(\MN[50]\)    & new      & ---           & 1e-03      & 5 & 35e-07                          & 4 & 62e-07                               \\
    Source-beam  & 1200              & 2.5     & \(\MN[50]\)    & standard & 0.002250      & ---        & 2 & 28e-05                          & 8 & 26e-04                               \\
    Source-beam  & 1200              & 2.5     & \(\MN[100]\)   & new      & ---           & 1e-03      & 1 & 60e-06                          & 1 & 55e-06                               \\
    Source-beam  & 1200              & 2.5     & \(\MN[100]\)   & standard & 0.002250      & ---        & 2 & 24e-05                          & 8 & 26e-04                               \\
  \end{tabular}
  \caption[Errors for the one-dimensional test cases.]{\(\Lp{1}\)/\(\Lp{\infty}\) errors compared to reference solution (new scheme with \(\rktol = 10^{-6}\))
    for the one-dimensional test cases.}\label{tab:AdditionalErrors1d}
\end{table}

\begin{table}[htbp]
  \centering
  \small
  \resizebox{0.97\textwidth}{!}{\begin{minipage}{\textwidth}
      \begin{tabular}{l c c l c c c r@{.}l r@{.}l}
        Test case    & \(\gridsize\) & \(\tf\) & Model          & Scheme   & \(\timestep\) & \(\rktol\) & \multicolumn{2}{c}{\(\LpError{1}\)} & \multicolumn{2}{c}{\(\LpError{\infty}\)} \\
        \midrule
        Point-source & \(50^3\)      & 0.75    & \(\MN[1]\)     & new      & ---           & 1e-02      & 1 & 35e-04                          & 6 & 56e-05                               \\
        Point-source & \(50^3\)      & 0.75    & \(\MN[1]\)     & standard & 0.020785      & ---        & 1 & 62e-02                          & 1 & 24e-02                               \\
        Point-source & \(50^3\)      & 0.75    & \(\MN[2]\)     & new      & ---           & 1e-02      & 1 & 36e-04                          & 1 & 13e-04                               \\
        Point-source & \(50^3\)      & 0.75    & \(\MN[2]\)     & standard & 0.020785      & ---        & 1 & 56e-02                          & 1 & 45e-02                               \\
        Point-source & \(50^3\)      & 0.75    & \(\MN[3]\)     & new      & ---           & 1e-02      & 1 & 68e-04                          & 8 & 16e-05                               \\
        Point-source & \(50^3\)      & 0.75    & \(\MN[3]\)     & standard & 0.020785      & ---        & 1 & 40e-02                          & 9 & 21e-03                               \\
        Point-source & \(50^3\)      & 0.75    & \(\MN[4]\)     & new      & ---           & 1e-02      & 1 & 74e-04                          & 7 & 35e-05                               \\
        Point-source & \(50^3\)      & 0.75    & \(\MN[4]\)     & standard & 0.020785      & ---        & 1 & 30e-02                          & 7 & 88e-03                               \\
        Point-source & \(50^3\)      & 0.75    & \(\HFMN[6]\)   & new      & ---           & 1e-02      & 2 & 97e-04                          & 5 & 92e-04                               \\
        Point-source & \(50^3\)      & 0.75    & \(\HFMN[6]\)   & standard & 0.020785      & ---        & 1 & 27e-02                          & 2 & 43e-02                               \\
        Point-source & \(50^3\)      & 0.75    & \(\HFMN[18]\)  & new      & ---           & 1e-02      & 2 & 72e-04                          & 2 & 34e-04                               \\
        Point-source & \(50^3\)      & 0.75    & \(\HFMN[18]\)  & standard & 0.020785      & ---        & 1 & 34e-02                          & 1 & 48e-02                               \\
        Point-source & \(50^3\)      & 0.75    & \(\HFMN[66]\)  & new      & ---           & 1e-02      & 2 & 86e-04                          & 1 & 51e-04                               \\
        Point-source & \(50^3\)      & 0.75    & \(\HFMN[66]\)  & standard & 0.020785      & ---        & 1 & 30e-02                          & 8 & 56e-03                               \\
        Point-source & \(50^3\)      & 0.75    & \(\PMMN[32]\)  & new      & ---           & 1e-02      & 4 & 37e-05                          & 2 & 84e-05                               \\
        Point-source & \(50^3\)      & 0.75    & \(\PMMN[32]\)  & standard & 0.020785      & ---        & 1 & 31e-02                          & 8 & 02e-03                               \\
        Point-source & \(50^3\)      & 0.75    & \(\PMMN[128]\) & new      & ---           & 1e-02      & 1 & 24e-05                          & 1 & 14e-05                               \\
        Point-source & \(50^3\)      & 0.75    & \(\PMMN[128]\) & standard & 0.020785      & ---        & 1 & 30e-02                          & 7 & 66e-03                               \\
        \midrule
        Checkerboard & \(70^3\)      & 3.2     & \(\MN[1]\)     & new      & ---           & 1e-02      & 4 & 50e-05                          & 3 & 67e-06                               \\
        Checkerboard & \(70^3\)      & 3.2     & \(\MN[1]\)     & standard & 0.051962      & ---        & 2 & 10e-02                          & 2 & 35e-02                               \\
        Checkerboard & \(70^3\)      & 3.2     & \(\MN[2]\)     & new      & ---           & 1e-02      & 5 & 10e-05                          & 7 & 04e-06                               \\
        Checkerboard & \(70^3\)      & 3.2     & \(\MN[2]\)     & standard & 0.051962      & ---        & 2 & 36e-02                          & 3 & 27e-02                               \\
        Checkerboard & \(70^3\)      & 3.2     & \(\MN[3]\)     & new      & ---           & 1e-02      & 4 & 54e-05                          & 6 & 06e-06                               \\
        Checkerboard & \(70^3\)      & 3.2     & \(\MN[3]\)     & standard & 0.051962      & ---        & 2 & 43e-02                          & 2 & 85e-02                               \\
        Checkerboard & \(70^3\)      & 3.2     & \(\PMMN[32]\)  & new      & ---           & 1e-02      & 4 & 29e-05                          & 5 & 52e-06                               \\
        Checkerboard & \(70^3\)      & 3.2     & \(\PMMN[32]\)  & standard & 0.051962      & ---        & 2 & 41e-02                          & 2 & 93e-02                               \\
        Checkerboard & \(70^3\)      & 3.2     & \(\HFMN[6]\)   & new      & ---           & 1e-02      & 1 & 14e-04                          & 1 & 54e-05                               \\
        Checkerboard & \(70^3\)      & 3.2     & \(\HFMN[6]\)   & standard & 0.051962      & ---        & 2 & 25e-02                          & 2 & 75e-02                               \\
        Checkerboard & \(70^3\)      & 3.2     & \(\HFMN[18]\)  & new      & ---           & 1e-02      & 1 & 19e-04                          & 1 & 48e-05                               \\
        Checkerboard & \(70^3\)      & 3.2     & \(\HFMN[18]\)  & standard & 0.051962      & ---        & 2 & 39e-02                          & 2 & 95e-02                               \\
        \midrule
        Shadow       & \(18000\)     & 20      & \(\MN[1]\)     & new      & ---           & 1e-02      & 1 & 34e-08                          & 3 & 84e-09                               \\
        Shadow       & \(18000\)     & 20      & \(\MN[1]\)     & standard & 0.103923      & ---        & 4 & 32e-02                          & 9 & 14e-03                               \\
        Shadow       & \(18000\)     & 20      & \(\MN[1]\)     & standard & 0.040000      & ---        & 1 & 12e-02                          & 5 & 28e-03                               \\
        Shadow       & \(18000\)     & 20      & \(\MN[2]\)     & new      & ---           & 1e-02      & 5 & 98e-10                          & 1 & 25e-10                               \\
        Shadow       & \(18000\)     & 20      & \(\MN[2]\)     & standard & 0.103923      & ---        & 3 & 40e-02                          & 1 & 06e-02                               \\
        Shadow       & \(18000\)     & 20      & \(\MN[2]\)     & standard & 0.040000      & ---        & 8 & 94e-03                          & 4 & 42e-03                               \\
        Shadow       & \(18000\)     & 20      & \(\PMMN[32]\)  & new      & ---           & 1e-02      & 2 & 07e-10                          & 7 & 18e-11                               \\
        Shadow       & \(18000\)     & 20      & \(\PMMN[32]\)  & standard & 0.103923      & ---        & 3 & 30e-02                          & 9 & 87e-03                               \\
        Shadow       & \(18000\)     & 20      & \(\PMMN[32]\)  & standard & 0.040000      & ---        & 8 & 77e-03                          & 4 & 39e-03                               \\
        Shadow       & \(18000\)     & 20      & \(\PMMN[128]\) & new      & ---           & 1e-02      & 1 & 29e-10                          & 4 & 14e-11                               \\
        Shadow       & \(18000\)     & 20      & \(\PMMN[128]\) & standard & 0.103923      & ---        & 3 & 27e-02                          & 9 & 77e-03                               \\
        Shadow       & \(18000\)     & 20      & \(\PMMN[128]\) & standard & 0.040000      & ---        & 8 & 72e-03                          & 4 & 39e-03                               \\
        Shadow       & \(18000\)     & 20      & \(\HFMN[6]\)   & new      & ---           & 1e-02      & 3 & 42e-08                          & 4 & 82e-09                               \\
        Shadow       & \(18000\)     & 20      & \(\HFMN[6]\)   & standard & 0.103923      & ---        & 2 & 75e-02                          & 7 & 22e-03                               \\
        Shadow       & \(18000\)     & 20      & \(\HFMN[6]\)   & standard & 0.040000      & ---        & 7 & 61e-03                          & 3 & 55e-03                               \\
        Shadow       & \(18000\)     & 20      & \(\HFMN[18]\)  & new      & ---           & 1e-02      & 1 & 05e-09                          & 2 & 73e-10                               \\
        Shadow       & \(18000\)     & 20      & \(\HFMN[18]\)  & standard & 0.103923      & ---        & 3 & 04e-02                          & 9 & 69e-03                               \\
        Shadow       & \(18000\)     & 20      & \(\HFMN[18]\)  & standard & 0.040000      & ---        & 8 & 30e-03                          & 4 & 31e-03                               \\
        Shadow       & \(18000\)     & 20      & \(\HFMN[66]\)  & new      & ---           & 1e-02      & 2 & 61e-10                          & 9 & 21e-11                               \\
        Shadow       & \(18000\)     & 20      & \(\HFMN[66]\)  & standard & 0.103923      & ---        & 3 & 29e-02                          & 9 & 69e-03                               \\
      \end{tabular}
    \end{minipage}}
  \caption[Errors for the three-dimensional test cases.]{\(\Lp{1}\)/\(\Lp{\infty}\) errors compared to reference solution (new scheme with \(\rktol = 10^{-6}\))
    for the three-dimensional test cases.}\label{tab:AdditionalErrors3d}
\end{table}

%% file: Images/SourcebeamConvergence.tex
\begin{tikzpicture}
  \def\pfad{Images/Results/}
  \begin{groupplot}[group style={group size=2 by 1, horizontal sep = 2cm,  vertical sep = 2cm},
      width=\figurewidth,
      height=\figureheight,
      scale only axis,
      cycle list name=color parula pairwise,
      xmode = log,
      ymode = log,
      grid = major,
      title style = {yshift = -0.2cm},
      x dir = reverse,
      ymax = 1e-2,
    ]
    \nextgroupplot[
      title = \tikztitle{New scheme},
      xlabel = {\(\rktol\)},
      ylabel = {\(\LpError{1}\)},
      legend style={at={(1.2,1.2)},xshift=-0cm,yshift=0cm,anchor=center,nodes=right},
      legend columns = 6,
    ]
    \node [text width=5cm,anchor=south]  {{\phantomsubcaption\normalfont\label{fig:AppendixConvergenceSourcebeam_a}}};
    \addplot+ [thick] plot table[x=tau,y=E1] {\pfad sourcebeam_error_m10.txt};
    \addplot+ [thick] plot table[x=tau,y=E1] {\pfad sourcebeam_error_m50.txt};
    \addplot+ [thick] plot table[x=tau,y=E1] {\pfad sourcebeam_error_hfm10.txt};
    \addplot+ [thick] plot table[x=tau,y=E1] {\pfad sourcebeam_error_hfm50.txt};
    \addplot+ [thick] plot table[x=tau,y=E1] {\pfad sourcebeam_error_pmm10.txt};
    \addplot+ [thick] plot table[x=tau,y=E1] {\pfad sourcebeam_error_pmm50.txt};
    \addlegendentry{\(\MN[10]\)};
    \addlegendentry{\(\MN[50]\)};
    \addlegendentry{\(\HFMN[10]\)};
    \addlegendentry{\(\HFMN[50]\)};
    \addlegendentry{\(\PMMN[10]\)};
    \addlegendentry{\(\PMMN[50]\)};

    \nextgroupplot[
      title = \tikztitle{Standard scheme},
      xlabel = {\(\timestep\)},
      ylabel = {\(\LpError{1}\)},
    ]
    \node [text width=5cm,anchor=south]  {{\phantomsubcaption\normalfont\label{fig:AppendixConvergenceSourcebeam_b}}};
    \addplot+ [thick] plot table[x=dt,y=E1] {\pfad sourcebeam_fv_error_m10.txt};
    \addplot+ [thick] plot table[x=dt,y=E1] {\pfad sourcebeam_fv_error_m50.txt};
    \addplot+ [thick] plot table[x=dt,y=E1] {\pfad sourcebeam_fv_error_hfm10.txt};
    \addplot+ [thick] plot table[x=dt,y=E1] {\pfad sourcebeam_fv_error_hfm50.txt};
    \addplot+ [thick] plot table[x=dt,y=E1] {\pfad sourcebeam_fv_error_pmm10.txt};
    \addplot+ [thick] plot table[x=dt,y=E1] {\pfad sourcebeam_fv_error_pmm50.txt};
  \end{groupplot}
\end{tikzpicture}%

%% file: Images/PointsourceConvergence.tex
\begin{tikzpicture}
  \def\pfad{Images/Results/}
  \begin{groupplot}[group style={group size=2 by 1, horizontal sep = 2cm,  vertical sep = 2cm},
      width=\figurewidth,
      height=\figureheight,
      scale only axis,
      cycle list name=color parula pairwise,
      xmode = log,
      ymode = log,
      grid = major,
      title style = {yshift = -0.2cm},
      x dir = reverse,
    ]
    \nextgroupplot[
      title = \tikztitle{New scheme},
      xlabel = {\(\rktol\)},
      ylabel = {\(\LpError{1}\)},
      legend style={at={(1.2,1.2)},xshift=-0cm,yshift=0cm,anchor=center,nodes=right},
      legend columns = 6,
    ]
    \node [text width=5cm,anchor=south]  {{\phantomsubcaption\normalfont\label{fig:AppendixConvergencePointsource_a}}};
    \addplot+ [thick] plot table[x=tau,y=E1] {\pfad pointsource_convergence_m2.txt};
    \addplot+ [thick] plot table[x=tau,y=E1] {\pfad pointsource_convergence_m4.txt};
    \addplot+ [thick] plot table[x=tau,y=E1] {\pfad pointsource_convergence_hfm6.txt};
    \addplot+ [thick] plot table[x=tau,y=E1] {\pfad pointsource_convergence_hfm66.txt};
    \addplot+ [thick] plot table[x=tau,y=E1] {\pfad pointsource_convergence_pmm32.txt};
    \addplot+ [thick] plot table[x=tau,y=E1] {\pfad pointsource_convergence_pmm128.txt};
    \addlegendentry{\(\MN[2]\)};
    \addlegendentry{\(\MN[4]\)};
    \addlegendentry{\(\HFMN[6]\)};
    \addlegendentry{\(\HFMN[66]\)};
    \addlegendentry{\(\PMMN[32]\)};
    \addlegendentry{\(\PMMN[128]\)};

    \nextgroupplot[
      title = \tikztitle{Standard scheme},
      xlabel = {\(\timestep\)},
      ylabel = {\(\LpError{1}\)},
    ]
    \node [text width=5cm,anchor=south]  {{\phantomsubcaption\normalfont\label{fig:AppendixConvergencePointsource_b}}};
    \addplot+ [thick] plot table[x=dt,y=E1] {\pfad pointsource_fv_convergence_m2.txt};
    \addplot+ [thick] plot table[x=dt,y=E1] {\pfad pointsource_fv_convergence_m4.txt};
    \addplot+ [thick] plot table[x=dt,y=E1] {\pfad pointsource_fv_convergence_hfm6.txt};
    \addplot+ [thick] plot table[x=dt,y=E1] {\pfad pointsource_fv_convergence_hfm66.txt};
    \addplot+ [thick] plot table[x=dt,y=E1] {\pfad pointsource_fv_convergence_pmm32.txt};
    \addplot+ [thick] plot table[x=dt,y=E1] {\pfad pointsource_fv_convergence_pmm128.txt};
  \end{groupplot}
\end{tikzpicture}%

%% file: Images/Timesteptimes3d.tex
\begin{tikzpicture}
  \def\pfad{Images/Results/}
  \begin{groupplot}[group style={group size=2 by 2, horizontal sep = 2cm,  vertical sep = 2cm},
      width=\figurewidth,
      height=\figureheight,
      scale only axis,
      cycle list name=color parula pairwise,
      ymode = log,
      grid = major,
      title style = {yshift = -0.2cm},
    ]
    \nextgroupplot[
      title = \tikztitle{Pointsource},
      xlabel = {\(\timevar\)},
      ylabel = {\(\timestep\)},
      legend style={at={(1.2,1.2)},xshift=-0cm,yshift=0cm,anchor=center,nodes=right},
      legend columns = 6,
    ]
    \addplot+ [thick] plot table[x=t,y=steptime,mark=none] {\pfad pointsource_timesteptimes_tol_0.01_hfm258.txt};
    \addplot+ [thick] plot table[x=t,y=steptime,mark=none] {\pfad pointsource_timesteptimes_ref_hfm258.txt};
    \addplot+ [thick] plot table[x=t,y=steptime,mark=none] {\pfad pointsource_timesteptimes_tol_0.01_pmm512.txt};
    \addplot+ [thick] plot table[x=t,y=steptime,mark=none] {\pfad pointsource_timesteptimes_ref_pmm512.txt};
    \addplot+ [thick] plot table[x=t,y=steptime,mark=none] {\pfad pointsource_timesteptimes_tol_0.01_m6.txt};
    \addplot+ [thick] plot table[x=t,y=steptime,mark=none] {\pfad pointsource_timesteptimes_ref_m6.txt};
    \addlegendentry{\(\HFMN[258]\) (\(\multipliers\))};
    \addlegendentry{\(\HFMN[258]\) (\(\moments\))};
    \addlegendentry{\(\PMMN[512]\) (\(\multipliers\))};
    \addlegendentry{\(\PMMN[512]\) (\(\moments\))};
    \addlegendentry{\(\MN[6]\) (\(\multipliers\))};
    \addlegendentry{\(\MN[6]\) (\(\moments\))};
    \node [text width=5cm,anchor=south]  {{\phantomsubcaption\normalfont\label{fig:Timesteptimes3d_a}}};

    \nextgroupplot[
      title = \tikztitle{Checkerboard},
      xlabel = {\(\timevar\)},
      ylabel = {\(\timestep\)},
    ]
    \addplot+ [thick] plot table[x=t,y=steptime,mark=none] {\pfad checkerboard_timesteptimes_tol_0.01_hfm258.txt};
    \addplot+ [thick] plot table[x=t,y=steptime,mark=none] {\pfad checkerboard_timesteptimes_ref_hfm258.txt};
    \addplot+ [thick] plot table[x=t,y=steptime,mark=none] {\pfad checkerboard_timesteptimes_tol_0.01_pmm512.txt};
    \addplot+ [thick] plot table[x=t,y=steptime,mark=none] {\pfad checkerboard_timesteptimes_ref_pmm512.txt};
    \addplot+ [thick] plot table[x=t,y=steptime,mark=none] {\pfad checkerboard_timesteptimes_tol_0.01_m6.txt};
    \addplot+ [thick] plot table[x=t,y=steptime,mark=none] {\pfad checkerboard_timesteptimes_ref_m6.txt};
    \node [text width=5cm,anchor=south]  {{\phantomsubcaption\normalfont\label{fig:Timesteptimes3d_b}}};

    \nextgroupplot[
      title = \tikztitle{Shadow \(\HFMN[258]\) and \(\PMMN[512]\)},
      xlabel = {\(\timevar\)},
      ylabel = {\(\timestep\)},
    ]
    \addplot+ [thick] plot table[x=t,y=steptime,mark=none] {\pfad shadow_timesteptimes_tol_0.01_adjust_0_hfm258.txt};
    \addplot+ [thick] plot table[x=t,y=steptime,mark=none] {\pfad shadow_timesteptimes_ref_hfm258.txt};
    \addplot+ [thick] plot table[x=t,y=steptime,mark=none] {\pfad shadow_timesteptimes_tol_0.01_adjust_0_pmm512.txt};
    \addplot+ [thick] plot table[x=t,y=steptime,mark=none] {\pfad shadow_timesteptimes_ref_pmm512.txt};
    \node [text width=5cm,anchor=south]  {{\phantomsubcaption\normalfont\label{fig:Timesteptimes3d_c}}};

    \nextgroupplot[
      title = \tikztitle{Shadow \(\MN[6]\)},
      xlabel = {\(\timevar\)},
      ylabel = {\(\timestep\)},
    ]
    \pgfplotsset{cycle list shift=4}
    \addplot+ [thick] plot table[x=t,y=steptime,mark=none] {\pfad shadow_timesteptimes_tol_0.01_adjust_0_m6.txt};
    \addplot+ [thick] plot table[x=t,y=steptime,mark=none] {\pfad shadow_timesteptimes_ref_m6.txt};
    \node [text width=5cm,anchor=south]  {{\phantomsubcaption\normalfont\label{fig:Timesteptimes3d_d}}};

  \end{groupplot}
\end{tikzpicture}%

%% file: Images/AppendixTimestepsSourcebeamM.tex
\begin{tikzpicture}
  \def\pfad{Images/Results/}
  \begin{groupplot}[group style={group size=2 by 3, horizontal sep = 2cm,  vertical sep = 1.7cm},
      width=0.9\figurewidth,
      height=0.8\figureheight,
      scale only axis,
      cycle list name=color parula pairwise,
      ymode = log,
      grid = major,
      title style = {yshift = -0.2cm},
      ymax = 5e-3,
      ymin=1e-06,
    ]

    \nextgroupplot[
      title = \tikztitle{\(\MN[10]\)},
      xlabel = {\(\timevar\)},
      ylabel = {\(\timestep\)},
    ]
    \addplot+ [thick, mark=none] plot table[x=t,y=dt] {\pfad sourcebeam_timesteps_tol_0.001_m10.txt};
    \draw (-5,0.0025) -- (5,0.0025);

    \nextgroupplot[
      title = \tikztitle{\(\MN[20]\)},
      xlabel = {\(\timevar\)},
      ylabel = {\(\timestep\)},
    ]
    \addplot+ [thick, mark=none] plot table[x=t,y=dt] {\pfad sourcebeam_timesteps_tol_0.001_m20.txt};
    \draw (-5,0.0025) -- (5,0.0025);

    \nextgroupplot[
      title = \tikztitle{\(\MN[30]\)},
      xlabel = {\(\timevar\)},
      ylabel = {\(\timestep\)},
    ]
    \addplot+ [thick, mark=none] plot table[x=t,y=dt] {\pfad sourcebeam_timesteps_tol_0.001_m30.txt};
    \draw (-5,0.0025) -- (5,0.0025);

    \nextgroupplot[
      title = \tikztitle{\(\MN[40]\)},
      xlabel = {\(\timevar\)},
      ylabel = {\(\timestep\)},
    ]
    \addplot+ [thick, mark=none] plot table[x=t,y=dt] {\pfad sourcebeam_timesteps_tol_0.001_m40.txt};
    \draw (-5,0.0025) -- (5,0.0025);

    \nextgroupplot[
      title = \tikztitle{\(\MN[50]\)},
      xlabel = {\(\timevar\)},
      ylabel = {\(\timestep\)},
    ]
    \addplot+ [thick, mark=none] plot table[x=t,y=dt] {\pfad sourcebeam_timesteps_tol_0.001_m50.txt};
    \draw (-5,0.0025) -- (5,0.0025);

    \nextgroupplot[
      title = \tikztitle{\(\MN[60]\)},
      xlabel = {\(\timevar\)},
      ylabel = {\(\timestep\)},
    ]
    \addplot+ [thick, mark=none] plot table[x=t,y=dt] {\pfad sourcebeam_timesteps_tol_0.001_m60.txt};
    \draw (-5,0.0025) -- (5,0.0025);

  \end{groupplot}
\end{tikzpicture}%

%% file: Images/AppendixTimestepsSourcebeamM2.tex
\begin{tikzpicture}
  \def\pfad{Images/Results/}
  \begin{groupplot}[group style={group size=2 by 2, horizontal sep = 2cm,  vertical sep = 2cm},
      width=\figurewidth,
      height=\figureheight,
      scale only axis,
      cycle list name=color parula pairwise,
      ymode = log,
      grid = major,
      title style = {yshift = -0.2cm},
      ymax = 5e-3,
      ymin=1e-06,
    ]

    \nextgroupplot[
      title = \tikztitle{\(\MN[70]\)},
      xlabel = {\(\timevar\)},
      ylabel = {\(\timestep\)},
    ]
    \addplot+ [thick, mark=none] plot table[x=t,y=dt] {\pfad sourcebeam_timesteps_tol_0.001_m70.txt};
    \draw (-5,0.0025) -- (5,0.0025);

    \nextgroupplot[
      title = \tikztitle{\(\MN[80]\)},
      xlabel = {\(\timevar\)},
      ylabel = {\(\timestep\)},
    ]
    \addplot+ [thick, mark=none] plot table[x=t,y=dt] {\pfad sourcebeam_timesteps_tol_0.001_m80.txt};
    \draw (-5,0.0025) -- (5,0.0025);

    \nextgroupplot[
      title = \tikztitle{\(\HFMN\)},
      xlabel = {\(\timevar\)},
      ylabel = {\(\timestep\)},
      legend pos = south east,
      legend columns = 2,
    ]
    \addplot+ [thick, mark=none] plot table[x=t,y=dt] {\pfad sourcebeam_timesteps_tol_0.001_hfm2.txt};
    \addplot+ [thick, mark=none] plot table[x=t,y=dt] {\pfad sourcebeam_timesteps_tol_0.001_hfm10.txt};
    \addplot+ [thick, mark=none] plot table[x=t,y=dt] {\pfad sourcebeam_timesteps_tol_0.001_hfm20.txt};
    \addplot+ [thick, mark=none] plot table[x=t,y=dt] {\pfad sourcebeam_timesteps_tol_0.001_hfm30.txt};
    \addplot+ [thick, mark=none] plot table[x=t,y=dt] {\pfad sourcebeam_timesteps_tol_0.001_hfm40.txt};
    \addplot+ [thick, mark=none] plot table[x=t,y=dt] {\pfad sourcebeam_timesteps_tol_0.001_hfm50.txt};
    \addplot+ [thick, mark=none] plot table[x=t,y=dt] {\pfad sourcebeam_timesteps_tol_0.001_hfm60.txt};
    \addplot+ [thick, mark=none] plot table[x=t,y=dt] {\pfad sourcebeam_timesteps_tol_0.001_hfm70.txt};
    \addplot+ [thick, mark=none] plot table[x=t,y=dt] {\pfad sourcebeam_timesteps_tol_0.001_hfm80.txt};
    \addplot+ [thick, mark=none] plot table[x=t,y=dt] {\pfad sourcebeam_timesteps_tol_0.001_hfm90.txt};
    \addplot+ [thick, mark=none] plot table[x=t,y=dt] {\pfad sourcebeam_timesteps_tol_0.001_hfm100.txt};
    \draw (-5,0.0025) -- (5,0.0025);
    \addlegendentry{\(\HFMN[2]\)}
    \addlegendentry{\(\HFMN[10]\)}
    \addlegendentry{\(\HFMN[20]\)}
    \addlegendentry{\(\HFMN[30]\)}
    \addlegendentry{\(\HFMN[40]\)}
    \addlegendentry{\(\HFMN[50]\)}
    \addlegendentry{\(\HFMN[60]\)}
    \addlegendentry{\(\HFMN[70]\)}
    \addlegendentry{\(\HFMN[80]\)}
    \addlegendentry{\(\HFMN[90]\)}
    \addlegendentry{\(\HFMN[100]\)}

    \nextgroupplot[
      title = \tikztitle{\(\PMMN\)},
      xlabel = {\(\timevar\)},
      ylabel = {\(\timestep\)},
      legend pos = south east,
      legend columns = 2,
    ]
    \addplot+ [thick, mark=none] plot table[x=t,y=dt] {\pfad sourcebeam_timesteps_tol_0.001_pmm2.txt};
    \addplot+ [thick, mark=none] plot table[x=t,y=dt] {\pfad sourcebeam_timesteps_tol_0.001_pmm10.txt};
    \addplot+ [thick, mark=none] plot table[x=t,y=dt] {\pfad sourcebeam_timesteps_tol_0.001_pmm20.txt};
    \addplot+ [thick, mark=none] plot table[x=t,y=dt] {\pfad sourcebeam_timesteps_tol_0.001_pmm30.txt};
    \addplot+ [thick, mark=none] plot table[x=t,y=dt] {\pfad sourcebeam_timesteps_tol_0.001_pmm40.txt};
    \addplot+ [thick, mark=none] plot table[x=t,y=dt] {\pfad sourcebeam_timesteps_tol_0.001_pmm50.txt};
    \addplot+ [thick, mark=none] plot table[x=t,y=dt] {\pfad sourcebeam_timesteps_tol_0.001_pmm60.txt};
    \addplot+ [thick, mark=none] plot table[x=t,y=dt] {\pfad sourcebeam_timesteps_tol_0.001_pmm70.txt};
    \addplot+ [thick, mark=none] plot table[x=t,y=dt] {\pfad sourcebeam_timesteps_tol_0.001_pmm80.txt};
    \addplot+ [thick, mark=none] plot table[x=t,y=dt] {\pfad sourcebeam_timesteps_tol_0.001_pmm90.txt};
    \addplot+ [thick, mark=none] plot table[x=t,y=dt] {\pfad sourcebeam_timesteps_tol_0.001_pmm100.txt};
    \draw (-5,0.0025) -- (5,0.0025);
    \addlegendentry{\(\PMMN[2]\)}
    \addlegendentry{\(\PMMN[10]\)}
    \addlegendentry{\(\PMMN[20]\)}
    \addlegendentry{\(\PMMN[30]\)}
    \addlegendentry{\(\PMMN[40]\)}
    \addlegendentry{\(\PMMN[50]\)}
    \addlegendentry{\(\PMMN[60]\)}
    \addlegendentry{\(\PMMN[70]\)}
    \addlegendentry{\(\PMMN[80]\)}
    \addlegendentry{\(\PMMN[90]\)}
    \addlegendentry{\(\PMMN[100]\)}

  \end{groupplot}
\end{tikzpicture}%

%% file: Images/Timesteps3d.tex
\begin{tikzpicture}
  \def\pfad{Images/Results/}
  \begin{groupplot}[group style={group size=2 by 3, horizontal sep = 2cm,  vertical sep = 2cm},
      width=\figurewidth,
      height=\figureheight,
      scale only axis,
      cycle list name=color parula pairwise,
      ymode = log,
      grid = major,
      title style = {yshift = -0.2cm},
    ]
    \nextgroupplot[
      title = \tikztitle{Pointsource},
      xlabel = {\(\timevar\)},
      ylabel = {\(\timestep\)},
      legend style={at={(1.2,1.2)},xshift=-0cm,yshift=0cm,anchor=center,nodes=right},
      legend columns = 6,
      ymin = 3e-4,
    ]
    \addplot+ [thick] plot table[x=t,y=dt,mark=none] {\pfad pointsource_timesteps_tol_0.01_hfm6.txt};
    \addplot+ [thick] plot table[x=t,y=dt,mark=none] {\pfad pointsource_timesteps_tol_0.01_hfm258.txt};
    \addplot+ [thick] plot table[x=t,y=dt,mark=none] {\pfad pointsource_timesteps_tol_0.01_pmm32.txt};
    \addplot+ [thick] plot table[x=t,y=dt,mark=none] {\pfad pointsource_timesteps_tol_0.01_pmm512.txt};
    \addplot+ [thick] plot table[x=t,y=dt,mark=none] {\pfad pointsource_timesteps_tol_0.01_m2.txt};
    \addplot+ [thick] plot table[x=t,y=dt,mark=none] {\pfad pointsource_timesteps_tol_0.01_m6.txt};
    \draw (-1,0.02309401076) -- (2.,0.02309401076);
    \addlegendentry{\(\HFMN[6]\)};
    \addlegendentry{\(\HFMN[258]\)};
    \addlegendentry{\(\PMMN[32]\)};
    \addlegendentry{\(\PMMN[512]\)};
    \addlegendentry{\(\MN[2]\)};
    \addlegendentry{\(\MN[6]\)};
    \node [text width=5cm,anchor=south]  {{\phantomsubcaption\normalfont\label{fig:Timesteps3d_a}}};

    \nextgroupplot[
      title = \tikztitle{Checkerboard},
      xlabel = {\(\timevar\)},
      ylabel = {\(\timestep\)},
      ymin = 1e-3,
    ]
    \addplot+ [thick] plot table[x=t,y=dt,mark=none] {\pfad checkerboard_timesteps_tol_0.01_hfm6.txt};
    \addplot+ [thick] plot table[x=t,y=dt,mark=none] {\pfad checkerboard_timesteps_tol_0.01_hfm258.txt};
    \addplot+ [thick] plot table[x=t,y=dt,mark=none] {\pfad checkerboard_timesteps_tol_0.01_pmm32.txt};
    \addplot+ [thick] plot table[x=t,y=dt,mark=none] {\pfad checkerboard_timesteps_tol_0.01_pmm512.txt};
    \addplot+ [thick] plot table[x=t,y=dt,mark=none] {\pfad checkerboard_timesteps_tol_0.01_m2.txt};
    \addplot+ [thick] plot table[x=t,y=dt,mark=none] {\pfad checkerboard_timesteps_tol_0.01_m6.txt};
    \draw (-5,0.05773502691) -- (12,0.05773502691);
    \draw[thick,densely dotted] (-5,0.03660254037) -- (12,0.03660254037);
    \node [text width=5cm,anchor=south]  {{\phantomsubcaption\normalfont\label{fig:Timesteps3d_b}}};
  \end{groupplot}
\end{tikzpicture}%

%% file: Images/AdditionalTimesteps3dShadowOtherModels.tex
\begin{tikzpicture}
  \def\pfad{Images/Results/}
  \begin{groupplot}[group style={group size=2 by 1, horizontal sep = 2cm,  vertical sep = 2cm},
      width=\figurewidth,
      height=\figureheight,
      scale only axis,
      cycle list name=color parula,
      ymode = log,
      grid = major,
      title style = {yshift = -0.2cm},
      ymin = 1e-5,
      ymax = 0.2,
    ]

    \nextgroupplot[
      title = \tikztitle{\(\HFMN\)},
      xlabel = {\(\timevar\)},
      ylabel = {\(\timestep\)},
      legend pos = south east,
    ]
    \addplot+ [thick] plot table[x=t,y=dt,mark=none] {\pfad shadow_timesteps_tol_0.01_adjust_0_hfm18.txt};
    \addplot+ [thick] plot table[x=t,y=dt,mark=none] {\pfad shadow_timesteps_tol_0.01_adjust_0_hfm66.txt};
    \addplot+ [thick] plot table[x=t,y=dt,mark=none] {\pfad shadow_timesteps_tol_0.01_adjust_0_hfm258.txt};
    \addlegendentry{\(\HFMN[18]\)}
    \addlegendentry{\(\HFMN[66]\)}
    \addlegendentry{\(\HFMN[258]\)}
    \draw (-5,0.11547005383) -- (25,0.11547005383);
    \draw[thick,densely dotted] (-5,0.01704731792) -- (25,0.01704731792);

    \nextgroupplot[
      title = \tikztitle{\(\PMMN\)},
      xlabel = {\(\timevar\)},
      ylabel = {\(\timestep\)},
      legend pos = south east,
    ]
    \addplot+ [thick] plot table[x=t,y=dt,mark=none] {\pfad shadow_timesteps_tol_0.01_adjust_0_pmm32.txt};
    \addplot+ [thick] plot table[x=t,y=dt,mark=none] {\pfad shadow_timesteps_tol_0.01_adjust_0_pmm128.txt};
    \addplot+ [thick] plot table[x=t,y=dt,mark=none] {\pfad shadow_timesteps_tol_0.01_adjust_0_pmm512.txt};
    \addlegendentry{\(\PMMN[32]\)}
    \addlegendentry{\(\PMMN[128]\)}
    \addlegendentry{\(\PMMN[512]\)}
    \draw (-5,0.11547005383) -- (25,0.11547005383);
    \draw[thick,densely dotted] (-5,0.01704731792) -- (25,0.01704731792);

  \end{groupplot}
\end{tikzpicture}%

%% file: Images/AdditionalTimesteps3dShadowM.tex
\begin{tikzpicture}
  \def\pfad{Images/Results/}
  \begin{groupplot}[group style={group size=2 by 3, horizontal sep = 2cm,  vertical sep = 1.5cm},
      width=\figurewidth,
      height=0.94\figureheight,
      scale only axis,
      cycle list name=color parula pairwise,
      ymode = log,
      grid = major,
      title style = {yshift = -0.2cm},
      legend style={at={(1.2,1.2)},xshift=-0cm,yshift=0cm,anchor=center,nodes=right},
      legend columns = 6,
      ymin = 1e-5,
      ymax = 0.2,
    ]

    \nextgroupplot[
      title = \tikztitle{\(\MN[1]\)},
      xlabel = {\(\timevar\)},
      ylabel = {\(\timestep\)},
    ]
    \pgfplotsset{cycle list shift=5}
    \addplot+ [thick] plot table[x=t,y=dt,mark=none] {\pfad shadow_timesteps_tol_0.01_adjust_0_m1.txt};
    \draw (-5,0.11547005383) -- (25,0.11547005383);
    \draw[thick,densely dotted] (-5,0.01704731792) -- (25,0.01704731792);

    \nextgroupplot[
      title = \tikztitle{\(\MN[2]\)},
      xlabel = {\(\timevar\)},
      ylabel = {\(\timestep\)},
    ]
    \pgfplotsset{cycle list shift=5}
    \addplot+ [thick] plot table[x=t,y=dt,mark=none] {\pfad shadow_timesteps_tol_0.01_adjust_0_m2.txt};
    \draw (-5,0.11547005383) -- (25,0.11547005383);
    \draw[thick,densely dotted] (-5,0.01704731792) -- (25,0.01704731792);

    \nextgroupplot[
      title = \tikztitle{\(\MN[3]\)},
      xlabel = {\(\timevar\)},
      ylabel = {\(\timestep\)},
    ]
    \pgfplotsset{cycle list shift=5}
    \addplot+ [thick] plot table[x=t,y=dt,mark=none] {\pfad shadow_timesteps_tol_0.01_adjust_0_m3.txt};
    \draw (-5,0.11547005383) -- (25,0.11547005383);
    \draw[thick,densely dotted] (-5,0.01704731792) -- (25,0.01704731792);

    \nextgroupplot[
      title = \tikztitle{\(\MN[4]\)},
      xlabel = {\(\timevar\)},
      ylabel = {\(\timestep\)},
    ]
    \pgfplotsset{cycle list shift=5}
    \addplot+ [thick] plot table[x=t,y=dt,mark=none] {\pfad shadow_timesteps_tol_0.01_adjust_0_m4.txt};
    \draw (-5,0.11547005383) -- (25,0.11547005383);
    \draw[thick,densely dotted] (-5,0.01704731792) -- (25,0.01704731792);

    \nextgroupplot[
      title = \tikztitle{\(\MN[5]\)},
      xlabel = {\(\timevar\)},
      ylabel = {\(\timestep\)},
    ]
    \pgfplotsset{cycle list shift=5}
    \addplot+ [thick] plot table[x=t,y=dt,mark=none] {\pfad shadow_timesteps_tol_0.01_adjust_0_m5.txt};
    \draw (-5,0.11547005383) -- (25,0.11547005383);
    \draw[thick,densely dotted] (-5,0.01704731792) -- (25,0.01704731792);

    \nextgroupplot[
      title = \tikztitle{\(\MN[6]\)},
      xlabel = {\(\timevar\)},
      ylabel = {\(\timestep\)},
    ]
    \pgfplotsset{cycle list shift=5}
    \addplot+ [thick] plot table[x=t,y=dt,mark=none] {\pfad shadow_timesteps_tol_0.01_adjust_0_m6.txt};
    \draw (-5,0.11547005383) -- (25,0.11547005383);
    \draw[thick,densely dotted] (-5,0.01704731792) -- (25,0.01704731792);

  \end{groupplot}
\end{tikzpicture}%

%% file: Images/AppendixTimestepsConvergencePlanesource.tex
\begin{tikzpicture}
  \def\pfad{Images/Results/}
  \begin{groupplot}[group style={group size=2 by 3, horizontal sep = 2cm,  vertical sep = 1.6cm},
      width=\figurewidth,
      height=0.88\figureheight,
      scale only axis,
      cycle list name=color parula,
      ymode = log,
      grid = major,
      title style = {yshift = -0.15cm},
    ]
    \nextgroupplot[
      title = \tikztitle{\(\MN[10]\)},
      xlabel = {\(\rktol\)},
      ylabel = {\(\timestep\)},
      legend style={at={(1.2,1.2)},xshift=-0cm,yshift=0.15cm,anchor=center,nodes=right},
      legend columns = 6,
      ymin = 0.0001,
    ]
    \addplot+ [thick,mark=none] plot table[x=t,y=dt] {\pfad planesource_timesteps_convergence_tol_0.1_m10.txt};
    \addplot+ [thick,mark=none] plot table[x=t,y=dt] {\pfad planesource_timesteps_convergence_tol_0.01_m10.txt};
    \addplot+ [thick,mark=none] plot table[x=t,y=dt] {\pfad planesource_timesteps_convergence_tol_0.001_m10.txt};
    \addplot+ [thick,mark=none] plot table[x=t,y=dt] {\pfad planesource_timesteps_convergence_tol_0.0001_m10.txt};
    \addplot+ [thick,mark=none] plot table[x=t,y=dt] {\pfad planesource_timesteps_convergence_tol_1e-05_m10.txt};
    \addplot+ [thick,mark=none] plot table[x=t,y=dt] {\pfad planesource_timesteps_convergence_tol_1e-06_m10.txt};
    \addlegendentry{\(\rktol=10^{-1}\)};
    \addlegendentry{\(\rktol=10^{-2}\)};
    \addlegendentry{\(\rktol=10^{-3}\)};
    \addlegendentry{\(\rktol=10^{-4}\)};
    \addlegendentry{\(\rktol=10^{-5}\)};
    \addlegendentry{\(\rktol=10^{-6}\)};

    \nextgroupplot[
      title = \tikztitle{\(\MN[50]\)},
      xlabel = {\(\rktol\)},
      ylabel = {\(\timestep\)},
      legend columns = 6,
      ymin = 0.0001,
    ]
    \addplot+ [thick,mark=none] plot table[x=t,y=dt] {\pfad planesource_timesteps_convergence_tol_0.1_m50.txt};
    \addplot+ [thick,mark=none] plot table[x=t,y=dt] {\pfad planesource_timesteps_convergence_tol_0.01_m50.txt};
    \addplot+ [thick,mark=none] plot table[x=t,y=dt] {\pfad planesource_timesteps_convergence_tol_0.001_m50.txt};
    \addplot+ [thick,mark=none] plot table[x=t,y=dt] {\pfad planesource_timesteps_convergence_tol_0.0001_m50.txt};
    \addplot+ [thick,mark=none] plot table[x=t,y=dt] {\pfad planesource_timesteps_convergence_tol_1e-05_m50.txt};
    \addplot+ [thick,mark=none] plot table[x=t,y=dt] {\pfad planesource_timesteps_convergence_tol_1e-06_m50.txt};

    \nextgroupplot[
      title = \tikztitle{\(\HFMN[10]\)},
      xlabel = {\(\rktol\)},
      ylabel = {\(\timestep\)},
      legend columns = 6,
      ymin = 0.0001,
    ]
    \addplot+ [thick,mark=none] plot table[x=t,y=dt] {\pfad planesource_timesteps_convergence_tol_0.1_hfm10.txt};
    \addplot+ [thick,mark=none] plot table[x=t,y=dt] {\pfad planesource_timesteps_convergence_tol_0.01_hfm10.txt};
    \addplot+ [thick,mark=none] plot table[x=t,y=dt] {\pfad planesource_timesteps_convergence_tol_0.001_hfm10.txt};
    \addplot+ [thick,mark=none] plot table[x=t,y=dt] {\pfad planesource_timesteps_convergence_tol_0.0001_hfm10.txt};
    \addplot+ [thick,mark=none] plot table[x=t,y=dt] {\pfad planesource_timesteps_convergence_tol_1e-05_hfm10.txt};
    \addplot+ [thick,mark=none] plot table[x=t,y=dt] {\pfad planesource_timesteps_convergence_tol_1e-06_hfm10.txt};

    \nextgroupplot[
      title = \tikztitle{\(\HFMN[50]\)},
      xlabel = {\(\rktol\)},
      ylabel = {\(\timestep\)},
      legend columns = 6,
      ymin = 0.0001,
    ]
    \addplot+ [thick,mark=none] plot table[x=t,y=dt] {\pfad planesource_timesteps_convergence_tol_0.1_hfm50.txt};
    \addplot+ [thick,mark=none] plot table[x=t,y=dt] {\pfad planesource_timesteps_convergence_tol_0.01_hfm50.txt};
    \addplot+ [thick,mark=none] plot table[x=t,y=dt] {\pfad planesource_timesteps_convergence_tol_0.001_hfm50.txt};
    \addplot+ [thick,mark=none] plot table[x=t,y=dt] {\pfad planesource_timesteps_convergence_tol_0.0001_hfm50.txt};
    \addplot+ [thick,mark=none] plot table[x=t,y=dt] {\pfad planesource_timesteps_convergence_tol_1e-05_hfm50.txt};
    \addplot+ [thick,mark=none] plot table[x=t,y=dt] {\pfad planesource_timesteps_convergence_tol_1e-06_hfm50.txt};

    \nextgroupplot[
      title = \tikztitle{\(\PMMN[10]\)},
      xlabel = {\(\rktol\)},
      ylabel = {\(\timestep\)},
      legend columns = 6,
      ymin = 0.0001,
    ]
    \addplot+ [thick,mark=none] plot table[x=t,y=dt] {\pfad planesource_timesteps_convergence_tol_0.1_pmm10.txt};
    \addplot+ [thick,mark=none] plot table[x=t,y=dt] {\pfad planesource_timesteps_convergence_tol_0.01_pmm10.txt};
    \addplot+ [thick,mark=none] plot table[x=t,y=dt] {\pfad planesource_timesteps_convergence_tol_0.001_pmm10.txt};
    \addplot+ [thick,mark=none] plot table[x=t,y=dt] {\pfad planesource_timesteps_convergence_tol_0.0001_pmm10.txt};
    \addplot+ [thick,mark=none] plot table[x=t,y=dt] {\pfad planesource_timesteps_convergence_tol_1e-05_pmm10.txt};
    \addplot+ [thick,mark=none] plot table[x=t,y=dt] {\pfad planesource_timesteps_convergence_tol_1e-06_pmm10.txt};

    \nextgroupplot[
      title = \tikztitle{\(\PMMN[50]\)},
      xlabel = {\(\rktol\)},
      ylabel = {\(\timestep\)},
      legend columns = 6,
      ymin = 0.0001,
    ]
    \addplot+ [thick,mark=none] plot table[x=t,y=dt] {\pfad planesource_timesteps_convergence_tol_0.1_pmm50.txt};
    \addplot+ [thick,mark=none] plot table[x=t,y=dt] {\pfad planesource_timesteps_convergence_tol_0.01_pmm50.txt};
    \addplot+ [thick,mark=none] plot table[x=t,y=dt] {\pfad planesource_timesteps_convergence_tol_0.001_pmm50.txt};
    \addplot+ [thick,mark=none] plot table[x=t,y=dt] {\pfad planesource_timesteps_convergence_tol_0.0001_pmm50.txt};
    \addplot+ [thick,mark=none] plot table[x=t,y=dt] {\pfad planesource_timesteps_convergence_tol_1e-05_pmm50.txt};
    \addplot+ [thick,mark=none] plot table[x=t,y=dt] {\pfad planesource_timesteps_convergence_tol_1e-06_pmm50.txt};

  \end{groupplot}
\end{tikzpicture}%

%% file: Images/AppendixTimestepsConvergencePointsource.tex
\begin{tikzpicture}
  \def\pfad{Images/Results/}
  \begin{groupplot}[group style={group size=2 by 3, horizontal sep = 2cm,  vertical sep = 1.6cm},
      width=\figurewidth,
      height=0.88\figureheight,
      scale only axis,
      cycle list name=color parula,
      ymode = log,
      grid = major,
      title style = {yshift = -0.15cm},
    ]
    \nextgroupplot[
      title = \tikztitle{\(\MN[2]\)},
      xlabel = {\(\rktol\)},
      ylabel = {\(\timestep\)},
      legend style={at={(1.2,1.2)},xshift=-0cm,yshift=0.15cm,anchor=center,nodes=right},
      legend columns = 6,
      ymin = 0.0001,
    ]
    \addplot+ [thick,mark=none] plot table[x=t,y=dt] {\pfad pointsource_timesteps_convergence_tol_0.1_m2.txt};
    \addplot+ [thick,mark=none] plot table[x=t,y=dt] {\pfad pointsource_timesteps_convergence_tol_0.01_m2.txt};
    \addplot+ [thick,mark=none] plot table[x=t,y=dt] {\pfad pointsource_timesteps_convergence_tol_0.001_m2.txt};
    \addplot+ [thick,mark=none] plot table[x=t,y=dt] {\pfad pointsource_timesteps_convergence_tol_0.0001_m2.txt};
    \addplot+ [thick,mark=none] plot table[x=t,y=dt] {\pfad pointsource_timesteps_convergence_tol_1e-05_m2.txt};
    \addplot+ [thick,mark=none] plot table[x=t,y=dt] {\pfad pointsource_timesteps_convergence_tol_1e-06_m2.txt};
    \addlegendentry{\(\rktol=10^{-1}\)};
    \addlegendentry{\(\rktol=10^{-2}\)};
    \addlegendentry{\(\rktol=10^{-3}\)};
    \addlegendentry{\(\rktol=10^{-4}\)};
    \addlegendentry{\(\rktol=10^{-5}\)};
    \addlegendentry{\(\rktol=10^{-6}\)};

    \nextgroupplot[
      title = \tikztitle{\(\MN[4]\)},
      xlabel = {\(\rktol\)},
      ylabel = {\(\timestep\)},
      legend columns = 6,
      ymin = 0.0001,
    ]
    \addplot+ [thick,mark=none] plot table[x=t,y=dt] {\pfad pointsource_timesteps_convergence_tol_0.1_m4.txt};
    \addplot+ [thick,mark=none] plot table[x=t,y=dt] {\pfad pointsource_timesteps_convergence_tol_0.01_m4.txt};
    \addplot+ [thick,mark=none] plot table[x=t,y=dt] {\pfad pointsource_timesteps_convergence_tol_0.001_m4.txt};
    \addplot+ [thick,mark=none] plot table[x=t,y=dt] {\pfad pointsource_timesteps_convergence_tol_0.0001_m4.txt};
    \addplot+ [thick,mark=none] plot table[x=t,y=dt] {\pfad pointsource_timesteps_convergence_tol_1e-05_m4.txt};
    \addplot+ [thick,mark=none] plot table[x=t,y=dt] {\pfad pointsource_timesteps_convergence_tol_1e-06_m4.txt};

    \nextgroupplot[
      title = \tikztitle{\(\HFMN[6]\)},
      xlabel = {\(\rktol\)},
      ylabel = {\(\timestep\)},
      legend columns = 6,
      ymin = 0.0001,
    ]
    \addplot+ [thick,mark=none] plot table[x=t,y=dt] {\pfad pointsource_timesteps_convergence_tol_0.1_hfm6.txt};
    \addplot+ [thick,mark=none] plot table[x=t,y=dt] {\pfad pointsource_timesteps_convergence_tol_0.01_hfm6.txt};
    \addplot+ [thick,mark=none] plot table[x=t,y=dt] {\pfad pointsource_timesteps_convergence_tol_0.001_hfm6.txt};
    \addplot+ [thick,mark=none] plot table[x=t,y=dt] {\pfad pointsource_timesteps_convergence_tol_0.0001_hfm6.txt};
    \addplot+ [thick,mark=none] plot table[x=t,y=dt] {\pfad pointsource_timesteps_convergence_tol_1e-05_hfm6.txt};
    \addplot+ [thick,mark=none] plot table[x=t,y=dt] {\pfad pointsource_timesteps_convergence_tol_1e-06_hfm6.txt};

    \nextgroupplot[
      title = \tikztitle{\(\HFMN[66]\)},
      xlabel = {\(\rktol\)},
      ylabel = {\(\timestep\)},
      legend columns = 6,
      ymin = 0.0001,
    ]
    \addplot+ [thick,mark=none] plot table[x=t,y=dt] {\pfad pointsource_timesteps_convergence_tol_0.1_hfm66.txt};
    \addplot+ [thick,mark=none] plot table[x=t,y=dt] {\pfad pointsource_timesteps_convergence_tol_0.01_hfm66.txt};
    \addplot+ [thick,mark=none] plot table[x=t,y=dt] {\pfad pointsource_timesteps_convergence_tol_0.001_hfm66.txt};
    \addplot+ [thick,mark=none] plot table[x=t,y=dt] {\pfad pointsource_timesteps_convergence_tol_0.0001_hfm66.txt};
    \addplot+ [thick,mark=none] plot table[x=t,y=dt] {\pfad pointsource_timesteps_convergence_tol_1e-05_hfm66.txt};
    \addplot+ [thick,mark=none] plot table[x=t,y=dt] {\pfad pointsource_timesteps_convergence_tol_1e-06_hfm66.txt};

    \nextgroupplot[
      title = \tikztitle{\(\PMMN[32]\)},
      xlabel = {\(\rktol\)},
      ylabel = {\(\timestep\)},
      legend columns = 6,
      ymin = 0.0001,
    ]
    \addplot+ [thick,mark=none] plot table[x=t,y=dt] {\pfad pointsource_timesteps_convergence_tol_0.1_pmm32.txt};
    \addplot+ [thick,mark=none] plot table[x=t,y=dt] {\pfad pointsource_timesteps_convergence_tol_0.01_pmm32.txt};
    \addplot+ [thick,mark=none] plot table[x=t,y=dt] {\pfad pointsource_timesteps_convergence_tol_0.001_pmm32.txt};
    \addplot+ [thick,mark=none] plot table[x=t,y=dt] {\pfad pointsource_timesteps_convergence_tol_0.0001_pmm32.txt};
    \addplot+ [thick,mark=none] plot table[x=t,y=dt] {\pfad pointsource_timesteps_convergence_tol_1e-05_pmm32.txt};
    \addplot+ [thick,mark=none] plot table[x=t,y=dt] {\pfad pointsource_timesteps_convergence_tol_1e-06_pmm32.txt};

    \nextgroupplot[
      title = \tikztitle{\(\PMMN[128]\)},
      xlabel = {\(\rktol\)},
      ylabel = {\(\timestep\)},
      legend columns = 6,
      ymin = 0.0001,
    ]
    \addplot+ [thick,mark=none] plot table[x=t,y=dt] {\pfad pointsource_timesteps_convergence_tol_0.1_pmm128.txt};
    \addplot+ [thick,mark=none] plot table[x=t,y=dt] {\pfad pointsource_timesteps_convergence_tol_0.01_pmm128.txt};
    \addplot+ [thick,mark=none] plot table[x=t,y=dt] {\pfad pointsource_timesteps_convergence_tol_0.001_pmm128.txt};
    \addplot+ [thick,mark=none] plot table[x=t,y=dt] {\pfad pointsource_timesteps_convergence_tol_0.0001_pmm128.txt};
    \addplot+ [thick,mark=none] plot table[x=t,y=dt] {\pfad pointsource_timesteps_convergence_tol_1e-05_pmm128.txt};
    \addplot+ [thick,mark=none] plot table[x=t,y=dt] {\pfad pointsource_timesteps_convergence_tol_1e-06_pmm128.txt};

  \end{groupplot}
\end{tikzpicture}%